\documentclass[10pt,a4paper]{article}
\usepackage{amsmath,amsfonts,amssymb,amsthm}
\usepackage{paralist}

\newenvironment{mabstract}
{\begin{quote}\small {\bfseries Abstract.}}{\end{quote}\par}

\newenvironment{mkeywords}
{\begin{quote}\small {\bfseries Keywords.}}{\end{quote}\par}

\newenvironment{msubjclass}
{\begin{quote}\small {\bfseries AMS Subject Classification.}}{\end{quote}\par}

\newtheorem{theorem}{Theorem}[section]

\newcounter{appendixcounter}

\newtheorem{theoremappendix}{Theorem}[appendixcounter]

\newtheorem{lemma}{Lemma}[section]

\newtheorem{assumption}{Assumption}[section]

\newtheorem{assumptionabc}{Assumption}[assumption]

\newtheorem{assumptionappendix}{Assumption}[appendixcounter]

\newtheorem{proposition}{Proposition}[section]

\newtheorem{remark}{Remark}[section]

\newtheorem*{vremark}{Remark}

\newtheorem*{rproof}{\bfseries Proof}
\newenvironment{sproof}
{\begin{rproof}\rm}{\qed \end{rproof}}

\newtheorem{bproof}{}

\newtheorem{bcase}{}[section]

\makeatletter
\renewcommand\@seccntformat[1]{{\csname thesection\endcsname}.\hspace{0.5em}}
\makeatother

\begin{document}

\title{Asymptotic Bias of Stochastic Gradient Search} 

\author{Vladislav B. Tadi\'{c}
\thanks{Department of Mathematics, University of Bristol,
Bristol, United Kingdom 
(email: v.b.tadic@bristol.ac.uk). }
\and 
Arnaud Doucet\thanks{Department of Statistics, 
University of Oxford, Oxford, United Kingdom
(doucet@stats.ox.ac.uk). } }
\date{}

\maketitle

\begin{mabstract}
The asymptotic behavior of the stochastic gradient algorithm 
with a biased gradient estimator is analyzed. 
Relying on arguments based on 
the dynamic system theory (chain-recurrence) and the differential geometry 
(Yomdin theorem and Lojasiewicz inequality), 
tight bounds on the asymptotic bias of the iterates 
generated by such an algorithm are derived. 
The obtained results hold under mild conditions 
and cover a broad class of high-dimensional nonlinear algorithms. 
Using these results, 
the asymptotic properties of the policy-gradient (reinforcement) learning 
and adaptive population Monte Carlo sampling are studied. 
Relying on the same results, 
the asymptotic behavior of the recursive maximum split-likelihood estimation 
in hidden Markov models is analyzed, too. 
\end{mabstract}

\begin{mkeywords}
Stochastic gradient search, biased gradient estimation, 
chain-recurrence, Yomdin theorem, Lojasiewicz inequalities, 
reinforcement learning, adaptive Monte Carlo sampling, system identification. 
\end{mkeywords}

\begin{msubjclass}
Primary 62L20; Secondary 90C15, 93E12, 93E35. 
\end{msubjclass}

\section{Introduction}\label{section0}

Many problems in automatic control, system identification, signal processing, 
machine learning, operations research and statistics can be posed as 
a stochastic optimization problem, 
i.e., as a minimization (or maximization) of an unknown objective function 
whose values are available only through noisy observations. 
Such a problem can efficiently be solved by stochastic gradient search
(also known as the stochastic gradient algorithm). 
Stochastic gradient search is a procedure of the stochastic approximation type
which iteratively approximates the minima of the objective function 
using a statistical or Monte Carlo estimator of the gradient (of the objective function). 
Often, the estimator is biased, since the consistent gradient estimation is usually 
computationally expensive or not available at all. 
As a result of the biased gradient estimation, 
the stochastic gradient search is biased, too, 
i.e., the corresponding algorithm does not converge to the minima, but to their vicinity. 
In order to interpret the results produced by such an algorithm 
and to tune the algorithm's parameters (e.g., to achieve a better bias/variance balance 
and a better convergence rate), 
the knowledge about the asymptotic behavior and the asymptotic bias of the algorithm iterates
is crucially needed. 

Despite its practical and theoretical importance, 
the asymptotic behavior of the stochastic gradient search with biased gradient estimation 
(also referred to as the biased stochastic gradient search) 
has not attracted much attention in the literature on stochastic optimization and 
stochastic approximation. 
To the best of the present authors' knowledge, 
the asymptotic properties of the biased stochastic gradient search 
(and the biased stochastic approximation) 
have only been analyzed in \cite{borkar}, \cite{chen1}, \cite{chen2} and \cite{chen3}. 
Although the results of \cite{borkar}, \cite{chen1}, \cite{chen2}, \cite{chen3} 
provide a good insight into the asymptotic behavior of the biased gradient search, 
they hold under restrictive conditions which are very hard to verify 
for complex stochastic gradient algorithms. 
Moreover, unless the objective function is of a simple form (e.g., convex or polynomial), 
none of \cite{borkar}, \cite{chen1}, \cite{chen2}, \cite{chen3} offers 
explicit bounds on the asymptotic bias of the algorithm iterates. 

In this paper, we study the asymptotic behavior of the biased gradient search. 
Using arguments based on the dynamic system theory (chain-recurrence)
and the differential geometry 
(Yomdin theorem and Lojasiewicz inequalities), 
we prove that the algorithm iterates converge to a vicinity of the set of minima. 
Relying on the same arguments, 
we also derive relatively tight bounds on the radius of the vicinity, 
i.e., on the asymptotic bias of the algorithm iterates. 
The obtained results hold under mild and easily verifiable conditions 
and cover a broad class of complex stochastic gradient algorithms. 
In this paper, we show how the obtained results can be applied to the asymptotic analysis of 
policy-gradient (reinforcement) learning and adaptive population Monte Carlo sampling.  
We also demonstrate how the obtained results can be used to assess the asymptotic bias of 
the recursive maximum split-likelihood estimation in hidden Markov models. 

The paper is organized as follows. 
The main results are presented in Section \ref{section1}, 
where the stochastic gradient search with additive noise is analyzed. 
In Section \ref{section2}, the asymptotic bias of the stochastic gradient search with 
Markovian dynamics is studied. 
Sections \ref{section3} -- \ref{section4} provide examples of the results of 
Sections \ref{section1} and \ref{section2}. 
In Section \ref{section3}, the policy-gradient (reinforcement) learning is considered, 
while the adaptive population Monte Carlo sampling is analyzed in Section \ref{section5}. 
Section \ref{section4} is devoted to the recursive maximum split-likelihood 
estimation in hidden Markov models. 
The results of Sections \ref{section1} -- \ref{section4} are proved in 
Sections \ref{section1.a*} -- \ref{section4*}. 

\section{Main Results}\label{section1}

In this section, the asymptotic behavior of the following algorithm is analyzed: 
\begin{align} \label{1.1}
	\theta_{n+1}
	=
	\theta_{n}
	-
	\alpha_{n} (\nabla f(\theta_{n} ) + \xi_{n} ), 
	\;\;\; 
	n\geq 0. 
\end{align}
Here, $f:\mathbb{R}^{d_{\theta} } \rightarrow \mathbb{R}$ 
is a differentiable function, 
while 
$\{\alpha_{n} \}_{n\geq 0}$ is a sequence of positive real numbers. 
$\theta_{0}$ is an $\mathbb{R}^{d_{\theta} }$-valued random variable 
defined on a probability space $(\Omega, {\cal F}, P )$, 
while 
$\{\xi_{n} \}_{n\geq 0}$ is an $\mathbb{R}^{d_{\theta} }$-valued 
stochastic process defined on 
the same probability space. 
To allow more generality, we assume that for each $n \geq 0$, 
$\xi_{n}$ is a random function of $\theta_{0},\dots,\theta_{n}$. 
In the area of stochastic optimization, 
recursion (\ref{1.1}) is known as a stochastic gradient search
(or stochastic gradient algorithm). 
The recursion minimizes function $f(\cdot )$, 
which is usually referred to as the objective function. 
Term $\nabla f(\theta_{n} ) + \xi_{n}$ is interpreted as a gradient estimator 
(i.e., an estimator of $\nabla f(\theta_{n} )$), 
while $\xi_{n}$ represents the estimator's noise (or error). 
For further details, see \cite{pflug}, \cite{spall} and 
references given therein. 

Throughout the paper, the following notation is used. 
$\|\cdot \|$ and $d(\cdot,\cdot )$ stand for the Euclidean norm 
and the distance induced by the Euclidean norm (respectively). 
For $t\in (0,\infty )$ and $n\geq 0$, 
$a(n,t)$ is the integer defined as 
\begin{align*}
	a(n,t)=\max\left\{k\geq n: \sum_{i=n}^{k-1} \alpha_{i} \leq t \right\}. 
\end{align*}
${\cal S}$ and $f({\cal S} )$ are the sets of stationary points and critical values of $f(\cdot )$, i.e., 
\begin{align}\label{1.21}
	{\cal S}=\{\theta \in \mathbb{R}^{d_{\theta } }:\nabla f(\theta ) = 0 \}, 
	\;\;\;\;\; 
	f({\cal S} )=\{f(\theta ): \theta \in S \}. 
\end{align}	
For $\theta\in\mathbb{R}^{d_{\theta } }$, 
$\pi(\cdot\; ;\theta )$ is the solution to the ODE 
$d\theta/dt=-\nabla f(\theta )$ satisfying $\pi(0;\theta )=\theta$. 
${\cal R}$ denotes the set of chain-recurrent points of this ODE, 
i.e., $\theta\in {\cal R}$ if and only if 
for any $\delta, t \in (0,\infty )$, 
there exist an integer $N\geq 1$, 
real numbers $t_{1},\dots,t_{N}\in [t,\infty )$ and 
vectors $\vartheta_{1},\dots,\vartheta_{N}\in\mathbb{R}^{d_{\theta } }$
(each of which can depend on $\theta$, $\delta$, $t$)
such that 
\begin{align}\label{1.11}
	\|\vartheta_{1}-\theta \|\leq\delta, 
	\;\;\;\;\; 
	\|\pi(t_{N};\vartheta_{N} ) - \theta \|\leq\delta, 
	\;\;\;\;\; 
	\|\vartheta_{k+1} - \pi(t_{k}; \vartheta_{k} ) \|\leq\delta 
\end{align}
for $1\leq k<N$.

Elements of ${\cal R}$ can be considered as limits to slightly perturbed solutions to 
the ODE $d\theta/dt=-\nabla f(\theta)$. 
As the piecewise linear interpolation of sequence $\{\theta_{n} \}_{n\geq 0}$
falls into the category of such solutions, 
the concept of chain-recurrence is tightly connected to the asymptotic behavior 
of stochastic gradient search. In \cite{benaim1}, \cite{benaim2}, it has been shown that for unbiased gradient estimates, 
all limit points of $\{\theta_{n} \}_{n\geq 0}$ belong to ${\cal R}$ and that each element of  ${\cal R}$ can potentially be 
a limit point of $\{\theta_{n} \}_{n\geq 0}$ with a non-zero probability. 

If $f(\cdot )$ is Lipschitz continuously differentiable, it can be established that ${\cal S}\subseteq{\cal R}$. 
If additionally $f({\cal S})$ is of a zero Lebesgue measure
(which holds when $f({\cal S})$ is discrete or 
when $f(\cdot )$ is $d_{\theta}$-times continuously differentiable), 
then ${\cal S}={\cal R}$. However, if $f(\cdot)$ is only Lipschitz continuously differentiable, 
then it is possible to have ${\cal R}\setminus{\cal S}\neq\emptyset$
(see \cite[Section 4]{hurley}). 
Hence, in general, a limit point of $\{\theta_{n} \}_{n\geq 0}$ is in ${\cal R}$
but not necessarily in ${\cal S}$. For more details on chain-recurrence, see 
\cite{benaim1}, \cite{benaim2}, \cite{borkar} and references therein. 
Given these results, it will prove useful to involve both ${\cal R}$ and ${\cal S}$ in
the asymptotic analysis of biased stochastic gradient search.

The algorithm (\ref{1.1}) is analyzed under the following assumptions:  

\begin{assumption} \label{a1.1}
$\lim_{n\rightarrow \infty } \alpha_{n} = 0$ 
and 
$\sum_{n=0}^{\infty } \alpha_{n} = \infty$. 
\end{assumption}

\begin{assumption} \label{a1.2}
$\{\xi_{n} \}_{n\geq 0}$ admits the decomposition $\xi_{n}=\zeta_{n}+\eta_{n}$ for each $n\geq 0$, 
where  
$\{\zeta_{n} \}_{n\geq 0}$ and 
$\{\eta_{n} \}_{n\geq 0}$ 
are
$\mathbb{R}^{d_{\theta} }$-valued stochastic processes 
(defined on $(\Omega,{\cal F},P)$) 
satisfying 
\begin{align}
	& \label{a1.2.1}
	\lim_{n\rightarrow \infty } 
	\max_{n\leq k < a(n,t) }
	\left\|\sum_{i=n}^{k} \alpha_{i} \zeta_{i} \right\|
	=0, 
	\;\;\;\;\; 
	\limsup_{n\rightarrow \infty } \|\eta_{n} \|
	<\infty
\end{align}
almost surely on $\{\sup_{n\geq 0}\|\theta_{n} \|<\infty \}$
for any $t\in (0,\infty )$.
\end{assumption}

\addtocounter{assumption}{1}

\begin{assumptionabc} \label{a1.3.a}
$\nabla f(\cdot )$ is locally Lipschitz continuous on $\mathbb{R}^{d_{\theta } }$. 
\end{assumptionabc}

\begin{assumptionabc} \label{a1.3.b}
$f(\cdot )$ is $p$-times differentiable on $\mathbb{R}^{d_{\theta } }$, 
where $p>d_{\theta }$. 
\end{assumptionabc}

\begin{assumptionabc} \label{a1.3.c}
$f(\cdot )$ is real-analytic on $\mathbb{R}^{d_{\theta } }$. 
\end{assumptionabc}

\begin{remark}\label{remark1.5}
Due to Assumption \ref{a1.1}, $a(n,t)$ is well-defined, finite and satisfies 
\begin{align}
	t
	\geq 
	\sum_{i=n}^{a(n,t)-1} \alpha_{i}
	=
	\sum_{i=n}^{a(n,t)} \alpha_{i} 
	-
	\alpha_{a(n,t)} 
	\geq 
	t - \alpha_{a(n,t)} 
\end{align}
for all $t\in(0,\infty)$, $n\geq 0$. 
Consequently, Assumption \ref{a1.1} yields 
\begin{align}\label{1.1501}
	\lim_{n\rightarrow\infty} \sum_{i=n}^{a(n,t)-1} \alpha_{i} 
	=
	\lim_{n\rightarrow\infty} \sum_{i=n}^{a(n,t)} \alpha_{i} 
	=
	t
\end{align}
for each $t\in(0,\infty )$. 
\end{remark}

Assumption \ref{a1.1} corresponds to the step-size sequence 
$\{\alpha_{n} \}_{n\geq 0}$ and is commonly used in the asymptotic analysis of 
stochastic gradient and stochastic approximation algorithms. 
In this or similar form, it is an ingredient of practically any asymptotic analysis of 
stochastic gradient search and stochastic approximation. 
Assumption \ref{a1.1} is satisfied if $\alpha_{n}=n^{-a}$ for $n\geq 1$, where $a\in(0,1]$. 

Assumption \ref{a1.2} is a noise condition. 
It can be interpreted as a decomposition of the gradient estimator's noise 
$\{\xi_{n} \}_{n\geq 0}$ 
into a zero-mean sequence $\{\zeta_{n} \}_{n\geq 0}$
(which is averaged out by step-sizes $\{\alpha_{n} \}_{n\geq 0}$) 
and the estimator's bias $\{\eta_{n} \}_{n\geq 0}$. 
Assumption \ref{a1.2} is satisfied if 
$\{\zeta_{n} \}_{n\geq 0}$ is a martingale-difference or mixingale 
sequence, 
and if $\{\eta_{n} \}_{n\geq 0}$ are continuous functions of $\{\theta_{n} \}_{n\geq 0}$. 
It also holds for gradient search with Markovian dynamics (see Section \ref{section2}). 
If the gradient estimator is unbiased  
(i.e., $\lim_{n\rightarrow \infty } \eta_{n} = 0$ almost surely), 
Assumption \ref{a1.2} reduces to the well-known
Kushner-Clark condition, the weakest noise assumption under which the almost sure 
convergence of (\ref{1.1}) can be demonstrated. 

Assumptions \ref{a1.3.a}, \ref{a1.3.b} and \ref{a1.3.c} 
are related to the objective function $f(\cdot )$ and its analytical properties. 
Assumption \ref{a1.3.a} is involved in practically any asymptotic result 
for stochastic gradient search 
(as well as in many other asymptotic and non-asymptotic results for stochastic and 
deterministic optimization). 
Although much more restrictive than Assumption \ref{a1.3.a}, 
Assumptions \ref{a1.3.b} and \ref{a1.3.c} hold for a number of algorithms 
routinely used in engineering, statistics, machine learning and operations research. 
In Sections \ref{section3} -- \ref{section4}, 
Assumptions \ref{a1.3.b} and \ref{a1.3.c} are shown for policy-gradient (reinforcement) learning, 
adaptive population Monte Carlo sampling 
and recursive maximum split-likelihood estimation in hidden Markov models. 
In \cite{tadic5}, Assumption \ref{a1.3.c} (which is a special case of Assumption \ref{a1.3.b}) 
has been demonstrated for recursive maximum (full) likelihood estimation in hidden Markov models. 
In \cite{tadic4}, the same assumption has also been demonstrated 
for supervised and temporal-difference learning, 
online principal component analysis, 
Monte Carlo optimization of controlled Markov chains
and recursive parameter estimation in linear stochastic systems. 
In \cite{tadic6}, we show Assumptions \ref{a1.3.b} and \ref{a1.3.c} 
for sequential Monte Carlo methods for the parameter estimation in non-linear non-Gaussian 
state-space models. 
It is also worth mentioning that the objective functions associated 
with online principal and independent component analysis 
(as well as with many other adaptive signal processing algorithms)
are often polynomial or rational, and hence, smooth and analytic, too
(see e.g., \cite{cichocki&amari} and references cited therein). 

As opposed to Assumption \ref{a1.3.a}, 
Assumptions \ref{a1.3.b} and \ref{a1.3.c} allow some sophisticated results
from the differential geometry to be applied to 
the asymptotic analysis of stochastic gradient search. 
More specifically, Yomdin theorem (qualitative version of Morse-Sard theorem; 
see \cite{yomdin} and Proposition \ref{proposition1.1} in Section \ref{section1.bc*})
can be applied to functions satisfying Assumption \ref{a1.3.b}, 
while Lojasiewicz inequalities 
(see \cite{lojasiewicz1}, \cite{lojasiewicz2};  
see also \cite{bierstone&milman}, \cite{kurdyka} and Proposition \ref{proposition1.2} in Section \ref{section1.bc*})
hold for functions fulfilling Assumption \ref{a1.3.c}. 
Using Yomdin theorem and Lojasiewicz inequalities, 
a more precise characterization of the asymptotic bias of the stochastic gradient search
can be obtained
(see Parts (ii) and (iii) of Theorem \ref{theorem1.1}). 

In order to state the main results of this section, we need some further notation. 
$\eta$ is 
the asymptotic magnitude of the gradient estimator's bias 
$\{\eta_{n} \}_{n\geq 0}$, i.e., 
\begin{align}\label{1.101}
	\eta=\limsup_{n\rightarrow \infty } \|\eta_{n} \|. 
\end{align}
For a compact set $Q\subset \mathbb{R}^{d_{\theta } }$, 
$\Lambda_{Q}$ denotes the event 
\begin{align}\label{1.103}
	\Lambda_{Q} 
	=
	\liminf_{n\rightarrow \infty } 
	\{\theta_{n} \in Q \} 
	=
	\bigcup_{n=0}^{\infty } \bigcap_{k=n}^{\infty } 
	\{\theta_{k} \in Q \}. 
\end{align}
With this notation,  
our main result on the asymptotic bias of the recursion (\ref{1.1}) can be stated as follows.  

\begin{theorem} \label{theorem1.1}
Suppose that Assumptions \ref{a1.1} and \ref{a1.2} hold. 
Let $Q\subset\mathbb{R}^{d_{\theta } }$ be any compact set. 
Then, the following is true: 
\begin{compactenum}[(i)]
\item
If $f(\cdot )$ satisfies Assumption \ref{a1.3.a}, 
there exists a (deterministic) non-decreasing function 
$\psi_{Q}:[0,\infty )\rightarrow[0,\infty)$ 
(independent of  $\eta$ and depending only on $f(\cdot )$)
such that 
$\lim_{t\rightarrow 0} \psi_{Q}(t) = \psi_{Q}(0) = 0$ 
and 
\begin{align}\label{t1.1.1*}
	\limsup_{n\rightarrow\infty } d(\theta_{n}, {\cal R}  )
	\leq 
	\psi_{Q}(\eta )
\end{align}
almost surely on $\Lambda_{Q}$. 
\item
If $f(\cdot )$ satisfies Assumption \ref{a1.3.b}, 
there exists a real number $K_{Q}\in (0,\infty )$
(independent of $\eta$ and depending only on $f(\cdot )$)
such that 
\begin{align}\label{t1.1.3*}
	\limsup_{n\rightarrow\infty } \|\nabla f(\theta_{n} ) \| 
	\leq 
	K_{Q} \eta^{q/2}, 
	\;\;\;\;\; 
	\limsup_{n\rightarrow\infty } f(\theta_{n} ) 
	-
	\liminf_{n\rightarrow\infty } f(\theta_{n} ) 
	\leq 
	K_{Q} \eta^{q}
\end{align}
almost surely on $\Lambda_{Q}$, 
where $q=(p-d_{\theta} )/(p-1)$. 
\item
If $f(\cdot )$ satisfies Assumption \ref{a1.3.c}, 
there exist real numbers $r_{Q}\in (0,1)$, $L_{Q}\in (0,\infty )$
(independent of $\eta$ and depending only on $f(\cdot )$)
such that 
\begin{align}\label{t1.1.5*}
	\limsup_{n\rightarrow\infty } \|\nabla f(\theta_{n} ) \| 
	\leq 
	L_{Q} \eta^{1/2}, 
	\;\;\;\;\; 
	\limsup_{n\rightarrow\infty } d(f(\theta_{n} ), f({\cal S} ) 
	\leq 
	L_{Q} \eta, 
	\;\;\;\;\; 
	\limsup_{n\rightarrow\infty } d(\theta_{n}, {\cal S} )
	\leq 
	L_{Q} \eta^{r_{Q} }
\end{align}
almost surely on $\Lambda_{Q}$. 
\end{compactenum}
\end{theorem}

Theorem \ref{theorem1.1} is proved in Sections \ref{section1.a*} and \ref{section1.bc*}, 
while its global version is provided in Appendix \ref{appendix2}.  

\begin{vremark}
If Assumption \ref{a1.3.b} (or Assumption \ref{a1.3.c}) is satisfied, 
then ${\cal S}={\cal R}$. 
Hence, under Assumption \ref{a1.3.b}, 
(\ref{t1.1.1*}) still holds if ${\cal R}$ is replaced with ${\cal S}$. 
\end{vremark}

\begin{remark}\label{remark1.1}
Function $\psi_{Q}(\cdot )$ depends on $f(\cdot )$ in the following two ways. 
First, $\psi_{Q}(\cdot )$ depends on $f(\cdot )$ 
through the chain-recurrent set ${\cal R}$ and its geometric properties. 
In addition to this, $\psi_{Q}(\cdot )$ depends on $f(\cdot )$ 
through upper bounds of $\|\nabla f(\cdot ) \|$
and Lipschitz constants of $\nabla f(\cdot )$. 
An explicit construction of $\psi_{Q}(\cdot )$ is provided in the proof of 
Part (i) of Theorem \ref{theorem1.1} (Section \ref{section1.a*}). 
\end{remark}

\begin{remark}\label{remark1.2}
As $\psi_{Q}(\cdot )$, constants $K_{Q}$ and $L_{Q}$ depend on $f(\cdot )$ 
through upper bounds of $\|\nabla f(\cdot ) \|$
and Lipschitz constants of 
$\nabla f(\cdot )$. 
$K_{Q}$ and $L_{Q}$ also depend on $f(\cdot )$ through the Yomdin and Lojasiewicz constants 
(quantities $M_{Q}$, $M_{1,Q}$, $M_{2,Q}$ specified in 
Propositions \ref{proposition1.1}, \ref{proposition1.2}). 
Explicit formulas for $K_{Q}$ and $L_{Q}$ are included in the proof of 
Parts (ii) and (iii) of Theorem \ref{theorem1.1}
(Section \ref{section1.bc*}). 
\end{remark}

According to the literature on stochastic optimization and stochastic approximation, 
stochastic gradient search with unbiased gradient estimates (the case when $\eta=0$) exhibits 
the following asymptotic behavior. 
Under mild conditions, sequences $\{\theta_{n} \}_{n\geq 0}$ and 
$\{f(\theta_{n} ) \}_{n\geq 0}$ converge to ${\cal R}$ and $f({\cal R})$ (respectively), 
i.e, 
\begin{align}\label{1.501}
	\lim_{n\rightarrow\infty} d(\theta_{n}, {\cal R} ) =0, 
	\;\;\;\;\; 
	\lim_{n\rightarrow\infty} d(f(\theta_{n}), f({\cal R}) ) = 0
\end{align}
almost surely on $\{\sup_{n\geq 0} \|\theta_{n} \|<\infty \}$ 
(see \cite[Proposition 4.1, Theorem 5.7]{benaim2} which hold under
Assumptions \ref{a1.1}, \ref{a1.2}, \ref{a1.3.a}). 
Under more restrictive conditions, 
sequences $\{\theta_{n} \}_{n\geq 0}$ and $\{ f(\theta_{n} ) \}_{n\geq 0}$ converge
to $\cal S$ and a point in $f(\cal S)$ (respectively), 
i.e., 
\begin{align}\label{1.503}
	\lim_{n\rightarrow \infty } d(\theta_{n}, {\cal S} ) = 0, 
	\;\;\;\;\; 
	\lim_{n\rightarrow \infty } \nabla f(\theta_{n} ) = 0, 
	\;\;\;\;\; 
	\lim_{n\rightarrow \infty } d(f(\theta_{n} ), f({\cal S} ) ) = 0, 
	\;\;\;\;\; 
	\limsup_{n\rightarrow \infty } f(\theta_{n} ) = 
\liminf_{n\rightarrow \infty } f(\theta_{n} ) 
\end{align}
almost surely on $\{\sup_{n\geq 0} \|\theta_{n} \|<\infty \}$ 
(see \cite[Corollary 6.7]{benaim2} 
which holds under Assumptions \ref{a1.1}, \ref{a1.2}, \ref{a1.3.b}). 
The same asymptotic behavior occurs when Assumptions \ref{a1.1}, \ref{a1.3.a} hold and 
$\{\xi_{n} \}_{n\geq 0}$ is a martingale-difference sequence 
(see \cite[Proposition 1]{bertsekas&tsitsiklis2}).
When the gradient estimator is biased (the case where $\eta > 0$), 
this is not true any more. 
Now, the quantities 
\begin{align}\label{1.5}
	\limsup_{n\rightarrow \infty } d(\theta_{n}, {\cal S} ), \;\;\;\;\; 
	\limsup_{n\rightarrow \infty } \|\nabla f(\theta_{n} ) \|, \;\;\;\;\;  
	\limsup_{n\rightarrow \infty } d(f(\theta_{n} ), f({\cal S} ) ), \;\;\;\;\;  
	\limsup_{n\rightarrow \infty } f(\theta_{n} ) - \liminf_{n\rightarrow \infty } f(\theta_{n} )
\end{align}
are strictly positive and depend on $\eta$
(it is reasonable to expect these quantities 
to decrease in $\eta$ and to tend to zero as $\eta\rightarrow 0$). 
Hence, the quantities (\ref{1.5}) and their dependence on $\eta$
can be considered as a sensible characterization of 
the asymptotic bias of 
the gradient search with biased gradient estimation
(i.e., these quantities describe how biased stochastic gradient search 
deviates from the nominal behavior). 
In the case of algorithm (\ref{1.1}), such a characterization is provided by
Theorem \ref{theorem1.1}.  
The theorem includes tight, explicit 
bounds on the quantities (\ref{1.5}) in the terms 
of the gradient estimator's bias $\eta$ 
and analytical properties of $f(\cdot )$. 

The results of Theorem \ref{theorem1.1}  are of a local nature.  
They hold only on the event where algorithm (\ref{1.1}) is stable
(i.e., where sequence $\{\theta_{n} \}_{n\geq 0}$ belongs to a compact set $Q$). 
Stating results on the asymptotic bias of 
stochastic gradient search in such a local form is quite sensible due to the following 
reasons. 
The stability of stochastic gradient search is based on 
well-understood arguments which are rather different from 
the arguments used here to analyze the asymptotic bias. 
Moreover and more importantly, 
as demonstrated in Appendix \ref{appendix2}, 
it is relatively easy to get a global version of 
Theorem \ref{theorem1.1}  
by combining the theorem with 
the methods for verifying or ensuring the stability
(e.g., with the results of \cite{borkar&meyn} and \cite{chen3}). 
It is also worth mentioning that local asymptotic results are quite common  
in the areas of stochastic optimization and stochastic approximation 
(e.g., most of the results of \cite[Part II]{benveniste}, 
similarly as Theorem \ref{theorem1.1}, 
hold only on set $\Lambda_{Q}$). 

Gradient algorithms with biased gradient estimation are extensively used 
in  
system identification 
\cite{andrieu&doucet&singh&tadic}, 
\cite{doucet&defreitas&gordon}, 
\cite{doucet&tadic}, 
\cite{kantas&doucet&singh}, 
\cite{ljung}, 
discrete-event system optimization 
\cite{cassandras&lafortune}, \cite{henderson&meyn&tadic}, \cite{rubinstein1}, \cite{rubinstein2}, 
machine learning 
\cite{baxter&bartlett}, 
\cite{bertsekas&tsitsiklis1}, 
\cite{cao}, 
\cite{konda&tsitsiklis}, 
\cite{powell}, 
and statistics
\cite{andrieu&doucet&tadic}, 
\cite{cappe&moulines&ryden}, 
\cite{doucet&defreitas&gordon}, 
\cite{poyiadjis&doucet&singh}
\cite{ryden2}. 
To interpret results obtained by such an algorithm 
and to tune the algorithm parameters 
(e.g., to achieve better bias/variance balance and convergence rate), 
it is crucially important to 
understand the asymptotic properties 
of the biased stochastic gradient search. 
Despite its importance, 
the asymptotic behavior of the stochastic gradient search 
with biased gradient estimation 
has not received much attention in the literature on 
stochastic optimization and stochastic approximation. 
To the best of the present authors' knowledge, 
the asymptotic properties of the biased stochastic gradient search and 
biased stochastic approximation have been studied only in 
\cite[Section 5.3]{borkar}, \cite{chen1}, \cite{chen2}, \cite[Section 2.7]{chen3}. 
Although these results 
provide a good insight into 
the asymptotic behavior of the biased gradient search, 
they are based on restrictive conditions. 
More specifically, 
the results of 
\cite[Section 5.3]{borkar}, \cite{chen1}, \cite{chen2}, \cite[Section 2.7]{chen3}
hold only if $f(\cdot )$ is unimodal 
or if $\{\theta_{n} \}_{n\geq 0}$ belongs to the domain of an asymptotically stable 
attractor of $d\theta/dt = - \nabla f(\theta )$. 
In addition to this, 
the results of \cite[Section 5.3]{borkar}, \cite{chen1}, \cite{chen2}, \cite[Section 2.7]{chen3}
do not provide any explicit bound 
on the asymptotic bias of the stochastic gradient search 
unless $f(\cdot )$ is of a simple form (e.g., convex or polynomial).  
Unfortunately, in the case of complex stochastic gradient algorithms
(such as those studied in Sections \ref{section3} -- \ref{section4}), 
$f(\cdot )$ is usually multimodal 
with lot of unisolated local extrema and saddle points. 
For such algorithms, not only it is hard to verify 
the assumptions adopted in 
\cite[Section 5.3]{borkar}, \cite{chen1}, \cite{chen2}, \cite[Section 2.7]{chen3}, 
but these assumptions are likely not to hold at all. 

Relying on the chain-recurrence, 
Yomdin theorem 
and Lojasiewicz inequalities, 
Theorem \ref{theorem1.1} 
overcomes the described difficulties.  
The theorem allows the objective function $f(\cdot )$
to be multimodal (with manifolds of unisolated extrema and saddle points) 
and does not require $d\theta/dt = - \nabla f(\theta )$ 
to have an asymptotically stable attractor 
which is infinitely often visited by $\{\theta_{n} \}_{n\geq 0}$. 
In addition to this, 
Theorem \ref{theorem1.1} 
provides relatively tight explicit bounds 
on the asymptotic bias of algorithm (\ref{1.1}). 
Furthermore, 
as demonstrated in Sections \ref{section3} -- \ref{section4} and \cite{tadic6}, 
the theorem covers a broad class 
of stochastic gradient algorithms used in machine learning, 
Monte Carlo sampling and system identification.

\section{Stochastic Gradient Search with Markovian Dynamics}\label{section2}

In order to illustrate the results of Section \ref{section1} and 
to set up a framework for the analysis carried out in Sections 
\ref{section3} -- \ref{section4}, 
we apply Theorem \ref{theorem1.1} 
to stochastic gradient algorithms with Markovian dynamics. 
These algorithms are defined by the following difference equation: 
\begin{align} \label{2.1}
	\theta_{n+1}
	=
	\theta_{n} 
	-
	\alpha_{n} (F(\theta_{n}, Z_{n+1} ) + \eta_{n} ), 
	\;\;\; 
	n\geq 0. 
\end{align}
In this recursion, 
$F: \mathbb{R}^{d_{\theta} } \times \mathbb{R}^{d_{z} } \rightarrow \mathbb{R}^{d_{\theta } }$ is 
a Borel-measurable function, 
while $\{\alpha_{n} \}_{n\geq 0}$ is a sequence of positive real numbers. 
$\theta_{0}$ is an $\mathbb{R}^{d_{\theta } }$-valued random variable defined on 
a probability space $(\Omega, {\cal F}, P)$. 
$\{Z_{n} \}_{n\geq 0}$ is an $\mathbb{R}^{d_{z} }$-valued stochastic process 
defined on $(\Omega,{\cal F}, P)$, 
while $\{\eta_{n} \}_{n\geq 0}$ is an $\mathbb{R}^{d_{\theta} }$-valued stochastic process 
defined on the same probability space. 
$\{Z_{n} \}_{n\geq 0}$ is a Markov process controlled by 
$\{\theta_{n} \}_{n\geq 0}$, i.e., 
there exists a family of transition probability kernels 
$\{\Pi_{\theta }(\cdot,\cdot ) \}_{\theta \in \mathbb{R}^{d_{\theta } } }$
defined on $\mathbb{R}^{d_{z} }$
such that 
\begin{align}\label{2.3}
	P(Z_{n+1} \in B|\theta_{0},Z_{0},\dots,\theta_{n},Z_{n} ) 
	=
	\Pi_{\theta_{n} }(Z_{n}, B )
\end{align}
almost surely for any Borel-measurable set $B \subseteq \mathbb{R}^{d_{z} }$ and $n\geq 0$. 
$\{\eta_{n} \}_{n\geq 0}$ are random function of $\{\theta_{n} \}_{n\geq 0}$, 
i.e., $\eta_{n}$ is a random function of 
$\theta_{0},\dots,\theta_{n}$ for each $n\geq 0$. 
In the context of stochastic gradient search, 
$F(\theta_{n},Z_{n+1} ) + \eta_{n}$ represents a gradient estimator 
(i.e., an estimator of 
$\nabla f(\theta_{n} )$). 

The algorithm (\ref{2.1}) is analyzed under the following assumptions. 

\begin{assumption} \label{a2.1}
$\limsup_{n\rightarrow \infty } |\alpha_{n+1}^{-1} - \alpha_{n}^{-1} | < \infty$,    
$\sum_{n=0}^{\infty } \alpha_{n} = \infty$
and  
$\sum_{n=0}^{\infty } \alpha_{n}^{2} < \infty$. 
\end{assumption}

\begin{assumption} \label{a2.2}
There exist a differentiable function 
$f: \mathbb{R}^{d_{\theta } } \rightarrow \mathbb{R}$ 
and 
a Borel-measurable function 
$\tilde{F}: \mathbb{R}^{d_{\theta} } \times \mathbb{R}^{d_{z} } \rightarrow 
\mathbb{R}^{d_{\theta } }$
such that 
$\nabla f(\cdot )$ is locally Lipschitz continuous and 
\begin{align}\label{a2.2.1*}
	F(\theta, z ) 
	-
	\nabla f(\theta )
	=
	\tilde{F}(\theta, z ) 
	-
	(\Pi\tilde{F} )(\theta, z )
\end{align}
for each $\theta\in \mathbb{R}^{d_{\theta } }$, 
$z \in \mathbb{R}^{d_{z } }$, 
where 
$(\Pi\tilde{F} )(\theta, z ) 
= \int \tilde{F}(\theta, z' ) \Pi_{\theta }(z, dz' )$. 
\end{assumption}

\begin{assumption} \label{a2.3}
For any compact set $Q\subset \mathbb{R}^{d_{\theta } }$, 
there exists a Borel-measurable function 
$\varphi_{Q}: \mathbb{R}^{d_{z} } \rightarrow [1,\infty )$ such that 
\begin{align*}
	&
	\max\{
	\|F(\theta,z ) \|, \|\tilde{F}(\theta,z ) \|, \|(\Pi \tilde{F} )(\theta,z ) \|
	\}
	\leq 
	\varphi_{Q}(z ), 
	\\
	&
	\|(\Pi \tilde{F} )(\theta',z ) - (\Pi \tilde{F} )(\theta'',z ) \|
	\leq 
	\varphi_{Q}(z) \|\theta' - \theta'' \| 
\end{align*}
for all 
$\theta, \theta', \theta'' \in Q$, $z \in \mathbb{R}^{d_{z} }$.  
Moreover, 
\begin{align*} 
	\sup_{n\geq 0}
	E\left(
	\varphi_{Q}^{2}(Z_{n+1} ) 
	I_{\{\tau_{Q} > n \} }
	|\theta_{0}=\theta, Z_{0}=z 
	\right)
	< 
	\infty
\end{align*}
for all $\theta \in \mathbb{R}^{d_{\theta } }$, $z \in \mathbb{R}^{d_{z} }$, 
where $\tau_{Q}$ is the stopping time defined by 
$
	\tau_{Q} = 
	\inf\left(\{n\geq 0: \theta_{n} \not\in Q \}\cup\{\infty \} \right) 
$.
\end{assumption}

\begin{assumption} \label{a2.4}
$
	\limsup_{n\rightarrow \infty } 
	\|\eta_{n} \| 
	< 
	\infty 
$ 
almost surely on $\{\sup_{n\geq 0} \|\theta_{n} \| < \infty \}$. 
\end{assumption}

Let ${\cal S}$ and $f({\cal S} )$ have the same meaning as in (\ref{1.21})
($f(\cdot )$ is now specified in Assumption \ref{a2.2}), 
while ${\cal R}$ is the set of chain-recurrent points of 
the ODE $d\theta/dt=-\nabla f(\theta )$ (for details on chain-recurrence, 
see Section \ref{section1}). 
Moreover, let $\eta$ have the same meaning as in (\ref{1.101}).  
Then, our results on the asymptotic behavior of the recursion 
(\ref{2.1}) read as follows. 

\begin{theorem}\label{theorem2.1}
Suppose that Assumptions \ref{a2.1} -- \ref{a2.4} hold. 
Let $Q\subset\mathbb{R}^{d_{\theta} }$ be any compact set. 
Then, the following is true: 
\begin{compactenum}[(i)]
\item
If $f(\cdot )$ (specified in Assumption \ref{a2.2}) satisfies Assumption \ref{a1.3.a}, 
Part (i) of Theorem \ref{theorem1.1} holds. 
\item
If $f(\cdot )$ (specified in Assumption \ref{a2.2}) satisfies Assumption \ref{a1.3.b}, 
Part (ii) of Theorem \ref{theorem1.1} holds. 
\item
If $f(\cdot )$ (specified in Assumption \ref{a2.2}) satisfies Assumption \ref{a1.3.c}, 
Part (iii) of Theorem \ref{theorem1.1} holds. 
\end{compactenum}
\end{theorem}

Theorem \ref{theorem2.1} is proved in Section \ref{section2*}, 
while its global version is provided in Appendix \ref{appendix3}.  

Assumption \ref{a2.1} is related to the sequence  
$\{\alpha_{n} \}_{n\geq 0}$. 
It is satisfied if 
$\alpha_{n} = 1/n^{a}$ for $n\geq 1$, 
where $a \in (1/2,1]$ is a constant. 
Assumptions \ref{a2.2} and \ref{a2.3} 
correspond to   
the stochastic process 
$\{Z_{n} \}_{n\geq 0}$
and are standard for the asymptotic analysis of 
stochastic approximation algorithms with Markovian dynamics. 
Basically, Assumptions \ref{a2.2} and \ref{a2.3} require 
the Poisson equation associated with algorithm (\ref{2.1}) 
to have a solution which is Lipschitz continuous in $\theta$. 
They hold if the following is satisfied: 
(i) $\Pi_{\theta}(\cdot,\cdot)$ is geometrically ergodic for each $\theta\in\mathbb{R}^{d_{\theta}}$, 
(ii) the convergence rate of $\Pi_{\theta}^{n}(\cdot,\cdot)$ is locally uniform in $\theta$, 
and 
(iii) $\Pi_{\theta}(\cdot,\cdot)$ is locally Lipschitz continuous in $\theta$ on $\mathbb{R}^{d_{\theta}}$ 
(for further details see, \cite[Chapter II.2]{benveniste}, 
\cite[Chapter 17]{meyn&tweedie} and references cited therein). 
Assumptions \ref{a2.2} and \ref{a2.3} have been introduced by 
M{\'e}tivier and Priouret in \cite{metivier&priouret1} 
(see also \cite[Part II]{benveniste}), 
and later generalized by Kushner and his co-workers 
(see \cite{kushner&yin} and references cited therein). 
However, none of these results cover the scenario where biased gradient estimates are used. Theorem \ref{theorem2.1} fills this gap 
in the literature on stochastic optimization and stochastic approximation.

Regarding Theorem \ref{theorem2.1}, 
the following note is in order.  
As already mentioned in the beginning of the section, 
the purpose of the theorem 
is illustrating the results of Section \ref{section1} 
and providing a framework for studying the examples 
presented in the next few sections. 
Since these examples perfectly fit into the framework 
developed by Metivier and Priouret, 
more general assumptions and settings of 
\cite{kushner&yin} are not considered here 
in order to keep the exposition as concise as possible. 

\section{Example 1: Reinforcement Learning}\label{section3} 

In this section, Theorems \ref{theorem1.1} and \ref{theorem2.1} are applied 
to the asymptotic analysis of policy-gradient search for average-cost Markov decision problems. 
Policy-gradient search is one of the most important classes of 
reinforcement learning  
(for further details, see e.g., \cite{bertsekas&tsitsiklis1}, \cite{powell}). 

In order to define controlled Markov chains with parametrized randomized control 
and to formulate the corresponding average-cost decision problems,  
we use the following notation. 
$d_{\theta }\geq 1$, $N_{x} > 1$, $N_{y} > 1$ are integers, 
while ${\cal X}$, ${\cal Y}$ are the sets 
\begin{align*}
	{\cal X} = \{1,\dots, N_{x} \}, 
	\;\;\;\;\; 
	{\cal Y} = \{1,\dots, N_{y} \}. 
\end{align*}
$\phi(x,y)$ is a non-negative (real-valued) function of $(x,y) \in {\cal X} \times {\cal Y}$. 
$p(x'|x,y)$ and $q_{\theta}(y|x)$ are non-negative (real-valued) functions of 
$(\theta,x,x',y) \in \mathbb{R}^{d_{\theta} }\times{\cal X}\times{\cal X}\times{\cal Y}$
with the following properties: 
$q_{\theta}(y|x)$ is differentiable in $\theta$ for each $\theta\in\mathbb{R}^{d_{\theta} }$, 
$x\in{\cal X}$, $y\in{\cal Y}$, 
and
\begin{align*}
	\sum_{x'\in{\cal X} } p(x'|x,y) = 1, 
	\;\;\;\;\; 
	\sum_{y'\in{\cal Y} } q_{\theta}(y'|x) = 1
\end{align*}
for the same $\theta$, $x$, $y$. 
For $\theta\in \mathbb{R}^{d_{\theta } }$, 
$\{(X_{n}^{\theta }, Y_{n}^{\theta } ) \}_{n\geq 0}$
is an ${\cal X} \times {\cal Y}$-valued Markov chain 
which is defined on a (canonical) probability space 
$(\Omega, {\cal F}, P_{\theta } )$
and which admits  
\begin{align*}
	P_{\theta }(X_{n+1}^{\theta } = x', Y_{n+1}^{\theta } = y'
	|X_{n}^{\theta } = x, Y_{n}^{\theta } = y )
	=
	q_{\theta }(y'|x') p(x'|x,y)
\end{align*}
for each $x,x' \in {\cal X}$, $y,y' \in {\cal Y}$. 
$f(\cdot )$ is a function defined by 
\begin{align}\label{3.501}
	f(\theta )
	=
	\lim_{n\rightarrow \infty }
	E_{\theta }\left(\frac{1}{n} \sum_{i=1}^{n} \phi(X_{i}^{\theta }, Y_{i}^{\theta } ) \right)
\end{align}
for $\theta \in \mathbb{R}^{d_{\theta } }$. 
With this notation, 
an average-cost Markov decision problem with parameterized randomized control 
can be defined as the minimization of $f(\cdot )$. 
In the literature on reinforcement learning and operations research, 
$\{X_{n}^{\theta } \}_{n\geq 0}$ are referred to as a controlled Markov chain, 
while $\{Y_{n}^{\theta } \}_{n\geq 0}$ are called control actions. 
$p(x'|x,y)$ is referred to as the (chain) transition probability, 
while 
$q_{\theta }(y|x)$
is called the (control) action probability. 
$\theta$ is a parameter indexing the action probability. 
For further details on Markov decision processes, 
see \cite{bertsekas&tsitsiklis1}, \cite{powell}, and references cited therein. 

Since $f(\cdot )$ and its gradient rarely admit a close-form expression, 
$f(\cdot )$ is minimized using methods based on stochastic gradient search 
and Monte Carlo gradient estimation.   
Such a method can be derived as follows. 
Let 
\begin{align*}
	s_{\theta }(x,y)
	=
	\frac{\nabla_{\theta } q_{\theta }(y|x) }{q_{\theta }(y|x) } 
\end{align*}
for $\theta\in \mathbb{R}^{d_{\theta } }$, $x\in {\cal X}$, $y\in {\cal Y}$. 
If $\{ (X_{n}^{\theta }, Y_{n}^{\theta } ) \}_{n\geq 0}$ is geometrically ergodic, 
we have
\begin{align*}
	\nabla f(\theta )
	= &
	\lim_{n\rightarrow \infty } 
	E_{\theta }\left(
	\phi(X_{n}^{\theta }, Y_{n}^{\theta } )
	\sum_{i=0}^{n-1} s_{\theta}(X_{n-i}^{\theta }, Y_{n-i}^{\theta }) 
	\right) 
\end{align*}
(see the proof of Lemma \ref{lemma3.2} and in particular (\ref{l3.2.1'''})). 
Hence, 
quantity 
\begin{align*}
	\phi(X_{n}^{\theta }, Y_{n}^{\theta } )
	\sum_{i=0}^{n-1} 
	s_{\theta }(X_{n-i}^{\theta }, Y_{n-i}^{\theta } )
\end{align*}
is an asymptotically consistent estimator of $\nabla f(\theta )$. 
To reduce its variance (which is usually very large for $n\gg 1$), 
term $s_{\theta }(X_{n-i}^{\theta }, Y_{n-i}^{\theta } )$ is `discounted'
by $\lambda^{i}$, 
where $\lambda \in [0,1)$ is a constant referred to as the discounting factor. 
This leads to the following gradient estimator: 
\begin{align}\label{3.1'}
	\phi(X_{n}^{\theta }, Y_{n}^{\theta } )
	\sum_{i=0}^{n-1} 
	\lambda^{i} 
	s_{\theta }(X_{n-i}^{\theta }, Y_{n-i}^{\theta } ). 
\end{align}
Gradient estimator (\ref{3.1'}) is biased and its bias is of the order $O(1-\lambda)$
when $\lambda\rightarrow 1$
(see Lemma \ref{lemma3.2}). 
Combining gradient search with estimator (\ref{3.1'}), we get the policy-gradient algorithm proposed in 
\cite{baxter&bartlett}. 
This algorithm is defined by the following difference equations:  
\begin{align}
	&\label{3.1}
	W_{n+1} 
	= 
	\lambda W_{n} + s_{\theta_{n} }(X_{n+1}, Y_{n+1} ), 
	\nonumber\\
	&
	\theta_{n+1} 
	=
	\theta_{n} 
	-
	\alpha_{n} \phi(X_{n+1}, Y_{n+1} ) 
	W_{n+1}, 
	\;\;\; n\geq 0. 
\end{align}
In the recursion (\ref{3.1}), 
$\{\alpha_{n} \}_{n\geq 0}$ is a sequence of positive reals, 
while  
$\theta_{0}, W_{0} \in \mathbb{R}^{d_{\theta } }$ are any (deterministic) vectors. 
$\{X_{n} \}_{n\geq 1}$ and $\{Y_{n} \}_{n\geq 1}$ 
are ${\cal X}$ and ${\cal Y}$ valued stochastic processes (respectively) 
generated 
through the following Monte Carlo simulations: 
\begin{align}\label{3.301}
	X_{n+1} &|\theta_{n}, X_{n}, Y_{n}, \dots, \theta_{0}, X_{0}, Y_{0} 
	\sim
	p(\cdot|X_{n},Y_{n} ), 
	\nonumber\\
	Y_{n+1} &|X_{n+1}, \theta_{n}, X_{n}, Y_{n}, \dots, \theta_{0}, X_{0}, Y_{0} 
	\sim
	q_{\theta_{n} }(\cdot|X_{n+1} ), 
	\;\;\;\;\; n\geq 0, 
\end{align}
where $X_{0}\in{\cal X}$, $Y_{0}\in{\cal Y}$ are deterministic quantities.\footnote 
{In (\ref{3.301}), $X_{n+1}$ is simulated from 
$p(\cdot|X_{n}, Y_{n} )$ independently of 
$\theta_{n}, \theta_{n-1}, X_{n-1}, Y_{n-1}, \dots, \theta_{0}, X_{0}, Y_{0}$, 
while $Y_{n+1}$ is simulated from $q_{\theta_{n} }(\cdot|X_{n+1} )$ independently of 
$X_{n}, Y_{n}, \theta_{n-1}, X_{n-1}, Y_{n-1}, \dots, \theta_{0}, X_{0}, Y_{0}$. }
Hence, $\{ (X_{n}, Y_{n} ) \}_{n\geq 1}$ satisfies 
\begin{align*}
	P(X_{n+1} = x, Y_{n+1} = y 
	|\theta_{n}, X_{n}, Y_{n}, \dots, \theta_{0}, X_{0}, Y_{0} )
	=
	q_{\theta_{n} }(y|x) p(x|X_{n}, Y_{n} )
\end{align*}
for all $x\in {\cal X}$, $y\in {\cal Y}$, $n\geq 1$. 

Algorithm (\ref{3.1}) is analyzed under the following assumptions. 

\begin{assumption} \label{a3.1}
For all $\theta\in\mathbb{R}^{d_{\theta } }$, $\{X_{n}^{\theta } \}_{n\geq 0}$ is irreducible and aperiodic. 
\end{assumption}

\begin{assumption} \label{a3.2}
For all $\theta\in\mathbb{R}^{d_{\theta } }$, $x\in{\cal X}$, $y\in{\cal Y}$, 
$s_{\theta}(x,y)$ is well-defined (and finite). 
Moreover, for each $x\in{\cal X}$, $y\in{\cal Y}$, 
$s_{\theta}(x,y)$ is locally Lipschitz continuous in $\theta$ 
on $\mathbb{R}^{d_{\theta } }$. 
\end{assumption}

\addtocounter{assumption}{1}

\begin{assumptionabc} \label{a3.3.a}
For each $x\in{\cal X}$, $y\in{\cal Y}$, 
$q_{\theta}(y|x)$ is $p$-times differentiable in $\theta$ 
on $\mathbb{R}^{d_{\theta } }$, 
where $p>d_{\theta }$. 
\end{assumptionabc}

\begin{assumptionabc} \label{a3.3.b}
For each $x\in{\cal X}$, $y\in{\cal Y}$, 
$q_{\theta}(y|x)$ is real-analytic in $\theta$ 
on $\mathbb{R}^{d_{\theta } }$. 
\end{assumptionabc}

Assumption \ref{a3.1} is related to the stability of 
the controlled Markov chain $\{X_{n}^{\theta } \}_{n\geq 0}$. 
In this or similar form, it is often involved in the asymptotic analysis of reinforcement learning algorithms
(see e.g., \cite{bertsekas&tsitsiklis1}, \cite{powell}). 
Assumptions \ref{a3.2}, \ref{a3.3.a} and \ref{a3.3.b} correspond to the parameterization of 
the action probabilities $q_{\theta }(y|x)$. 
They are satisfied for many commonly used parameterizations 
(such as natural, exponential and trigonometric).  

Let ${\cal S}$ and $f({\cal S} )$ have the same meaning as in (\ref{1.21})
($f(\cdot )$ is now defined in (\ref{3.501})), 
while ${\cal R}$ is the set of chain-recurrent points of 
the ODE $d\theta/dt=-\nabla f(\theta )$ (for details on chain-recurrence, 
see Section \ref{section1}). 
Moreover, for a compact set $Q\subset\mathbb{R}^{d_{\theta} }$, 
let $\Lambda_{Q}$ have the same meaning as in (\ref{1.103}). 
Then, our results on the asymptotic behavior of the recursion (\ref{3.1}) read as follows. 

\begin{theorem}\label{theorem3.1}
Suppose that Assumptions \ref{a2.1}, \ref{a3.1} and \ref{a3.2} hold. 
Let $Q\subset\mathbb{R}^{d_{\theta } }$ be any compact set. 
Then, the following is true: 
\begin{compactenum}[(i)]
\item
There exists a (deterministic) non-decreasing function 
$\psi_{Q}:[0,\infty )\rightarrow[0,\infty)$ 
(independent of  $\lambda$ and depending only on $\phi(x,y)$, $p(x'|x,y)$, $q_{\theta}(y|x)$)
such that 
$\lim_{t\rightarrow 0} \psi_{Q}(t) = \psi_{Q}(0) = 0$ 
and 
\begin{align*}
	\limsup_{n\rightarrow\infty } d(\theta_{n}, {\cal R} )
	\leq 
	\psi_{Q}(1-\lambda)
\end{align*}
almost surely on $\Lambda_{Q}$. 
\item
If (in addition to Assumptions \ref{a2.1}, \ref{a3.1} and \ref{a3.2}) 
Assumption \ref{a3.3.a} is satisfied, 
there exists a real number $K_{Q}\in (0,\infty )$
(independent of $\lambda$ and depending only on $\phi(x,y)$, $p(x'|x,y)$, $q_{\theta}(y|x)$)
such that 
\begin{align*}
	&
	\limsup_{n\rightarrow\infty } \|\nabla f(\theta_{n} ) \| 
	\leq 
	K_{Q} (1-\lambda)^{q/2}, 
	\;\;\;\;\; 
	\limsup_{n\rightarrow\infty } f(\theta_{n} ) 
	-
	\liminf_{n\rightarrow\infty } f(\theta_{n} ) 
	\leq 
	K_{Q} (1-\lambda)^{q}
\end{align*}
almost surely on $\Lambda_{Q}$, 
where $q=(p-d_{\theta} )/(p-1)$. 
\item
If (in addition to Assumptions \ref{a2.1}, \ref{a3.1} and \ref{a3.2})
Assumption \ref{a3.3.b} is satisfied, 
there exist real numbers $r_{Q}\in (0,1)$, $L_{Q}\in (0,\infty )$
(independent of $\lambda$ and depending only on $\phi(x,y)$, $p(x'|x,y)$, $q_{\theta}(y|x)$)
such that 
\begin{align*}
	&
	\limsup_{n\rightarrow\infty } \|\nabla f(\theta_{n} ) \| 
	\leq 
	L_{Q} (1\!-\!\lambda)^{1/2},  
	\;\;\; 
	\limsup_{n\rightarrow\infty } d(f(\theta_{n} ), f({\cal S} ) ) 
	\leq 
	L_{Q} (1\!-\!\lambda), 
	\;\;\; 
	\limsup_{n\rightarrow\infty } d(\theta_{n}, {\cal S} )
	\leq 
	L_{Q} (1\!-\!\lambda)^{r_{Q} }
\end{align*}
almost surely on $\Lambda_{Q}$. 
\end{compactenum}
\end{theorem}

Theorem \ref{theorem3.1} is proved in Section \ref{section3*}. 

\begin{vremark}
Function $\psi_{Q}(\cdot )$ depends on $\phi(x,y)$, $p(x'|x,y)$, $q_{\theta }(y|x)$ through
function $f(\cdot )$ (defined in (\ref{3.501})) and its properties
(see Remark \ref{remark1.1} for details). 
Function $\psi_{Q}(\cdot )$ also depends on $p(x'|x,y)$, $q_{\theta }(y|x)$ 
through the ergodicity properties of $\{ (X_{n}^{\theta }, Y_{n}^{\theta} ) \}_{n\geq 0}$
(see Lemma \ref{lemma3.1}). 
In addition to this, $\psi_{Q}(\cdot )$ depends on $\phi(x,y)$, $q_{\theta }(y|x)$ through
upper bounds of $|\phi(x,y)|$, $\|s_{\theta}(x,y) \|$. 
Further details can be found in the proofs of Lemmas \ref{lemma3.1}, \ref{lemma3.2} 
and Theorem \ref{theorem3.1} (Section \ref{section3*}). 
\end{vremark}

\begin{vremark}
As $\psi_{Q}(\cdot )$, constants $K_{Q}$ and $L_{Q}$ depend on 
$\phi(x,y)$, $p(x'|x,y)$, $q_{\theta }(y|x)$ through
function $f(\cdot )$ (defined in (\ref{3.501})) and its properties
(see Remark \ref{remark1.2} for details). 
$K_{Q}$ and $L_{Q}$ also depend on 
$\phi(x,y)$, $p(x'|x,y)$, $q_{\theta }(y|x)$ through 
the ergodicity properties of $\{ (X_{n}^{\theta }, Y_{n}^{\theta} ) \}_{n\geq 0}$. 
In addition to this, $K_{Q}$ and $L_{Q}$ depend on 
$\phi(x,y)$, $p(x'|x,y)$, $q_{\theta }(y|x)$ through 
upper bounds of $|\phi(x,y)|$, $\|s_{\theta}(x,y) \|$. 
For further details, see the proofs of Lemmas \ref{lemma3.1}, \ref{lemma3.2} 
and Theorem \ref{theorem3.1} (Section \ref{section3*}).
\end{vremark}

Although gradient search with  
`discounted' gradient estimation (\ref{3.1'}) is widely used 
in reinforcement learning 
(besides policy-gradient search, temporal-difference and actor-critic learning
also rely on the same approach), 
the available literature does not give a quite satisfactory answer 
to the problem of its asymptotic behavior. 
To the best of the present authors' knowledge, 
the existing results do not offer even the guarantee 
that the asymptotic bias of recursion (\ref{3.1}) goes to zero 
as $\lambda\rightarrow 1$
(i.e., that 
$\{\theta_{n} \}_{n\geq 0}$ converges to a vicinity of 
${\cal S}$ whose radius tends to zero as $\lambda\rightarrow 1$).\footnote{
Paper \cite{konda&tsitsiklis} can be considered as the strongest result 
on the asymptotic behavior of reinforcement learning with 
`discounted' gradient estimation. 
However, \cite{konda&tsitsiklis} only claims that a subsequence of 
$\{\theta_{n} \}_{n\geq 0}$ converges to a vicinity to ${\cal S}$
whose radius goes to zero as $\lambda\rightarrow 1$. 
} 
The main difficulty stems from the fact that 
reinforcement learning algorithms are so complex 
that the existing asymptotic results for biased stochastic gradient search 
and biased stochastic approximation \cite[Section 5.3]{borkar}, \cite{chen1}, \cite{chen2}, 
\cite[Section 2.7]{chen3}
cannot be applied. 
Relying on the results presented in Sections \ref{section1} and \ref{section2}, 
Theorem \ref{theorem3.1} overcomes these difficulties. 
Under mild and easily verifiable conditions, 
Theorem \ref{theorem3.1} guarantees 
that the asymptotic bias of algorithm (\ref{3.1}) converges 
to zero as $\lambda\rightarrow 1$ (Part (i)). 
Theorem \ref{theorem3.1} also provides relatively tight polynomial bounds on the rate
at which the bias goes to zero (Parts (ii), (iii)). 
In addition to this, Theorem \ref{theorem3.1} can be extended to other reinforcement learning algorithms 
such as temporal-difference and actor-critic learning. 

\section{Example 2: Adaptive Monte Carlo Sampling}\label{section5}

In this section, Theorems \ref{theorem1.1} and \ref{theorem2.1} are used to analyze the
asymptotic behavior of adaptive population Monte Carlo methods. 

In order to describe the population Monte Carlo methods 
and explain how their performance can adaptively be improved,
we use the following notation. 
$d_{\theta}\geq 1$, $d_{x}\geq 1$, $N>1$ are integers.
$\Theta\subseteq\mathbb{R}^{d_{\theta } }$ is an open set, 
while ${\cal X}\subseteq\mathbb{R}^{d_{x} }$ is a Borel-set. 
$p(x)$ is a probability density on ${\cal X}$, 
while $q(x)$ is a non-negative function proportional to $p(\cdot )$
(i.e., $p(x)\geq 0$, $q(x)\geq 0$, $p(x)=q(x)/\int_{\cal X} q(x') dx'$ for all $x\in{\cal X}$). 
$p_{\theta }(x'|x)$ is a non-negative (real-valued) function of 
$(\theta,x,x')\in \Theta\times{\cal X}\times{\cal X}$ which satisfies 
$\int_{\cal X} p_{\theta }(x'|x) dx' = 1$ for all $\theta\in\Theta$, $x,\in{\cal X}$
(notice that $p_{\theta }(\cdot|x)$ is a transition density on ${\cal X}$). 
$w_{\theta }(x,\tilde{x})$ is the function defined by 
\begin{align*}
	w_{\theta }(x,\tilde{x})
	=
	\frac{q(\tilde{x} ) }{p_{\theta }(\tilde{x}|x) }
\end{align*}
for $\theta\in\Theta$, $x,\tilde{x}\in{\cal X}$. 
$\tilde{r}_{N,\theta }(\cdot|x_{1:N} )$ is the transition density on ${\cal X}^{N}$ defined as
\begin{align*}
	\tilde{r}_{N,\theta }(\tilde{x}_{1:N}|x_{1:N} )
	=
	\prod_{i=1}^{N} p_{\theta }(\tilde{x}_{i}|x_{i} )
\end{align*}
for $\theta\in\Theta$, $x_{1:N}=(x_{1},\dots,x_{N} )\in {\cal X}^{N}$, 
$\tilde{x}_{1:N}=(\tilde{x}_{1},\dots,\tilde{x}_{N} )\in {\cal X}^{N}$. 
$R_{N,\theta }(\cdot|x_{1:N}, \tilde{x}_{1:N} )$ is the probability measure on ${\cal X}^{N}$ defined by 
\begin{align*}
	R_{N,\theta }(B|x_{1:N}, \tilde{x}_{1:N} )
	=
	\int_{\cal X}\cdots\int_{\cal X} 
	I_{B}(x'_{1}, \dots, x'_{N} ) 
	\prod_{i=1}^{N}
	\left(
	\frac{\sum_{j=1}^{N} w_{\theta }(x_{j},\tilde{x}_{j} ) \delta_{\tilde{x}_{j} }(dx'_{i} ) }
	{\sum_{j=1}^{N} w_{\theta }(x_{j},\tilde{x}_{j} ) }
	\right)
\end{align*}
for Borel-set $B \subseteq {\cal X}^{N}$ 
and $\theta\in\Theta$, $x_{1:N}=(x_{1},\dots,x_{N} )\in {\cal X}^{N}$, 
$\tilde{x}_{1:N}=(\tilde{x}_{1},\dots,\tilde{x}_{N} )\in {\cal X}^{N}$ 
($\delta_{\tilde{x} }(\cdot )$ represents the Dirac measure centered at $\tilde{x}$). 

Population Monte Carlo method is a method for simulating samples from 
$p(x)$ in a situation when only $q(x)$ is available 
(i.e., when $p(x)$ is known up to a normalizing constant). 
Population Monte Carlo method generates sequences of random 
variables 
$\{(X_{n}^{\theta }(1), \dots, X_{n}^{\theta }(N) ) \}_{n\geq 0}$, 
$\{\tilde{X}_{n}^{\theta }(1), \dots, \tilde{X}_{n}^{\theta }(N) ) \}_{n\geq 0}$
(defined on a canonical probability space $(\Omega,{\cal F}, P_{\theta } )$)
using the following sampling importance-resampling scheme: 
\begin{align}
	\label{5.101}
	\tilde{X}_{n+1}^{\theta }
	&|
	X_{n}^{\theta }, \tilde{X}_{n}^{\theta }, \dots, X_{0}^{\theta }, \tilde{X}_{0}^{\theta }
	\sim
	\tilde{r}_{N,\theta }(\cdot|X_{n}^{\theta } )
	\\
	\label{5.103}
	X_{n+1}^{\theta }
	&|
	\tilde{X}_{n+1}^{\theta }, 
	X_{n}^{\theta }, \tilde{X}_{n}^{\theta }, \dots, X_{0}^{\theta }, \tilde{X}_{0}^{\theta }
	\sim
	R_{N,\theta }(\cdot|X_{n}^{\theta }, \tilde{X}_{n+1}^{\theta } ), 
	\;\;\;\;\; n\geq 0, 
\end{align}
where $X_{n}^{\theta } = (X_{n}^{\theta }(1),\dots,X_{n}^{\theta }(N) )$, 
$\tilde{X}_{n}^{\theta } = (\tilde{X}_{n}^{\theta }(1),\dots,\tilde{X}_{n}^{\theta }(N) )$, 
while $X_{0}^{\theta }, \tilde{X}_{0}^{\theta } \in {\cal X}^{N}$ are any deterministic vectors.\footnote
{At the sampling step (\ref{5.101}), for each $1\leq i\leq N$, 
$\tilde{X}_{n+1}^{\theta }(i)$ is sampled from $p_{\theta }(\cdot|X_{n}^{\theta }(i) )$. 
At the resampling step (\ref{5.103}), for each $1\leq i\leq N$, 
random integer $I_{n}^{\theta }(i)$ is sampled proportionally from 
$\big(w_{\theta }(X_{n}^{\theta }(1), \tilde{X}_{n}^{\theta }(1) ), \cdots, 
w_{\theta }(X_{n}^{\theta }(N), \tilde{X}_{n}^{\theta }(N) ) \big)$
and then, random number $X_{n+1}^{\theta }(i)$ is selected according to 
$X_{n+1}^{\theta }(i)=\tilde{X}_{n+1}^{\theta }(I_{n+1}^{\theta }(i) )$. }
For further details on the population Monte Carlo method, see e.g., \cite{cappe1}, 
\cite{robert&casella} and references cited therein. 

In order to improve the performance of the population Monte Carlo method, 
parameter $\theta$ is selected so as the Kullback-Leibler distance between 
$p(x')p(x)$ and $p(x'|x)p(x)$ is minimum.  
Hence, $\theta$ minimizes 
\begin{align*}
	\int_{\cal X} \int_{\cal X} \log\left(\frac{p(x') }{p_{\theta }(x'|x) } \right) p(x')p(x) dx'dx
\end{align*}
on $\Theta$. 
It is straightforward to show that $\theta$ also minimizes 
\begin{align}\label{5.701}
	f(\theta )
	=
	-
	\int_{\cal X} \int_{\cal X} \log\left(p_{\theta }(x'|x) \right) p(x')p(x) dx'dx 
\end{align}
on $\Theta$. 
As $f(\cdot )$ and its gradient are not available analytically, 
$f(\cdot )$ is minimized using stochastic gradient search 
and Monte Carlo gradient estimation 
(or their variants such as online EM algorithm). 
$\nabla f(\cdot )$ can be estimated by the quantity
\begin{align}\label{5.105}
	-\frac{1}{N}
	\sum_{i=1}^{N} s_{\theta }(X_{n}^{\theta }(i), X_{n+1}^{\theta }(i) ), 
\end{align}
where $s_{\theta }(x,x')$ is the function defined by 
\begin{align}
	s_{\theta }(x,x')
	=
	\frac{\nabla_{\theta } p_{\theta }(x'|x) }{p_{\theta }(x'|x) }
\end{align}
for $\theta\in\Theta$, $x,x'\in{\cal X}$. 
Estimator (\ref{5.105}) is biased, and its bias is of the order $O(1/N)$ when $N\rightarrow\infty$
(see Lemma \ref{lemma5.2}). 
Combining gradient search with estimator (\ref{5.105}), we get an adaptive population Monte Carlo method. 
This method is defined by the following difference equation: 
\begin{align}\label{5.1}
	\theta_{n+1}
	=
	\theta_{n}
	+
	\frac{\alpha_{n} }{N}
	\sum_{i=1}^{N} s_{\theta_{n} }(X_{n}(i), X_{n+1}(i) ), 
	\;\;\;\;\; n\geq 0. 
\end{align}
In the recursion (\ref{5.1}), 
$\{\alpha_{n} \}_{n\geq 0}$ is a sequence of positive reals, 
while  
$\theta_{0}\in\Theta$ is any (deterministic) vector. 
$\left\{\left( X_{n}(1),\dots,X_{n}(N) \right) \right\}_{n\geq 1}$ is 
an ${\cal X}^{N}$-valued stochastic process generated 
through the following Monte Carlo simulations: 
\begin{align}\label{5.3}
	\tilde{X}_{n+1}
	&|
	X_{n}, \tilde{X}_{n}, \dots, X_{0}, \tilde{X}_{0}
	\sim
	\tilde{r}_{N,\theta_{n} }(\cdot|X_{n} )
	\\
	X_{n+1}
	&|
	\tilde{X}_{n+1}, 
	X_{n}, \tilde{X}_{n}, \dots, X_{0}, \tilde{X}_{0}
	\sim
	R_{N,\theta_{n} }(\cdot|X_{n}, \tilde{X}_{n+1} ), 
	\;\;\;\;\; n\geq 0, 
\end{align}
where $X_{n} = (X_{n}(1),\dots,X_{n}(N) )$, 
$\tilde{X}_{n} = (\tilde{X}_{n}(1),\dots,\tilde{X}_{n}(N) )$, 
while $X_{0}, \tilde{X}_{0} \in {\cal X}^{N}$ are any deterministic vectors. 
For more details on adaptive population Monte Carlo methods, see e.g., \cite{cappe2}, \cite{douc}. 

\begin{vremark}
Recursion (\ref{5.3}) usually includes a projection (or truncation) scheme 
which keeps $\{\theta_{n} \}_{n\geq 0}$ within $\Theta$
(see \cite{ljung}). 
For the sake of exposition, this aspect of (\ref{5.3}) is not studied here. 
Instead, similarly as in \cite{benveniste} and \cite{ljung}, 
the asymptotic results are stated in a local form. 
\end{vremark}

Algorithm (\ref{5.1}) is analyzed under the following assumptions. 

\begin{assumption} \label{a5.1}
${\cal X}$ is compact. 
\end{assumption}

\begin{assumption} \label{a5.2}
$p(x)>0$ for all $x\in{\cal X}$. Moreover, $p(x)$ is continuous for each $x\in{\cal X}$. 
\end{assumption}

\begin{assumption} \label{a5.3}
$p_{\theta}(x'|x)>0$ for all $\theta\in\Theta$, $x,x'\in{\cal X}$. 
Moreover, 
$\nabla_{\theta } p_{\theta}(x'|x)$ is locally Lipschitz continuous in 
$(\theta,x,x')$ on $\Theta\times{\cal X}\times{\cal X}$. 
\end{assumption}

\addtocounter{assumption}{1}

\begin{assumptionabc} \label{a5.4.a}
For each $x,x'\in{\cal X}$, 
$p_{\theta}(x'|x)$ is $p$-times differentiable in $\theta$ 
on $\Theta$, 
where $p>d_{\theta }$. 
Moreover, the $p$-th order derivatives (in $\theta$) of $p_{\theta }(x'|x)$
are continuous in $(\theta,x,x')$ on $\Theta\times{\cal X}\times{\cal X}$. 
\end{assumptionabc}

\begin{assumptionabc} \label{a5.4.b}
For each $x,x'\in{\cal X}$, 
$p_{\theta}(x'|x)$ is real-analytic in $\theta$ 
on $\Theta$. 
Moreover, 
$p_{\theta }(x'|x)$ has 
a (complex-valued) continuation 
$\hat{p}_{\eta}(x'|x)$ with the following properties: 
\begin{compactenum}[(i)]
\item
$\hat{p}_{\eta}(x'|x)$ maps $(\eta,x,x')\in \mathbb{C}^{d_{\theta } } \times {\cal X} \times {\cal X}$ 
to $\mathbb{C}$. 
\item
$\hat{p}_{\theta}(x'|x) = p_{\theta}(x'|x)$
for all $\theta\in\Theta$, $x,x'\in {\cal X}$. 
\item
For any $\theta\in\Theta$, 
there exists a real number
$\delta_{\theta } \in (0,1)$ such that 
$\hat{p}_{\eta}(x'|x)$ is 
analytic in $\eta$ 
and 
continuous in $(\eta,x,x')$
for any 
$\eta\in\mathbb{C}^{d_{\theta } }$, $x,x'\in {\cal X}$ 
satisfying $\|\eta-\theta\|\leq\delta_{\theta }$. 
\end{compactenum}
\end{assumptionabc}

Assumptions \ref{a5.1} and \ref{a5.2} correspond to the target density $p(\cdot )$, 
while Assumptions \ref{a5.3}, \ref{a5.4.a} and \ref{a5.4.b} are related to the instrumental 
density $p_{\theta}(\cdot|\cdot )$. 
These assumptions are rather restrictive from the theoretical perspective, 
since they require $p(\cdot )$ and $p_{\theta}(\cdot|\cdot)$ to be compactly supported. 
We rely on such restrictive conditions for the sake of exposition. 
However, the results presented here can easily be extended to the case 
where ${\cal X}$ is unbounded 
and where the fourth moments of $p(\cdot )$ and $p_{\theta}(\cdot|\cdot )$ are finite. 

Let ${\cal S}$ and $f({\cal S} )$ have the same meaning as in (\ref{1.21})
($f(\cdot )$ is now defined in (\ref{5.701})), 
while ${\cal R}$ is the set of chain-recurrent points of 
the ODE $d\theta/dt=-\nabla f(\theta )$ (for details on chain-recurrence, 
see Section \ref{section1}). 
Moreover, for a compact set $Q\subset\mathbb{R}^{d_{\theta} }$, 
let $\Lambda_{Q}$ have the same meaning as in (\ref{1.103}). 
Then, our results on the asymptotic behavior of algorithm (\ref{5.1}) read as follows. 

\begin{theorem}\label{theorem5.1}
Suppose that Assumptions \ref{a2.1}, \ref{a5.1} and \ref{a5.2} hold. 
Let $Q\subset\Theta$ be any compact set. 
Then, the following is true: 
\begin{compactenum}[(i)]
\item
There exists a (deterministic) non-decreasing function 
$\psi_{Q}:[0,\infty )\rightarrow[0,\infty)$ 
(independent of  $N$ and depending only on 
$p_{\theta }(x'|x)$, $q(x)$)
such that 
$\lim_{t\rightarrow 0} \psi_{Q}(t) = \psi_{Q}(0) = 0$ 
and  
\begin{align*}
	\limsup_{n\rightarrow\infty } d(\theta_{n}, {\cal R} )
	\leq 
	\psi_{Q}\left(\frac{1}{N} \right)
\end{align*}
almost surely on $\Lambda_{Q}$. 
\item
If (in addition to Assumptions \ref{a2.1}, \ref{a5.1} and \ref{a5.2})
Assumption \ref{a5.4.a} is satisfied, 
there exists a real number $K_{Q}\in (0,\infty )$
(independent of $N$ and depending only on 
$p_{\theta }(x'|x)$, $q(x)$)
such that 
\begin{align*}
	&
	\limsup_{n\rightarrow\infty } \|\nabla f(\theta_{n} ) \| 
	\leq 
	\frac{K_{Q} }{N^{q/2} }, 
	\;\;\;\;\; 
	\limsup_{n\rightarrow\infty } f(\theta_{n} ) 
	-
	\liminf_{n\rightarrow\infty } f(\theta_{n} ) 
	\leq 
	\frac{K_{Q} }{N^{q} }
\end{align*}
almost surely on $\Lambda_{Q}$, 
where $q=(p-d_{\theta} )/(p-1)$. 
\item
If (in addition to Assumptions \ref{a2.1}, \ref{a5.1} and \ref{a5.2})
Assumption \ref{a5.4.b} is satisfied, 
there exist real numbers $r_{Q}\in (0,1)$, $L_{Q}\in (0,\infty )$
(independent of $N$ and depending only on 
$p_{\theta }(x'|x)$, $q(x)$)
such that 
\begin{align*}
	&
	\limsup_{n\rightarrow\infty } \|\nabla f(\theta_{n} ) \| 
	\leq 
	\frac{L_{Q} }{N^{1/2} }, 
	\;\;\;\;\; 
	\limsup_{n\rightarrow\infty } d(f(\theta_{n} ), f({\cal S} ) ) 
	\leq 
	\frac{L_{Q} }{N}, 
	\;\;\;\;\; 
	\limsup_{n\rightarrow\infty } d(\theta_{n}, {\cal S} )
	\leq 
	\frac{L_{Q} }{N^{r_{Q} } }
\end{align*}
almost surely on $\Lambda_{Q}$. 
\end{compactenum}
\end{theorem}

Theorem \ref{theorem5.1} is proved in Section \ref{section5*}. 

\begin{vremark}
Function $\psi_{Q}(\cdot )$ depends on $p_{\theta }(x'|x)$, $q(x)$ through
function $f(\cdot )$ (defined in (\ref{5.701})) and its properties
(see Remark \ref{remark1.1} for details). 
Function $\psi_{Q}(\cdot )$ also depends on $p_{\theta }(x'|x)$, $q(x)$ 
through lower bounds of $p_{\theta }(x'|x)$, $q(x)$
and upper bounds of $p_{\theta }(x'|x)$, $q(x)$, $\|\nabla_{\theta } p_{\theta }(x'|x) \|$. 
Further details can be found in the proofs of Lemma \ref{lemma5.2} (Part (ii))
and Theorem \ref{theorem5.1} (Section \ref{section5*}). 
\end{vremark}

\begin{vremark}
As $\psi_{Q}(\cdot )$, constants $K_{Q}$ and $L_{Q}$ depend on 
$\phi(x,y)$, $p(x'|x,y)$, $q_{\theta }(y|x)$ through
function $f(\cdot )$ (defined in (\ref{5.701})) and its properties
(see Remark \ref{remark1.2} for details). 
$K_{Q}$ and $L_{Q}$ also depend on 
$p_{\theta }(x'|x)$, $q(x)$ 
through lower bounds of $p_{\theta }(x'|x)$, $q(x)$
and upper bounds of $p_{\theta }(x'|x)$, $q(x)$, $\|\nabla_{\theta } p_{\theta }(x'|x) \|$. 
For further details, see the proofs of Lemma \ref{lemma5.2} (Part (ii))
and Theorem \ref{theorem5.1} (Section \ref{section5*}).
\end{vremark}

Population Monte Carlo methods have been proposed and studied in \cite{cappe1}, 
while their adaptive versions have been developed and analyzed in \cite{cappe2}, \cite{douc}. 
Although based on the same principle as (\ref{5.3}) 
(minimization of function $f(\cdot )$), 
the adaptive methods considered in \cite{cappe2}, \cite{douc} 
compute optimal values of $\theta$ using iterative techniques 
(slightly) different from stochastic gradient search
(i.e., using EM algorithm). 
Unfortunately, unless $f(\cdot )$ is convex, 
\cite{cappe2}, \cite{douc} do not offer much information on the asymptotic behavior of 
$\{\theta_{n} \}_{n\geq 0}$.\footnote
{The results of \cite{cappe2}, \cite{douc} are focused only on the case
where $p_{\theta}(\cdot|\cdot)$ is a mixture of transition kernels parameterized by the mixture weights. } 
The purpose of Theorem \ref{theorem5.1} (besides illustrating Theorems \ref{theorem1.1}, \ref{theorem2.1}) 
is to fill this gap in the literature on population Monte Carlo methods. 

\section{Example 3: Identification of Hidden Markov Models}\label{section4}

In this section, Theorems \ref{theorem1.1} and \ref{theorem2.1}
are applied to the asymptotic analysis of 
recursive maximum split-likelihood methods for the identification of 
hidden Markov models. 

In order to define hidden Markov models and to formulate the problem of their identification, 
we use the following notation. 
$N_{x}>1$, $N_{y}>1$ are integers, 
while ${\cal X}$, ${\cal Y}$ are the sets 
\begin{align*}
	{\cal X} = \{1,\dots,N_{x} \}, 
	\;\;\;\;\; 
	{\cal Y} = \{1,\dots,N_{y} \}. 
\end{align*}
$p(x'|x)$ and $q(y|x)$ are non-negative (real-valued) functions of 
$(x,x',y) \in {\cal X} \times {\cal X} \times {\cal Y}$
satisfying  
\begin{align*}
	\sum_{x' \in {\cal X} } p(x'|x) = 1, 
	\;\;\;\;\; 
	\sum_{y \in {\cal Y} } q(y|x) = 1
\end{align*}
for each $x\in {\cal X}$. 
$\{ (X_{n}, Y_{n} ) \}_{n\geq 0}$ is an ${\cal X} \times {\cal Y}$-valued 
Markov chain 
which is defined on a (canonical) probability space 
$(\Omega, {\cal F}, P )$ and which admits  
\begin{align*}
	P(X_{n+1} = x', Y_{n+1} = y' |X_{n} = x, Y_{n} = y )
	=
	q(y'|x') p(x'|x)
\end{align*}
for all $x,x' \in {\cal X}$, $y,y' \in {\cal Y}$. 
On the other side, 
$d_{\theta } \geq 1$ is an integer, while $\Theta \subseteq R^{d_{\theta } }$ is an open set. 
$\pi_{\theta }(x)$,
$p_{\theta }(x'|x)$ and $q_{\theta }(y|x)$ are non-negative functions of
$(\theta,x,x',y) \in \Theta\times{\cal X}\times{\cal X}\times{\cal Y}$
with the following properties: 
They are differentiable in $\theta$ 
for each $\theta\in \Theta$, $x,x' \in {\cal X}$, $y \in {\cal Y}$
and satisfy 
\begin{align*}
	\sum_{x'\in {\cal X} } \pi_{\theta }(x') = 1, 
	\;\;\;\;\;
	\sum_{x' \in {\cal X} } p_{\theta }(x'|x) = 1,  
	\;\;\;\;\; 
	\sum_{y \in {\cal Y} } q_{\theta }(y|x) = 1
\end{align*}
for the same $\theta$, $x$. 
For $\theta\in\Theta$, 
$\{ (X_{n}^{\theta }, Y_{n}^{\theta } ) \}_{n\geq 0}$
is an ${\cal X} \times {\cal Y}$-valued Markov chain 
which is defined on a (canonical) probability space $(\Omega, {\cal F}, P_{\theta } )$
and which admits  
\begin{align*}
	&
	P_{\theta }(X_{0}^{\theta } = x, Y_{0}^{\theta } = y ) 
	=
	q_{\theta }(y|x) \pi_{\theta }(x), 
	\\
	&
	P_{\theta }(X_{n+1}^{\theta } = x', Y_{n+1}^{\theta } = y'
	|X_{n}^{\theta } = x, Y_{n}^{\theta } = y )
	= 
	q_{\theta }(y'|x') p_{\theta }(x'|x)
\end{align*}
for all $x,x' \in {\cal X}$, $y,y' \in {\cal Y}$, $n\geq 0$. 
$\phi_{N,\theta }(y_{1:N} )$ is the function defined as  
\begin{align*}
	\phi_{N,\theta }(y_{1:N} )
	= &
	-
	\frac{1}{N}
	\log\left(
	\sum_{x_{0},\dots,x_{N} \in {\cal X} } 
	\left(
	\prod_{i=1}^{N}
	\left(
	q_{\theta}(y_{i}|x_{i} ) p_{\theta }(x_{i} |x_{i-1} ) 
	\right)
	\right)
	\pi_{\theta }(x_{0} ) 
	\right), 
\end{align*}
for $\theta\in\Theta$, $y_{1:N} = (y_{1}, \dots, y_{N} ) \in {\cal Y}^{N}$, $N\geq 1$. 
$f_{N}(\cdot )$ and $f(\cdot )$ are the functions defined by 
\begin{align}\label{4.701}
	&
	f_{N}(\theta ) 
	= 
	\lim_{n\rightarrow \infty } 
	E(\phi_{N,\theta }(Y_{nN+1:(n+1)N } ) ), 
	\;\;\;\;\; 
	f(\theta ) 
	=
	\lim_{N\rightarrow\infty }
	E(\phi_{N,\theta }(Y_{1:N} ) )
\end{align}
for $\theta\in\Theta$, $N\geq 1$ 
(here, $Y_{nN+1:(n+1)N}$ stands for $(Y_{nN+1}, \dots, Y_{(n+1)N} )$). 
Then, it is straightforward to show that $f(\cdot )$ is 
the negative (asymptotic) log-likelihood associated with $\{Y_{n} \}_{n\geq 0}$. 

In the statistics and engineering literature, 
$\{ (X_{n}, Y_{n} ) \}_{n\geq 0}$  
(as well as 
$\{ (X_{n}^{\theta }, Y_{n}^{\theta } ) \}_{n\geq 0}$)
is known as a hidden Markov model, 
while 
$X_{n}$ and $Y_{n}$ are 
the model's (unobservable) state and (observable) output at discrete-time $n$. 
Using the notation introduced in this section, 
the identification of $\{ (X_{n}, Y_{n} ) \}_{n\geq 0}$
can be stated as follows: 
Given a realization of the output sequence $\{Y_{n} \}_{n\geq 0}$, estimate 
$\{p(x'|x) \}_{x,x' \in {\cal X} }$ and
$\{q(y|x) \}_{x \in {\cal X}, y \in {\cal Y} }$. 
If the identification is based on the maximum likelihood principle
and the parameterized model 
$\{ (X_{n}^{\theta}, Y_{n}^{\theta} ) \}_{n\geq 0}$, 
the estimation reduces to the minimization of $f(\cdot )$
over $\Theta$. 
In this context, 
$\{ (X_{n}^{\theta }, Y_{n}^{\theta } ) \}_{n\geq 0}$
can be considered as a candidate model for 
the unknown system $\{ (X_{n}, Y_{n} ) \}_{n\geq 0}$. 
For more details on hidden Markov models and their identification see 
\cite{cappe&moulines&ryden} and references cited therein. 

As the negative log-likelihood $f(\cdot )$ and its gradient are rarely 
available analytically, 
$f(\cdot )$ is usually minimized by stochastic gradient search. 
Unfortunately, the consistent estimation of $\nabla f(\cdot )$
is computationally expensive (even for moderately large $N_{x}$, $N_{y}$), 
since it relies on the optimal filter and the optimal filter derivatives 
(see e.g., \cite{cappe&moulines&ryden}). 
To reduce the computational complexity, 
a number of approaches based on approximate maximum likelihood (also known 
as pseudo-likelihood) has been proposed. 
Among them, 
the maximum split-likelihood method 
\cite{ryden1}, \cite{ryden2} has attracted a considerable attention in 
the literature. 
This approach is based on the following fact. 
If $\{X_{n} \}_{n\geq 0}$ is geometrically ergodic and 
if the optimal filter for the candidate model 
$\{ (X_{n}^{\theta }, Y_{n}^{\theta } ) \}_{n\geq 0}$ is stable, 
then 
\begin{align*}
	\nabla f(\theta ) 
	=
	\lim_{N\rightarrow\infty } \nabla f_{N}(\theta )
	=
	\lim_{N\rightarrow \infty } \lim_{n\rightarrow \infty } 
	E\left(
	\nabla_{\theta } \phi_{N,\theta }(Y_{nN+1:(n+1)N} )
	\right)
\end{align*}
(see Lemma \ref{lemma4.1}). 
Hence, 
$\nabla_{\theta }\phi_{N,\theta }(Y_{nN+1:(n+1)N} )$
is a reasonably good estimator of $\nabla f(\theta )$ when $n,N$ are large. 
This estimator is biased and the bias is of the order $O(1/N)$ 
when $N\rightarrow\infty$ (see Lemma \ref{lemma4.1}). 
Combining gradient search with the estimator $\nabla_{\theta }\phi_{N,\theta }(Y_{nN+1:(n+1)N} )$, 
we get the recursive maximum split-likelihood algorithm: 
\begin{align}\label{4.1}	
	\theta_{n+1}
	=
	\theta_{n} 
	-
	\alpha_{n} 
	\psi_{N, \theta_{n} }(Y_{nN+1:(n+1)N} ), 
	\;\;\; n\geq 0. 
\end{align}
Here, $\{\alpha_{n} \}_{n\geq 0}$ is a sequence of positive real numbers, 
$N\geq 1$ is a fixed integer, 
and $\psi_{N,\theta }(\cdot ) = \nabla_{\theta } \phi_{N, \theta }(\cdot )$. 

\begin{vremark}
Since ${\cal X}$ is a finite set, 
$\psi_{N, \theta_{n} }(Y_{nN+1:(n+1)N} )$ can be computed exactly. 
When ${\cal X}$ has infinitely many elements, 
$\psi_{N,\theta_{n} }(Y_{nN+1:(n+1)N} )$ can accurately be approximated  
using Monte Carlo methods 
(for the developments of this kind see \cite{andrieu&doucet&tadic}). 
To avoid unnecessary technical details, 
we consider only the case when ${\cal X}$ is finite. 
\end{vremark}

\begin{vremark}
As (\ref{5.3}), recursion (\ref{4.1}) usually includes a projection (or truncation) scheme 
which keeps $\{\theta_{n} \}_{n\geq 0}$ within $\Theta$
(see \cite{ljung}). 
For the sake of exposition, this aspect of (\ref{4.1}) is not studied here. 
Instead, similarly as in \cite{benveniste} and \cite{ljung}, 
the asymptotic results are stated in a local form. 
\end{vremark}

Algorithm (\ref{4.1}) is analyzed under the following assumptions. 

\begin{assumption}\label{a4.1} 
$\{X_{n} \}_{n\geq 0}$ is geometrically ergodic. 
\end{assumption}

\begin{assumption}\label{a4.2}
$p_{\theta}(x'|x) > 0$ and 
$q_{\theta}(y|x) > 0$ 
for all $\theta\in\Theta$, $x,x'\in {\cal X}$, $y\in {\cal Y}$. 
\end{assumption}

\begin{assumption}\label{a4.3} 
For each $x,x'\in {\cal X}$, $y\in {\cal Y}$, 
$\nabla_{\theta} p_{\theta }(x'|x)$, 
$\nabla_{\theta} q_{\theta }(y|x)$ 
and $\nabla_{\theta } \pi_{\theta }(x)$
are locally Lipschitz continuous in $\theta$ on $\Theta$. 
\end{assumption}

\addtocounter{assumption}{1}

\begin{assumptionabc} \label{a4.4.a}
For all $x,x'\in{\cal X}$, $y\in{\cal Y}$, 
$p_{\theta}(x'|x)$ and $q_{\theta}(y|x)$ are $p$-times differentiable in $\theta$ 
on $\Theta$, 
where $p>d_{\theta }$. 
\end{assumptionabc}

\begin{assumptionabc} \label{a4.4.b}
For all $x,x'\in{\cal X}$, $y\in{\cal Y}$, 
$p_{\theta}(x'|x)$ and $q_{\theta}(y|x)$ are real-analytic in $\theta$ 
on $\Theta$. 
\end{assumptionabc}

Assumption \ref{a4.1} is related to the stability 
of the unknown system $\{ (X_{n}, Y_{n} ) \}_{n\geq 0}$, 
while Assumption \ref{a4.2} corresponds to 
the stability of the optimal filter associated with the candidate model 
$\{ (X_{n}^{\theta }, Y_{n}^{\theta } ) \}_{n\geq 0}$. 
In this or similar form, Assumptions \ref{a4.1} and \ref{a4.2} are involved in any asymptotic analysis 
of identification methods for hidden Markov models
(see e.g. \cite{cappe&moulines&ryden}, \cite{tadic5} and references cited therein). 
Assumptions \ref{a4.3}, \ref{a4.4.a} and \ref{a4.4.b} 
correspond to the parameterization of the 
candidate model $\{ (X_{n}^{\theta }, Y_{n}^{\theta } ) \}_{n\geq 0}$
and often hold in practice. 
For some commonly used parameterizations 
(such as natural, trigonometric and exponential), 
$p_{\theta }(x'|x)$, 
$q_{\theta }(y|x)$ and 
$\pi_{\theta }(x)$
are not only Lipschitz continuously differentiable in $\theta$, 
but also real-analytic
(see \cite{tadic5} for further details). 

Let ${\cal S}$ and $f({\cal S} )$ have the same meaning as in (\ref{1.21})
($f(\cdot )$ is now defined in (\ref{4.701})), 
while ${\cal R}$ is the set of chain-recurrent points of 
the ODE $d\theta/dt=-\nabla f(\theta )$ (for details on chain-recurrence, 
see Section \ref{section1}). 
Moreover, for a compact set $Q\subset\mathbb{R}^{d_{\theta} }$, 
let $\Lambda_{Q}$ have the same meaning as in (\ref{1.103}). 
Our results on the asymptotic behavior of algorithm (\ref{4.1}) read as follows. 

\begin{theorem}\label{theorem4.1}
Suppose that Assumptions \ref{a2.1} and \ref{a4.1} -- \ref{a4.3} hold. 
Let $Q\subset\Theta$ be any compact set. 
Then, the following is true: 
\begin{compactenum}[(i)]
\item
There exists a (deterministic) non-decreasing function 
$\psi_{Q}:[0,\infty )\rightarrow[0,\infty)$ 
(independent of  $N$ and depending only on 
$p_{\theta }(x'|x)$, $q_{\theta}(y|x)$, $\pi_{\theta }(x)$, $\{X_{n} \}_{n\geq 0}$)
such that 
$\lim_{t\rightarrow 0} \psi_{Q}(t) = \psi_{Q}(0) = 0$ 
and  
\begin{align*}
	\limsup_{n\rightarrow\infty } d(\theta_{n}, {\cal R} )
	\leq 
	\psi_{Q}\left(\frac{1}{N} \right)
\end{align*}
almost surely on $\Lambda_{Q}$. 
\item
If (in addition to Assumptions \ref{a2.1} and \ref{a4.1} -- \ref{a4.3})
Assumption \ref{a4.4.a} is satisfied, 
there exists a real number $K_{Q}\in (0,\infty )$
(independent of $N$ and depending only on 
$p_{\theta }(x'|x)$, $q_{\theta}(y|x)$, $\pi_{\theta }(x)$, $\{X_{n} \}_{n\geq 0}$)
such that 
\begin{align*}
	&
	\limsup_{n\rightarrow\infty } \|\nabla f(\theta_{n} ) \| 
	\leq 
	\frac{K_{Q} }{N^{q/2} }, 
	\;\;\;\;\; 
	\limsup_{n\rightarrow\infty } f(\theta_{n} ) 
	-
	\liminf_{n\rightarrow\infty } f(\theta_{n} ) 
	\leq 
	\frac{K_{Q} }{N^{q} }
\end{align*}
almost surely on $\Lambda_{Q}$, 
where $q=(p-d_{\theta} )/(p-1)$. 
\item
If (in addition to Assumptions \ref{a2.1} and \ref{a4.1} -- \ref{a4.3})
Assumption \ref{a4.4.b} is satisfied, 
there exist real numbers $r_{Q}\in (0,1)$, $L_{Q}\in (0,\infty )$
(independent of $N$ and depending only on 
$p_{\theta }(x'|x)$, $q_{\theta}(y|x)$, $\pi_{\theta }(x)$, $\{X_{n} \}_{n\geq 0}$)
such that 
\begin{align*}
	&
	\limsup_{n\rightarrow\infty } \|\nabla f(\theta_{n} ) \| 
	\leq 
	\frac{L_{Q} }{N^{1/2} }, 
	\;\;\;\;\; 
	\limsup_{n\rightarrow\infty } d(f(\theta_{n} ), f({\cal S} ) ) 
	\leq 
	\frac{L_{Q} }{N}, 
	\;\;\;\;\; 
	\limsup_{n\rightarrow\infty } d(\theta_{n}, {\cal S} )
	\leq 
	\frac{L_{Q} }{N^{r_{Q} } }
\end{align*}
almost surely on $\Lambda_{Q}$. 
\end{compactenum}
\end{theorem}

Theorem \ref{theorem4.1} is proved in Section \ref{section4*}. 

\begin{vremark}
Function $\psi_{Q}(\cdot )$ depends on 
$p_{\theta }(x'|x,y)$, $q_{\theta }(y|x)$, $\pi_{\theta}(x)$, $\{X_{n} \}_{n\geq 0}$ 
through function $f(\cdot )$ (defined in (\ref{4.701})) and its properties
(see Remark \ref{remark1.1} for details). 
Function $\psi_{Q}(\cdot )$ also depends on 
$p_{\theta }(x'|x,y)$, $q_{\theta }(y|x)$, $\pi_{\theta}(x)$, $\{X_{n} \}_{n\geq 0}$ 
through the ergodicity properties of the optimal filter 
(see (\ref{l4.1.5})). 
In addition to this, $\psi_{Q}(\cdot )$ depends on 
$p_{\theta }(x'|x,y)$, $q_{\theta }(y|x)$
through upper and lower bounds of $p_{\theta }(x'|x,y)$, $q_{\theta }(y|x)$
and Lipschitz constants of 
$p_{\theta }(x'|x,y)$, $q_{\theta }(y|x)$, 
$\nabla_{\theta } p_{\theta }(x'|x,y)$, $\nabla_{\theta } q_{\theta }(y|x)$
(see (\ref{l4.1.901}) -- (\ref{l4.1.905})). 
Further details can be found in the proofs of Lemma \ref{lemma4.1} 
and Theorem \ref{theorem4.1} (Section \ref{section4*}). 
\end{vremark}

\begin{vremark}
As $\psi_{Q}(\cdot )$, constants $K_{Q}$ and $L_{Q}$ depend on 
$p_{\theta }(x'|x,y)$, $q_{\theta }(y|x)$, $\pi_{\theta}(x)$, $\{X_{n} \}_{n\geq 0}$ 
through function $f(\cdot )$ (defined in (\ref{4.701})) and its properties
(see Remark \ref{remark1.2} for details). 
$K_{Q}$ and $L_{Q}$ also depend on 
$p_{\theta }(x'|x,y)$, $q_{\theta }(y|x)$, $\pi_{\theta}(x)$, $\{X_{n} \}_{n\geq 0}$ 
through the ergodicity properties of the optimal filter. 
In addition to this, $K_{Q}$ and $L_{Q}$ depend on 
$p_{\theta }(x'|x,y)$, $q_{\theta }(y|x)$
through upper and lower bounds of $p_{\theta }(x'|x,y)$, $q_{\theta }(y|x)$
and Lipschitz constants of 
$p_{\theta }(x'|x,y)$, $q_{\theta }(y|x)$, 
$\nabla_{\theta } p_{\theta }(x'|x,y)$, $\nabla_{\theta } q_{\theta }(y|x)$. 
For further details, see the proofs of Lemma \ref{lemma4.1} 
and Theorem \ref{theorem4.1} (Section \ref{section4*}).
\end{vremark}

The recursive maximum split-likelihood method (\ref{4.1}) 
has been proposed and thoroughly analyzed in \cite{ryden1}, \cite{ryden2}. 
Although the results of \cite{ryden2} provide a good insight into 
its asymptotic behavior, 
they do not offer any information about the asymptotic bias
(i.e., bounds on quantities (\ref{1.5})). 
The main difficulty is the same as in the case of policy-gradient search: 
The recursive maximum split-likelihood (\ref{4.1}) is so complex 
(even for moderately large $N_{x}$, $N_{y}$)
that the existing results on the biased stochastic gradient search 
and the biased stochastic approximation 
\cite[Section 5.3]{borkar}, \cite{chen1}, \cite{chen2}, \cite[Section 2.7]{chen3} 
cannot be applied. 
As opposed to the results of \cite{ryden2}, 
under mild and easily verifiable conditions, 
Theorem \ref{theorem4.1} provides (relatively) tight upper bounds on 
the asymptotic bias of (\ref{4.1}) in terms of $N$. 
It is worth mentioning that Theorem \ref{theorem4.1} can be extended to more general models  
and more sophisticated identification algorithms 
(such as those based on MCMC and SMC sampling; see \cite{tadic6}).

\section{Proof of Part (i) of Theorem \ref{theorem1.1}}\label{section1.a*}

In this section, we rely on the following notation. 
For a set $A\subseteq\mathbb{R}^{d_{\theta } }$ and $\varepsilon\in(0,\infty )$, 
let $V_{\varepsilon}(A)$ be the $\varepsilon$-vicinity of $A$, i.e., 
$V_{\varepsilon}(A) = \{\theta\in\mathbb{R}^{d_{\theta } }: d(\theta, A ) \leq \varepsilon \}$. 
For $\theta\in{\mathbb R}^{d_{\theta}}$ and $\gamma\in[0,\infty)$, let $F_{\gamma}(\theta)$ be the set
defined by 
\begin{align*}
	F_{\gamma}(\theta)
	=
	\left\{
	-\nabla f(\theta)+\vartheta: \vartheta\in\mathbb{R}^{d_{\theta}}, \|\vartheta\|\leq\gamma 
	\right\} 
\end{align*}
(notice that $F_{\gamma}(\theta)$ is a set-valued function of $\theta$). 
For $\gamma\in[0,\infty)$, let $\Phi_{\gamma}$ be the family of solutions to 
the differential inclusion $d\theta/dt\in F_{\gamma}(\theta)$, 
i.e., $\Phi_{\gamma}$ is the collection of absolutely continuous functions 
$\varphi:[0,\infty)\rightarrow\mathbb{R}^{d_{\theta}}$ satisfying 
$d\varphi(t)/dt\in F_{\gamma}(\varphi(t))$ almost everywhere (in $t$) on $[0,\infty)$. 
For a compact set $Q\subset\mathbb{R}^{d_{\theta}}$ and $\gamma\in[0,\infty)$, 
let ${\cal H}_{Q,\gamma}$ be the largest invariant set of the differential inclusion 
$d\theta/dt\in F_{\gamma}(\theta)$ contained in $Q$, 
i.e., ${\cal H}_{Q,\gamma}$ is the largest set ${\cal H}$ with the following property: 
For any $\theta\in{\cal H}$, 
there exists a solution $\varphi\in\Phi_{\gamma}$ such that $\varphi(0)=\theta$
and $\varphi(t)\in{\cal H}$ for all $t\in[0,\infty)$. 
For a compact set $Q\subset{\mathbb R}^{d_{\theta}}$ and $\gamma\in[0,\infty)$, 
let ${\cal R}_{Q,\gamma}$ be the set of chain-recurrent points of 
the differential inclusion $d\theta/dt\in F_{\gamma}(\theta)$ contained in $Q$, 
i.e., $\theta\in{\cal R}_{Q,\gamma}$ if and only if for any $\delta,t\in(0,\infty)$, 
there exist an integer $N\geq 1$, real numbers $t_{1},\dots,t_{N}\in[t,\infty)$ 
and solutions $\varphi_{1},\dots,\varphi_{N}\in\Phi_{\gamma}$
(each of which can depend on $\theta,\delta,t$) such that 
$\varphi_{k}(0)\in{\cal H}_{Q,\gamma}$ for $1\leq k\leq N$ and 
\begin{align*}
	\|\varphi_{1}(0)-\theta \|\leq \delta, 
	\;\;\;\;\; 
	\|\varphi_{N}(t_{N})-\theta \|\leq \delta, 
	\;\;\;\;\; 
	\|\varphi_{k}(t_{k})-\varphi_{k+1}(0) \|\leq \delta
\end{align*}
for $1\leq k< N$. 
For more details on differential inclusions and their solutions, invariant sets and chain-recurrent points, 
see \cite{aubin&celina} and references cited therein. 

\begin{lemma}\label{lemma1.a.2}
Suppose that Assumption \ref{a1.3.a} holds. 
Then, given a compact set $Q\subset\mathbb{R}^{d_{\theta}}$, 
there exists a non-decreasing function $\phi_{Q}:[0,\infty)\rightarrow[0,\infty)$
such that 
$\lim_{\gamma\rightarrow 0} \phi_{Q}(\gamma)=\phi_{Q}(0)=0$ and 
${\cal R}_{Q,\gamma}\subseteq V_{\phi_{Q}(\gamma)}({\cal R})$ for all $\gamma\in[0,\infty)$. 
\end{lemma}

\begin{proof}
Let $Q\subset{\mathbb R}^{d_{\theta}}$ be any compact set. 
Moreover, let $\phi_{Q}:[0,\infty)\rightarrow[0,\infty)$ be the function defined 
by $\phi_{Q}(0)=0$ and 
\begin{align*}
	\phi_{Q}(\gamma)
	=
	\sup\left( \left\{ d(\theta,{\cal R}): \theta\in{\cal R}_{Q,\gamma} \right\} \cup \{0\} \right)
\end{align*}
for $\gamma\in (0,\infty)$. 
Then, it is easy to show that $\phi_{Q}(\cdot)$ is well-defined and satisfies 
${\cal R}_{Q,\gamma}\subseteq V_{\phi_{Q}(\gamma)}({\cal R})$ for all $\gamma\in[0,\infty)$. 
It is also easy to check that 
$F_{\gamma}(\theta)\subseteq F_{\delta}(\theta)$ for all $\theta\in\mathbb{R}^{d_{\theta}}$, 
$\gamma,\delta\in[0,\infty)$ satisfying $\gamma\leq\delta$. 
Consequently, $\Phi_{\gamma}\subseteq\Phi_{\delta}$, 
${\cal H}_{Q,\gamma}\subseteq{\cal H}_{Q,\delta}$, 
${\cal R}_{Q,\gamma}\subseteq{\cal R}_{Q,\delta}$
for all $\gamma,\delta\in[0,\infty)$ satisfying $\gamma\leq\delta$. 
Thus, $\phi_{Q}(\cdot)$ is non-decreasing.\footnote
{Notice that 
$\{d(\theta,{\cal R}): \theta\in{\cal R}_{Q,\gamma} \} \subseteq 
\{d(\theta,{\cal R}): \theta\in{\cal R}_{Q,\delta} \}$ whenever $\gamma\leq\delta$. } 
Moreover, \cite[Theorem 3.1]{benaim4} implies that given $\varepsilon\in(0,\infty)$, 
there exists a real number $\gamma_{Q}(\varepsilon)\in(0,\infty)$ such that 
${\cal R}_{Q,\gamma}\subseteq V_{\varepsilon}({\cal R})$ 
for all $\gamma\in[0,\gamma_{Q}(\varepsilon) )$. 
Therefore, $\phi_{Q}(\gamma)\leq\varepsilon$ 
for all $\varepsilon\in(0,\infty)$, $\gamma\in[0,\gamma_{Q}(\varepsilon) )$.\footnote
{Notice that $d(\theta,{\cal R})\leq\varepsilon$ 
whenever $\theta\in{\cal R}_{Q,\gamma}$, $\gamma\in[0,\gamma_{Q}(\varepsilon) )$. }
Consequently, $\lim_{\gamma\rightarrow 0}\phi_{Q}(\gamma)=\phi_{Q}(0)=0$. 
\end{proof}

\begin{vproof}{Part (i) of Theorem \ref{theorem1.1}}
{Let $Q\subset\mathbb{R}^{d_{\theta}}$ be any compact set and let 
$\psi_{Q}:[0,\infty)\rightarrow[0,\infty)$ be the function defined by
$\psi_{Q}(t)=\phi_{Q}(2t)$ for $t\in[0,\infty)$ 
($\phi_{Q}(\cdot)$ is specified in the statement of Lemma \ref{lemma1.a.2}). 
Then, due to Lemma \ref{lemma1.a.2}, 
$\psi_{Q}(\cdot)$ is non-decreasing and 
$\lim_{t\rightarrow 0} \psi_{Q}(t) = \psi_{Q}(0) = 0$. 
Moreover, owing to Assumption \ref{a1.2}, 
there exists an event $N_{Q}\in{\cal F}$
such that the following holds: 
$P(N_{Q})=0$ and (\ref{a1.2.1}) is satisfied on $\Lambda_{Q}\setminus N_{Q}$
for all $t\in(0,\infty)$. 
Let $\omega$ be an arbitrary sample in $\Lambda_{Q}\setminus N_{Q}$. 
To prove Part (i) of Theorem \ref{theorem1.1}, 
it is sufficient to show (\ref{t1.1.1*}) for $\omega$. 
Notice that all formulas that follow in the proof correspond to $\omega$. 

If $\eta=0$, then \cite[Proposition 4.1, Theorem 5.7]{benaim2} imply
that all limit points of $\{\theta_{n} \}_{n\geq 0}$ are included in ${\cal R}$. 
Hence, (\ref{t1.1.1*}) holds when $\eta=0$. 

Now, suppose $\eta>0$. Then, there exists $n_{0}\geq 0$ (depending on $\omega$)
such that $\theta_{n}\in Q$, $\|\eta_{n}\|\leq 2\eta$ for $n\geq n_{0}$. 
Therefore, 
\begin{align*}
	\frac{\theta_{n+1}-\theta_{n} }{\alpha_{n} } + \zeta_{n}
	=
	-\left(\nabla f(\theta_{n} ) + \eta_{n} \right)
	\in 
	F_{2\eta}(\theta_{n})
\end{align*}
for $n\geq n_{0}$. 
Consequently, \cite[Proposition 1.3, Theorem 3.6]{benaim3} imply that all limit points of 
$\{\theta_{n} \}_{n\geq 0}$ are contained in ${\cal R}_{Q,2\eta}$. 
Combining this with Lemma \ref{lemma1.a.2}, we conclude that the limit points of $\{\theta_{n} \}_{n\geq 0}$
are included in $V_{\phi_{Q}(2\eta)}({\cal R}) = V_{\psi_{Q}(\eta)}({\cal R})$. 
Thus, (\ref{t1.1.1*}) holds when $\eta>0$. 
}
\end{vproof}

\section{Proof of Parts (ii), (iii) of Theorem \ref{theorem1.1}}\label{section1.bc*}

In this section, the following notation is used. 
$\phi$ is the random variable defined by  
\begin{align*}
	\phi=\limsup_{n\rightarrow \infty } \|\nabla f(\theta_{n} ) \|. 
\end{align*}	
For $t\in (0,\infty )$ and $n\geq 0$, 
$\phi_{1,n}(t), \phi_{2,n}(t), \phi_{n}(t)$ are the random quantities 
defined as 
\begin{align*}	
	&
	\phi_{1,n}(t)
	=
	- (\nabla f(\theta_{n} ) )^{T} 
	\sum_{i=n}^{a(n,t)-1} \alpha_{i} 
	\left(\nabla f(\theta_{i} ) - \nabla f(\theta_{n} ) \right), 
	\\
	&
	\phi_{2,n}(t)
	=
	\int_{0}^{1} 
	\left(\nabla f(\theta_{n} + s (\theta_{a(n,t) } - \theta_{n} ) ) - \nabla f(\theta_{n} ) \right)^{T}
	(\theta_{a(n,t) } - \theta_{n} ) ds, 
	\\
	&
	\phi_{n}(t) = \phi_{1,n}(t) + \phi_{2,n}(t).
\end{align*}
Then, it is straightforward to demonstrate 
\begin{align} \label{1.1*}
	f(\theta_{a(n,t) } ) - f(\theta_{n} )
	= &
	-
	\|\nabla f(\theta_{n} ) \|^{2} 
	\sum_{i=n}^{a(n,t)-1} \alpha_{i} 
	-
	(\nabla f(\theta_{n} ) )^{T} 
	\sum_{i=n}^{a(n,t)-1} \alpha_{i} \xi_{i}
	+
	\phi_{n}(t)	
	\nonumber\\
	\leq &
	-
	\|\nabla f(\theta_{n} ) \|
	\left(
	\|\nabla f(\theta_{n} ) \|
	\sum_{i=n}^{a(n,t)-1} \alpha_{i} 
	-
	\left\|
	\sum_{i=n}^{a(n,t)-1} \alpha_{i} \xi_{i}
	\right\|
	\right)
	+
	|\phi_{n}(t)|
\end{align}
for $t\in(0,\infty )$, $n\geq 0$. 

In this section, the following notation is also relied on. 
The Lebesgue measure is denoted by $m(\cdot )$. 
For a compact set $Q\subset\mathbb{R}^{d_{\theta } }$ and $\varepsilon\in(0,\infty )$, 
$A_{Q,\varepsilon }$ is the set defined by 
\begin{align}\label{1.3*}
	A_{Q,\varepsilon } 
	=
	\{f(\theta ): \theta\in Q, \|\nabla f(\theta ) \|\leq\varepsilon \}. 
\end{align}

In order to treat Assumptions \ref{a1.3.b}, \ref{a1.3.c} in a unified way 
and to 
provide a unified proof of Parts (ii), (iii) of Theorem \ref{theorem1.1}, 
we introduce the following assumption. 

\begin{assumption}\label{a1.4}
There exists a real number $s\in (0,1]$ 
and for any compact set $Q\subset\mathbb{R}^{d_{\theta } }$, 
there exists a real number $M_{Q}\in[1,\infty )$ such that 
$m(A_{Q,\varepsilon } ) \leq M_{Q}\varepsilon^{s}$ for all $\varepsilon\in(0,\infty )$.
\end{assumption}

\begin{proposition}\label{proposition1.1}
Suppose that Assumption \ref{a1.3.b} holds. 
Let $Q\subset\mathbb{R}^{d_{\theta } }$ be any compact set.  
Then, there exists a real number $M_{Q}\in [1,\infty )$
(depending only on $f(\cdot )$) such that 
$m(A_{Q,\varepsilon } ) \leq M_{Q}\varepsilon^{q}$
for all $\varepsilon\in(0,\infty )$
($q$ is specified in the statement of Theorem \ref{theorem1.1}). 
\end{proposition}

\begin{proof}
The proposition is a particular case of 
Yomdin theorem \cite[Theorem 1.2]{yomdin}. 
\end{proof}

\begin{proposition} \label{proposition1.2}
Suppose that Assumption \ref{a1.3.c} holds. 
Let $Q\subset\mathbb{R}^{d_{\theta } }$ be any compact set. 
Then, the following is true: 
\begin{compactenum}[(i)]
\item
There exists a real number $M_{Q}\in[1,\infty )$
(depending only on $f(\cdot )$) such that $m(A_{Q,\varepsilon } )\leq M_{Q}\varepsilon$
for all $\varepsilon\in(0,\infty )$. 
\item
There exist real numbers 
$r_{Q}\in (0,1)$, $M_{1,Q}, M_{2,Q} \in[1,\infty )$ (depending only on $f(\cdot )$) 
such that 
\begin{align}\label{p1.2.1*}
	d(\theta, {\cal S} ) 
	\leq 
	M_{1,Q} \|\nabla f(\theta ) \|^{r_{Q} },
	\;\;\;\;\; 
	d(f(\theta ) ), f({\cal S} ) ) 
	\leq 
	M_{2,Q} \|\nabla f(\theta ) \|
\end{align}
for all $\theta\in Q$
(${\cal S}$ and $f({\cal S} )$ are specified in (\ref{1.21})). 
\end{compactenum}
\end{proposition}

\begin{proof}
Let $Q\subset\mathbb{R}^{d_{\theta } }$ be any compact set. 
Owing to Lojasiewicz (ordinary) inequality 
(see \cite[Theorem 6.4, Remark 6.5]{bierstone&milman}), 
there exist real numbers $r_{Q}\in (0,1)$, $M_{1,Q} \in[1,\infty )$
such the first inequality in (\ref{p1.2.1*}) holds for all $\theta\in {\cal S}$. 
On the other side, due to Lojasiewicz gradient inequality 
(see \cite[Theorem \L I, Page 775]{kurdyka}), we have the following:  
For any $a\in f(Q)=\{f(\theta ): \theta\in Q\}$, there exist real numbers 
$\delta_{Q,a}\in (0,1)$, $\nu_{Q,a}\in(1,2]$, $N_{Q,a}\in[1,\infty )$ such that 
\begin{align}\label{p1.2.1}
	|f(\theta ) - a |
	\leq 
	N_{Q,a} \|\nabla f(\theta ) \|^{\nu_{Q,a} }
\end{align}
for all $\theta\in Q$ satisfying 
$|f(\theta ) - a |\leq\delta_{Q,a}$. 

Now, we show by contradiction that 
$f({\cal S}\cap Q )=\{f(\theta ): \theta\in {\cal S}\cap Q \}$ has finitely many elements. 
Suppose the opposite. 
Then, there exists a sequence $\{\vartheta_{n} \}_{n\geq 0}$ in ${\cal S}\cap Q$
such that $\{f(\vartheta_{n} ) \}_{n\geq 0}$ contains infinitely many different elements. 
Since ${\cal S}\cap Q$ is compact, 
$\{\vartheta_{n} \}_{n\geq 0}$ has a convergent subsequence 
$\{\tilde{\vartheta}_{n} \}_{n\geq 0}$
such that $\{f(\tilde{\vartheta}_{n} ) \}_{n\geq 0}$ also contains infinitely many different elements. 
Let $\vartheta=\lim_{n\rightarrow\infty } \tilde{\vartheta}_{n}$,
$a=f(\vartheta )$. 
As $\delta_{Q,a}>0$, there exists an integer $n_{0}\geq 0$
such that $|f(\tilde{\vartheta}_{n} ) - a |\leq \delta_{Q,a}$
for $n\geq n_{0}$. 
Since $\nabla f(\tilde{\vartheta}_{n} ) = 0$ for $n\geq 0$, 
(\ref{p1.2.1}) implies $f(\tilde{\vartheta}_{n} )=a$ for $n\geq n_{0}$. 
However, this is impossible, 
since $\{f(\tilde{\vartheta}_{n} ) \}_{n\geq 0}$ has infinitely many different elements. 

Let $a_{1},\dots,a_{N}$ be the elements of $f({\cal S}\cap Q)$, 
while 
$\tilde{C}_{1,Q} = \max_{1\leq i\leq N } N_{Q,a_{i} }$. 
For $1\leq i\leq N$, let 
\begin{align*}
	B_{Q,i}
	=
	\left\{\theta\in Q: 
	\|\nabla f(\theta ) \|<1, 
	f(\theta )\in (a_{i}-\delta_{Q,a_{i} },a_{i}+\delta_{Q,a_{i} } ) 
	\right\}, 
\end{align*}
while $B_{Q} = \bigcup_{i=1}^{N} B_{Q,i}$, 
$\varepsilon_{Q} = \inf\{\|\nabla f(\theta ) \|: \theta\in Q\setminus B_{Q} \}$. 
As $B_{Q}$ is open and ${\cal S}\cap Q \subset B_{Q}$, 
we have $\varepsilon_{Q}>0$. 

Let $\tilde{C}_{2,Q}\in[1,\infty )$ be an upper bound of $|f(\cdot )|$ on $Q$, 
while $M_{2,Q}=2\max\{\tilde{C}_{1,Q}, \tilde{C}_{2,Q}/\varepsilon_{Q} \}$. 
Then, if $\theta\in B_{Q}$, we have 
\begin{align*}
	d(f(\theta ), f({\cal S} ) ) 
	=
	\min_{1\leq i\leq N} 
	|f(\theta ) - a_{i} |
	\leq 
	\max_{1\leq i\leq N} 
	N_{Q,a_{i} } \|\nabla f(\theta ) \|^{\nu_{Q,a_{i} } }
	\leq 
	M_{2,Q} \|\nabla f(\theta ) \|
\end{align*}
(notice that $\|\nabla f(\theta ) \|<1$, $\nu_{Q,a_{i} }>1$). 
On the other side, 
if $\theta\in Q\setminus B_{Q}$, we get 
\begin{align*}
	d(f(\theta ), f({\cal S} ) ) 
	=
	\min_{1\leq i\leq N} 
	|f(\theta ) - a_{i} |
	\leq 
	2\tilde{C}_{2,Q} 
	\leq 
	2\tilde{C}_{2,Q}\varepsilon_{Q}^{-1} \|\nabla f(\theta ) \|
	\leq 
	M_{2,Q} \|\nabla f(\theta ) \|
\end{align*}
(notice that $\|\nabla f(\theta ) \|\geq \varepsilon_{Q}$). 
Hence, the second inequality in (\ref{p1.2.1*}) holds for all $\theta\in Q$. 

Let $M_{Q}=2M_{2,Q}N$. 
Owing to the second inequality in (\ref{p1.2.1*}), we have 
\begin{align*}
	A_{Q,\varepsilon } 
	\subseteq
	\bigcup_{i=1}^{N} 
	[f(a_{i} ) - M_{2,Q}\varepsilon, f(a_{i} ) + M_{2,Q}\varepsilon ]
\end{align*}
for each $\varepsilon\in(0,\infty )$. 
Consequently, 
$m(A_{Q,\varepsilon } ) \leq 2M_{2,Q}N\varepsilon = M_{Q}\varepsilon$
for all $\varepsilon\in(0,\infty )$. 
\end{proof}

\begin{lemma} \label{lemma1.1}
Let Assumptions \ref{a1.1} and \ref{a1.2} hold. 
Then, there exists 
an event $N_{0} \in {\cal F}$ such that 
$P(N_{0} ) = 0$ and 
\begin{align}
	& \label{l1.1.1*}
	\limsup_{n\rightarrow \infty } 
	\max_{n\leq k < a(n,t) } 
	\left\|
	\sum_{i=n}^{k} \alpha_{i} \xi_{i} 
	\right\|
	\leq 
	\eta t, 
	\\ \label{l1.1.3*}
	&
	\lim_{n\rightarrow \infty } |f(\theta_{n+1} ) - f(\theta_{n} ) | 
	=
	0
\end{align}
on $\{\sup_{n\geq 0} \|\theta_{n}\| < \infty \} \setminus N_{0}$
for all $t\in (0,\infty )$. 
Moreover, 
given a compact set 
$Q\subset \mathbb{R}^{d_{\theta } }$, there exists a real number  
$C_{1,Q} \in [1,\infty )$ 
(independent of $\eta$ and depending only on $f(\cdot )$)
such that 
\begin{align}	
	& \label{l1.1.5*}
	\limsup_{n\rightarrow \infty } 
	\max_{n\leq k \leq a(n,t) } 
	|f(\theta_{k} ) - f(\theta_{n} ) |	
	\leq 
	C_{1,Q} t(\phi+\eta), 	
	\\
	& \label{l1.1.7*} 
	\limsup_{n\rightarrow \infty } 
	|\phi_{n}(t) |	
	\leq 
	C_{1,Q} 
	t^{2}(\phi + \eta )^{2} 
\end{align}
on $\Lambda_{Q}\setminus N_{0}$
for all $t \in (0,\infty )$. 
\end{lemma}

\begin{proof}
In the same way as in the proof of Lemma \ref{lemma1.a.2}, 
it can be shown that there exists $N_{0}\in{\cal F}$ such that the following holds: 
(i) $P(N_{0} ) = 0$, and  
(ii) (\ref{a1.2.1}), (\ref{l1.1.1*}) are satisfied on 
$\{\sup_{n\geq 0} \|\theta_{n}\| < \infty \}\setminus N_{0}^{c}$ 
for all $t\in(0,\infty )$. 
%

Let $Q\subset \mathbb{R}^{d_{\theta } }$ be any compact set, 
while $\tilde{C}_{Q}\in[1,\infty )$ stands 
for a Lipschitz constant of $f(\cdot )$, $\nabla f(\cdot )$ on $Q$. 
Moreover, let 
$C_{1,Q} = 2 \tilde{C}_{Q}$, 
while $\omega$ is an arbitrary sample from 
$\Lambda_{Q} \setminus N_{0}$. 
In order to prove the lemma, it is sufficient to show that 
(\ref{l1.1.3*}) -- (\ref{l1.1.7*}) hold for $\omega$ and any $t\in (0,\infty )$. 
Notice that all formulas which follow in the proof correspond to $\omega$. 

Let $\varepsilon\in(0,\infty )$ be any real number. 
Then, there exists $n_{0}\geq 0$ (depending on $\omega$, $\varepsilon$)
such that $\theta_{n}\in Q$, $\|\nabla f(\theta_{n} ) \|\leq \phi+\varepsilon$
for $n\geq n_{0}$
(notice that these relations hold for all but finitely many $n$).
Therefore, 
\begin{align*}
	\|\theta_{k} - \theta_{n} \|
	\leq 
	\sum_{i=n}^{k-1} \alpha_{i} \|\nabla f(\theta_{i} ) \| 
	+
	\left\|
	\sum_{i=n}^{k-1} \alpha_{i} \xi_{i} 
	\right\|
	\leq 
	t(\phi+\varepsilon ) 
	+
	\max_{n\leq j<a(n,t) }
	\left\|
	\sum_{i=n}^{j} \alpha_{i} \xi_{i} 
	\right\|
\end{align*}
for $n_{0}\leq n\leq k\leq a(n,t)$, $t\in(0,\infty )$. 
Combining this with (\ref{l1.1.1*}), we get 
\begin{align*}
	\limsup_{n\rightarrow\infty} 
	\max_{n\leq k\leq a(n,t)}
	\|\theta_{k} - \theta_{n} \|
	\leq 
	t(\phi+\eta+\varepsilon)
\end{align*}
for $t\in(0,\infty)$. 
Then, 
the limit process $\varepsilon\rightarrow 0$ yields
\begin{align*}
	\limsup_{n\rightarrow\infty} 
	\max_{n\leq k\leq a(n,t)}
	\|\theta_{k} - \theta_{n} \|
	\leq 
	t(\phi+\eta)
\end{align*}
for $t\in(0,\infty)$
(notice that $\varepsilon\in(0,\infty)$ is any real number). 
As  
\begin{align*}
	|f(\theta_{k} ) - f(\theta_{n} ) |
	\leq 
	\tilde{C}_{Q} \|\theta_{k} - \theta_{n} \|
\end{align*}
for $k\geq n\geq n_{0}$ (notice that $\theta_{n}\in Q$ for $n\geq n_{0}$),  
we have 
\begin{align*}
	\limsup_{n\rightarrow\infty }
	\max_{n\leq k\leq a(n,t) } 
	|f(\theta_{k} ) - f(\theta_{n} ) |
	\leq 
	\tilde{C}_{Q}t(\phi+\eta)
	\leq 
	C_{1,Q} t(\phi+\eta)
\end{align*}
for $t\in(0,\infty )$. 
Since 
\begin{align*}
	|f(\theta_{n+1} ) - f(\theta_{n} ) |
	\leq 
	\max_{n\leq k\leq a(n,t) } 
	|f(\theta_{k} ) - f(\theta_{n} ) |
\end{align*}
for $t\in(0,\infty )$ and sufficiently large $n$
(notice that $a(n,t)\geq n+1$ for sufficiently large $n$), 
we conclude 
\begin{align*}
	\limsup_{n\rightarrow\infty }
	|f(\theta_{n+1} ) - f(\theta_{n} ) |
	\leq 
	\tilde{C}_{Q}t(\phi+\eta)
\end{align*}
for $t\in(0,\infty )$. 
Then, the limit process $t\rightarrow 0$ implies (\ref{l1.1.3*}). 
On the other side, we have 
\begin{align*}
	&
	|\phi_{1,n}(t) |
	\leq 
	\tilde{C}_{Q} \|\nabla f(\theta_{n} ) \| 
	\sum_{i=n}^{a(n,t)-1} \alpha_{i} \|\theta_{i} - \theta_{n} \| 
	\leq 
	\tilde{C}_{Q}t \|\nabla f(\theta_{n} ) \| 
	\max_{n\leq k\leq a(n,t) }
	\|\theta_{k} - \theta_{n} \|, 
	\\
	&
	|\phi_{2,n}(t)|
	\leq 
	\tilde{C}_{Q} \|\theta_{a(n,t)} - \theta_{n} \|^{2} 
	\leq 
	\tilde{C}_{Q} 
	\max_{n\leq k\leq a(n,t) }
	\|\theta_{k} - \theta_{n} \|^{2}
\end{align*}
for $n\geq n_{0}$, $t\in(0,\infty )$. 
Therefore, 
\begin{align*}
	&
	\limsup_{n\rightarrow\infty } 
	|\phi_{1,n}(t) |
	\leq 
	\tilde{C}_{Q}t^{2} \phi(\phi+\eta), 
	\;\;\;\;\; 
	\limsup_{n\rightarrow\infty } 
	|\phi_{2,n}(t) |
	\leq 
	\tilde{C}_{Q} t^{2} (\phi+\eta)^{2}
\end{align*}
for $t\in(0,\infty )$. 
Hence, 
\begin{align*}
	\limsup_{n\rightarrow\infty } 
	|\phi_{n}(t) |
	\leq 
	2\tilde{C}_{Q} t^{2} (\phi+\eta)^{2}
	=
	C_{1,Q} t^{2} (\phi+\eta)^{2}
\end{align*}
for $t\in(0,\infty )$. 
\end{proof}

\begin{lemma} \label{lemma1.2}
Let Assumptions \ref{a1.1}, \ref{a1.2} and \ref{a1.4} hold. 
Then, given a compact set $Q\subset \mathbb{R}^{d_{\theta } }$, there exists 
a real number $C_{2,Q} \in [1,\infty )$
(independent of $\eta$ and depending only on $f(\cdot )$)
such that 
\begin{align} \label{l1.2.1*}
	\limsup_{n\rightarrow \infty } 
	f(\theta_{n} ) 
	- 
	\liminf_{n\rightarrow \infty } 
	f(\theta_{n} )
	\leq 
	C_{2,Q} \eta^{s}
\end{align}
on $\Lambda_{Q} \setminus N_{0}$
($s$ is specified in Assumption \ref{a1.4}).
\end{lemma}

\begin{proof}
Let $Q\subset \mathbb{R}^{d_{\theta } }$ be any compact set, 
while $\tilde{C}_{Q}$ stands for an upper bound of $\|\nabla f(\cdot )\|$ on $Q$. 
Moreover, let $C_{2,Q}=4 M_{Q}$.
In order to avoid considering separately the cases $\eta=0$ and $\eta>0$, 
we show  
\begin{align} \label{l1.2.1}
	\limsup_{n\rightarrow \infty } 
	f(\theta_{n} ) 
	- 
	\liminf_{n\rightarrow \infty } 
	f(\theta_{n} )
	\leq 
	C_{2,Q} (\varepsilon + \eta )^{s}
\end{align}
on $\Lambda_{Q}\setminus N_{0}$ for all $\varepsilon \in (0,\infty )$. Then,
(\ref{l1.2.1*}) follows directly from (\ref{l1.2.1}) by letting $\varepsilon \rightarrow 0$. 

Inequality (\ref{l1.2.1}) is proved by contradiction: 
Suppose that there exist a sample 
$\omega \in \Lambda_{Q} \setminus N_{0}$ and a real number $\varepsilon \in (0,\infty )$
such that (\ref{l1.2.1}) does not hold for them. 
Notice that all formulas which follow in the proof correspond to $\omega$. 

Let 
$\gamma=2(\varepsilon + \eta )$, 
$\delta=M_{Q} \gamma^{s}$, 
while 
\begin{align*}
	\mu=\delta/(C_{1,Q} (\tilde{C}_{Q} + \eta ) ), 
	\;\;\;
	\nu=\gamma^{2}/(4 C_{1,Q} (\tilde{C}_{Q} + \eta )^{2} ), 
	\;\;\;
	\tau=\min\{\mu,\nu/2\}. 
\end{align*}
Since $\{\theta_{n} \}_{n\geq 0}$ is bounded and 
(\ref{l1.2.1}) is not satisfied, 
there exist real numbers 
$a,b \in \mathbb{R}$
(depending on $\omega,\varepsilon$)
such that 
$b-a > 2\delta$ 
and such that inequalities 
$f(\theta_{n} ) < a$, $f(\theta_{k} ) > b$ hold for infinitely many 
$n,k\geq 0$
(notice that $C_{2,Q}(\varepsilon+\eta)^{s}\geq 2\delta$). 
As $m(A_{Q,\gamma } ) \leq M_{Q} \gamma^{s} = \delta$, 
there exists a real number 
$c$ such that $c\not\in A_{Q,\gamma }$ and $a<c<b-\delta$
(otherwise, 
$(a, b-\delta ) \subset A_{Q,\varepsilon }$, 
which is impossible as 
$(b-\delta) - a > \delta$). 

Let $n_{0}=0$, while 
\begin{align*}
	&
	l_{k}
	=
	\min\{n\geq n_{k-1}: f(\theta_{n} ) \leq c \}, 
	\;\;\;\;\;	
	n_{k}
	=
	\min\{n\geq l_{k}: f(\theta_{n} ) \geq b \}, 
	\;\;\;\;\; 
	m_{k} 
	=
	\max\{n\leq n_{k}: f(\theta_{n} ) \leq c \}
\end{align*}
for $k\geq 1$. 
It can easily be deduced that 
sequences 
$\{l_{k} \}_{k\geq 1}$, $\{m_{k} \}_{k\geq 1}$, $\{n_{k} \}_{k\geq 1}$
are well-defined
and satisfy 
$l_{k} < m_{k} < n_{k} < l_{k+1}$ 
and 
\begin{align}
	& \label{l1.2.3}
	f(\theta_{m_{k} } ) \leq c < f(\theta_{m_{k}+1} ), 
	\;\;\;\;\; 
	f(\theta_{n_{k} } ) - f(\theta_{m_{k} } )
	\geq 
	b-c, 
	\;\;\;\;\; 
	\min_{m_{k}<n\leq n_{k} } f(\theta_{n} ) > c
\end{align}
for $k\geq 1$. 
On the other side, Lemma \ref{lemma1.1} implies 
\begin{align}
	& \label{l1.2.9}
	\lim_{k\rightarrow \infty } 
	|f(\theta_{m_{k}+1}) - f(\theta_{m_{k} } ) |
	=
	0,
	\\
	& \label{l1.2.21}
	\limsup_{k\rightarrow \infty } 
	\max_{m_{k} \leq j \leq a(m_{k},\tau) }
	|f(\theta_{j} ) - f(\theta_{m_{k} } ) |
	\leq 
	C_{1,Q} \tau (\tilde{C}_{Q} + \eta )	
	\leq 
	\delta
	< b-c
\end{align}
(to get (\ref{l1.2.21}), notice that 
$\theta_{n} \in Q$ for all but finitely many $n$ and that 
$\phi \leq \tilde{C}_{Q}$). 
Owing to (\ref{l1.2.21}) and the second inequality in (\ref{l1.2.3}), 
there exists $k_{0} \geq 1$ such that 
$a(m_{k},\tau) \leq n_{k}$ for $k\geq k_{0}$.\footnote
{If $a(m_{k},\tau) > n_{k}$ for infinitely many $k$, 
then (\ref{l1.2.21}) yields 
\begin{align*}
	\liminf_{k\rightarrow \infty } (f(\theta_{n_{k} } ) - f(\theta_{m_{k} } ) ) 
	\leq \delta < b - c. 
\end{align*}
However, this contradicts the second inequality in (\ref{l1.2.3}).
}
Then, the last inequality in (\ref{l1.2.3}) implies   
$f(\theta_{a(m_{k},\tau)} ) \geq c$ for $k\geq k_{0}$, 
while  
$\lim_{k\rightarrow \infty } f(\theta_{m_{k} } ) = c$
follows from (\ref{l1.2.9}) and the first inequality in (\ref{l1.2.3}). 
Since $\|\nabla f(\theta ) \| > \gamma$ for any $\theta \in Q$
satisfying $f(\theta )=c$
(due to the way $c$ is selected), 
we have 
$\liminf_{k\rightarrow \infty } \|\nabla f(\theta_{m_{k} } ) \| \geq \gamma$. 
Consequently, 
Lemma \ref{lemma1.1} and (\ref{1.1501}) yield 
\begin{align*}
	\liminf_{k\rightarrow\infty } 
	\left(
	\|\nabla f(\theta_{m_{k} } ) \| 
	\sum_{i=m_{k} }^{a(m_{k},\tau ) -1} \alpha_{i} 
	-
	\left\|
	\sum_{i=m_{k} }^{a(m_{k},\tau ) - 1} \alpha_{i} \xi_{i} 
	\right\|
	\right)
	\geq 
	\tau (\gamma-\eta)
	\geq 
	\tau\gamma/2 > 0
\end{align*}
(notice that $\eta<\gamma/2$). Therefore, 
\begin{align*}
	\liminf_{k\rightarrow\infty } 
	\|\nabla f(\theta_{m_{k} } ) \| 
	\left(
	\|\nabla f(\theta_{m_{k} } ) \| 
	\sum_{i=m_{k} }^{a(m_{k},\tau ) -1} \alpha_{i} 
	-
	\left\|
	\sum_{i=m_{k} }^{a(m_{k},\tau ) - 1} \alpha_{i} \xi_{i} 
	\right\|
	\right)
	\geq 
	\tau\gamma^{2}/2. 
\end{align*}
Combining this with Lemma \ref{lemma1.1} and (\ref{1.1*}), we get 
\begin{align*}
	\limsup_{k\rightarrow\infty } 
	( f(\theta_{a(m_{k},\tau) } ) - f(\theta_{m_{k} } ) )
	\leq &
	-
	\liminf_{k\rightarrow\infty } 
	\|\nabla f(\theta_{m_{k} } ) \| 
	\left(
	\|\nabla f(\theta_{m_{k} } ) \| 
	\sum_{i=m_{k} }^{a(m_{k},\tau ) -1} \alpha_{i} 
	-
	\left\|
	\sum_{i=m_{k} }^{a(m_{k},\tau ) - 1} \alpha_{i} \xi_{i} 
	\right\|
	\right)
	\\
	&
	+
	\limsup_{k\rightarrow\infty } |\phi_{m_{k} }(\tau ) |
	\\
	\leq &
	-\tau\gamma^{2}/2 + C_{1,Q}\tau^{2}(\phi+\eta)^{2}
	< 0
\end{align*}
(notice that $\phi\leq\tilde{C}_{Q}$, $C_{1,Q}\tau(\tilde{C}_{Q}+\eta)^{2}\leq\gamma^{2}/4$). 
However, this is not possible, as 
$f(\theta_{a(m_{k},\tau ) } ) \geq c \geq f(\theta_{m_{k} } )$ for each $k\geq k_{0}$. 
Hence, (\ref{l1.2.1}) is true. 
\end{proof}

\begin{lemma} \label{lemma1.3}
Let Assumptions \ref{a1.1} and \ref{a1.2} hold. 
Then, given a compact set $Q\subset \mathbb{R}^{d_{\theta } }$, there exists 
a real number $C_{3,Q} \in (0,1)$
(independent of $\eta$ and depending only on $f(\cdot )$)
such that 
\begin{align} \label{l3.2.1*}
	\limsup_{n\rightarrow \infty } 
	f(\theta_{n} ) 
	- 
	\liminf_{n\rightarrow \infty } 
	f(\theta_{n} )
	\geq 
	C_{3,Q} \phi^{2}
\end{align}
on 
$(\Lambda_{Q} \setminus N_{0} )
\cap \{\phi > 2\eta \}$.
\end{lemma}

\begin{proof}
Let $Q\subset \mathbb{R}^{d_{\theta } }$ be any compact set, 
while
$C_{3,Q} = 1/(64 C_{1,Q} )$ and 
$\tau=1/(16 C_{1,Q} )$. 
Moreover, 
let $\omega$ be an arbitrary sample from 
$(\Lambda_{Q}\setminus N_{0} ) \cap \{\phi > 2\eta \}$. 
In order to prove the lemma's assertion, 
it is sufficient to show that (\ref{l3.2.1*}) holds for $\omega$. 
Notice that all formulas which follow in the proof correspond to $\omega$. 
 
Let $n_{0}=0$ and 
\begin{align*}
	n_{k}
	=
	\min\{n>n_{k-1}: \|\nabla f(\theta_{n} ) \| \geq \phi - 1/k \}
\end{align*}
for $k\geq 1$. 
Obviously, 
sequence $\{n_{k} \}_{k\geq 0}$ is well-defined 
and satisfies 
$\lim_{k\rightarrow \infty } \|\nabla f(\theta_{n_{k} } ) \| = \phi$. 
Then, Lemma \ref{lemma1.1} and (\ref{1.1501}) yield 
\begin{align*}
	\liminf_{k\rightarrow\infty } 
	\|\nabla f(\theta_{n_{k} } ) \| 
	\left(
	\|\nabla f(\theta_{n_{k} } ) \| 
	\sum_{i=n_{k} }^{a(n_{k},\tau ) - 1 } \alpha_{i} 
	-
	\left\|
	\sum_{i=n_{k} }^{a(n_{k},\tau )- 1} \alpha_{i} \xi_{i}
	\right\|
	\right)
	\geq 
	\tau\phi(\phi-\eta)
	\geq
	\tau\phi^{2}/2
	>0. 
\end{align*}
Combining this with Lemma \ref{lemma1.1} and (\ref{1.1*}), we get 
\begin{align*} 
	\limsup_{k\rightarrow \infty } 
	(f(\theta_{a(n_{k},\tau ) } ) - f(\theta_{n_{k} } ) )
	\leq &
	-
	\liminf_{k\rightarrow\infty} 
	\|\nabla f(\theta_{n_{k} } ) \| 
	\left(
	\|\nabla f(\theta_{n_{k} } ) \| 
	\sum_{i=n_{k} }^{a(n_{k},\tau ) - 1 } \alpha_{i} 
	-
	\left\|
	\sum_{i=n_{k} }^{a(n_{k},\tau )- 1} \alpha_{i} \xi_{i}
	\right\|
	\right)
	\nonumber\\
	&
	+
	\limsup_{k\rightarrow\infty} |\phi_{n_{k} }(\tau ) |
	\nonumber\\
	\leq &
	-\tau\phi^{2}/2 + C_{1,Q} \tau^{2} (\phi+\eta)^{2}
	\leq 
	-C_{3,Q} \phi^{2}
\end{align*}
(notice that $\eta<\phi$).
Consequently, 
\begin{align*}
	\limsup_{n\rightarrow \infty } 
	f(\theta_{n} ) 
	- 
	\liminf_{n\rightarrow \infty } 
	f(\theta_{n} )
	\geq 
	-
	\limsup_{k\rightarrow \infty } 
	(f(\theta_{a(n_{k},\tau ) } ) - f(\theta_{n_{k} } ) )
	\geq 
	C_{3,Q} \phi^{2}. 
\end{align*}
Hence, (\ref{l3.2.1*}) is true. 
\end{proof}

\begin{proposition}\label{proposition1.3}
Suppose that Assumptions \ref{a1.1}, \ref{a1.2} and \ref{a1.4} hold. 
Let $Q\subset\mathbb{R}^{d_{\theta } }$ be any compact set. 
Then, there exists a real number $K_{Q}\in[1,\infty )$
(independent of $\eta$ and depending only on $f(\cdot )$) 
such that 
\begin{align}\label{p1.3.1*}
	\limsup_{n\rightarrow\infty } \|\nabla f(\theta_{n} ) \|
	\leq 
	K_{Q} \eta^{s/2}, 
	\;\;\;\;\; 
	\limsup_{n\rightarrow\infty } f(\theta_{n} ) 
	-
	\liminf_{n\rightarrow\infty } f(\theta_{n} ) 
	\leq 
	K_{Q}\eta^{s}
\end{align}
on $\Lambda_{Q}\setminus N_{0}$. 
\end{proposition}

\begin{proof}
Let $Q\subset \mathbb{R}^{d_{\theta } }$ be any compact set, 
while $\tilde{C}_{Q}\in [1,\infty )$ stands for an upper bound of 
$\|\nabla f(\cdot ) \|$ on $Q$. 
Moreover, let 
$K_{Q} = \max\{2,\tilde{C}_{Q},C_{2,Q} \}$. 
Obviously, 
it is sufficient to show 
$\phi\leq K_{Q} \eta^{s/2}$ 
on $\Lambda_{Q}\setminus N_{0}$
(notice that the second inequality in (\ref{p1.3.1*}) 
is a direct consequence of Lemma \ref{lemma1.2}). 

Owing to Lemmas \ref{lemma1.2}, \ref{lemma1.3}, we have 
$C_{3,Q} \phi^{2} \leq C_{2,Q} \eta^{s}$
on 
$(\Lambda_{Q}\setminus N_{0} )\cap \{\phi > 2\eta \}$. 
Therefore, 
$\phi \leq (C_{2,Q}/C_{3,Q} )^{1/2} \eta^{s/2}
\leq K_{Q} \eta^{s/2}$
on 
$(\Lambda_{Q}\setminus N_{0} )\cap \{\phi > 2\eta \}$. 
On the other side, 
$\phi \leq 2\eta \leq K_{Q} \eta^{s/2}$
on 
$(\Lambda_{Q}\setminus N_{0} )\cap \{\phi \leq 2\eta, \eta\leq 1 \}$
(notice that $s/2 < 1$), 
while 
$\phi\leq \tilde{C}_{Q} \leq K_{Q} \eta^{s/2}$ on 
$(\Lambda_{Q}\setminus N_{0} )\cap \{\phi \leq 2\eta, \eta>1 \}$. 
Thus, 
$\phi \leq K_{Q} \eta^{s/2}$ indeed holds on 
$\Lambda_{Q}\setminus N_{0}$.
\end{proof}

\begin{vproof}{Parts (ii), (iii) of Theorem \ref{theorem1.1}}
{Part (ii) of the theorem directly follows from 
Propositions \ref{proposition1.1}, \ref{proposition1.3}, 
while  Part (iii) is a direct consequence of
Propositions \ref{proposition1.2}, \ref{proposition1.3} 
}
\end{vproof}

\section{Proof of Theorem \ref{theorem2.1}}\label{section2*}

The following notation is used in this section. 
For $\theta \in \mathbb{R}^{d_{\theta } }$, 
$z \in \mathbb{R}^{d_{z} }$, 
$E_{\theta, z }(\cdot )$ denotes 
the conditional expectation given 
$\theta_{0}=\theta$, $Z_{0}=z$. 
For $n\geq 1$, 
$\zeta_{n}$, $\xi_{n}$ are the random variables defined by 
\begin{align}\label{2.1*}
	\zeta_{n} 
	=
	F(\theta_{n}, Z_{n+1} ) - \nabla f(\theta_{n} ), 
	\;\;\;\;\; 
	\xi_{n} = \zeta_{n} + \eta_{n}, 
\end{align}
while 
$\zeta_{1,n}$, $\zeta_{2,n}$, $\zeta_{3,n}$
are random variables defined as 
\begin{align*} 
	&
	\zeta_{1,n} 
	=
	\tilde{F}(\theta_{n}, Z_{n+1} ) - (\Pi\tilde{F} )(\theta_{n}, Z_{n} ), 
	\;\;\;\: 
	\zeta_{2,n}
	=
	(\Pi\tilde{F} )(\theta_{n}, Z_{n} ) - (\Pi\tilde{F} )(\theta_{n-1}, Z_{n} ), 
	\;\;\;\:
	\zeta_{3,n}
	=
	-(\Pi\tilde{F} )(\theta_{n}, Z_{n+1} ).    
\end{align*} 
Then, it is straightforward to verify that 
algorithm (\ref{2.1}) admits the form (\ref{1.1}). 
Moreover,  
using Assumption \ref{a2.2}, it is easy to show 
\begin{align} \label{2.3*}
	\sum_{i=n}^{k} 
	\alpha_{i} \zeta_{i} 
	= &
	\sum_{i=n}^{k} 
	\alpha_{i} \zeta_{1,i} 
	+
	\sum_{i=n}^{k} 
	\alpha_{i} \zeta_{2,i} 
	+
	\sum_{i=n}^{k} 
	(\alpha_{i} - \alpha_{i+1} ) \zeta_{3,i} 
	+
	\alpha_{k+1} \zeta_{3,k} 
	-
	\alpha_{n} \zeta_{3,n-1} 
\end{align}
for $1\leq n \leq k$. 

\begin{vproof}{Theorem \ref{theorem2.1}}
{Let $Q\subset \mathbb{R}^{d_{\theta } }$ be any compact set and $\tilde{\Lambda}_{Q}$ be the event 
defined by $\tilde{\Lambda}_{Q}=\bigcap_{n=0}^{\infty} \{\theta_{n}\in Q \}$.  
Then, owing to Assumptions \ref{a2.1} and \ref{a2.3}, we have 
\begin{align}\label{t2.1.901}
	E_{\theta,z}\left(
	\sum_{n=0}^{\infty} (\alpha_{n}^{2} + \alpha_{n+1}^{2} ) 
	\varphi_{Q}^{2}(Z_{n+1} ) I_{ \{\tau_{Q}>n \} }
	\right)
	<
	\infty, 
	\;\;\;\;\; 
	E_{\theta,z}\left(
	\sum_{n=0}^{\infty} |\alpha_{n} - \alpha_{n+1} |
	\varphi_{Q}^{2}(Z_{n+1} ) I_{ \{\tau_{Q}>n \} }
	\right)
	<
	\infty
\end{align}
for all $\theta\in\mathbb{R}^{d_{\theta}}$, $z\in\mathbb{R}^{d_{z}}$. 

Let ${\cal F}_{n}=\sigma\{\theta_{0},Z_{0},\dots,\theta_{n},Z_{n} \}$ for $n\geq 0$. 
Since $\{\tau_{Q}>n\}\in{\cal F}_{n}$ for $n\geq 0$, 
Assumption \ref{a2.2} implies 
\begin{align*}
	E_{\theta,z}
	\left(
	\zeta_{1,n} 
	I_{ \{\tau_{Q}>n \} }
	|
	{\cal F}_{n} 
	\right)
	=
	\left(
	E_{\theta,z}
	(
	\tilde{F}(\theta_{n}, Z_{n+1} )
	|
	{\cal F}_{n} 
	)
	-
	(\Pi\tilde{F} )(\theta_{n},Z_{n} )
	\right) 
	I_{ \{\tau_{Q}>n \} }
	=
	0
\end{align*}
almost surely for each $\theta\in\mathbb{R}^{d_{\theta } }$, $z\in\mathbb{R}^{d_{z} }$, $n\geq 0$. 
Assumption \ref{a2.3} also yields 
\begin{align*}
	\|\zeta_{1,n} \| I_{ \{\tau_{Q}>n \} }
	\leq 
	\varphi_{Q}(Z_{n} ) I_{ \{\tau_{Q}>n-1 \} } + \varphi_{Q}(Z_{n+1} ) I_{ \{\tau_{Q}>n \} }  
\end{align*}
for $n\geq 0$. 
Combining this with (\ref{t2.1.901}), we get 
\begin{align*}
	E_{\theta,z}\left(
	\sum_{n=0}^{\infty } \alpha_{n}^{2} \|\zeta_{1,n} \|^{2} I_{ \{\tau_{Q}>n \} }
	\right)
	\leq
	2E_{\theta,z}\left(
	\sum_{n=0}^{\infty } (\alpha_{n}^{2} + \alpha_{n+1}^{2} ) \varphi_{Q}^{2}(Z_{n+1} ) I_{ \{\tau_{Q}>n \} }
	\right)
	<\infty
\end{align*}
for all $\theta\in\mathbb{R}^{d_{\theta } }$, $z\in\mathbb{R}^{d_{z} }$. 
Then, using Doob theorem, we conclude that 
$\sum_{n=0}^{\infty } \alpha_{n} \zeta_{1,n} I_{ \{\tau_{Q}>n \} }$
converges almost surely. 
As $\tilde{\Lambda}_{Q}\subseteq\{\tau_{Q}>n\}$ for $n\geq 0$, 
$\sum_{n=0}^{\infty } \alpha_{n}\zeta_{1,n}$ converges almost surely on $\tilde{\Lambda}_{Q}$.\footnote
{Notice that 
$\sum_{n=0}^{\infty } \alpha_{n} \zeta_{1,n} I_{ \{\tau_{Q}>n \} } = 
\sum_{n=0}^{\infty } \alpha_{n} \zeta_{1,n}$ on $\tilde{\Lambda}_{Q}$. } 

Due to Assumption \ref{a2.3}, we have 
\begin{align*}
	\|\zeta_{2,n} \| I_{\tilde{\Lambda}_{Q} }
	\leq &
	\varphi_{Q}(Z_{n} ) \|\theta_{n} - \theta_{n-1} \| I_{\tilde{\Lambda}_{Q} } 
	\\
	\leq &
	\alpha_{n-1} \varphi_{Q}(Z_{n} ) 
	(\|F(\theta_{n-1},Z_{n} ) \| + \|\eta_{n-1} \| ) I_{\tilde{\Lambda}_{Q} } 
	\\
	\leq &
	\alpha_{n-1} \varphi_{Q}(Z_{n} ) 
	(\varphi_{Q}(Z_{n} ) + \|\eta_{n-1} \| ) I_{\tilde{\Lambda}_{Q} }
	\\
	\leq &
	2\alpha_{n-1} ( \varphi_{Q}^{2}(Z_{n} ) + \|\eta_{n-1} \|^{2} ) I_{\tilde{\Lambda}_{Q} }
\end{align*}
for $n\geq 1$ (notice that $\varphi_{Q}(z)\geq 1$ for any $z\in\mathbb{R}^{d_{z} }$).  
Thus, 
\begin{align*}
	\sum_{n=1}^{j} \alpha_{n} \|\zeta_{2,n} \| I_{\tilde{\Lambda}_{Q} }
	\leq &
	2 \sum_{n=0}^{\infty } 
	\alpha_{n} \alpha_{n+1} 
	\left( \varphi_{Q}^{2}(Z_{n+1} ) + \|\eta_{n+1} \|^{2} \right) 
	I_{\tilde{\Lambda}_{Q} }
	\\
	\leq &
	\sum_{n=0}^{\infty } 
	(\alpha_{n}^{2} + \alpha_{n+1}^{2} ) \varphi_{Q}^{2}(Z_{n+1} ) 
	I_{ \{\tau_{Q}>n \} }
	+
	\sup_{n\geq 0} \|\eta_{n} \|^{2} I_{\tilde{\Lambda}_{Q} }
	\sum_{n=0}^{\infty } (\alpha_{n}^{2} +  \alpha_{n+1}^{2} )  
\end{align*}
(notice that $2\alpha_{n}\alpha_{n+1}\leq\alpha_{n}^{2}+\alpha_{n+1}^{2}$). 
Then, Assumption \ref{a2.4} and (\ref{t2.1.901}) imply that 
$\sum_{n=1}^{\infty } \alpha_{n} \zeta_{2,n}$ converges almost surely on $\tilde{\Lambda}_{Q}$. 

Owing to Assumption \ref{a2.3}, we have 
\begin{align*}
	&
	\|\zeta_{3,n} \| I_{\tilde{\Lambda}_{Q} } 
	\leq 
	\varphi_{Q}(Z_{n+1} ) I_{\tilde{\Lambda}_{Q} }
	\leq 
	\varphi_{Q}^{2}(Z_{n+1} ) I_{ \{\tau_{Q}>n \} }
\end{align*}
for $n\geq 0$. 
Hence, 
\begin{align*}
	&
	\sum_{n=0}^{\infty} \alpha_{n+1}^{2} \|\zeta_{3,n}\|^{2} I_{\tilde{\Lambda}_{Q} }
	\leq 
	\sum_{n=0}^{\infty} \alpha_{n+1}^{2} \varphi_{Q}^{2}(Z_{n+1} ) I_{ \{\tau_{Q}>n \} }, 
	\\
	&
	\sum_{n=0}^{\infty} |\alpha_{n} - \alpha_{n+1} | \: \|\zeta_{3,n} \| I_{\tilde{\Lambda}_{Q} }
	\leq 
	\sum_{n=0}^{\infty} |\alpha_{n} - \alpha_{n+1} | \varphi_{Q}^{2}(Z_{n+1} ) I_{ \{\tau_{Q}>n \} }. 
\end{align*}
Combining this with (\ref{t2.1.901}), we conclude 
$\lim_{n\rightarrow\infty} \alpha_{n+1}\zeta_{3,n}=0$
almost surely on $\tilde{\Lambda}_{Q}$. 
We also deduce that
$\sum_{n=0}^{\infty } (\alpha_{n} - \alpha_{n+1} ) \zeta_{3,n}$
converges almost surely on $\tilde{\Lambda}_{Q}$. 
Since $\sum_{n=0}^{\infty } \alpha_{n}\zeta_{1,n}$, 
$\sum_{n=1}^{\infty } \alpha_{n}\zeta_{2,n}$ converge almost surely on $\tilde{\Lambda}_{Q}$, 
(\ref{2.3*}) implies that 
$\sum_{n=0}^{\infty } \alpha_{n}\zeta_{n}$ also converges almost surely on $\tilde{\Lambda}_{Q}$. 
As $Q$ is any compact set in $\mathbb{R}^{d_{\theta}}$, 
$\sum_{n=0}^{\infty} \alpha_{n}\zeta_{n}$ converges almost surely on 
$\{\sup_{n\geq 0}\|\theta_{n}\|<\infty \}$. 
Consequently, Assumption \ref{a2.4} yields that 
$\{\xi_{n} \}_{n\geq 0}$ defined in (\ref{2.1*}) satisfies Assumption \ref{a1.2}. 
Then, the theorem's assertion directly follows from Theorem \ref{theorem1.1}. 
}
\end{vproof}

\section{Proof of Theorem \ref{theorem3.1}}\label{section3*}

In this section, we use the following notation. 
$\phi(v)$, $s_{\theta }(v)$ are the functions defined by 
\begin{align*}
	\phi(v)=\phi(x,y), 
	\;\;\;\;\; 
	s_{\theta }(v) = s_{\theta }(x,y)
\end{align*}
for $\theta\in\mathbb{R}^{d_{\theta } }$, $v = (x,y) \in {\cal X} \times {\cal Y}$. 
For $\theta\in\mathbb{R}^{d_{\theta} }$, 
$\{V_{n}^{\theta} \}_{n\geq 0}$, 
$\{W_{n}^{\theta} \}_{n\geq 0}$ and 
$\{Z_{n}^{\theta} \}_{n\geq 0}$ are stochastic processes defined by 
\begin{align*}
	V_{n}^{\theta} = (X_{n}^{\theta}, Y_{n}^{\theta} ), 
	\;\;\;\;\; 
	W_{n+1}^{\theta} = \lambda W_{n}^{\theta} + s_{\theta}(V_{n}^{\theta} ), 
	\;\;\;\;\; 
	Z_{n}^{\theta} = (V_{n}^{\theta}, W_{n}^{\theta} )
\end{align*} 
for $n\geq 0$, where $W_{0}^{\theta}\in\mathbb{R}^{d_{\theta} }$
is a (deterministic) vector
(notice that $\{V_{n}^{\theta} \}_{n\geq 0}$, $\{Z_{n}^{\theta} \}_{n\geq 0}$ 
are Markov chains). 
Moreover, for $\theta\in\mathbb{R}^{d_{\theta} }$, 
$r_{\theta }(\cdot|\cdot )$ and $\nu_{\theta }(\cdot )$
are the transition kernel and invariant probability 
of $\{V_{n}^{\theta } \}_{n\geq 0}$,\footnote
{Under Assumption \ref{a3.1}, 
$\nu_{\theta}(\cdot )$ exists and is unique (the details are provided in Lemma \ref{lemma3.1}). 
The transition $r_{\theta}(\cdot|\cdot )$ can be defined by 
$r_{\theta }(v'|v) = q_{\theta }(y'|x') p(x'|x,y)$ 
for $v = (x,y) \in {\cal X} \times {\cal Y}$, 
$v' = (x',y') \in {\cal X} \times {\cal Y}$. } 
while $\Pi_{\theta }(\cdot,\cdot )$ is the transition kernel
of $\{Z_{n}^{\theta} \}_{n\geq 0}$.\footnote 
{$\Pi_{\theta }(\cdot,\cdot )$ can be defined by 
$\Pi_{\theta }(z,\{v'\}\times B ) = I_{B}(\lambda w + s_{\theta }(v') ) r_{\theta }(v'|v)$ for 
$z=(v,w) \in ({\cal X} \times {\cal Y} ) \times \mathbb{R}^{d_{\theta } }$
and a Borel-measurable set $B\subseteq \mathbb{R}^{d_{\theta } }$. }
For $\theta\in\mathbb{R}^{d_{\theta } }$, $n\geq 0$, 
$r_{\theta }^{n}(\cdot|\cdot )$ is the $n$-th transition probability of 
$\{V_{n}^{\theta } \}_{n\geq 0}$, 
while 
\begin{align*}
	\tilde{r}_{\theta}^{n}(v'|v)
	=
	r_{\theta }^{n}(v'|v)
	-
	\nu_{\theta }(v')
\end{align*}
for $\theta\in\mathbb{R}^{d_{\theta } }$, 
$v, v' \in {\cal X} \times {\cal Y}$, $n\geq 0$. 
Additionally, the functions $\eta(\cdot )$, $F(\cdot,\cdot )$ are defined by 
\begin{align*}
	&
	\eta(\theta )
	=
	\sum_{n=0}^{\infty } 
	\sum_{v,v'\in {\cal X} \times {\cal Y} } 
	\lambda^{n} \phi(v') 
	\tilde{r}_{\theta}^{n}(v'|v) 
	s_{\theta }(v) 
	\nu_{\theta }(v)
	-
	\nabla f(\theta ), 
	\;\;\;\;\; 
	F(\theta, z ) 
	=
	\phi(v) w - \eta(\theta )
\end{align*}
for $\theta\in R^{d_{\theta } }$, 
$z=(v,w) \in ({\cal X} \times {\cal Y} ) \times \mathbb{R}^{d_{\theta } }$.\footnote
{Under Assumptions \ref{a3.1}, \ref{a3.2}, $f(\cdot )$ is differentiable 
(the details are provided in Lemma \ref{lemma3.2}). } 
$\{Z_{n} \}_{n\geq 0}$, $\{\eta_{n} \}_{n\geq }$ are the stochastic processes defined as 
\begin{align*}
	Z_{n}=(X_{n},Y_{n},W_{n} ), 
	\;\;\;\;\; 
	\eta_{n}=\eta(\theta_{n} )
\end{align*}
for $n\geq 0$. 
Then, it is straightforward to show that the 
algorithm (\ref{3.1}) is of the same form 
as the recursion studied in Section \ref{section2}
(i.e., $\{\theta_{n} \}_{n\geq 0}$, $\{\eta_{n} \}_{n\geq 0}$, 
$F(\cdot,\cdot )$, $\Pi_{\theta }(\cdot,\cdot )$
defined in Section \ref{section3} and here 
admit (\ref{2.1}), (\ref{2.3})).

We will use the following additional notation.  
$N_{v}$ is the integer defined by $N_{v} = N_{x}N_{y}$, 
while $e\in\mathbb{R}^{N_{v} }$ is the vector whose all components are one. 
For $v\in{\cal X}\times{\cal Y}$, 
$e(v)\in\mathbb{R}^{N_{v} }$ is the vector representation of $I_{v}(\cdot )$,
while $\phi\in\mathbb{R}^{N_{v} }$ is the vector representation of $\phi(\cdot)$.\footnote{
For $v=(x,y)\in{\cal X}\times{\cal Y}$, 
element $i$ of $e(v)$ is one if $i=(x-1)N_{y}+y$ and zero otherwise. 
For the same $v$, $\phi(v)$ is element $(x-1)N_{y}+y$ of $\phi$. } 
For $\theta\in\mathbb{R}^{d_{\theta} }$, 
$R_{\theta}\in\mathbb{R}^{N_{v}\times N_{v} }$ and $\nu_{\theta}\in\mathbb{R}^{N_{v}}$ are the transition matrix 
and the invariant probability vector of 
$\{V_{n}^{\theta} \}_{n\geq 0}$,\footnote{
For $v=(x,y)\in{\cal X}\times{\cal Y}$, $v'=(x',y')\in{\cal X}\times{\cal Y}$, 
$r_{\theta}(v'|v)$ is entry $((x-1)N_{y}+y,(x'-1)N_{y}+y')$ of $R_{\theta}$, 
while $\nu_{\theta}(v)$ is element $(x-1)N_{y}+y$ of $\nu_{\theta}$. 
} 
while $\tilde{R}_{\theta} = R_{\theta}-e\nu_{\theta}^{T}$. 
For $\theta\in\mathbb{R}^{d_{\theta} }$, $1\leq j\leq d_{\theta}$, 
$s_{\theta,j}(\cdot )$ is the $j$-th component of $s_{\theta}(\cdot )$, 
while $S_{\theta,j}\in\mathbb{R}^{N_{v}\times N_{v} }$
is the diagonal matrix representation of $s_{\theta,j}(\cdot )$.\footnote{
For $v=(x,y)\in{\cal X}\times{\cal Y}$, $s_{\theta,j}(v)$ is entry 
$((x-1)N_{y}+y,(x-1)N_{y}+y)$ of $S_{\theta,j}$. 
The off-diagonal elements of $S_{\theta,j}$ are zero. }

\begin{lemma}\label{lemma3.1}
Suppose that Assumptions \ref{a3.1} and \ref{a3.2} hold. 
Let $Q\subset\mathbb{R}^{d_{\theta } }$ be any compact set. 
Then, the following is true: 
\begin{compactenum}[(i)]
\item
$\{V_{n}^{\theta } \}_{n\geq 0}$ is geometrically ergodic for each 
$\theta\in \mathbb{R}^{d_{\theta } }$. 
Moreover, there exist real numbers $\varepsilon_{Q}\in(0,1)$, $C_{1,Q}\in[1,\infty )$ 
(independent of $\lambda$) such 
$\|\tilde{R}_{\theta}^{n} \| \leq C_{1,Q}\varepsilon_{Q}^{n}$ 
for all $\theta\in Q$, $n\geq 0$. 
\item
There exists a real number 
$C_{2,Q} \in [1,\infty )$ (independent of $\lambda$) such that 
\begin{align}
	&\label{l3.1.3*}
	\max\{\|\nu_{\theta'} - \nu_{\theta''} \|, \|R_{\theta'}^{n} - R_{\theta''}^{n} \| \}
	\leq 
	C_{2,Q} \|\theta' - \theta'' \|, 
	\\
	&\label{l3.1.5*}
	\|\tilde{R}_{\theta'}^{n} - \tilde{R}_{\theta''}^{n} \|
	\leq 
	C_{2,Q} \varepsilon_{Q}^{n} \|\theta' - \theta'' \|
\end{align}
for all $\theta', \theta'' \in Q$, $n\geq 0$. 
\item
$\nu_{\theta }$ is differentiable on $\mathbb{R}^{d_{\theta} }$. 
Moreover,  
$\nabla_{\theta } \nu_{\theta }$ is locally Lipschitz continuous 
on $\mathbb{R}^{d_{\theta } }$. 
\item
If Assumption \ref{a3.3.a} is satisfied, 
$\nu_{\theta }$ is $p$ times differentiable on $\mathbb{R}^{d_{\theta } }$. 
\item
If Assumption \ref{a3.3.b} is satisfied, 
$\nu_{\theta }$ is real-analytic on $\mathbb{R}^{d_{\theta } }$. 
\end{compactenum}
\end{lemma}

\begin{proof}
(i) 
For $\theta\in\mathbb{R}^{d_{\theta } }$, $n\geq 0$, 
let $p_{\theta}^{n}(\cdot|\cdot )$ and $\mu_{\theta }(\cdot )$ be the $n$-th transition probability 
and the invariant probability of $\{X_{n}^{\theta } \}_{n\geq 0}$. 
Moreover, for $\theta\in\mathbb{R}^{d_{\theta}}$, $v=(x,y)\in{\cal X}\times{\cal Y}$, 
let $\tilde{\nu}_{\theta}(v)=q_{\theta}(y|x)\mu_{\theta}(x)$. 
Then, it is straightforward to verify 
\begin{align*}
	r_{\theta}^{n+1}(v'|v) - \tilde{\nu}_{\theta}(v')
	=
	\sum_{x''\in{\cal X} } 
	q_{\theta}(y'|x') (p_{\theta}^{n}(x'|x'') - \mu_{\theta}(x') ) p(x''|x,y)
\end{align*}
for $\theta\in\mathbb{R}^{d_{\theta} }$, 
$v=(x,y)\in{\cal X}\times{\cal Y}$, $v'=(x',y')\in{\cal X}\times{\cal Y}$, 
$n\geq 0$. 
Therefore, 
\begin{align*}
	|r_{\theta}^{n+1}(v'|v) - \tilde{\nu}_{\theta}(v') |
	\leq 
	\sum_{x''\in{\cal X} } 
	q_{\theta}(y'|x') |p_{\theta}^{n}(x'|x'') - \mu_{\theta}(x') | p(x''|x,y)
	\leq 
	N_{x} \max_{x''\in{\cal X} } 
	|p_{\theta}^{n}(x'|x'')-\mu_{\theta}(x')|
\end{align*}
for all $\theta\in \mathbb{R}^{d_{\theta}}$, 
$v=(x,y)\in{\cal X}\times{\cal Y}$, $v'=(x',y')\in{\cal X}\times{\cal Y}$, 
$n\geq 0$. 
Combining this with Assumption \ref{a3.1}, we conclude that 
$\{V_{n}^{\theta} \}_{n\geq 0}$ is geometrically ergodic for each $\theta\in\mathbb{R}^{d_{\theta}}$. 
We also conclude that $\tilde{\nu}_{\theta}(\cdot)$ is the invariant probability of $\{V_{n}^{\theta} \}_{n\geq 0}$
for each $\theta\in\mathbb{R}^{d_{\theta}}$, 
i.e., $\nu_{\theta}(v)=\tilde{\nu}_{\theta}(v)=q_{\theta}(y|x)\mu_{\theta}(x)$
for $\theta\in\mathbb{R}^{d_{\theta}}$, $v=(x,y)\in{\cal X}\times{\cal Y}$. 

For $\theta\in\mathbb{R}^{d_{\theta}}$, let 
$\rho_{\theta }=\min_{v\in{\cal V} }\nu_{\theta }(x)/3$. 
Then, we have 
$0<\rho_{\theta }\leq 1/(3N_{v} )$, 
$\rho_{\theta }\leq\nu_{\theta }(v)/3$ for all 
$\theta\in\mathbb{R}^{d_{\theta } }$, $v\in{\cal V}$. 
Moreover, for any $\theta\in\mathbb{R}^{d_{\theta } }$, 
there exists an integer $n_{\theta }\geq 0$ such that 
$|r_{\theta }^{n}(v'|v)-\nu_{\theta }(v')|\leq\rho_{\theta }$ 
for each $v,v'\in{\cal V}$, $n\geq n_{\theta }$. 
Hence, 
$r_{\theta }^{n}(v'|v)\geq\nu_{\theta }(v')-\rho_{\theta }\geq 2\rho_{\theta }$ 
for all $\theta\in\mathbb{R}^{d_{\theta } }$, $v,v'\in{\cal V}$, $n\geq n_{\theta }$. 
Additionally, Assumption \ref{a3.2} implies that for each $v,v'\in{\cal V}$, $n\geq 0$, 
$r_{\theta }^{n}(v'|v)$ is locally Lipschitz continuous in $\theta$ on $\mathbb{R}^{d_{\theta } }$.\footnote
{Notice that, due to Assumption \ref{a3.2},  
$q_{\theta }(y|x)$ is locally Lipschitz continuous in $\theta$ for each $x\in{\cal X}$, $y\in{\cal Y}$ and that
 $r_{\theta }^{n}(\cdot|\cdot)$ is a polynomial function of $p(\cdot|\cdot,\cdot )$, 
$q_{\theta}(\cdot|\cdot )$. } 
Consequently, for any $\theta\in\mathbb{R}^{d_{\theta } }$, 
there exists a real number $\delta_{\theta }\in (0,1)$ such that 
$|r_{\vartheta }^{n_{\theta } }(v'|v) - r_{\theta }^{n_{\theta } }(v'|v) |\leq \rho_{\theta }$
for all $\vartheta\in\mathbb{R}^{d_{\theta } }$, $v,v'\in{\cal V}$
satisfying $\|\vartheta-\theta\|\leq\delta_{\theta }$. 
Thus, 
$r_{\vartheta }^{n_{\theta } }(v'|v)\geq r_{\theta }^{n_{\theta } }(v'|v) - \rho_{\theta }\geq\rho_{\theta}$
for each $\vartheta\in\mathbb{R}^{d_{\theta } }$, $v,v'\in{\cal V}$
satisfying $\|\vartheta-\theta\|\leq\delta_{\theta }$. 
Since 
\begin{align*}
	r_{\vartheta }^{n}(v'|v)
	=
	\sum_{v''\in{\cal V} } 
	r_{\vartheta }^{n_{\theta } }(v'|v'') r_{\vartheta }^{n-n_{\theta } }(v''|v) 
	\geq 
	\rho_{\theta } 
	\sum_{v''\in{\cal V} } 
	r_{\vartheta }^{n-n_{\theta } }(v''|v) 
	=
	\rho_{\theta }
\end{align*}
for any $\vartheta\in\mathbb{R}^{d_{\theta } }$, $v,v'\in{\cal V}$, $n\geq n_{\theta}$
satisfying $\|\vartheta-\theta\|\leq\delta_{\theta }$, 
we conclude $r_{\vartheta }^{n}(v'|v)\geq\rho_{\theta }$ 
for the same $\vartheta$, $v,v'$, $n$. 

Let $B_{\theta } = \{\vartheta\in\mathbb{R}^{d_{\theta } }: \|\vartheta-\theta\|<\delta_{\theta } \}$
for $\theta\in\mathbb{R}^{d_{\theta } }$. 
As $\{B_{\theta } \}_{\theta\in Q}$ is an open covering of $Q$, 
there exists a finite set $\tilde{Q}\subseteq Q$ such that 
$\bigcup_{\theta\in\tilde{Q} } B_{\theta } \supset Q$. 
Let $\tilde{n}_{Q}=\max_{\theta\in\tilde{Q} } n_{\theta }$, 
$\tilde{\rho}_{Q}=\min_{\theta\in\tilde{Q} } \rho_{\theta }$,  
$\tilde{\varepsilon}_{Q} = (1-\tilde{\rho}_{Q} )^{1/\tilde{n}_{Q}}$. 
Since each element of $Q$ is also an element of one of $\{ B_{\theta } \}_{\theta\in\tilde{Q} }$, 
we have $r_{\theta}^{n}(v'|v)\geq\tilde{\rho}_{Q}$ 
for all $\theta\in Q$, $v,v'\in{\cal V}$, $n\geq\tilde{n}_{Q}$.\footnote
{If $\theta\in B_{\vartheta }$ and $\vartheta\in\tilde{Q}$, 
then $n_{\vartheta}\leq \tilde{n}_{Q}$ 
and $r_{\theta }^{n}(v'|v)\geq\rho_{\vartheta}\geq\tilde{\rho}_{Q}$ for $n\geq n_{\vartheta }$. } 
Then, standard results of Markov chain theory (see e.g., \cite[Theorem 16.0.2]{meyn&tweedie})
imply 
\begin{align*}
	|r_{\theta}^{n}(v'|v) - \nu_{\theta}(v') |
	\leq 
	(1-\tilde{\rho}_{Q}N_{v} )^{n/\tilde{n}_{Q} }
	\leq 
	\tilde{\varepsilon}_{Q}^{n}
\end{align*}
for all $\theta\in Q$, $v,v'\in{\cal V}$, $n\geq 0$. 

Let $\varepsilon_{Q}=\tilde{\varepsilon}_{Q}^{1/2}$, 
$C_{1,Q}=N_{v}$. 
Then, we have 
\begin{align}\label{l3.1.51}
	\|\tilde{R}_{\theta}^{n} \|
	\leq 
	N_{v} \max_{v,v'\in{\cal X}\times{\cal Y} } |\tilde{r}_{\theta}^{n}(v'|v) |
	\leq 
	N_{v}\tilde{\varepsilon}_{Q}^{n}
	=
	C_{1,Q} \varepsilon_{Q}^{2n}
\end{align}
for all $\theta\in Q$, $n\geq 0$. 

(ii) 
Let $g$ be the $N_{v}$-th standard unit vector in $\mathbb{R}^{N_{v}}$
(i.e., the first $N_{v}-1$ elements of $g$ are zero, while the last element of $g$ is one) and, for $A\in\mathbb{R}^{N_{v}\times N_{v} }$, 
let $G(A)$ be the $N_{v}\times N_{v}$ matrix obtained when the last row of 
$I-A^{T}$ is replaced by $e^{T}$
(here, $I$ is the $N_{v}\times N_{v}$ unit matrix). 
Additionally, 
let ${\cal Q}_{0}^{N_{v}\times N_{v} } = \{A\in \mathbb{R}^{N_{v}\times N_{v} }: \text{det}(G(A))\neq 0 \}$ and, for $A\in {\cal Q}_{0}^{N_{v}\times N_{v} }$, 
let $h(A)=(G(A))^{-1}g$. 
Then, it is easy to conclude that ${\cal Q}_{0}^{N_{v}\times N_{v} }$ is an open set 
(notice that $\text{det}(G(A))$ is a polynomial function of the entries of $A$). 
It is also easy to deduce that $h(\cdot)$ is well-defined and real-analytic on ${\cal Q}_{0}^{N_{v}\times N_{v} }$
(notice that due to the Cramer's rule, 
all elements of $h(A)$ are rational functions of the entries of $A$). 

Let ${\cal P}_{0}^{N_{v}\times N_{v} }$ be the set of $N_{v}\times N_{v}$
geometrically ergodic stochastic matrices.
Then, each $P\in{\cal P}_{0}^{N_{v}\times N_{v} }$ has a unique invariant probability vector. 
Moreover, the invariant probability vector of $P\in{\cal P}_{0}^{N_{v}\times N_{v} }$ 
is the unique solution to the linear system of equations $G(P)x=g$, where $x\in\mathbb{R}^{N_{v}}$ is the unknown. 
Hence, $\text{det}(G(P))\neq 0$ for each $P\in{\cal P}_{0}^{N_{v}\times N_{v} }$ so ${\cal P}_{0}^{N_{v}\times N_{v} }\subset{\cal Q}_{0}^{N_{v}\times N_{v} }$.

Owing to (i), $R_{\theta}\in{\cal P}_{0}^{N_{v}\times N_{v} }$ for each $\theta\in\mathbb{R}^{d_{\theta}}$. 
Thus, $\nu_{\theta}=h(R_{\theta})$ for all $\theta\in\mathbb{R}^{d_{\theta}}$. 
Moreover, due to Assumption \ref{a3.2}, $R_{\theta}$ is locally Lipschitz continuous on 
$\mathbb{R}^{d_{\theta}}$.\footnote{
Notice that $r_{\theta}(v'|v)= q_{\theta}(y'|x')p(x'|x,y)$
for $v=(x,y)$, $v'=(x',y')$ and that 
$q_{\theta}(y|x)$ is locally Lipschitz continuous in $\theta$. }  
Since $h(\cdot )$ is real-analytic on ${\cal Q}_{0}^{N_{v}\times N_{v} }$
and ${\cal P}_{0}^{N_{v}\times N_{v} }\subset{\cal Q}_{0}^{N_{v}\times N_{v} }$, 
$\nu_{\theta}$ is locally Lipschitz continuous on $\mathbb{R}^{d_{\theta}}$. 

Let $\tilde{C}_{1,Q}\in[1,\infty )$ be a Lipschitz constant of $R_{\theta }$, $\nu_{\theta }$ on $Q$, 
while $\tilde{C}_{2,Q}\in[1,\infty )$ is an upper bound of the sequence 
$\{n\varepsilon_{Q}^{n} \}_{n\geq 1}$. 
Let $C_{2,Q}=3\varepsilon_{Q}^{-1}C_{1,Q}^{2}\tilde{C}_{1,Q}\tilde{C}_{2,Q}$. 
It is straightforward to verify 
\begin{align*}
	\tilde{R}_{\theta'}^{n+1} 
	-
	\tilde{R}_{\theta''}^{n+1} 
	=&
	\sum_{i=0}^{n} 
	\tilde{R}_{\theta'}^{i} 
	(R_{\theta'} - R_{\theta''} - e(\nu_{\theta'} - \nu_{\theta''} )^{T} )
	\tilde{R}_{\theta''}^{n-i} 
\end{align*}
for $\theta',\theta'' \in \mathbb{R}^{d_{\theta } }$, $n\geq 0$. 
Combining this with (\ref{l3.1.51}), we get 
\begin{align*}
	\|\tilde{R}_{\theta'}^{n+1} 
	-
	\tilde{R}_{\theta''}^{n+1} \|
	\leq &
	\sum_{i=0}^{n} 
	\|\tilde{R}_{\theta'}^{i} \| \|\tilde{R}_{\theta''}^{n-i} \|
	\left(
	\|R_{\theta'} - R_{\theta''} \| + \|\nu_{\theta'} - \nu_{\theta''} \| 
	\right)
	\nonumber\\
	\leq &
	2 C_{1,Q}^{2} \tilde{C}_{1,Q} 
	(n+1) \varepsilon_{Q}^{2n} \|\theta' - \theta'' \|
	\nonumber\\
	\leq &
	C_{2,Q}\varepsilon_{Q}^{n} \|\theta' - \theta'' \|
\end{align*}
for each $\theta',\theta'' \in Q$, $n\geq 0$. 
Therefore, 
\begin{align*}
	\|R_{\theta'}^{n} - R_{\theta''}^{n} \|
	\leq 
	\|\tilde{R}_{\theta'}^{n} - \tilde{R}_{\theta''}^{n} \|
	+
	\|\nu_{\theta'} - \nu_{\theta''} \|
	\leq 
	\tilde{C}_{1,Q} (2C_{1,Q}^{2}\tilde{C}_{2,Q}n\varepsilon_{Q}^{n-1}  + 1 ) \|\theta' - \theta'' \|
	\leq 
	C_{2,Q} \|\theta' - \theta'' \|
\end{align*}
for all $\theta',\theta'' \in Q$, $n\geq 0$
(notice that $\tilde{R}_{\theta}^{k} = R_{\theta}^{k} - e\nu_{\theta}^{T}$). 

(iii)  
Due to (i), $R_{\theta}\in{\cal P}_{0}^{N_{v}\times N_{v} }$ for each $\theta\in\mathbb{R}^{d_{\theta}}$. 
Hence, $\nu_{\theta}=h(R_{\theta})$ for all $\theta\in\mathbb{R}^{d_{\theta}}$. 
Moreover, owing to Assumption \ref{a3.2}, 
$R_{\theta}$ is differentiable on $\mathbb{R}^{d_{\theta} }$ and its first-order derivatives are locally Lipschitz continuous on the same space.\footnote
{Notice that 
$\nabla_{\theta } r_{\theta }(v'|v)=
s_{\theta}(x',y') q_{\theta}(y'|x') p(x'|x,y)$ for $v=(x,y)$, $v'=(x',y')$. 
}
As $h(\cdot )$ is real-analytic on ${\cal Q}_{0}^{N_{v}\times N_{v} }$
and ${\cal P}_{0}^{N_{v}\times N_{v} }\subset{\cal Q}_{0}^{N_{v}\times N_{v} }$, 
$\nu_{\theta}$ is differentiable on $\mathbb{R}^{d_{\theta} }$. 
The same arguments also imply that  
$\nabla_{\theta}\nu_{\theta}$ is locally Lipschitz continuous on $\mathbb{R}^{d_{\theta} }$. 

(iv), (v) 
If Assumption \ref{a3.3.a} is satisfied, 
then $R_{\theta }$ is $p$ times differentiable on $\mathbb{R}^{d_{\theta} }$, 
and consequently, $\nu_{\theta}$ is $p$ times differentiable on $\mathbb{R}^{d_{\theta} }$, too.\footnote
{Notice that $R_{\theta}\in{\cal P}_{0}^{N_{v}\times N_{v} }\subset{\cal Q}_{0}^{N_{v}\times N_{v} }$, 
$\nu_{\theta}=h(R_{\theta})$ for all $\theta\in\mathbb{R}^{d_{\theta}}$.
Notice also that $h(\cdot)$ is real-analytic on ${\cal Q}_{0}^{N_{v}\times N_{v} }$. } 
Similarly, if Assumption \ref{a3.3.b} is satisfied, 
then $R_{\theta }$ is real-analytic on $\mathbb{R}^{d_{\theta} }$, 
and therefore, $\nu_{\theta }$ is also real-analytic on $\mathbb{R}^{d_{\theta} }$. 
\end{proof}

\begin{lemma}\label{lemma3.2}
Suppose that Assumptions \ref{a3.1} and \ref{a3.2} hold. 
Let $Q\subset\mathbb{R}^{d_{\theta } }$ be any compact set. 
Then, the following is true: 
\begin{compactenum}[(i)]
\item
$f(\cdot )$ is differentiable and $\nabla f(\cdot )$ is locally Lipschitz continuous. 
\item
There exists a real number 
$C_{3,Q} \in [1,\infty )$ (independent of $\lambda$) such that 
$\|\eta(\theta ) \|\leq C_{3,Q} (1-\lambda )$
for all $\theta \in Q$. 
\item
If Assumption \ref{a3.3.a} is satisfied, 
$f(\cdot )$ is $p$ times differentiable. 
\item
If Assumption \ref{a3.3.b} is satisfied, 
$f(\cdot )$ is real-analytic. 
\end{compactenum}
\end{lemma}

\begin{proof}
(i), (iii), (iv) Owing to Lemma \ref{lemma3.1}, we have 
\begin{align}\label{l3.2.901}
	f(\theta )
	=
	\lim_{n\rightarrow\infty } E_{\theta }(\phi(V_{n}^{\theta } ) )
	=
	\sum_{v\in{\cal X}\times {\cal Y} }
	\phi(v) \nu_{\theta }(v)
	=
	\phi^{T}\nu_{\theta} 
\end{align}
for all $\theta\in\mathbb{R}^{d_{\theta } }$. 
Then, these parts of the lemma directly follow from Lemma \ref{lemma3.1}. 

(ii) 
For each $1\leq j\leq d_{\theta}$, let  
$\tilde{C}_{Q}\in[1,\infty )$ be an upper bound of $\|S_{\theta,j} \|$
on $Q$. For $\theta\in \mathbb{R}^{d_{\theta } }$, $v\in {\cal X} \times {\cal Y}$, $n\geq 0$, 
let also define
\begin{align}\label{l3.2.1}
	f_{n}(\theta, v )
	=
	\sum_{v'\in{\cal X}\times{\cal Y} } \phi(v') r_{\theta}^{n}(v'|v), 
	\;\;\;\;\; 
	h(\theta )
	=
	\sum_{n=0}^{\infty } \;
	\sum_{v,v' \in {\cal X} \times {\cal Y} }
	\phi(v') \tilde{r}_{\theta}^{n}(v'|v) s_{\theta}(v) \nu_{\theta }(v). 
\end{align}

Owing to Lemma \ref{lemma3.1}, 
$f_{n}(\theta, v)$ converges to $f(\theta )$ as $n\rightarrow\infty$
uniformly in $(\theta, v )$ on $Q \times ({\cal X} \times {\cal Y} )$. 
Due to the same lemma, 
$h(\cdot )$ is well-defined on $Q$ 
(notice that when $\theta\in Q$, each term in the sum in (\ref{l3.2.1}) tends to zero 
at the rate $\varepsilon_{Q}^{n}$). 
Moreover, it is straightforward to show
\begin{align}\label{l3.2.903}
	\nabla_{\theta } f_{n}(\theta, v_{0} )
	=&
	\nabla_{\theta }\left(
	\sum_{v_{1},\dots,v_{n}\in {\cal X}\times{\cal Y} } 
	\phi(v_{n} ) 
	\left(\prod_{i=1}^{n} r_{\theta}(v_{i}|v_{i-1} ) \right)
	\right)
	\nonumber\\
	=&
	\sum_{v_{1},\dots,v_{n}\in {\cal X}\times{\cal Y} } 
	\phi(v_{n} ) 
	\left(\sum_{i=1}^{n} \frac{\nabla_{\theta} r_{\theta}(v_{i}|v_{i-1} ) }{r_{\theta}(v_{i}|v_{i-1} ) } \right)
	\left(\prod_{i=1}^{n} r_{\theta}(v_{i}|v_{i-1} ) \right)
	\nonumber\\
	=&
	\sum_{i=1}^{n} \; 
	\sum_{v',v'' \in {\cal X} \times {\cal Y} } 
	\phi(v'') 
	r_{\theta}^{n-i}(v''|v') 
	s_{\theta }(v') 
	r_{\theta}^{i}(v'|v_{0} )
\end{align}
for all $\theta\in \mathbb{R}^{d_{\theta } }$, $v_{0}\in {\cal X} \times {\cal Y}$, $n\geq 1$. 
Therefore, 
\begin{align}\label{l3.2.1'''}
	\partial_{\theta}^{j} f_{n}(\theta, v )
	=
	\sum_{i=1}^{n} e^{T}(v) R_{\theta}^{i} S_{\theta,j} R_{\theta}^{n-i} \phi
	=
	\sum_{i=0}^{n-1} e^{T}(v) R_{\theta}^{n-i} S_{\theta,j} R_{\theta}^{i} \phi
\end{align}
for $\theta\in \mathbb{R}^{d_{\theta } }$, $v\in {\cal X} \times {\cal Y}$, 
$1\leq j\leq d_{\theta}$, $n\geq 1$, 
where $\partial_{\theta}^{j} f_{n}(\theta, v )$ is the $j$-th component of 
$\nabla_{\theta} f_{n}(\theta,v)$. 
We also have 
\begin{align}\label{l3.2.151}
	\sum_{i=0}^{n-1} e^{T}(v) R_{\theta}^{n-i} S_{\theta,j} R_{\theta}^{i} e = 0
\end{align}
for $\theta\in \mathbb{R}^{d_{\theta } }$, $v\in {\cal X} \times {\cal Y}$, 
$1\leq j\leq d_{\theta}$, $n\geq 1$.\footnote{
If $\phi=e$, then $f_{n}(\theta,v)$ is identically one, 
while $\nabla_{\theta} f_{n}(\theta,v)$ is identically zero. 
Hence, (\ref{l3.2.1'''}) reduces to (\ref{l3.2.151}) when $\phi=e$. } 
Hence, 
\begin{align*}
	\sum_{i=0}^{n-1} 
	e^{T}(v) R_{\theta}^{n-i} S_{\theta,j} e \nu_{\theta}^{T} \phi
	=
	\nu_{\theta}^{T} \phi
	\sum_{i=0}^{n-1} e^{T}(v) R_{\theta}^{n-i} S_{\theta,j} R_{\theta}^{i} e
	=
	0
\end{align*}
for $\theta\in \mathbb{R}^{d_{\theta } }$, $v\in {\cal X} \times {\cal Y}$, 
$1\leq j\leq d_{\theta}$, $n\geq 1$
(notice that $R_{\theta}^{i}e = e$).
Therefore, 
\begin{align*}
	\partial_{\theta}^{j} f_{n}(\theta,v) 
	=
	\sum_{i=0}^{n-1} e^{T}(v) R_{\theta}^{n-i} S_{\theta,j} \tilde{R}_{\theta}^{i} \phi
\end{align*}
for $\theta\in \mathbb{R}^{d_{\theta } }$, $v\in {\cal X} \times {\cal Y}$, 
$1\leq j\leq d_{\theta}$, $n\geq 1$. 
Additionally, we have 
\begin{align*}
	h_{j}(\theta)
	=
	\sum_{n=0}^{\infty } 
	\nu_{\theta}^{T} S_{\theta,j} \tilde{R}_{\theta}^{n} \phi
	=
	\sum_{i=0}^{n-1} 
	e^{T}(v) e
	\nu_{\theta}^{T} S_{\theta,j} \tilde{R}_{\theta}^{i} \phi
	+
	\sum_{i=n}^{\infty } 
	\nu_{\theta}^{T} S_{\theta,j} \tilde{R}_{\theta}^{i} \phi
\end{align*}
for $\theta\in \mathbb{R}^{d_{\theta } }$, $v\in {\cal X} \times {\cal Y}$, 
$1\leq j\leq d_{\theta}$, $n\geq 1$
(notice that $e^{T}(v) e = 1$), 
where $h_{j}(\theta)$ is the $j$-th component of $h(\theta)$. 
Thus, 
\begin{align*}
	\partial_{\theta}^{j} f_{n}(\theta,v) - h_{j}(\theta )
	=
	\sum_{i=0}^{n-1} e^{T}(v) \tilde{R}_{\theta}^{n-i} S_{\theta,j} \tilde{R}_{\theta}^{i} \phi
	-
	\sum_{i=n}^{\infty } 
	\nu_{\theta}^{T} S_{\theta,j} \tilde{R}_{\theta}^{i} \phi
\end{align*}
for $\theta\in \mathbb{R}^{d_{\theta } }$, $v\in {\cal X} \times {\cal Y}$, 
$1\leq j\leq d_{\theta}$, $n\geq 1$. 
Then, Lemma \ref{lemma3.1} implies 
\begin{align*}
	|\partial_{\theta }^{j} f_{n}(\theta, v )
	-
	h_{j}(\theta ) |
	\leq &
	\|\phi\| \|e(v)\| \|S_{\theta,j}\| 
	\sum_{i=0}^{n-1}
	\|\tilde{R}_{\theta}^{i}\| \|\tilde{R}_{\theta}^{n-i}\| 
	+
	\|\phi\| \|\nu_{\theta}\| \|S_{\theta,j}\| 
	\sum_{i=n}^{\infty } \|\tilde{R}_{\theta}^{i}\| 
	\\
	\leq &
	\tilde{C}_{Q} C_{1,Q}^{2} \|\phi\| n\varepsilon_{Q}^{n} 
	+
	\frac{\tilde{C}_{Q} C_{1,Q} \|\phi\| \varepsilon_{Q}^{n} }{1-\varepsilon_{Q} }
\end{align*}
for all $\theta\in Q$, $v\in {\cal X} \times {\cal Y}$, 
$1\leq j\leq d_{\theta}$, $n\geq 1$. 
Hence, 
$\nabla_{\theta } f_{n}(\theta, v )$ converges to $h(\theta )$
as $n\rightarrow\infty$ uniformly in 
$(\theta, v )$ on $Q \times ({\cal X} \times {\cal Y} )$. 
Therefore, 
$\nabla f(\theta ) = h(\theta )$ for all $\theta\in\mathbb{R}^{d_{\theta} }$
(notice that $Q$ is any compact set). 
Consequently, 
\begin{align*}
	\eta_{j}(\theta )
	=&
	\sum_{n=0}^{\infty } 
	\lambda^{n} \nu_{\theta}^{T} S_{\theta,j} \tilde{R}_{\theta}^{n} \phi
	-
	h_{j}(\theta) 
	=
	-
	\sum_{n=0}^{\infty } 
	(1-\lambda^{n} ) \nu_{\theta}^{T} S_{\theta,j} \tilde{R}_{\theta}^{n} \phi
\end{align*}
for $\theta\in \mathbb{R}^{d_{\theta } }$, $1\leq j\leq d_{\theta}$, 
where $\eta_{j}(\theta)$ is the $j$-th component of $\eta(\theta)$. 
Combining this with Lemma \ref{lemma3.1}, we get 
\begin{align*}
	|\eta_{j}(\theta ) |
	\leq &
	\|\phi\| \|\nu_{\theta} \| \|S_{\theta,j}\| 
	\sum_{n=0}^{\infty } 
	(1 - \lambda^{n} ) \|\tilde{R}_{\theta}^{n} \|
	\leq 
	\tilde{C}_{Q} C_{1,Q} \|\phi\|  
	\sum_{n=0}^{\infty } 
	(1 - \lambda^{n} ) \varepsilon_{Q}^{n} 
	\leq 
	\frac{\tilde{C}_{Q} C_{1,Q} \|\phi\| (1 - \lambda )}{(1-\varepsilon_{Q} )^{2} } 
\end{align*}
for all $\theta\in Q$, $1\leq j\leq d_{\theta}$. 
Then, we conclude that there exists a real number $C_{3,Q}\in[1,\infty )$
with the properties specified in (ii). 
\end{proof}

\begin{lemma}\label{lemma3.3}  
Suppose that Assumptions \ref{a3.1} and \ref{a3.2} hold. 
Let $Q\subset\mathbb{R}^{d_{\theta } }$ be any compact set. 
Then, the following is true: 
\begin{compactenum}[(i)]
\item
There exist real numbers $\delta_{Q}\in (0,1)$, $C_{4,Q} \in [1,\infty )$
(possibly depending on $\lambda$) 
such that 
\begin{align*}
	&
	\|(\Pi^{n} F)(\theta,z) - \nabla f(\theta ) \|
	\leq 
	C_{4,Q} n \delta_{Q}^{n} (1 + \|w\| ), 
	\\
	&
	\|((\Pi^{n} F)(\theta',z) - \nabla f(\theta') )
	-
	((\Pi^{n} F)(\theta'',z) - \nabla f(\theta'') ) \|
	\leq
	C_{4,Q} n \delta_{Q}^{n} \|\theta' - \theta'' \| (1 + \|w\| ),
\end{align*}
for all $\theta, \theta', \theta'' \in Q$, 
$z=(x,y,w) \in {\cal X} \times {\cal Y} \times \mathbb{R}^{d_{\theta } }$, 
$n\geq 0$. 
\item
There exits a real number $C_{5,Q} \in [1,\infty )$
(possibly depending on $\lambda$) such that 
\begin{align*}
	\|W_{n+1} \| I_{ \{\tau_{Q} > n \} }
	\leq 
	C_{5,Q} (1 + \|W_{0} \| )
\end{align*}
for all $n\geq 0$
($\tau_{Q}$ is specified in Assumption \ref{a2.3}). 
\end{compactenum}
\end{lemma}

\begin{proof}
(i) 
For each $1\leq j\leq d_{\theta}$, 
let $\tilde{C}_{1,Q}\in[1,\infty )$ be an upper bound of $\|S_{\theta,j}\|$ on $Q$
and a Lipschitz constant of $S_{\theta,j}$ on the same set. 
Moreover, let $\tilde{C}_{2,Q}=3\tilde{C}_{1,Q}C_{1,Q}C_{2,Q}N_{v}$, 
$\tilde{C}_{3,Q} = 2\tilde{C}_{2,Q}(1-\varepsilon_{Q} )^{-1}$, 
while $\delta_{Q} = \max\{\lambda,\varepsilon_{Q} \}$.

Owing to Lemma \ref{lemma3.1}, we have
\begin{align}\label{l3.3.101}
	&
	\|\tilde{R}_{\theta}^{k} S_{\theta,j} R_{\theta}^{l} \|
	\leq 
	\|\tilde{R}_{\theta}^{k} \| \|S_{\theta,j} \| \|R_{\theta}^{l} \|
	\leq 
	\tilde{C}_{2,Q} \varepsilon_{Q}^{k}, 
	\\
	&\label{l3.3.103}
	\|\nu_{\theta}^{T} S_{\theta,j} R_{\theta}^{l} \|
	\leq 
	\|\nu_{\theta}^{T} \| \|S_{\theta,j} \| \|R_{\theta}^{l} \|
	\leq 
	\tilde{C}_{2,Q} 
\end{align}
for all $\theta\in Q$, $1\leq j\leq d_{\theta}$, $k,l\geq 1$. 
Due to the same lemma, we also have 
\begin{align}\label{l3.3.1}
	\|\tilde{R}_{\theta'}^{k} S_{\theta',j} R_{\theta'}^{l} 
	- 
	\tilde{R}_{\theta''}^{k} S_{\theta'',j} R_{\theta''}^{l} \|
	\leq &
	\|\tilde{R}_{\theta'}^{k} - \tilde{R}_{\theta''}^{k} \|
	\|S_{\theta',j} \| \|R_{\theta'}^{l} \| 
	+
	\|\tilde{R}_{\theta''}^{k} \|
	\|S_{\theta',j} - S_{\theta'',j} \| \|R_{\theta'}^{l} \| 
	\nonumber\\
	&+
	\|\tilde{R}_{\theta''}^{k} \|
	\|S_{\theta'',j} \|
	\|R_{\theta'}^{l} - R_{\theta''}^{l} \|
	\nonumber\\
	\leq &
	\tilde{C}_{2,Q} \varepsilon_{Q}^{k} \|\theta' - \theta'' \|
\end{align}
for all $\theta',\theta'' \in Q$, $1\leq j\leq d_{\theta}$, 
$k,l\geq 1$. 
In addition to this, Lemma \ref{lemma3.1} implies 
\begin{align}\label{l3.3.3}
	\|\nu_{\theta'}^{T} S_{\theta',j} R_{\theta'}^{l} - \nu_{\theta''}^{T} S_{\theta'',j} R_{\theta''}^{l} \|
	\leq &
	\|\nu_{\theta'}^{T} - \nu_{\theta''}^{T} \| 
	\|S_{\theta',j} \| \|R_{\theta'}^{l} \| 
	+
	\|\nu_{\theta''}^{T} \| 
	\|S_{\theta',j} - S_{\theta'',j} \| \|R_{\theta'}^{l} \| 
	\nonumber\\
	&+
	\|\nu_{\theta''}^{T} \| 
	\|S_{\theta'',j} \| \|R_{\theta'}^{l} - R_{\theta''}^{l} \| 
	\nonumber\\
	\leq &
	\tilde{C}_{2,Q} \|\theta'-\theta''\|
\end{align}
for each $\theta',\theta'' \in Q$, $1\leq j\leq d_{\theta}$, 
$l\geq 1$. 
Moreover, 
it is straightforward to show
\begin{align*}
	(\Pi^{n} F)(\theta, z ) 
	= &
	-\eta(\theta )
	+
	E_{\theta }\left(
	\phi(V_{n}^{\theta } ) W_{n}^{\theta }
	|V_{0}^{\theta } = v, W_{0}^{\theta } = w
	\right)
	\\
	= &
	-\eta(\theta )
	+
	E_{\theta }\left(\left.
	\phi(V_{n}^{\theta } ) 
	\left(
	\lambda^{n} w 
	+
	\sum_{i=0}^{n-1} \lambda^{i} s_{\theta }(V_{n-i}^{\theta } )
	\right)
	\right|V_{0}^{\theta } = v
	\right)
	\\
	= &
	-\eta(\theta )
	+
	\sum_{i=0}^{n-1} \;
	\sum_{v',v'' \in {\cal X} \times {\cal Y} }
	\lambda^{i}
	\phi(v'') r_{\theta }^{i}(v''|v') s_{\theta }(v') r_{\theta }^{n-i}(v'|v)  
	+
	\lambda^{n} w 
	\sum_{v' \in {\cal X} \times {\cal Y} }
	\phi(v') r_{\theta }^{n}(v'|v) 
\end{align*}
for $\theta\in \mathbb{R}^{d_{\theta } }$, 
$z = (v,w) \in ({\cal X} \times {\cal Y} ) \times 
\mathbb{R}^{d_{\theta } }$, 
$n\geq 1$. 
Therefore, 
\begin{align*}
	(\Pi^{n} F_{j} )(\theta, z ) 
	=
	-\eta_{j}(\theta )
	+
	\sum_{i=0}^{n-1} \lambda^{i} 
	e^{T}(v) R_{\theta}^{n-i} S_{\theta,j} R_{\theta}^{i} \phi 
	+
	\lambda^{n} e_{j}^{T} w \: e^{T}(v) R_{\theta}^{n} \phi
\end{align*}
for $\theta\in \mathbb{R}^{d_{\theta } }$, 
$z = (v,w) \in ({\cal X} \times {\cal Y} ) \times 
\mathbb{R}^{d_{\theta } }$, 
$1\leq j\leq d_{\theta}$, $n\geq 1$. 
Here, $F_{j}(\theta,z)$, $\eta_{j}(\theta)$ are the $j$-th components of 
$F(\theta,z)$, $\eta(\theta)$, 
while $e_{j}$ is the $j$-th standard unit vector in $\mathbb{R}^{d_{\theta} }$. 
Moreover, we have 
\begin{align*}
	\partial^{j} f(\theta )
	=
	-\eta_{j}(\theta) 
	+ 
	\sum_{n=0}^{\infty } 
	\lambda^{n} \nu_{\theta}^{T} S_{\theta,j} \tilde{R}_{\theta}^{n} \phi
\end{align*}
for $\theta\in\mathbb{R}^{d_{\theta} }$, $1\leq j\leq d_{\theta}$, 
where $\partial^{j} f(\theta )$ is the $j$-th component of $\nabla f(\theta)$. 
Since $e^{T}(v)e=1$, $\tilde{R}_{\theta}^{n}=R_{\theta}^{n}-e\nu_{\theta}^{T}$
and 
\begin{align*}
	\nu_{\theta}^{T} S_{\theta,j} e
	=
	\sum_{v\in{\cal X}\times{\cal Y} } \nu_{\theta}(v) s_{\theta,j}(v)
	=
	\sum_{x\in{\cal X} }
	\left(\sum_{y\in{\cal Y} } \partial_{\theta}^{j} q_{\theta}(y|x) \right) 
	\mu_{\theta}(x)
	=
	0
\end{align*}
for $\theta\in\mathbb{R}^{d_{\theta} }$, $v\in{\cal X}\times{\cal Y}$, 
$1\leq j\leq d_{\theta}$, $n\geq 0$,\footnote{
Notice that 
$\sum_{y\in{\cal Y} } \partial_{\theta}^{j} q_{\theta}(y|x) =
\partial_{\theta}^{j} \left(\sum_{y\in{\cal Y} } q_{\theta}(y|x) \right) = 0$. 
Notice also that 
$\nu_{\theta}(v)=q_{\theta}(y|x)\mu_{\theta}(x)$ for $v=(x,y)\in{\cal X}\times{\cal Y}$, 
where $\mu_{\theta}(x)$ is the invariant probability of $\{X_{n}^{\theta} \}_{n\geq 0}$
(see the proof of Part (i) of Lemma \ref{lemma3.1}). 
} 
we get
\begin{align*}
	\partial^{j} f(\theta)
	=&
	-
	\eta_{j}(\theta)
	+
	\sum_{i=0}^{n-1} \lambda^{i}\nu_{\theta}^{T} S_{\theta,j} R_{\theta}^{i} \phi
	+
	\sum_{i=0}^{n-1} \lambda^{i}\nu_{\theta}^{T} S_{\theta,j} e \nu_{\theta}^{T} \phi
	+
	\sum_{i=n}^{\infty} \lambda^{i} \nu_{\theta}^{T} S_{\theta,j} \tilde{R}_{\theta}^{i} \phi
	\\
	=&
	-
	\eta_{j}(\theta)
	+
	\sum_{i=0}^{n-1} \lambda^{i} e^{T}(v)e \nu_{\theta}^{T} S_{\theta,j} R_{\theta}^{i} \phi
	+
	\sum_{i=n}^{\infty} \lambda^{i} \nu_{\theta}^{T} S_{\theta,j} \tilde{R}_{\theta}^{i} \phi
\end{align*}
for the same $\theta$, $v$, $j$, $n$. 
Consequently, 
\begin{align*}
	(\Pi^{n} F_{j} )(\theta, z ) - \partial^{j} f(\theta ) 
	=
	\sum_{i=0}^{n-1} 
	\lambda^{i} e^{T}(v) \tilde{R}^{n-i} S_{\theta,j} R_{\theta}^{i} \phi
	-
	\sum_{i=n}^{\infty } 
	\lambda^{i} \nu_{\theta}^{T} S_{\theta,j} \tilde{R}_{\theta}^{i} \phi
	+
	\lambda^{n} e_{j}^{T}w \: e^{T}(v) R_{\theta}^{n} \phi
\end{align*}
for $\theta\in\mathbb{R}^{d_{\theta} }$, 
$z=(v,w)\in({\cal X}\times{\cal Y} )\times\mathbb{R}^{d_{\theta} }$, 
$1\leq j\leq d_{\theta}$, $n\geq 1$. 
Then, (\ref{l3.3.101}), (\ref{l3.3.103}) imply 
\begin{align}\label{l3.3.201}
	|(\Pi^{n} F_{j} )(\theta, z ) - \partial^{j} f(\theta ) |
	\leq &
	\|\phi\| \|e(v)\| 
	\sum_{i=0}^{n-1} 
	\lambda^{i}
	\|\tilde{R}_{\theta}^{n-i} S_{\theta,j} R_{\theta}^{i} \|
	+
	\|\phi\| 
	\sum_{i=n}^{\infty } 
	\lambda^{i}
	\|\nu_{\theta}^{T} S_{\theta,j} \tilde{R}_{\theta}^{i} \| 
	+
	\lambda^{n} \|\phi\| \|e(v)\| \|R_{\theta}^{n}\| \|w \|
	\nonumber\\
	\leq &
	\tilde{C}_{2,Q}
	\left(
	\sum_{i=1}^{n} \lambda^{i} \varepsilon_{Q}^{n-i} 
	+
	\sum_{i=n}^{\infty } \lambda^{i} \varepsilon_{Q}^{i} 
	+ 
	\lambda^{n} \|w\|
	\right) 
	\nonumber\\
	\leq & 
	C_{3,Q} n\delta_{Q}^{n} (1 + \|w \| )
\end{align}
for all $\theta\in Q$, $z = (v,w) \in ({\cal X} \times {\cal Y} ) \times 
\mathbb{R}^{d_{\theta } }$, 
$1\leq j\leq d_{\theta}$, $n\geq 1$. 
Similarly, (\ref{l3.3.1}), (\ref{l3.3.3}) yield  
\begin{align}\label{l3.3.203}
	&
	|((\Pi^{n} F_{j} )(\theta', z ) - \partial^{j} f(\theta' ) ) 
	-
	((\Pi^{n} F_{j} )(\theta'', z ) - \partial^{j} f(\theta'' ) ) |
	\nonumber\\
	&
	\begin{aligned}[t]
	\leq &
	\|\phi\| \|e(v)\| \sum_{i=0}^{n-1} \;
	\lambda^{i}
	\|\tilde{R}_{\theta'}^{n-i} S_{\theta',j} R_{\theta'}^{i} 
	-
	\tilde{R}_{\theta''}^{n-i} S_{\theta'',j} R_{\theta''}^{i} \|
	+
	\|\phi\| \sum_{i=n}^{\infty } \;
	\lambda^{i}
	\|\nu_{\theta'}^{T} S_{\theta',j} \tilde{R}_{\theta'}^{i} 
	-
	\nu_{\theta''}^{T} S_{\theta'',j} \tilde{R}_{\theta''}^{i} \|
	\nonumber\\
	&
	+
	\lambda^{n} \|\phi\| \|e(v) \| \|w \| 
	\|R_{\theta'}^{n} - R_{\theta''}^{n} \|
	\end{aligned}
	\nonumber\\
	&
	\leq
	\tilde{C}_{2,Q} \|\theta'-\theta'' \| 
	\left(
	\sum_{i=1}^{n} \lambda^{i} \varepsilon_{Q}^{n-i} 
	+
	\sum_{i=n}^{\infty } \lambda^{i} \varepsilon_{Q}^{i} 
	+ 
	\lambda^{n} \|w\|
	\right) 
	\nonumber\\
	&
	\leq 
	C_{3,Q} n \delta_{Q}^{n} \|\theta' - \theta'' \| (1 + \|w \| )
\end{align}
for all $\theta',\theta''\in Q$, 
$z = (v,w) \in ({\cal X} \times {\cal Y} ) \times \mathbb{R}^{d_{\theta } }$, 
$1\leq j\leq d_{\theta}$, $n\geq 1$. 
Using (\ref{l3.3.201}), (\ref{l3.3.203}), we conclude that there exist 
real numbers $\delta_{Q}$, $C_{4,Q}$ with properties specified in (i). 

(ii) 
Let 
$C_{5,Q} = \tilde{C}_{1,Q} (1-\lambda )^{-1}$ 
($\tilde{C}_{1,Q}$ is specified in the proof of (i)). 
Then, due to Assumption \ref{a3.2}, we have 
\begin{align*}
	\|W_{n+1} \|
	I_{ \{\tau_{Q} > n \} }
	= &
	\left\|
	\lambda^{n+1} W_{0} 
	+
	\sum_{i=0}^{n} 
	\lambda^{n-i} s_{\theta_{i} }(X_{i+1}, Y_{i+1} )
	\right\|
	I_{ \{\tau_{Q} > n \} }
	\\
	\leq &
	\lambda^{n+1} \|W_{0} \| 
	+
	\tilde{C}_{1,Q} \sum_{i=0}^{n} \lambda^{n-i} 
	\\
	\leq &
	C_{5,Q} (1 + \|W_{0} \| )
\end{align*}
for $n\geq 0$. 
\end{proof}

\begin{vproof}{Theorem \ref{theorem3.1}}
{For $\theta\in \mathbb{R}^{d_{\theta } }$, 
$z=(v,w) \in ({\cal X} \times {\cal Y} ) \times \mathbb{R}^{d_{\theta } }$,   
let 
\begin{align*}
	\tilde{F}(\theta,z)
	=
	\sum_{n=0}^{\infty } ((\Pi^{n} F)(\theta,z) - \nabla f(\theta ) ), 
	\;\;\;\;\;  
	\varphi(z)
	=
	1 + \|w\|.
\end{align*}
Then, using Lemma \ref{lemma3.3}, 
we conclude that for each $\theta\in \mathbb{R}^{d_{\theta } }$, 
$z\in {\cal X} \times {\cal Y} \times \mathbb{R}^{d_{\theta } }$, 
$\tilde{F}(\theta,z)$ is well-defined and satisfies 
$(\Pi\tilde{F})(\theta,z) = \sum_{n=1}^{\infty } ((\Pi^{n} F)(\theta,z) - \nabla f(\theta ) )$. 
Thus, Assumption \ref{a2.2} holds. 
Relying on Lemma \ref{lemma3.3}, we also deduce that 
for any compact set $Q\subset \mathbb{R}^{d_{\theta } }$, 
there exists a real number $\tilde{C}_{Q} \in [1,\infty )$
(possibly depending on $\lambda$) such that 
\begin{align*}
	&
	\max\{\|F(\theta,z) \|, \|\tilde{F}(\theta,z) \|, \|(\Pi \tilde{F})(\theta,z) \| \} 
	\leq 
	\tilde{C}_{Q} \varphi(z), 
	\\
	&
	\|(\Pi\tilde{F} )(\theta',z ) - (\Pi\tilde{F} )(\theta'',z ) \|
	\leq 
	\tilde{C}_{Q} \varphi(z) \|\theta' - \theta'' \|, 
	\\
	&
	E\left(\varphi^{2}(Z_{n+1} )I_{ \{\tau_{Q}>n\} }|\theta_{0}=\theta, Z_{0}=z \right)
	\leq 
	\tilde{C}_{Q} \varphi^{2}(z)
\end{align*}
for all $\theta,\theta',\theta''\in Q$, 
$z\in {\cal X} \times {\cal Y} \times \mathbb{R}^{d_{\theta } }$. 
Hence, Assumptions \ref{a2.3} is satisfied, too. 
Moreover, Lemma \ref{lemma3.2} yields 
\begin{align*}
	\eta
	=
	\limsup_{n\rightarrow \infty } 
	\|\eta_{n} \| 
	\leq 
	C_{3,Q} (1-\lambda )
\end{align*}
on $\Lambda_{Q}$
(notice that $C_{3,Q}$ does not depend on $\lambda$).
Then, the theorem's assertion directly follows from 
Theorem \ref{theorem2.1} and Parts (i), (iii), (iv) of Lemma \ref{lemma3.2}. 
}
\end{vproof}

\section{Proof of Theorem \ref{theorem5.1} } \label{section5*}

In this section, we rely on the following notation. 
$\{Z_{n}^{\theta } \}_{n\geq 0}$ is the ${\cal X}^{2N}$-valued Markov chain
defined by $Z_{n+1}^{\theta } = (X_{n}^{\theta }, X_{n+1}^{\theta } )$ for 
$\theta\in\Theta$, $n\geq 0$, 
while $\Pi_{\theta }(\cdot,\cdot )$ and $\pi_{\theta }(\cdot )$ are the transition kernel 
and the invariant probability of $\{Z_{n}^{\theta } \}_{n\geq 0}$. 
$G(\cdot,\cdot )$, $g(\cdot )$ are the functions defined by 
\begin{align*}
	G(\theta,z)
	=
	-\frac{1}{N} \sum_{i=1}^{N} s_{\theta}(x_{i},x_{i}'), 
	\;\;\;\;\; 
	g(\theta )
	=
	\int_{{\cal X}^{2N} } G(\theta,z') \pi_{\theta }(dz')
\end{align*}
for $\theta\in\Theta$, $z=(x_{1},\dots,x_{N},x'_{1},\dots,x'_{N} )\in{\cal X}^{2N}$. 
$\eta(\cdot )$, $F(\cdot,\cdot )$ are the functions defined as 
\begin{align*}
	\eta(\theta )=g(\theta ) - \nabla f(\theta ), 
	\;\;\;\;\; 
	F(\theta,z)=G(\theta,z)-\eta(\theta)
\end{align*}
for $\theta\in\Theta$, $z\in{\cal X}^{2N}$. 
On the other side, $\{Z_{n} \}_{n\geq 0}$, $\{\eta_{n} \}_{n\geq 0}$ are the stochastic processes defined by 
\begin{align*}
	Z_{n+1} = (X_{n}, X_{n+1} ), 
	\;\;\;\;\; 
	\eta_{n} = \eta(\theta_{n} )
\end{align*}
for $n\geq 0$. 
Then, it is straightforward to show that the 
algorithm (\ref{5.1}) is of the same form 
as the recursion studied in Section \ref{section2}
(i.e., $\{\theta_{n} \}_{n\geq 0}$, $\{\eta_{n} \}_{n\geq 0}$, 
$F(\cdot,\cdot )$, $\Pi_{\theta }(\cdot,\cdot )$
defined in Section \ref{section3} and here 
admit (\ref{2.1}), (\ref{2.3})). 

In this section, we also use the following notation. 
For $\theta\in\Theta$, $n\geq 0$, $\Pi_{\theta}^{n}(\cdot,\cdot )$ is the $n$-th transition distribution of 
$\{Z_{n}^{\theta } \}_{n\geq 0}$, 
while 
\begin{align*}
	\tilde{\Pi}_{\theta}^{n}(z,B)
	=
	\Pi_{\theta}^{n}(z,B)
	-
	\pi_{\theta}(B), 
	\;\;\;\;\; 
	(\tilde{\Pi}^{n} F)(\theta,z)
	=
	\int_{{\cal X}^{2N} } F(\theta,z') \tilde{\Pi}_{\theta }^{n}(z,dz')
\end{align*}
for a Borel-set $B\subseteq{\cal X}^{2N}$ and $\theta\in\Theta$, $z\in{\cal X}^{2N}$, $n\geq 0$.

\begin{lemma}\label{lemma5.1}  
Suppose that Assumptions \ref{a5.1} -- \ref{a5.3} hold. 
Let $Q\subset\Theta$ be any compact set. 
Then, the following is true: 
\begin{compactenum}[(i)]
\item
$\{Z_{n}^{\theta } \}_{n\geq 0}$ is geometrically ergodic for each $\theta\in\Theta$. 
\item
There exist real numbers $\delta_{Q}\in (0,1)$, $C_{1,Q} \in [1,\infty )$
(possibly depending on $N$) 
such that 
\begin{align}
	&
	|\tilde{\Pi}_{\theta }^{n}(z,B) |
	\leq 
	C_{1,Q} \delta_{Q}^{n}, 
	\nonumber\\
	&
	|\tilde{\Pi}_{\theta'}^{n}(z,B) - \tilde{\Pi}_{\theta''}^{n}(z,B) |
	\leq
	C_{1,Q} n \delta_{Q}^{n} \|\theta' - \theta'' \|, 
	\nonumber\\
	&\label{l5.1.5*}
	|\pi_{\theta'}(B) - \pi_{\theta''}(B) | 
	\leq 
	C_{1,Q} \|\theta'-\theta''\|
\end{align}
for any Borel-set $B\subseteq{\cal X}^{2N}$ and all $\theta, \theta', \theta'' \in Q$, 
$z\in {\cal X}^{2N}$, 
$n\geq 0$. 
\end{compactenum}
\end{lemma}

\begin{proof}
For $\theta\in\Theta$, let $P_{\theta }(\cdot,\cdot )$ be the transition kernel of 
$\{X_{n}^{\theta } \}_{n\geq 0}$, 
while $\lambda(\cdot )$ is the measure on ${\cal X}^{N}$ defined by 
\begin{align*}
	\lambda(B)
	=
	\sum_{1\leq i_{1},\dots,i_{N}\leq N}
	\int_{\cal X}\cdots\int_{\cal X}
	I_{B}(x_{i_{1} }, \dots, x_{i_{N} } ) 
	dx_{1} \cdots dx_{N}
\end{align*}
for a Borel-set $B\subseteq{\cal X}^{N}$. 
On the other side, let 
\begin{align*}
	v_{\theta }(x'_{1:N}|x_{1:N} )
	=
	\prod_{j=1}^{N} p_{\theta }(x'_{j}|x_{j} ), 
	\;\;\;\;\; 
	w_{\theta, i }(x_{1:N}, x'_{1:N} )
	=
	\frac{w_{\theta }(x_{i}, x'_{i} ) }{\sum_{j=1}^{N} w_{\theta }(x_{j},x'_{j} ) }
\end{align*}
for $\theta\in\Theta$, $x_{1:N}=(x_{1},\dots,x_{N} )\in{\cal X}^{N}$, 
$x'_{1:N}=(x'_{1},\dots,x'_{N} )\in{\cal X}^{N}$, 
$1\leq i\leq N$. 
Moreover, let 
\begin{align*}
	u_{\theta }(B| x_{1:N}, x'_{1:N} )
	=
	\sum_{1\leq i_{1},\dots,i_{N}\leq N} 
	I_{B}(x_{i_{1} }, \cdots, x_{i_{N} } )
	\frac{\prod_{j=1}^{N} w_{\theta}(x_{i_{j} },x'_{i_{j} } ) }
	{\left(\sum_{j=1}^{N} w_{\theta }(x_{j},x'_{j} ) \right)^{N} }
\end{align*}
for a Borel-set $B\subseteq{\cal X}^{N}$ 
and $\theta\in\Theta$, $x_{1:N}=(x_{1},\dots,x_{N} )\in{\cal X}^{N}$, 
$x'_{1:N}=(x'_{1},\dots,x'_{N} )\in{\cal X}^{N}$. 
Then, it is straightforward to show
\begin{align*}
	P_{\theta }(x_{1:N},B)
	=
	\int_{{\cal X}^{2N} } u_{\theta }(B|x_{1:N},x'_{1:N} ) v_{\theta}(x'_{1:N}|x_{1:N} ) dx'_{1:N}
\end{align*}
for any Borel-set $B\subseteq{\cal X}^{N}$ 
and all $\theta\in\Theta$, $x_{1:N}=(x_{1},\dots,x_{N} )\in{\cal X}^{N}$. 
It is also easy to demonstrate 
\begin{align*}
	u_{\theta }(B| x_{1:N}, x'_{1:N} )
	=
	\sum_{1\leq i_{1},\dots,i_{N}\leq N} 
	I_{B}(x_{i_{1} }, \cdots, x_{i_{N} } )
	\prod_{j=1}^{N} w_{\theta,i_{j} }(x_{1:N},x'_{1:N} ) 
\end{align*}
for any Borel-set $B\subseteq{\cal X}^{N}$ 
and all $\theta\in\Theta$, $x_{1:N}=(x_{1},\dots,x_{N} )\in{\cal X}^{N}$, 
$x'_{1:N}=(x'_{1},\dots,x'_{N} )\in{\cal X}^{N}$. 

Due to Assumptions \ref{a5.1} -- \ref{a5.3}, 
there exists a real number $\varepsilon_{1,Q}\in(0,1)$ such that 
$p_{\theta}(x'|x)\geq\varepsilon_{1,Q}$, 
$q(x)\geq\varepsilon_{1,Q}$ for all $\theta\in Q$, $x,x'\in{\cal X}$. 
Let $\tilde{C}\in[1,\infty )$ be an upper bound of $q(\cdot )$ on ${\cal X}$, 
while $\tilde{C}_{1,Q}\in[\tilde{C},\infty )$ is an upper bound in $(\theta,x,x')$
for $p_{\theta }(x'|x)$ on $Q\times{\cal X}\times{\cal X}$. 
Moreover, let $\tilde{C}_{2,Q}\in[1,\infty )$ be a Lipschitz constant 
in $(\theta,x,x')$ for $p_{\theta }(x'|x)$ on $Q\times{\cal X}\times{\cal X}$,  
while $\varepsilon_{2,Q}=\varepsilon_{1,Q}^{2}\tilde{C}_{1,Q}^{-2}N^{-1}$, 
$\tilde{C}_{3,Q}=2\varepsilon_{1,Q}^{-3}\tilde{C}_{1,Q}^{2}\tilde{C}_{2,Q}$. 
Then, we have 
$\varepsilon_{2,Q} \leq w_{\theta,i}(x_{1:N},x'_{1:N} ) \leq 1$
and 
\begin{align*}
	|w_{\theta',i}(x_{1:N},x'_{1:N} ) - w_{\theta'',i}(x_{1:N},x'_{1:N} ) |
	\leq &
	\frac{w_{\theta',i}(x_{1:N},x'_{1:N} ) 
	\sum_{j=1}^{N} |w_{\theta''}(x_{j},x'_{j} ) - w_{\theta'}(x_{j},x'_{j} ) | }
	{\sum_{j=1}^{N} w_{\theta''}(x_{j},x'_{j} ) }
	\\
	&
	+
	\frac{|w_{\theta'}(x_{i},x'_{i} ) - w_{\theta''}(x_{i},x'_{i} ) | }
	{\sum_{j=1}^{N} w_{\theta''}(x_{j},x'_{j} ) }
	\\
	\leq &
	\tilde{C}_{3,Q} \|\theta' - \theta'' \|
\end{align*}
for all $\theta,\theta',\theta''\in Q$, $x_{1:N}=(x_{1},\dots,x_{N} )\in {\cal X}^{N}$, 
$x'_{1:N}=(x'_{1},\dots,x'_{N} )\in {\cal X}^{N}$, $1\leq i\leq N$.\footnote{
Notice that 
$\varepsilon_{Q}\tilde{C}_{1,Q}^{-1}\leq w_{\theta}(x,x')\leq\varepsilon_{Q}^{-1}\tilde{C}_{1,Q}$
and 
$|w_{\theta'}(x,x')-w_{\theta''}(x,x')|\leq 
\varepsilon_{Q}^{-2}\tilde{C}_{1,Q}\tilde{C}_{2,Q}\|\theta'-\theta''\|$
when $\theta,\theta',\theta''\in Q$. 
} 
Consequently, 
\begin{align*}
	&
	\begin{aligned}[b]
	\left|
	\prod_{j=1}^{N} w_{\theta',i_{j} }(x_{1:N},x'_{1:N} ) 
	- 
	\prod_{j=1}^{N} w_{\theta'',i_{j} }(x_{1:N},x'_{1:N} ) 
	\right|
	\leq &
	\begin{aligned}[t]
	\sum_{j=1}^{N}
	&
	\left(\prod_{k=1}^{j-1} w_{\theta',i_{k} }(x_{1:N},x'_{1:N} ) \right) 
	\left(\prod_{k=j+1}^{N} w_{\theta',i_{k} }(x_{1:N},x'_{1:N} ) \right)
	\\
	&
	\cdot 
	\left|
	w_{\theta',i_{j} }(x_{1:N},x'_{1:N} ) 
	-
	w_{\theta'',i_{j} }(x_{1:N},x'_{1:N} )
	\right|
	\end{aligned}
	\\
	\leq &
	\tilde{C}_{3,Q}N \|\theta'-\theta''\| 
	\end{aligned}
	\end{align*}
for each $\theta',\theta''\in Q$, $x_{1:N}=(x_{1},\dots,x_{N} )\in {\cal X}^{N}$, 
$x'_{1:N}=(x'_{1},\dots,x'_{N} )\in {\cal X}^{N}$, $1\leq i_{1},\dots,i_{N}\leq N$. 
Similarly, we have 
	\begin{align*}
	&
	\begin{aligned}[b]
	\left|
	v_{\theta'}(x'_{1:N}|x_{1:N} ) 
	- 
	v_{\theta''}(x'_{1:N}|x_{1:N} ) 
	\right|
	\leq &
	\sum_{j=1}^{N}
	\left(\prod_{k=1}^{j-1} p_{\theta' }(x'_{k}|x_{k} ) \right) 
	\left(\prod_{k=j+1}^{N} p_{\theta' }(x'_{k}|x_{k} ) \right)
	\left|
	p_{\theta' }(x'_{j}|x_{j} ) 
	-
	p_{\theta'' }(x'_{j}|x_{j} )
	\right|
	\\
	\leq &
	\tilde{C}_{3,Q}^{N} \|\theta'-\theta''\| 
	\end{aligned}
\end{align*}
for all $\theta',\theta''\in Q$, $x_{1:N}=(x_{1},\dots,x_{N} )\in {\cal X}^{N}$, 
$x'_{1:N}=(x'_{1},\dots,x'_{N} )\in {\cal X}^{N}$. 
Hence, 
\begin{align*}
	\left|
	\left(\prod_{j=1}^{N} w_{\theta',i_{j} }(x_{1:N},x'_{1:N} ) \right)
	v_{\theta'}(x'_{1:N}|x_{1:N} )
	- 
	\left(\prod_{j=1}^{N} w_{\theta'',i_{j} }(x_{1:N},x'_{1:N} ) \right)
	v_{\theta''}(x'_{1:N}|x_{1:N} )
	\right|
	\leq 
	2\tilde{C}_{3,Q}^{2N} \|\theta'-\theta''\|
\end{align*}
for each $\theta',\theta''\in Q$, $x_{1:N}=(x_{1},\dots,x_{N} )\in {\cal X}^{N}$, 
$x'_{1:N}=(x'_{1},\dots,x'_{N} )\in {\cal X}^{N}$, $1\leq i_{1},\dots,i_{N}\leq N$
(notice that $v_{\theta}(x'_{1:N}|x_{1:N} )\leq \tilde{C}_{1,Q}^{N}\leq\tilde{C}_{3,Q}^{N}N^{-1}$
when $\theta\in Q$). 

Let $\varepsilon_{3,Q}=\varepsilon_{2,Q}^{4N}\min^{2}\{1,\lambda({\cal X}^{N} ) \}$, 
$\delta_{Q}=(1-\varepsilon_{3,Q} )^{1/2}$. 
Then, we have 
\begin{align*}
	P_{\theta}(x_{1:N}, B )
	\geq 
	\varepsilon_{2,Q}^{2N} 
	\sum_{1\leq i_{1},\dots,i_{N}\leq N} 
	\int_{\cal X}\cdots\int_{\cal X} I_{B}(x'_{i_{1} },\dots,x'_{i_{N} } )
	dx'_{1}\cdots dx'_{N}
	=
	\frac{\varepsilon_{3,Q}^{1/2} \lambda(B) }{\lambda({\cal X}^{N} ) }
\end{align*}
for any Borel-set $B\subseteq{\cal X}^{N}$ 
and all $\theta\in Q$, $x_{1:N}=(x_{1},\dots,x_{N} )\in {\cal X}^{N}$.\footnote{
Notice that $v_{\theta}(x'_{1:N}|x_{1:N})\geq \varepsilon_{1,Q}^{N}\geq\varepsilon_{2,Q}^{N}$
and 
$u_{\theta}(B|x_{1:N},x'_{1:N} ) \geq \varepsilon_{Q}^{N} 
\sum_{1\leq i_{1},\dots,i_{N}\leq N} I_{B}(x'_{i_{1} },\dots,x'_{i_{N} } )$
when $\theta\in Q$.
} 
Consequently, 
\begin{align*}
	\Pi_{\theta }^{2}(z,B)
	=&
	\int_{{\cal X}^{2N} }\int_{{\cal X}^{2N} } I_{B}(x''_{1:N}, x'''_{1:N} )
	P_{\theta}(x''_{1:N}, dx'''_{1:N} ) P_{\theta}(x'_{1:N}, dx''_{1:N} )
	\\
	\geq &
	\frac{\varepsilon_{3,Q} }{\lambda^{2}({\cal X}^{N} ) }
	\int_{{\cal X}^{2N} }\int_{{\cal X}^{2N} } I_{B}(x''_{1:N}, x'''_{1:N} )
	\lambda(dx'''_{1:N} ) \lambda(dx''_{1:N} )
\end{align*}
for any Borel-set $B\subseteq{\cal X}^{2N}$ 
and all $\theta\in Q$, $z=(x_{1},\dots,x_{N}, x'_{1},\dots, x'_{N} )\in {\cal X}^{2N}$. 
Combining this with well-known results from the Markov chain theory 
(see e.g., \cite[Theorem 16.02]{meyn&tweedie}), 
we conclude that 
$\{Z_{n}^{\theta } \}_{n\geq 0}$ is geometrically ergodic for each $\theta\in\Theta$
(notice that $Q$ is any compact set). 
We also deduce 
\begin{align*}  
	|\tilde{\Pi}_{\theta }^{n}(z,B) |\leq (1-\varepsilon_{3,Q} )^{n/2} = \delta_{Q}^{n}
\end{align*}
for any Borel-set $B\subseteq{\cal X}^{2N}$ and all $\theta\in Q$, $z\in{\cal X}^{2N}$. 

Let $\tilde{C}_{4,Q}=2\tilde{C}_{3,Q}^{2N}\max\{1,\lambda({\cal X}^{N} ) \}$, 
$C_{1,Q}=\delta_{Q}^{-1}\tilde{C}_{4,Q}$. 
Then, we have 
\begin{align*}
	|P_{\theta'}(x_{1:N},B) - P_{\theta''}(x_{1:N},B) |
	\leq &
	2\tilde{C}_{3,Q}^{2N} \|\theta' - \theta'' \|
	\sum_{1\leq i_{1},\dots,i_{N}\leq N} 
	\int_{\cal X}\cdots\int_{\cal X} I_{B}(x'_{i_{1} },\dots,x'_{i_{N} } )
	dx'_{1}\cdots dx'_{N}
	\\
	\leq &
	\tilde{C}_{4,Q} \|\theta'-\theta'' \|
\end{align*}
for any Borel-set $B\subseteq{\cal X}^{N}$ 
and all $\theta',\theta''\in Q$, $x_{1:N}=(x_{1},\dots,x_{N} )\in{\cal X}^{N}$. 
Therefore,
\begin{align*}
	|\Pi_{\theta'}(z,B) - \Pi_{\theta''}(z,B)|
	=
	\left|
	\int_{{\cal X}^{2N} } I_{B}(x'_{1:N},x''_{1:N} ) 
	\: (P_{\theta'} - P_{\theta''} )(x'_{1:N},dx''_{1:N} ) 
	\right|
	\leq 
	\tilde{C}_{4,Q} \|\theta'-\theta'' \|
\end{align*}
for any Borel-set $B\subseteq{\cal X}^{2N}$ 
and all $\theta',\theta''\in Q$, $z=(x_{1},\dots,x_{N},x'_{1},\dots,x'_{N} )\in{\cal X}^{2N}$. 
Consequently, we conclude 
\begin{align*}
	|\Pi_{\theta'}^{n+1}(z,B) - \Pi_{\theta''}^{n+1}(z,B) |
	=&
	\left|
	\sum_{j=0}^{n} 
	\int_{{\cal X}^{2N} }\int_{{\cal X}^{2N} } 
	\tilde{\Pi}_{\theta'}^{j}(z'',B)\:
	(\Pi_{\theta'} - \Pi_{\theta''} )(z',dz'')\: 
	\Pi_{\theta''}^{n-j}(z,dz') 
	\right|
	\\
	\leq &
	\sum_{j=0}^{n} 
	\int_{{\cal X}^{2N} }\int_{{\cal X}^{2N} } 
	|\tilde{\Pi}_{\theta'}^{j}(z'',B)| \:
	|\Pi_{\theta'} - \Pi_{\theta''} |(z',dz'')\: 
	\Pi_{\theta''}^{n-j}(z,dz') 
	\\
	\leq &
	\tilde{C}_{4,Q} (1-\delta_{Q} )^{-1} \|\theta'-\theta''\| 
	\\
	\leq &
	C_{1,Q} \|\theta'-\theta''\|
\end{align*}
for any Borel-set $B\subseteq{\cal X}^{2N}$ 
and all $\theta',\theta''\in Q$, $z\in{\cal X}^{2N}$, $n\geq 0$. 
Similarly, we deduce
\begin{align*}
	|\tilde{\Pi}_{\theta'}^{n+1}(z,B) - \tilde{\Pi}_{\theta''}^{n+1}(z,B) |
	=&
	\left|
	\sum_{j=0}^{n} 
	\int_{{\cal X}^{2N} }\int_{{\cal X}^{2N} } 
	\tilde{\Pi}_{\theta'}^{j}(z'',B)\:
	(\Pi_{\theta'} - \Pi_{\theta''} )(z',dz'')\: 
	\tilde{\Pi}_{\theta''}^{n-j}(z,dz') 
	\right|
	\\
	\leq &
	\sum_{j=0}^{n} 
	\int_{{\cal X}^{2N} }\int_{{\cal X}^{2N} } 
	|\tilde{\Pi}_{\theta'}^{j}(z'',B)| \:
	|\Pi_{\theta'} - \Pi_{\theta''} |(z',dz'')\: 
	|\tilde{\Pi}_{\theta''}^{n-j} |(z,dz') 
	\\
	\leq &
	\tilde{C}_{4,Q} \delta_{Q}^{n} (n+1) \|\theta'-\theta''\|
	\\
	= &
	C_{1,Q}\delta_{Q}^{n+1} (n+1) \|\theta'-\theta''\|
\end{align*}
for any Borel-set $B\subseteq{\cal X}^{2N}$ 
and all $\theta',\theta''\in Q$, $z\in{\cal X}^{2N}$, $n\geq 0$. 
Hence, we have  
\begin{align*}
	|\pi_{\theta'}(B) - \pi_{\theta''}(B) |
	\leq 
	|\Pi_{\theta'}^{n}(z,B) - \Pi_{\theta''}^{n}(z,B) |
	+
	|\tilde{\Pi}_{\theta'}^{n}(z,B) | 
	+
	|\tilde{\Pi}_{\theta''}^{n}(z,B) | 
	\leq 
	C_{1,Q}\|\theta'-\theta''\| + 2C_{1,Q}\delta_{Q}^{n}
\end{align*}
for any Borel-set $B\subseteq{\cal X}^{2N}$ and all $\theta',\theta''\in Q$, 
$z\in{\cal X}^{2N}$, $n\geq 1$. 
Then, letting $n\rightarrow\infty$, 
we conclude that (\ref{l5.1.5*}) holds for any Borel-set $B\subseteq{\cal X}^{2N}$ and all $\theta',\theta''\in Q$. 
\end{proof}

\begin{lemma}\label{lemma5.2}
Suppose that Assumptions \ref{a5.1} -- \ref{a5.3} hold. 
Let $Q\subset\Theta$ be any compact set. 
Then, the following is true: 
\begin{compactenum}[(i)]
\item
$f(\cdot )$ is differentiable and 
$\nabla f(\cdot )$, $g(\cdot )$ are locally Lipschitz continuous. 
\item
There exists a real number 
$C_{2,Q} \in [1,\infty )$ (independent of $N$) such that 
$\|\eta(\theta ) \|\leq C_{2,Q}/N$
for all $\theta \in Q$. 
\item
If Assumption \ref{a5.4.a} is satisfied,  
$f(\cdot )$ is $p$ times differentiable. 
\item
If Assumption \ref{a5.4.b} is satisfied, 
$f(\cdot )$ is real-analytic. 
\end{compactenum}
\end{lemma}

\begin{proof}
(i) 
Owing to Assumptions \ref{a5.2}, \ref{a5.3}, 
there exists a real number $\varepsilon_{Q}\in(0,1)$ 
such that $p_{\theta}(x'|x)\geq\varepsilon_{Q}$ for all $\theta\in Q$, $x,x'\in{\cal X}$. 
Let $\tilde{C}_{1,Q}\in[1,\infty )$ be an upper bound in $(\theta,x,x')$ 
for $p_{\theta }(x'|x)$, $\|\nabla_{\theta } p_{\theta }(x'|x) \|$
on $Q\times{\cal X}\times{\cal X}$, 
while $\tilde{C}_{2,Q}\in[1,\infty )$ is a Lipschitz constant in $(\theta,x,x')$ 
for $p_{\theta }(x'|x)$, $\nabla_{\theta } p_{\theta }(x'|x)$
on $Q\times{\cal X}\times{\cal X}$. 
Moreover, let $\tilde{C}_{3,Q}=2\varepsilon_{Q}^{-2}\tilde{C}_{1,Q}\tilde{C}_{2,Q}$. 
Then, we have $\|s_{\theta }(x,x') \|\leq \tilde{C}_{3,Q}$
and 
\begin{align}\label{l5.2.101}
	\|s_{\theta'}(x,x') - s_{\theta''}(x,x') \|
	\leq &
	\frac{\|\nabla_{\theta } p_{\theta' }(x'|x) - \nabla_{\theta } p_{\theta''}(x'|x) \| 
	+
	\|s_{\theta''}(x,x') \|\: |p_{\theta'}(x'|x) - p_{\theta''}(x'|x) | }
	{p_{\theta'}(x'|x) }
	\nonumber\\
	\leq &
	\tilde{C}_{3,Q}\|\theta' - \theta'' \|
\end{align}
for all $\theta,\theta',\theta''\in Q$, $x,x'\in{\cal X}$. 
Hence, 
\begin{align}\label{l5.2.103'}
	\|G(\theta,z)\|\leq\tilde{C}_{3,Q}, 
	\;\;\;\;\; 
	\|G(\theta',z) - G(\theta'',z) \| \leq \tilde{C}_{3,Q} \|\theta'-\theta''\|
\end{align}
for each $\theta,\theta',\theta''\in Q$, $z\in{\cal X}^{2N}$. 
As $s_{\theta}(x,x')=\nabla_{\theta }\log p_{\theta}(x'|x)$ 
for each $\theta\in\Theta$, $x,x'\in{\cal X}$, 
the dominated convergence theorem implies that $f(\cdot )$ is differentiable on $Q$
and satisfies
\begin{align}\label{l5.2.103}
	\nabla f(\theta )
	=
	-\int_{\cal X}\int_{\cal X} s_{\theta }(x,x') p(x) p(x') dx dx'
\end{align}
for any $\theta\in Q$. 
As $Q$ is any compact set, we conclude that $f(\cdot )$ is differentiable on $\Theta$. 
Combining (\ref{l5.2.101}) -- (\ref{l5.2.103}), we deduce 
\begin{align*}
	\|\nabla f(\theta') - \nabla f(\theta'') \| 
	\leq 
	\int_{\cal X}\int_{\cal X} \|s_{\theta'}(x,x') - s_{\theta''}(x,x') \| p(x)p(x') dx dx'
	\leq 
	\tilde{C}_{3,Q}\|\theta'-\theta''\|
\end{align*}
for all $\theta',\theta''\in Q$. 
Similarly, using Lemma \ref{lemma5.1} and (\ref{l5.2.103'}), 
we conclude 
\begin{align*}
	\|g(\theta') - g(\theta'') \|
	\leq &
	\int_{{\cal X}^{2N} } \|G(\theta',z) - G(\theta'',z) \| \: \pi_{\theta'}(dz) 
	+
	\int_{{\cal X}^{2N} } \|G(\theta'',z) \| \: |\pi_{\theta'}-\pi_{\theta''} |(dz) 
	\\
	\leq &
	2C_{1,Q}\tilde{C}_{3,Q} \|\theta'-\theta'' \|
\end{align*}
for each $\theta',\theta''\in Q$. 
Since $Q$ is any compact set, 
we deduce that $\nabla f(\cdot )$, $g(\cdot )$ are Lipschitz continuous on $\Theta$. 

(ii) Due to Assumptions \ref{a5.2}, \ref{a5.3}, 
there exists a real number $\varepsilon_{1,Q}\in(0,1)$ such that 
$p(x)\geq\varepsilon_{1,Q}$, $p_{\theta }(x'|x)\geq\varepsilon_{1,Q}$
for all $\theta\in Q$, $x,x'\in{\cal X}$. 
Let $\tilde{C}\in[1,\infty )$ be an upper bound of $p(\cdot )$ on ${\cal X}$, 
while $\tilde{C}_{1,Q}\in[\tilde{C},\infty )$ is an upper bound in $(\theta,x,x')$
of $p_{\theta}(x'|x)$, $\|\nabla_{\theta } p_{\theta }(x,x') \|$ 
on $\Theta\times{\cal X}\times{\cal X}$. 
Moreover, let $\varepsilon_{2,Q}=\varepsilon_{1,Q}\tilde{C}_{1,Q}^{-1}$, 
$\tilde{C}_{2,Q}=\varepsilon_{1,Q}^{-1} \tilde{C}_{1,Q}$, 
$\tilde{C}_{3,Q}=\varepsilon_{2,Q}^{-1}\tilde{C}_{2,Q}^{4}$. 
On the other side, let 
\begin{align*}
	t_{\theta }(x)
	=
	\int s_{\theta }(x,x'') p(x'') dx'', 
	\;\;\;\;\; 
	\bar{w}_{\theta }(x,x')
	=
	\frac{p(x') }{p_{\theta }(x'|x) }, 
\end{align*}
for $\theta\in\Theta$, $x,x'\in{\cal X}$. 
Then, we have 
$\varepsilon_{2,Q}\leq \bar{w}_{\theta }(x,x')\leq \tilde{C}_{2,Q}$, 
$\|s_{\theta }(x,x') \|\leq\tilde{C}_{2,Q}$ for all $\theta\in Q$, $x,x'\in{\cal X}$. 
We also have $\|t_{\theta }(x) \| \leq\tilde{C}_{2,Q}$ for each $\theta\in Q$, $x\in{\cal X}$. 

For $\theta\in\Theta$, $x,x'\in{\cal X}$, $1\leq i,j\leq N$, let 
$S_{n,i,j}^{\theta } = s_{\theta }(X_{n}^{\theta }(i), \tilde{X}_{n+1}^{\theta }(j) )$, 
$\bar{W}_{n,j}^{\theta} = \bar{w}_{\theta }(X_{n}^{\theta }(j), \tilde{X}_{n+1}^{\theta }(j) )$, 
while $E_{\theta }(\cdot )$ denotes expectation in the probability space 
$(\Omega,{\cal F}, P_{\theta } )$
($X_{n}^{\theta }(i)$, $\tilde{X}_{n+1}^{\theta }(i)$ are specified in Section \ref{section5}). 
Then, it is straightforward to verify 
\begin{align*}
	E_{\theta }\left(\left. 
	s_{\theta }(X_{n}^{\theta }(i), X_{n+1}^{\theta }(i) ) 
	\right|X_{n}^{\theta } \right)
	=&
	E_{\theta }\left(\left. 
	s_{\theta }(X_{n}^{\theta }(i), \tilde{X}_{n+1}^{\theta }(I_{n+1}^{\theta }(i) ) ) 
	\right|X_{n}^{\theta } \right)
	\\
	=&
	E_{\theta }\left(\left. 
	\frac{\sum_{j=1}^{N} s_{\theta }(X_{n}^{\theta }(i), \tilde{X}_{n+1}^{\theta }(j) ) 
	w_{\theta }(X_{n}^{\theta }(j), \tilde{X}_{n+1}^{\theta }(j) ) }
	{\sum_{j=1}^{N} w_{\theta }(X_{n}^{\theta }(j), \tilde{X}_{n+1}^{\theta }(j) ) }
	\right|X_{n}^{\theta } \right)
	\\
	=&
	E_{\theta }\left(\left. 
	\frac{\frac{1}{N} \sum_{j=1}^{N} S_{n,i,j}^{\theta } \bar{W}_{n,j}^{\theta } }
	{1 + \frac{1}{N} \sum_{j=1}^{N} (\bar{W}_{n,j}^{\theta } - 1 ) }
	\right|X_{n}^{\theta } \right)
\end{align*}
for $\theta\in\Theta$, $n\geq 0$, $1\leq i\leq N$.\footnote
{Notice that 
$w_{\theta }(X_{n}^{\theta }(j), \tilde{X}_{n+1}^{\theta }(j) ) = c \bar{W}_{n,j}^{\theta }$, 
where $c=\int_{\cal X} q(x)dx$. 
$I_{n+1}^{\theta}(i)$ is specified in Section \ref{section5}
(see the footnote after (\ref{5.103})).}
Consequently, 
\begin{align}\label{l5.2.21}
	E_{\theta }\left(\left. 
	s_{\theta }(X_{n}^{\theta }(i), X_{n+1}^{\theta }(i) ) 
	\right|X_{n}^{\theta } \right)
	=&
	E_{\theta }\left(\left. 
	\frac{1}{N} \sum_{j=1}^{N} S_{n,i,j}^{\theta } \bar{W}_{n,j}^{\theta } 
	\right|X_{n}^{\theta } \right)
	\nonumber\\
	&
	-
	E_{\theta }\left(\left. 
	\left(\frac{1}{N} \sum_{j=1}^{N} S_{n,i,j}^{\theta } \bar{W}_{n,j}^{\theta } \right) 
	\left(\frac{1}{N} \sum_{j=1}^{N} (\bar{W}_{n,j}^{\theta } - 1 ) \right)
	\right|X_{n}^{\theta } \right)
	\nonumber\\
	&
	+
	E_{\theta }\left(\left. 
	\left(\frac{1}{N} \sum_{j=1}^{N} S_{n,i,j}^{\theta } \bar{W}_{n,j}^{\theta } \right) 
	\phi\left(\frac{1}{N} \sum_{j=1}^{N} (\bar{W}_{n,j}^{\theta } - 1 ) \right)
	\right|X_{n}^{\theta } \right)
\end{align}
for $\theta\in\Theta$, $n\geq 0$, $1\leq i\leq N$, 
where $\phi(t)=t^{2}/(1+t)$ for $t\in(-1,\infty )$.\footnote
{Notice that $1/(1+t)=1-t+\phi(t)$. Notice also that 
$\frac{1}{N} \sum_{j=1}^{N} (\bar{W}_{n,j}^{\theta } - 1 )>-1$.}
On the other side, we have 
\begin{align*}
	E_{\theta }\left(\left. 
	\frac{1}{N} \sum_{j=1}^{N} S_{n,i,j}^{\theta } \bar{W}_{n,j}^{\theta } 
	\right|X_{n}^{\theta } \right)
	=&
	\frac{1}{N} \sum_{j=1}^{N} 
	E_{\theta }\left(\left. 
	s_{\theta }(X_{n}^{\theta }(i), \tilde{X}_{n+1}^{\theta }(j) )
	\bar{w}_{\theta }(X_{n}^{\theta }(j), \tilde{X}_{n+1}^{\theta }(j) )
	\right|X_{n}^{\theta } \right)
	\\
	=&
	\frac{1}{N} \sum_{j=1}^{N} 
	\int_{\cal X} s_{\theta }(X_{n}^{\theta }(i), x_{j} )
	\bar{w}_{\theta }(X_{n}^{\theta }(j), x_{j} )
	p_{\theta }(x_{j}|X_{n}^{\theta }(j) ) dx_{j} 
	\\
	=&
	t_{\theta }(X_{n}^{\theta }(i) )
\end{align*}
for $\theta\in\Theta$, $n\geq 0$, $1\leq i\leq N$. 
We also have
\begin{align}\label{l5.2.1}
	E_{\theta }\left(\left. 
	\bar{W}_{n,j}^{\theta } - 1  
	\right|X_{n}^{\theta } \right)
	=
	\int_{\cal X} \bar{w}_{\theta }(X_{n}^{\theta }(j), x_{j} )
	p_{\theta }(x_{j}|X_{n}^{\theta }(j) ) dx_{j} 
	-
	1
	=
	0
\end{align}
for $\theta\in\Theta$, $n\geq 0$, $1\leq j\leq N$. 
Therefore, 
\begin{align}\label{l5.2.23}
	E_{\theta }\left(\left. 
	\left(\frac{1}{N} \sum_{j=1}^{N} S_{n,i,j}^{\theta } \bar{W}_{n,j}^{\theta } \right) 
	\left(\frac{1}{N} \sum_{j=1}^{N} (\bar{W}_{n,j}^{\theta } - 1 ) \right)
	\right|X_{n}^{\theta } \right)
	=&
	E_{\theta }\left(\left. 
	\frac{1}{N^{2} } \sum_{j=1}^{N} 
	S_{n,i,j}^{\theta } \bar{W}_{n,j}^{\theta } (\bar{W}_{n,j}^{\theta } - 1 ) 
	\right|X_{n}^{\theta } \right)
	\nonumber\\
	&
	+
	\frac{1}{N^{2} }
	\sum_{\stackrel{\scriptstyle 1\leq j,k \leq N }{j\neq k } } 
	E_{\theta }\left(\left. 
	S_{n,i,j}^{\theta } \bar{W}_{n,j}^{\theta } (\bar{W}_{n,k}^{\theta } - 1 ) 
	\right|X_{n}^{\theta } \right)
	\nonumber\\
	=&
	E_{\theta }\left(\left. 
	\frac{1}{N^{2} } \sum_{j=1}^{N} 
	S_{n,i,j}^{\theta } \bar{W}_{n,j}^{\theta } (\bar{W}_{n,j}^{\theta } - 1 ) 
	\right|X_{n}^{\theta } \right)
\end{align}
for $\theta\in\Theta$, $n\geq 0$, $1\leq i\leq N$.\footnote
{Notice that $S_{n,i,j}^{\theta } \bar{W}_{n,j}^{\theta }$ and $\bar{W}_{n,k}^{\theta }$ 
are independent conditionally on $X_{n}^{\theta }$ whenever $j\neq k$. } 
Hence, 
\begin{align}\label{l5.2.3}
	E_{\theta }\left(\left. 
	s_{\theta }(X_{n}^{\theta }(i), X_{n+1}^{\theta }(i) ) 
	\right|X_{n}^{\theta } \right)
	-
	t_{\theta }(X_{n}^{\theta }(i) )
	=&
	-
	E_{\theta }\left(\left. 
	\frac{1}{N^{2} } \sum_{j=1}^{N} 
	S_{n,i,j}^{\theta } \bar{W}_{n,j}^{\theta } (\bar{W}_{n,j}^{\theta } - 1 ) 
	\right|X_{n}^{\theta } \right)
	\nonumber\\
	&
	+
	E_{\theta }\left(\left. 
	\left(\frac{1}{N} \sum_{j=1}^{N} S_{n,i,j}^{\theta } \bar{W}_{n,j}^{\theta } \right) 
	\phi\left(\frac{1}{N} \sum_{j=1}^{N} (\bar{W}_{n,j}^{\theta } - 1 ) \right)
	\right|X_{n}^{\theta } \right)
\end{align}
for $\theta\in\Theta$, $n\geq 0$, $1\leq i\leq N$. 
Since $\|S_{n,i,j}^{\theta } \|\leq\tilde{C}_{2,Q}$, 
$\varepsilon_{2,Q}\leq \bar{W}_{n,j}^{\theta }\leq\tilde{C}_{2,Q}$
for each $\theta\in Q$, $n\geq 0$, $1\leq i,j\leq N$, 
we conclude 
\begin{align*}
	\left\|
	\left(\frac{1}{N} \sum_{j=1}^{N} S_{n,i,j}^{\theta } \bar{W}_{n,j}^{\theta } \right) 
	\phi\left(\frac{1}{N} \sum_{j=1}^{N} (\bar{W}_{n,j}^{\theta } - 1 ) \right)
	\right\|
	\leq &
	\tilde{C}_{2,Q}^{2} 
	\phi\left(\frac{1}{N} \sum_{j=1}^{N} (\bar{W}_{n,j}^{\theta } - 1 ) \right)
	\\
	\leq &
	\tilde{C}_{2,Q}^{2} \varepsilon_{2,Q}^{-1} 
	\left|\frac{1}{N} \sum_{j=1}^{N} (\bar{W}_{n,j}^{\theta } - 1 ) \right|^{2}
\end{align*}
for the same $\theta$, $n$, $i$.\footnote 
{Notice that $1+\frac{1}{N} \sum_{j=1}^{N} (\bar{W}_{n,j}^{\theta } - 1 ) \geq\varepsilon_{2,Q}$. 
Notice also that $\phi(t)\leq\varepsilon_{2,Q}^{-1}t^{2}$ when $t+1\geq\varepsilon_{2,Q}$. }
Similarly, we get
\begin{align*}
	\left\|
	\frac{1}{N^{2} } \sum_{j=1}^{N} 
	S_{n,i,j}^{\theta } \bar{W}_{n,j}^{\theta } (\bar{W}_{n,j}^{\theta } - 1 ) 
	\right\|
	\leq
	\frac{\tilde{C}_{2,Q}^{3} }{N}
\end{align*}
for all $\theta\in Q$, $n\geq 0$, $1\leq i\leq N$. 
On the other side, (\ref{l5.2.1}) yields 
\begin{align}\label{l5.2.27}
	E_{\theta }\left(\left. 
	\left|\frac{1}{N} \sum_{j=1}^{N} (\bar{W}_{n,j}^{\theta } - 1 ) \right|^{2} 
	\right|X_{n}^{\theta } \right)
	=&
	E_{\theta }\left(\left. 
	\frac{1}{N^{2} } \sum_{j=1}^{N} (\bar{W}_{n,j}^{\theta } - 1 )^{2}  
	\right|X_{n}^{\theta } \right)
	\nonumber\\
	&
	+
	\frac{1}{N^{2} }
	\sum_{\stackrel{\scriptstyle 1\leq j,k \leq N }{j\neq k } } 
	E_{\theta }\left(\left. 
	(\bar{W}_{n,j}^{\theta } - 1 ) (\bar{W}_{n,k}^{\theta } - 1 ) 
	\right|X_{n}^{\theta } \right)
	\nonumber\\
	=&
	E_{\theta }\left(\left. 
	\frac{1}{N^{2} } \sum_{j=1}^{N} (\bar{W}_{n,j}^{\theta } - 1 )^{2}  
	\right|X_{n}^{\theta } \right)
	\nonumber\\
	\leq &
	\frac{\tilde{C}_{2,Q}^{2} }{N}
\end{align}
for each $\theta\in Q$, $n\geq 0$.\footnote
{Notice that $\bar{W}_{n,j}^{\theta }$ and $\bar{W}_{n,k}^{\theta }$ 
are independent conditionally on $X_{n}^{\theta }$ whenever $j\neq k$. } 
Thus, 
\begin{align*}
	&
	\left\|
	E_{\theta }\left(\left. 
	\frac{1}{N^{2} } \sum_{j=1}^{N} 
	S_{n,i,j}^{\theta } \bar{W}_{n,j}^{\theta } (\bar{W}_{n,j}^{\theta } - 1 ) 
	\right|X_{n}^{\theta } \right)
	\right\|
	\leq
	\frac{\tilde{C}_{3,Q} }{N}, 
	\\
	&
	\left\|
	E_{\theta }\left(\left. 
	\left(\frac{1}{N} \sum_{j=1}^{N} S_{n,i,j}^{\theta } \bar{W}_{n,j}^{\theta } \right) 
	\phi\left(\frac{1}{N} \sum_{j=1}^{N} (\bar{W}_{n,j}^{\theta } - 1 ) \right)
	\right|X_{n}^{\theta } \right)
	\right\|
	\leq 
	\frac{\tilde{C}_{3,Q} }{N}
\end{align*}
for all $\theta\in Q$, $n\geq 0$, $1\leq i\leq N$. 
Combining this with (\ref{l5.2.3}), we deduce  
\begin{align}\label{l5.2.5}
	\left\|
	E_{\theta }\left(\left. 
	s_{\theta }(X_{n}^{\theta }(i), X_{n+1}^{\theta }(i) ) 
	\right|X_{n}^{\theta } \right)
	-
	t_{\theta }(X_{n}^{\theta }(i) ) 
	\right\|
	\leq 
	\frac{2\tilde{C}_{3,Q} }{N}
\end{align}
for each $\theta\in Q$, $n\geq 0$, $1\leq i\leq N$. 

For $\theta\in\Theta$, $n\geq 0$, $1\leq j\leq N$, 
let $T_{n,j}^{\theta } = t_{\theta }(\tilde{X}_{n+1}^{\theta }(j) )$. 
Then, similarly as (\ref{l5.2.21}), we conclude 
\begin{align}\label{l5.2.31}
	E_{\theta }\left(\left. 
	t_{\theta }(X_{n+1}^{\theta }(i) ) 
	\right|X_{n}^{\theta } \right)
	=&
	E_{\theta }\left(\left. 
	\frac{1}{N} \sum_{j=1}^{N} T_{n,j}^{\theta } \bar{W}_{n,j}^{\theta } 
	\right|X_{n}^{\theta } \right)
	-
	E_{\theta }\left(\left. 
	\left(\frac{1}{N} \sum_{j=1}^{N} T_{n,j}^{\theta } \bar{W}_{n,j}^{\theta } \right) 
	\left(\frac{1}{N} \sum_{j=1}^{N} (\bar{W}_{n,j}^{\theta } - 1 ) \right)
	\right|X_{n}^{\theta } \right)
	\nonumber\\
	&
	+
	E_{\theta }\left(\left. 
	\left(\frac{1}{N} \sum_{j=1}^{N} T_{n,j}^{\theta } \bar{W}_{n,j}^{\theta } \right) 
	\phi\left(\frac{1}{N} \sum_{j=1}^{N} (\bar{W}_{n,j}^{\theta } - 1 ) \right)
	\right|X_{n}^{\theta } \right)
\end{align}
for $\theta\in\Theta$, $n\geq 0$, $1\leq i\leq N$. 
Moreover, similarly as (\ref{l5.2.23}), we deduce
\begin{align}\label{l5.2.33}
	E_{\theta }\left(\left. 
	\left(\frac{1}{N} \sum_{j=1}^{N} T_{n,j}^{\theta } \bar{W}_{n,j}^{\theta } \right) 
	\left(\frac{1}{N} \sum_{j=1}^{N} (\bar{W}_{n,j}^{\theta } - 1 ) \right)
	\right|X_{n}^{\theta } \right)
	=&
	E_{\theta }\left(\left. 
	\frac{1}{N^{2} } \sum_{j=1}^{N} 
	T_{n,j}^{\theta } \bar{W}_{n,j}^{\theta } (\bar{W}_{n,j}^{\theta } - 1 ) 
	\right|X_{n}^{\theta } \right)
	\nonumber\\
	&
	+
	\frac{1}{N^{2} }
	\sum_{\stackrel{\scriptstyle 1\leq j,k \leq N }{j\neq k } } 
	E_{\theta }\left(\left. 
	T_{n,j}^{\theta } \bar{W}_{n,j}^{\theta } (\bar{W}_{n,k}^{\theta } - 1 ) 
	\right|X_{n}^{\theta } \right)
	\nonumber\\
	=&
	E_{\theta }\left(\left. 
	\frac{1}{N^{2} } \sum_{j=1}^{N} 
	T_{n,j}^{\theta } \bar{W}_{n,j}^{\theta } (\bar{W}_{n,j}^{\theta } - 1 ) 
	\right|X_{n}^{\theta } \right)
\end{align}
for $\theta\in\Theta$, $n\geq 0$.
On the other side, we have 
\begin{align}\label{l5.2.35}
	E_{\theta }\left(\left. 
	\frac{1}{N} \sum_{j=1}^{N} T_{n,j}^{\theta } \bar{W}_{n,j}^{\theta } 
	\right|X_{n}^{\theta } \right)
	=&
	\frac{1}{N} \sum_{j=1}^{N} 
	E_{\theta }\left(\left. 
	t_{\theta }(\tilde{X}_{n+1}^{\theta }(j) )
	\bar{w}_{\theta }(X_{n}^{\theta }(j), \tilde{X}_{n+1}^{\theta }(j) )
	\right|X_{n}^{\theta } \right)
	\nonumber\\
	=&
	\frac{1}{N} \sum_{j=1}^{N} 
	\int_{\cal X} t_{\theta }(x_{j} )
	\bar{w}_{\theta }(X_{n}^{\theta }(j), x_{j} )
	p_{\theta }(x_{j}|X_{n}^{\theta }(j) ) dx_{j} 
	\nonumber\\
	=&
	-\nabla f(\theta )
\end{align}
for $\theta\in\Theta$, $n\geq 0$. 
As a result of (\ref{l5.2.31}) -- (\ref{l5.2.35}), we get 
\begin{align}\label{l5.2.37}
	E_{\theta }\left(\left. 
	t_{\theta }(X_{n+1}^{\theta }(i) ) 
	\right|X_{n}^{\theta } \right)
	+
	\nabla f(\theta )
	=&
	-
	E_{\theta }\left(\left. 
	\frac{1}{N^{2} } \sum_{j=1}^{N} 
	T_{n,j}^{\theta } \bar{W}_{n,j}^{\theta } (\bar{W}_{n,j}^{\theta } - 1 ) 
	\right|X_{n}^{\theta } \right)
	\nonumber\\
	&
	+
	E_{\theta }\left(\left. 
	\left(\frac{1}{N} \sum_{j=1}^{N} T_{n,j}^{\theta } \bar{W}_{n,j}^{\theta } \right) 
	\phi\left(\frac{1}{N} \sum_{j=1}^{N} (\bar{W}_{n,j}^{\theta } - 1 ) \right)
	\right|X_{n}^{\theta } \right)
\end{align}
for $\theta\in\Theta$, $n\geq 0$, $1\leq i\leq N$. 
Since $\|T_{n,j}^{\theta } \|\leq\tilde{C}_{2,Q}$, 
$\varepsilon_{2,Q}\leq \bar{W}_{n,j}^{\theta }\leq\tilde{C}_{2,Q}$
for each $\theta\in Q$, $n\geq 0$, $1\leq j\leq N$, 
we conclude 
\begin{align*}
	&
	\left\|
	\left(\frac{1}{N} \sum_{j=1}^{N} T_{n,j}^{\theta } \bar{W}_{n,j}^{\theta } \right) 
	\phi\left(\frac{1}{N} \sum_{j=1}^{N} (\bar{W}_{n,j}^{\theta } - 1 ) \right)
	\right\|
	\leq 
	\tilde{C}_{2,Q}^{2} \varepsilon_{2,Q}^{-1} 
	\left|\frac{1}{N} \sum_{j=1}^{N} (\bar{W}_{n,j}^{\theta } - 1 ) \right|^{2}, 
\end{align*}
for the same $\theta$, $n$.
We also deduce 
\begin{align*}
	\left\|
	\frac{1}{N^{2} } \sum_{j=1}^{N} 
	T_{n,j}^{\theta } \bar{W}_{n,j}^{\theta } (\bar{W}_{n,j}^{\theta } - 1 ) 
	\right\|
	\leq
	\frac{\tilde{C}_{2,Q}^{3} }{N}
\end{align*}
for all $\theta\in Q$, $n\geq 0$. 
Combining this with (\ref{l5.2.27}), we get 
\begin{align*}
	&
	\left\|
	E_{\theta }\left(\left. 
	\frac{1}{N^{2} } \sum_{j=1}^{N} 
	T_{n,j}^{\theta } \bar{W}_{n,j}^{\theta } (\bar{W}_{n,j}^{\theta } - 1 ) 
	\right|X_{n}^{\theta } \right)
	\right\|
	\leq
	\frac{\tilde{C}_{3,Q} }{N}, 
	\\
	&
	\left\|
	E_{\theta }\left(\left. 
	\left(\frac{1}{N} \sum_{j=1}^{N} T_{n,j}^{\theta } \bar{W}_{n,j}^{\theta } \right) 
	\phi\left(\frac{1}{N} \sum_{j=1}^{N} (\bar{W}_{n,j}^{\theta } - 1 ) \right)
	\right|X_{n}^{\theta } \right)
	\right\|
	\leq 
	\frac{\tilde{C}_{3,Q} }{N}
\end{align*}
for all $\theta\in Q$, $n\geq 0$. 
Hence, 
\begin{align}\label{l5.2.7}
	\left\|
	E_{\theta }\left(\left. 
	t_{\theta }(X_{n+1}^{\theta }(i) ) 
	\right|X_{n}^{\theta } \right)
	+
	\nabla f(\theta ) 
	\right\|
	\leq 
	2\tilde{C}_{3,Q}/N
\end{align}
for each $\theta\in Q$, $n\geq 0$. 

Let $C_{2,Q}=4\tilde{C}_{3,Q}$. 
It is straightforward to demonstrate
\begin{align*}
	E_{\theta }\left(\left.G(\theta, Z_{n+2}^{\theta } ) \right|Z_{n}^{\theta } \right)
	-
	\nabla f(\theta )
	=&
	-E_{\theta }\left(\left.
	\frac{1}{N} \sum_{i=1}^{N} s_{\theta }(X_{n+1}^{\theta }(i), X_{n+2}^{\theta }(i) )
	\right|X_{n}^{\theta } \right)
	-
	\nabla f(\theta )
	\\
	=&
	-E_{\theta }\left(\left.
	\frac{1}{N} \sum_{i=1}^{N} 
	\left(
	E_{\theta }\left(\left. 
	s_{\theta }(X_{n+1}^{\theta }(i), X_{n+2}^{\theta }(i) )\right|X_{n+1}^{\theta } 
	\right)
	-
	t_{\theta }(X_{n+1}^{\theta }(i) ) 
	\right)
	\right|X_{n}^{\theta } \right)
	\\
	&
	-
	\frac{1}{N} \sum_{i=1}^{N} 
	\left(
	E_{\theta }\left(\left. 
	t_{\theta }(X_{n+1}^{\theta }(i) )\right|X_{n}^{\theta } 
	\right)
	+
	\nabla f(\theta ) 
	\right)
\end{align*}
for $\theta\in\Theta$, $n\geq 0$. 
Then, (\ref{l5.2.5}), (\ref{l5.2.7}) imply 
\begin{align*}
	\left\|
	E_{\theta }\left(\left.G(\theta, Z_{n+2}^{\theta } ) \right|Z_{n}^{\theta } \right)
	-
	\nabla f(\theta )
	\right\|
	\leq &
	E_{\theta }\left(\left.
	\frac{1}{N} \sum_{i=1}^{N} 
	\left\|
	E_{\theta }\left(\left. 
	s_{\theta }(X_{n+1}^{\theta }(i), X_{n+2}^{\theta }(i) )\right|X_{n+1}^{\theta } 
	\right)
	-
	t_{\theta }(X_{n+1}^{\theta }(i) ) 
	\right\|
	\right|X_{n}^{\theta } \right)
	\\
	&
	+
	\frac{1}{N} \sum_{i=1}^{N} 
	\left\|
	E_{\theta }\left(\left. 
	t_{\theta }(X_{n+1}^{\theta }(i) )\right|X_{n}^{\theta } 
	\right)
	+
	\nabla f(\theta ) 
	\right\|
	\\
	\leq &
	\frac{C_{2,Q} }{N}
\end{align*}
for each $\theta\in Q$, $n\geq 0$. 
Consequently, 
\begin{align*}
	\|(\Pi^{2} G )(\theta, z) - \nabla f(\theta ) \|
	=
	\|E_{\theta }\left(\left.G(\theta, Z_{n+2}^{\theta } ) \right|Z_{n}^{\theta }=z \right)
	-
	\nabla f(\theta ) \|
	\leq 
	\frac{C_{2,Q} }{N}
\end{align*}
for all $\theta\in Q$, $z\in {\cal X}^{2N}$, $n\geq 0$. 
Therefore, 
\begin{align*}
	\|\eta(\theta ) \|
	=
	\|g(\theta ) - \nabla f(\theta ) \|
	=
	\left\|
	\int_{{\cal X}^{2N} } \left( (\Pi^{2} G )(\theta, z' ) - \nabla f(\theta ) \right) \pi_{\theta }(dz') 
	\right\|
	\leq 
	\frac{C_{2,Q} }{N}
\end{align*}
for each $\theta\in Q$.

(iii) Owing to Assumption \ref{a5.4.a}, 
there exists a real number $\tilde{C}_{Q}\in[1,\infty )$ such that 
\begin{align*}
	\left|
	\frac{\partial^{k} \log p_{\theta }(x'|x) }{\partial\vartheta_{i_{1} } \cdots \partial\vartheta_{i_{k} } }
	\right|
	\leq 
	\tilde{C}_{Q}
\end{align*}
for all $\theta\in Q$, $x,x'\in{\cal X}$, $1\leq k\leq p$, 
$1\leq i_{1},\dots,i_{k} \leq d_{\theta }$, 
where $\vartheta_{i}$ is the $i$-th component of $\theta$. 
Then, using the dominated convergence theorem, we conclude that 
$f(\cdot )$ is differentiable $p$ times. 

(iv) 
Let $\theta\in\Theta$ be an arbitrary vector, while 
\begin{align*}	
	\hat{H}_{\eta}(x,x') = \log\hat{p}_{\eta}(x'|x),
	\;\;\;\;\; 
	\hat{f}(\eta ) = \int_{\cal X}\int_{\cal X} \hat{H}_{\eta}(x,x') p(x) p(x') dx dx'
\end{align*} 
for $\eta\in\mathbb{C}^{d_{\theta } }$, $x,x'\in {\cal X}$.  
To prove this part of the lemma, it sufficient to show that $\hat{f}(\cdot )$ is analytic 
in an open vicinity of $\theta$. 

Let $V_{\delta_{\theta } }(\theta ) = \{\eta\in\mathbb{C}^{d_{\theta } }: 
\|\eta-\theta\|\leq \delta_{\theta } \}$
($\delta_{\theta }$ is specified in Assumption \ref{a5.4.b}). 
Since $V_{\delta_{\theta } }(\theta ) \times {\cal X} \times {\cal X}$ is a compact set, 
Assumptions \ref{a5.3}, \ref{a5.4.b} imply that there exist 
real numbers $\varepsilon_{\theta} \in (0,\delta_{\theta } )$, 
$C_{1,\theta } \in [1,\infty )$ such that 
$C_{1,\theta}^{-1} \leq |\hat{p}_{\eta}(x'|x) | \leq C_{1,\theta}$
for all $\eta\in V_{\varepsilon_{\theta } }(\theta )$, $x,x'\in {\cal X}$. 
Therefore, $\hat{H}_{\eta}(x,x')$ is analytic in $\eta$ for all 
$\eta\in V_{\varepsilon_{\theta } }(\theta )$, $x,x'\in {\cal X}$. 
Moreover, $|\hat{H}_{\eta}(x,x') | \leq \log C_{1,\theta }$ for all 
$\eta\in V_{\varepsilon_{\theta } }(\theta )$, $x,x'\in {\cal X}$. 
Then, using Cauchy inequality for complex analytic functions, 
we deduce that there exists a real number $C_{2,\theta } \in [1,\infty )$
such that $\|\nabla_{\eta } \hat{H}_{\eta}(x,x') \| \leq C_{2,\theta }$ for all 
$\eta\in V_{\varepsilon_{\theta } }(\theta )$, $x,x'\in {\cal X}$. 
Consequently, the dominated convergence theorem implies that $\hat{f}(\eta )$
is differentiable for all $\eta\in V_{\varepsilon_{\theta} }(\theta )$. 
Hence, $\hat{f}(\cdot )$ is analytic on $V_{\varepsilon_{\theta} }(\theta )$. 
\end{proof}

\begin{vproof}{Theorem \ref{theorem5.1}}
{For $\theta\in\Theta$, $z\in{\cal X}^{2N}$, $n\geq 0$, 
let 
\begin{align*}
	(\tilde{\Pi}^{n} F)(\theta,z)
	=
	\int_{{\cal X}^{2N} } F(\theta, z' ) \tilde{\Pi}_{\theta }^{n}(z,dz'), 
\end{align*}
while $\tilde{F}(\theta, z ) = \sum_{n=0}^{\infty } (\tilde{\Pi}^{n} F)(\theta,z)$. 
On the other side, let $\tilde{C}_{1,Q}\in[1,\infty )$ 
be an upper bound of $\|F(\theta, z)\|$ on $Q\times{\cal X}^{2N}$. 
Moreover, let $\tilde{C}_{2,Q}\in[1,\infty )$ be
a Lipschitz constant in $\theta$ for $F(\theta,z)$ on $Q\times{\cal X}^{2N}$.\footnote
{Owing to Assumption \ref{a5.3}, $G(\theta,z)$ is locally Lipschitz continuous 
in $\theta$ (see (\ref{l5.2.103'})). 
As $\nabla f(\cdot )$, $g(\cdot )$ are locally Lipschitz continuous
(due to Lemma \ref{lemma5.2}), $F(\theta,z)$ is locally Lipschitz continuous in $\theta$. } 
Then, Lemma \ref{lemma5.1} implies 
\begin{align}
	&\label{t5.1.1}
	\|(\tilde{\Pi}^{n} F)(\theta,z) \|
	\leq
	\int_{{\cal X}^{2N} } \|F(\theta,z') \|\: |\tilde{\Pi}_{\theta }^{n} |(z,dz')
	\leq 
	C_{1,Q}\tilde{C}_{1,Q} \delta_{Q}^{n} 
\end{align}
for all $\theta\in Q$, $z\in{\cal X}^{2N}$, $n\geq 0$. 
The same lemma also yields 
\begin{align}
	\label{t5.1.3}
	\|(\tilde{\Pi}^{n} F)(\theta',z) - (\tilde{\Pi}^{n} F)(\theta'',z) \|
	\leq &
	\int_{{\cal X}^{2N} } \|F(\theta',z') - F(\theta'', z') \| \: |\tilde{\Pi}_{\theta'}^{n} |(z,dz') 
	\\
	&
	+
	\int_{{\cal X}^{2N} } \|F(\theta'',z') \| \: |\tilde{\Pi}_{\theta'}^{n} - \tilde{\Pi}_{\theta''}^{n} |(z,dz') 
	\nonumber\\
	\leq &
	2C_{1,Q}\tilde{C}_{1,Q}\tilde{C}_{2,Q} \delta_{Q}^{n} n \|\theta' - \theta'' \|
\end{align}
for all $\theta',\theta''\in Q$, $z\in{\cal X}^{2N}$, $n\geq 0$. 
Therefore, $\sum_{n=0}^{\infty } \|(\tilde{\Pi}^{n} F)(\theta,z) \| < \infty$
for any $\theta\in Q$, $z\in{\cal X}^{2N}$. 
As $Q$ is any compact set, we conclude that 
for each $\theta\in \Theta$, $z\in{\cal X}^{2N}$, 
$\tilde{F}(\theta,z)$ is well-defined and satisfies
$(\Pi\tilde{F} )(\theta, z ) = \sum_{n=1}^{\infty } (\tilde{\Pi}^{n} F)(\theta,z)$. 
Consequently, 
\begin{align*}
	\tilde{F}(\theta,z)
	-
	(\Pi\tilde{F} )(\theta, z )
	=
	(\tilde{\Pi}^{0} F)(\theta, z )
	=
	F(\theta,z)
	-
	\int_{{\cal X}^{2N} }
	F(\theta,z') \pi_{\theta}(dz') 
	=
	F(\theta,z) - \nabla f(\theta )
\end{align*}
for all $\theta\in\Theta$, $z\in{\cal X}^{2N}$.\footnote
{Notice that 
$\int_{{\cal X}^{2N} } F(\theta,z') \pi_{\theta}(dz') = g(\theta ) - \eta(\theta )$.} 
Thus, Assumption \ref{a2.2} holds. 

Due to (\ref{t5.1.1}), we have 
\begin{align*}
	&
	\max\{\|\tilde{F}(\theta,z) \|, \|(\Pi\tilde{F} )(\theta,z ) \| \} 
	\leq 
	\sum_{n=0}^{\infty } \|(\tilde{\Pi}^{n} F(\theta,z) \|
	\leq
	C_{1,Q}\tilde{C}_{1,Q}(1-\delta_{Q} )^{-1}, 
\end{align*}
for each $\theta\in Q$, $z\in{\cal X}^{2N}$. 
On the other side, (\ref{t5.1.3}) yields 
\begin{align*}
	&
	\|(\Pi\tilde{F} )(\theta',z) - (\Pi\tilde{F} )(\theta'',z) \|
	\leq 
	\sum_{n=1}^{\infty } 
	\|(\tilde{\Pi} F)(\theta',z) - (\tilde{\Pi} F)(\theta'',z) \|
	\leq
	2C_{1,Q}\tilde{C}_{1,Q}\tilde{C}_{2,Q}(1-\delta_{Q} )^{-2} \|\theta' - \theta'' \|
\end{align*}
for all $\theta',\theta''\in Q$, $z\in{\cal X}^{2N}$. 
Hence, Assumption \ref{a2.3} holds, too. 
On the other side, Lemma \ref{lemma5.1} yields 
\begin{align*}
	\eta
	=
	\limsup_{n\rightarrow \infty } 
	\|\eta_{n} \| 
	\leq 
	C_{2,Q}/N
\end{align*}
on $\Lambda_{Q}$
(notice that $C_{2,Q}$ does not depend on $N$). 
Then, the theorem's assertion directly follows from 
Theorem \ref{theorem2.1} and Parts (i), (iii), (iv) of Lemma \ref{lemma5.1}. 
}
\end{vproof}

\section{Proof Theorem \ref{theorem4.1}} \label{section4*}

In this section, we rely on the following notation. 
Functions $F(\cdot,\cdot )$, $\eta(\cdot )$ are defined by
\begin{align*}
	\eta(\theta )
	=
	\nabla f_{N}(\theta ) - \nabla f(\theta ), 
	\;\;\;\;\; 
	F(\theta,z)
	=
	\psi_{N,\theta }(y_{1:N} )
	-
	\eta(\theta )
\end{align*}
for $\theta\in\Theta$, 
$z = (x_{1:N}, y_{1:N} ) \in {\cal X}^{N} \times {\cal Y}^{N}$. 
Stochastic processes 
$\{Z_{n} \}_{n\geq 0}$ and $\{\eta_{n} \}_{n\geq 0}$ are defined as 
\begin{align*}
	Z_{n+1} 
	=
	(X_{nN+1:(n+1)N}, Y_{nN+1:(n+1)N} ), 
	\;\;\;\;\; 
	\eta_{n} 
	=
	\eta(\theta_{n} )
\end{align*}
for $n\geq 0$.  
$\Pi(\cdot,\cdot )$ is the transition kernel of 
$\{Z_{n} \}_{n\geq 0}$
(notice that $\{Z_{n} \}_{n\geq 0}$ does not depend on $\{\theta_{n} \}_{n\geq 0}$, 
and consequently, $\Pi(\cdot,\cdot )$ does not depend on $\theta$). 
Then, it is straightforward to show that 
the algorithm (\ref{4.1}) is of the same form as the recursion 
studied in Section \ref{section2}
(i.e., $\{\theta_{n} \}_{n\geq 0}$, $\{\eta_{n} \}_{n\geq 0}$, 
$F(\cdot,\cdot )$, $\Pi(\cdot,\cdot )$
defined in Section \ref{section3} and here 
admit (\ref{2.1}), (\ref{2.3})).

\begin{lemma}\label{lemma4.1}
Suppose that Assumptions \ref{a4.1} -- \ref{a4.3} hold. 
Let $Q\subset\Theta$ be any compact set. 
Then, the following is true: 
\begin{compactenum}[(i)]
\item
$f(\cdot )$ is differentiable and $\nabla f(\cdot )$ is locally Lipschitz continuous. 
\item
There exists a real number $C_{Q} \in [1,\infty )$
(independent of $N$) such that 
$\|\eta(\theta ) \|\leq C_{Q}/N$ 
for all $\theta\in Q$. 
\item
If Assumption \ref{a4.4.a} is satisfied,  
$f(\cdot )$ is $p$ times differentiable. 
\item
If Assumption \ref{a4.4.b} is satisfied, 
$f(\cdot )$ is real-analytic. 
\end{compactenum}
\end{lemma}

\begin{proof}
(i), (iii), (iv) 
First, we consider the case when models 
$\{(X_{n}^{\theta }, Y_{n}^{\theta } ) \}_{n\geq 0}$ are naturally parameterized. 
To to so, we rely on the following notation. 
${\cal P}^{N_{x}\times N_{x} }$ is the set of $N_{x}\times N_{x}$ (row) stochastic matrices
whose all entries are strictly positive. 
Moreover, ${\cal P}^{N_{x}\times N_{y} }$ is the set of $N_{x}\times N_{y}$ (row) stochastic matrices
whose all entries are also strictly positive. 
On the other side, 
$h(\cdot,\cdot )$ is the function defined by 
$h(A,B)=f(\theta )$
for  
$A=[a_{i,j} ] \in {\cal P}^{N_{x}\times N_{x} }$, 
$B=[b_{i,k} ] \in {\cal P}^{N_{x}\times N_{y} }$, 
$\theta=[a_{1,1} \cdots a_{N_{x},N_{x} } \; b_{1,1} \cdots b_{N_{x},N_{y} } ]^{T}$. 
Then, \cite[Theorem 1]{tadic5} and Assumption \ref{a4.1} imply that $h(\cdot,\cdot )$ is real-analytic on 
${\cal P}^{N_{x}\times N_{x} }\times{\cal P}^{N_{x}\times N_{y} } $. 

Using function $h(\cdot,\cdot )$, we now consider any parameterization satisfying Assumptions 
\ref{a4.2}, \ref{a4.3}. 
For $\theta\in\Theta$, let $P_{\theta }$ be the $N_{x}\times N_{x}$ matrix whose $(i,j)$
entry is $p_{\theta }(j|i)$ ($1\leq i,j \leq N_{x}$), 
while $Q_{\theta }$ is the $N_{x}\times N_{y}$ matrix whose $(i,k)$
entry is $q_{\theta }(k|i)$ ($1\leq i\leq N_{x}$, $1\leq k\leq N_{y}$). 
Then, Assumption \ref{a4.2} implies 
$P_{\theta }\in {\cal P}^{N_{x}\times N_{x} }$, $Q_{\theta }\in {\cal P}^{N_{x}\times N_{y} }$
for any $\theta\in\Theta$. 
Consequently, $f(\theta ) = h(P_{\theta }, Q_{\theta } )$ for all $\theta\in\Theta$. 
On the other side, due to Assumption \ref{a4.3}, 
$P_{\theta }$, $Q_{\theta }$ are differentiable (in $\theta$), 
while their derivatives are locally Lipschitz continuous (in $\theta$). 
Therefore, 
$f(\cdot )$ is differentiable and $\nabla f(\cdot )$ is locally Lipschitz continuous
(notice that $h(\cdot,\cdot )$ is real-analytic). 
If Assumption \ref{a4.4.a} is satisfied, 
then $P_{\theta}$, $Q_{\theta}$ are $p$ times differentiable, 
and consequently, $f(\cdot )$ is $p$ times differentiable, too. 
Similarly, if Assumption \ref{a4.4.b} is satisfied, 
then $P_{\theta}$, $Q_{\theta}$ are real-analytic, 
and hence, $f(\cdot )$ is also real-analytic. 

(ii) 
Let ${\cal P}^{N_{x} }$ be the set of $N_{x}$-dimensional probability vectors, 
while $e$ is the $N_{x}$-dimensional vector whose all components are one. 
For $\theta\in\Theta$, $x,x'\in{\cal X}$, $y\in{\cal Y}$, let 
\begin{align*}
	r_{\theta }(y,x'|x)
	=
	q_{\theta }(y|x') p_{\theta }(x'|x), 
\end{align*}
while $R_{\theta }(y)$ is the $N_{x}\times N_{x}$ matrix 
whose $(i,j)$ entry is $r_{\theta }(y,i|j)$. 
For $\theta\in\Theta$, $y\in{\cal Y}$, $u\in{\cal P}^{N_{x} }$, 
$V\in\mathbb{R}^{N_{x}\times N_{x} }$, let 
\begin{align*}
	\Phi_{\theta }(y,u)
	=
	\log(e^{T} R_{\theta }(y) u ), 
	\;\;\;\;\; 
	\Psi_{\theta }(y,u,V)
	=
	\nabla_{\theta }\Phi_{\theta }(y,u) 
	+
	V\: \nabla_{u}\Phi_{\theta }(y,u). 
\end{align*}
Then, owing to Assumptions \ref{a4.2}, there exists a real number $\delta_{Q}\in(0,1)$
such that 
\begin{align}\label{l4.1.701}
	r_{\theta}(y,x'|x)\geq\delta_{Q}
\end{align}
for all $\theta\in Q$, $x,x'\in{\cal X}$, $y\in{\cal Y}$. 
Combining this with Assumption \ref{a4.3}, 
we conclude that 
there exists a real number $\tilde{C}_{1,Q}\in[1,\infty )$ such that 
\begin{align}
	&\label{l4.1.901}
	\|\Psi_{\theta}(y,u,V) \|
	\leq 
	\tilde{C}_{1,Q} (1 + \|V\| ), 
	\\
	&\label{l4.1.903}
	|\Phi_{\theta}(y,u' ) - \Phi_{\theta}(y,u'' ) |
	\leq 
	\tilde{C}_{1,Q} \|u' - u'' \|, 
	\\
	&\label{l4.1.905}
	\|\Psi_{\theta}(y,u',V') - \Psi_{\theta}(y,u'',V'') \|
	\leq 
	\tilde{C}_{1,Q} 
	(\|u'-u''\| + \|V'-V''\| ) (1 + \|V'\| + \|V''\| )
\end{align}
for each $\theta\in Q$, $y\in{\cal Y}$, 
$u, u', u'' \in{\cal P}^{N_{x} }$, $V, V', V'' \in\mathbb{R}^{N_{x}\times N_{x} }$.

For $\theta\in\Theta$, $y_{1:n}\in{\cal Y}^{n}$, $n\geq 1$, 
let $u_{0,\theta }$, $u_{n,\theta }(y_{1:n} )$ be the $N_{x}$-dimensional vectors 
whose $i$-th components are 
\begin{align*}
	u_{0,i,\theta }
	=
	\pi_{\theta }(i), 
	\;\;\;\;\; 
	u_{n,i,\theta }(y_{1:n} )
	=
	P_{\theta }(X_{n}^{\theta } = i|Y_{1:n}^{\theta } = y_{1:n} )
\end{align*}
(notice that $\{ u_{n,\theta }(y_{1:n} ) \}_{n\geq 1}$ is the optimal filter for the model 
$\{(X_{n}^{\theta }, Y_{n}^{\theta } ) \}_{n\geq 0})$. 
For the same $\theta$, $y_{1:n}$, $n$, let 
\begin{align*}
	V_{0,\theta } = \nabla_{\theta } u_{0,\theta }, 
	\;\;\;\;\; 
	V_{n,\theta }(y_{1:n} ) = \nabla_{\theta } u_{n,\theta }(y_{1:n} ). 
\end{align*}
Then, it is straightforward to verify 
\begin{align*}
	\log P_{\theta }(Y_{1}^{\theta } = y )
	=
	\Phi_{\theta}(y, u_{0,\theta } ), 
	\;\;\;\;\; 
	\log\left(
	\frac{P_{\theta }(Y_{1:n+1}^{\theta } = y_{1:n+1} ) }
	{P_{\theta }(Y_{1:n}^{\theta } = y_{1:n} ) }
	\right)
	=
	\Phi_{\theta}(y_{n+1}, u_{n,\theta }(y_{1:n} ) )
\end{align*}
for $\theta\in\Theta$, $y\in{\cal Y}$, $y_{1:n+1}=(y_{1},\dots,y_{n+1} )\in{\cal Y}^{n+1}$, 
$n\geq 0$. 
Since 
\begin{align*}
	\phi_{n,\theta }(y_{1:n} )
	=
	-
	\frac{1}{n}
	\left(
	\log P_{\theta }(Y_{1}^{\theta } = y_{1} )
	+
	\sum_{i=1}^{n-1} 
	\log\left(
	\frac{P_{\theta }(Y_{1:i+1}^{\theta } = y_{1:i+1} ) }
	{P_{\theta }(Y_{1:i}^{\theta } = y_{1:i} ) }
	\right)
	\right)
\end{align*}
for $\theta\in\Theta$, $y_{1:n}=(y_{1},\dots,y_{n} )\in{\cal Y}^{n}$, $n\geq 1$, 
we conclude 
\begin{align}\label{l4.1.3'}
	\phi_{n,\theta }(y_{1:n} )
	=
	-
	\frac{1}{n}
	\sum_{i=0}^{n-1} 
	\Phi_{\theta }(y_{i+1}, u_{i,\theta}(y_{1:i} ) ) 
\end{align}
for the same $\theta$, $y_{1:n}$, $n$. 
Differentiating (\ref{l4.1.3'}) (in $\theta$), we get
\begin{align}\label{l4.1.3''}
	\psi_{n,\theta }(y_{1:n} )
	=
	-
	\frac{1}{n}
	\sum_{i=0}^{n-1} 
	\Psi_{\theta }(y_{i+1}, u_{i,\theta}(y_{1:i} ), V_{i,\theta}(y_{1:i} ) ) 
\end{align}
for $\theta\in\Theta$, $y_{1:n} = (y_{1},\dots,y_{n} ) \in {\cal Y}^{n}$, 
$n\geq 0$. 

Let $U_{0}^{\theta } = u_{0,\theta }$, $U_{n}^{\theta } = u_{n,\theta }(Y_{1:n} )$ and  
$V_{0}^{\theta } = V_{0,\theta }$, 
$V_{n}^{\theta } = V_{n,\theta }(Y_{1:n} )$ for $\theta\in\Theta$, $n\geq 1$. 
Then, using \cite[Theorems 4.1, 4.2]{tadic&doucet} and (\ref{l4.1.701}) -- (\ref{l4.1.905}), 
we conclude that 
$\{(X_{n+1}, Y_{n+1}, U_{n}^{\theta }, V_{n}^{\theta } ) \}_{n\geq 0}$
is geometrically ergodic for each $\theta\in\Theta$. 
We also deduce that there exist functions 
$g:\Theta\rightarrow\mathbb{R}$, $h:\Theta\rightarrow\mathbb{R}^{d_{\theta} }$
and 
real numbers $\varepsilon_{Q}\in(0,1)$, 
$\tilde{C}_{2,Q}\in[1,\infty )$ (independent of $N$) 
such that 
\begin{align}\label{l4.1.5}
	\max\{
	|E(\Phi_{\theta }(Y_{n+1}, U_{n}^{\theta } ) ) - g(\theta ) |, 
	\|E(\Psi_{\theta }(Y_{n+1}, U_{n}^{\theta }, V_{n}^{\theta } ) ) - h(\theta ) \|
	\}
	\leq 
	\tilde{C}_{2,Q} \varepsilon_{Q}^{n}
\end{align}
for all $\theta\in Q$, $n\geq 0$. 
As a result of (\ref{l4.1.3'}) -- (\ref{l4.1.5}), we get 
\begin{align*}
	g(\theta ) 
	=
	\lim_{n\rightarrow\infty } 
	E(\phi_{n,\theta }(Y_{1:n} ) ), 
	\;\;\;\;\; 
	h(\theta )
	=
	\lim_{n\rightarrow\infty } 
	E(\psi_{n,\theta }(Y_{1:n} ) )
	=
	\lim_{n\rightarrow\infty } 
	\nabla_{\theta } 
	E(\phi_{n,\theta }(Y_{1:n} ) )
\end{align*}
for all $\theta\in\Theta$. 
Therefore, $g(\theta)=f(\theta)$, $h(\theta) = \nabla f(\theta)$
for all $\theta\in\Theta$
(notice that $E(\psi_{n,\theta}(Y_{1:n} ) )$ converges to $h(\theta )$
uniformly in $\theta$ on each compact subset of $\Theta$). 

In the rest of the proof, we assume that $\{X_{n} \}_{n\geq 0}$ is in steady-state
(i.e., $X_{0}$ is distributed according to the invariant distribution of $\{X_{n} \}_{n\geq 0}$). 
Then, we have 
\begin{align*}
	f_{N}(\theta ) 
	=
	E(\phi_{N,\theta }(Y_{1:N} ) ), 
	\;\;\;\;\; 
	\nabla f_{N}(\theta ) 
	=
	E(\psi_{N,\theta }(Y_{1:N} ) )
\end{align*}
for each $\theta\in\Theta$. 
Combining this with (\ref{l4.1.3''}), (\ref{l4.1.5}), we get 
\begin{align*}
	&
	\|\eta(\theta ) \|
	=
	\|\nabla f_{N}(\theta ) - \nabla f(\theta ) \|
	=
	\left\|
	\frac{1}{N} \sum_{i=0}^{N-1} 
	\left(
	E(\Psi_{\theta }(Y_{i+1}, U_{i}^{\theta }, V_{i}^{\theta } ) )
	-
	h(\theta )
	\right)
	\right\|
	\leq 
	\frac{\tilde{C}_{2,Q} }{N}
	\sum_{i=0}^{N-1} \varepsilon_{Q}^{i}
	\leq 
	\frac{\tilde{C}_{2,Q} }{(1 - \varepsilon_{Q} ) N }
\end{align*}
for all $\theta\in Q$. 
Then, it can easily be deduced that 
there exists $C_{Q}\in (0,\infty )$ (independent of $N$) such that 
$\|\eta(\theta ) \|\leq C_{Q}/N$ for all $\theta\in Q$. 
\end{proof}

\begin{vproof}{Theorem \ref{theorem4.1}}
{Due to Assumption \ref{a4.1}, $\{Z_{n} \}_{n\geq 0}$ is geometrically ergodic. 
Let $\nu(\cdot )$ be the invariant probability of $\{Z_{n} \}_{n\geq 0}$, 
while $\tilde{\Pi}^{n}(z,z')= \Pi^{n}(z,z') - \nu(z')$ 
for $z,z'\in {\cal X}^{N}\times {\cal Y}^{N}$, $n\geq 0$. 
Then, 
there exist real numbers $\rho\in(0,1)$, $\tilde{C}\in[1,\infty )$ such that 
$|\tilde{\Pi}^{n}(z,z')| \leq \tilde{C}\rho^{n}$ 
for each $z,z'\in {\cal X}^{N}\times {\cal Y}^{N}$, $n\geq 0$. 
On the other side, due to Assumption \ref{a4.3}, 
there exists a real number $\tilde{C}_{Q}\in[\tilde{C},\infty )$
such that 
\begin{align*}
	\|F(\theta,z) \|\leq \tilde{C}_{Q}, 
	\;\;\;\;\; 
	\|F(\theta',z) - F(\theta'',z) \| \leq \tilde{C}_{Q} \|\theta'-\theta'' \|
\end{align*}
for all $\theta,\theta',\theta''\in Q$, $z\in{\cal X}^{N}\times{\cal Y}^{N}$. 

For $\theta\in\Theta$, $z\in{\cal X}^{N}\times{\cal Y}^{N}$, $n\geq 0$, 
let 
\begin{align*}
	(\tilde{\Pi}^{n} F)(\theta,z)
	=
	\sum_{z'\in{\cal X}^{N}\times{\cal Y}^{N} }
	F(\theta, z' ) \tilde{\Pi}^{n}(z,z'), 
\end{align*}
while $\tilde{F}(\theta, z ) = \sum_{n=0}^{\infty } (\tilde{\Pi}^{n} F)(\theta,z)$. 
Then, we have 
\begin{align}
	&\label{t4.1.1}
	\|(\tilde{\Pi}^{n} F)(\theta,z) \|
	\leq 
	\tilde{C}_{Q}^{2}\rho^{n}, 
	\\
	&\label{t4.1.3}
	\|(\tilde{\Pi}^{n} F)(\theta',z) - (\tilde{\Pi}^{n} F)(\theta'',z) \|
	\leq 
	\tilde{C}_{Q}^{2}\rho^{n} \|\theta' - \theta'' \|
\end{align}
for all $\theta,\theta',\theta''\in Q$, $z\in{\cal X}^{N}\times{\cal Y}^{N}$, $n\geq 0$. 
Therefore, $\sum_{n=0}^{\infty } \|(\tilde{\Pi}^{n} F)(\theta,z) \| < \infty$
for any $\theta\in Q$, $z\in{\cal X}^{N}\times{\cal Y}^{N}$. 
As $Q$ is any compact set, we conclude that 
for each $\theta\in \Theta$, $z\in{\cal X}^{N}\times{\cal Y}^{N}$, 
$\tilde{F}(\theta,z)$ is well-defined and satisfies
$(\Pi\tilde{F} )(\theta, z ) = \sum_{n=1}^{\infty } (\tilde{\Pi}^{n} F)(\theta,z)$. 
Consequently, 
\begin{align*}
	\tilde{F}(\theta,z)
	-
	(\Pi\tilde{F} )(\theta, z )
	=
	(\tilde{\Pi}^{0} F)(\theta, z )
	=
	F(\theta,z)
	-
	\sum_{z'\in{\cal X}^{N}\times{\cal Y}^{N} }
	F(\theta,z') \nu(z') 
	=
	F(\theta,z) - \nabla f(\theta )
\end{align*}
for all $\theta\in\Theta$, $z\in{\cal X}^{N}\times{\cal Y}^{N}$.\footnote
{Notice that 
$\sum_{z'\in{\cal X}^{N}\times{\cal Y}^{N} } F(\theta,z') \nu(z') = \nabla f_{N}(\theta ) - \eta(\theta )$.} 
Thus, Assumption \ref{a2.2} holds. 

Owing to (\ref{t4.1.1}), we have 
\begin{align*}
	&
	\max\{\|\tilde{F}(\theta,z) \|, \|(\Pi\tilde{F} )(\theta,z ) \| \} 
	\leq 
	\sum_{n=0}^{\infty } \|(\tilde{\Pi}^{n} F(\theta,z) \|
	\leq
	\tilde{C}_{Q}^{2}(1-\rho )^{-1}, 
\end{align*}
for each $\theta\in Q$, $z\in{\cal X}^{N}\times{\cal Y}^{N}$. 
On the other side, (\ref{t4.1.3}) yields 
\begin{align*}
	&
	\|(\Pi\tilde{F} )(\theta',z) - (\Pi\tilde{F} )(\theta'',z) \|
	\leq 
	\sum_{n=1}^{\infty } 
	\|(\tilde{\Pi}^{n} F)(\theta',z) - (\tilde{\Pi}^{n} F)(\theta'',z) \|
	\leq
	\tilde{C}_{Q}^{2}(1-\rho )^{-1} \|\theta' - \theta'' \|
\end{align*}
for all $\theta',\theta''\in Q$, $z\in{\cal X}^{N}\times{\cal Y}^{N}$. 
Hence, Assumption \ref{a2.3} holds, too. 
On the other side, Lemma \ref{lemma4.1} yields 
\begin{align*}
	\eta
	=
	\limsup_{n\rightarrow \infty } 
	\|\eta_{n} \| 
	\leq 
	C_{Q}/N
\end{align*}
on $\Lambda_{Q}$
(notice that $C_{Q}$ does not depend on $N$). 
Then, the theorem's assertion directly follows from 
Theorem \ref{theorem2.1} and Parts (i), (iii), (iv) of Lemma \ref{lemma4.1}. 
}
\end{vproof}

\refstepcounter{appendixcounter}\label{appendix2}
\section*{Appendix \arabic{appendixcounter} }

In this section, a global version of Theorem \ref{theorem1.1} is presented. 
It is also demonstrated how Theorem \ref{theorem1.1} can be extended to
the randomly projected stochastic gradient search. 

First, the stability and the global asymptotic behavior of algorithm (\ref{1.1}) 
are considered. 
To analyze these properties, we introduce the following two assumptions. 

\begin{assumptionappendix}\label{aa2.1} 
$f(\cdot )$ is uniformly lower bounded
(i.e., $\inf_{\theta\in\mathbb{R}^{d_{\theta } } } f(\theta )>-\infty $), 
and
$\nabla f(\cdot )$ is (globally) Lipschitz continuous. 
Moreover, there exist real numbers $c\in(0,1)$, $\rho\in[1,\infty )$
such that $\|\nabla f(\theta ) \|\geq c$ for all $\theta\in\mathbb{R}^{d_{\theta } }$
satisfying $\|\theta \|\geq\rho$. 
\end{assumptionappendix}

\begin{assumptionappendix}\label{aa2.2} 
$\{\xi_{n} \}_{n\geq 0}$ admits the decomposition $\xi_{n}=\zeta_{n}+\eta_{n}$ for each $n\geq 0$, 
where  
$\{\zeta_{n} \}_{n\geq }$ and $\{\eta_{n} \}_{n\geq 0}$ are 
$\mathbb{R}^{d_{\theta } }$-valued stochastic processes 
satisfying 
\begin{align}\label{aa2.2.1}
	\lim_{n\rightarrow\infty } 
	g(\theta_{n} ) 
	\max_{n\leq j<a(n,t) } 
	\left\| 
	\sum_{i=n}^{j} \alpha_{i} \zeta_{i} 
	\right\|
	=
	0, 
	\;\;\;\;\; 
	\limsup_{n\rightarrow\infty } 
	g(\theta_{n} ) \|\eta_{n} \|
	<
	\infty
\end{align}
almost surely for any $t\in(0,\infty )$. 
In addition to this, there exists a real number $\delta\in(0,1)$ such that 
\begin{align}\label{aa2.2.3}
	\lim_{n\rightarrow\infty } 
	h(\theta_{n} ) \|\eta_{n} \|
	<
	\delta
\end{align}
almost surely. 
Here, $g, h: \mathbb{R}^{d_{\theta } }\rightarrow(0,\infty )$ are the (scaling) functions defined by 
\begin{align*}
	g(\theta ) 
	=
	(\|\nabla f(\theta ) \| + 1 )^{-1}, 
	\;\;\;\;\; 
	h(\theta )
	=
	\begin{cases}
	\|\nabla f(\theta ) \|^{-1}, &\text{ if } \|\theta \|\geq\rho
	\\
	0, &\text{ otherwise }
	\end{cases}
\end{align*}
for $\theta\in\mathbb{R}^{d_{\theta } }$
($\rho$ is specified in Assumption \ref{aa2.1}). 
\end{assumptionappendix}

Assumption \ref{aa2.1} is a stability condition. 
In this or a similar form, it is involved in the stability analysis of 
stochastic gradient search and stochastic approximation 
(see e.g., \cite{benveniste}, \cite{borkar}, \cite{chen3} and references cited therein). 
Assumption \ref{aa2.1} is restrictive, as it requires $\nabla^{2}f(\cdot )$ to be 
uniformly bounded. 
Assumption \ref{aa2.1} also requires $\nabla f(\cdot )$ to grow at most linearly as $\theta\rightarrow\infty$. 
Using the random projections (see (\ref{appendix2.1})), 
these restrictive conditions can considerably be relaxed. 

Assumption \ref{aa2.2} is a noise condition. 
It requires 
the effect of the gradient estimator's error $\{\xi_{n} \}_{n\geq 0}$
to be compensated by the gradient of the objective function $f(\cdot )$
(i.e., by the stability of the ODE $d\theta/dt=-\nabla f(\cdot )$). 
Assumption \ref{aa2.2} is true whenever (\ref{a1.2.1}) holds almost surely. 
It is also satisfied for stochastic gradient search with Markovian dynamics
(see Theorem \ref{theorema3.1}, Appendix \ref{appendix3}). 
Assumption \ref{aa2.2} and the results based on it (Theorem \ref{theorema2.1}, below) 
are motivated by the scaled ODE approach 
to the stability analysis of stochastic approximation \cite{borkar&meyn}.\footnote
{The main difference between \cite{borkar&meyn} and the results presented here 
is the choice of the scaling functions. 
The scaling adopted in \cite{borkar&meyn} is (asymptotically) proportional to 
$\|\theta\|$. 
In this paper, the scaling is (asymptotically) proportional to $\|\nabla f(\theta )\|$. } 

Our results on the stability and asymptotic bias of algorithm (\ref{1.1}) are provided in the next theorem. 

\begin{theoremappendix}\label{theorema2.1}
Suppose that Assumptions \ref{a1.1}, \ref{aa2.1} and \ref{aa2.2} hold. 
Then, the following is true: 
\begin{compactenum}[(i)]
\item
There exists a compact (deterministic) set $Q\subset\mathbb{R}^{d_{\theta } }$ 
such that $P(\Lambda_{Q} ) = 1$
($\Lambda_{Q}$ is specified in (\ref{1.103})). 
\item
There exists a (deterministic) non-decreasing function 
$\psi:[0,\infty )\rightarrow[0,\infty)$ 
(independent of  $\eta$ and depending only on $f(\cdot )$)
such that 
$\lim_{t\rightarrow 0} \psi(t) = \psi(0) = 0$ 
and 
\begin{align*}
	\limsup_{n\rightarrow\infty } d(\theta_{n}, {\cal R} )
	\leq 
	\psi(\eta )
\end{align*}
almost surely. 
\item
If $f(\cdot )$ satisfies Assumption \ref{a1.3.b}, 
there exists a real number $K\in (0,\infty )$
(independent of $\eta$ and depending only on $f(\cdot )$)
such that 
\begin{align*}
	\limsup_{n\rightarrow\infty } \|\nabla f(\theta_{n} ) \| 
	\leq 
	K \eta^{q/2}, 
	\;\;\;\;\; 
	\limsup_{n\rightarrow\infty } f(\theta_{n} ) 
	-
	\liminf_{n\rightarrow\infty } f(\theta_{n} ) 
	\leq 
	K \eta^{q}
\end{align*}
almost surely ($q$ is specified in the statement of Theorem \ref{theorem1.1}). 
\item
If $f(\cdot )$ satisfies Assumption \ref{a1.3.c}, 
there exist real numbers $r\in (0,1)$, $L\in (0,\infty )$
(independent of $\eta$ and depending only on $f(\cdot )$)
such that 
\begin{align*}
	\limsup_{n\rightarrow\infty } \|\nabla f(\theta_{n} ) \| 
	\leq 
	L \eta^{1/2}, 
	\;\;\;\;\; 
	\limsup_{n\rightarrow\infty } d(f(\theta_{n} ), f({\cal S} ) ) 
	\leq 
	L \eta, 
	\;\;\;\;\; 
	\limsup_{n\rightarrow\infty } d(\theta_{n}, {\cal S} )
	\leq 
	L \eta^{r}
\end{align*}
almost surely. 
\end{compactenum}
\end{theoremappendix}

\begin{sproof}
Owing to Assumption \ref{aa2.1}, there exists a real number 
$\tilde{C}_{1} \in [1,\infty )$ 
such that the following is true: 
(i) $f(\theta )> -\tilde{C}_{1}$ for all $\theta\in\mathbb{R}^{d_{\theta } }$, 
and 
(ii) $f(\theta )\leq \tilde{C}_{1}$ 
for any $\theta\in\mathbb{R}^{d_{\theta } }$ satisfying $\|\theta \|\leq \rho+1$. 
Moreover, due to Assumption \ref{aa2.2}, 
there also exists an event $N_{0}\in {\cal F}$ with the following properties: 
(i) $P(N_{0} ) = 0$,
and 
(ii) (\ref{aa2.2.1}), (\ref{aa2.2.3}) hold on $N_{0}^{c}$ for all $t\in(0,\infty )$. 

Let $\varepsilon = (1-\delta )/6$, 
$T=2\tilde{C}_{1}\varepsilon^{-1}c^{-2}$ and let $\phi:[0,\infty )\rightarrow[0,\infty )$ be the function defined by 
\begin{align*}
	\phi(z)
	=
	\sup\{\|\nabla f(\theta ) \|: \theta\in\mathbb{R}^{d_{\theta } }, \|\theta\|\leq z \}
\end{align*}
for $z\in [0,\infty )$. 
As $\nabla f(\cdot )$ is locally Lipschitz continuous, 
$\phi(\cdot )$ is locally Lipschitz continuous, too. 
$\phi(\cdot )$ is also non-negative and satisfies $\|\nabla f(\theta ) \|\leq \phi(\|\theta \| )$
for all $\theta\in\mathbb{R}^{d_{\theta } }$. 

For $z\in [0,\infty )$, let $\lambda(\cdot\:; z )$ be the solution to 
the ODE $dz/dt=2\phi(z)$ satisfying $\lambda(0;z)=z$. 
As $2\phi(\cdot )$ is non-negative and locally Lipschitz continuous, 
$\lambda(\cdot\; ;\cdot )$ is well-defined and locally Lipschitz continuous 
(in both arguments) on $[0,\infty )\times [0,\infty )$. 
We also have 
\begin{align}\label{ta2.1.701}
	\lambda(t;z)
	=
	z
	+
	2\int_{0}^{t} \phi(\lambda(s;z) ) ds
\end{align}
for all $t,z\in[0,\infty )$. 
Then, there exists $\rho_{1}\in [1,\infty )$ such that 
$\rho_{1}\geq\rho+1$ and such that 
$|\lambda(t;z) |\leq\rho_{1}$ for all $t\in [0,T]$, $z\in [0,\rho+1]$. 

Let $\rho_{2}=\rho_{1}+1$, $Q=\{\theta\in\mathbb{R}^{d_{\theta } }: \|\theta \|\leq \rho_{2} \}$, 
while $\Lambda$ is the event defined by 
\begin{align*}
	\Lambda
	=
	\limsup_{n\rightarrow\infty } \{\|\theta_{n} \|<\rho \} 
	=
	\bigcap_{m=0}^{\infty } \bigcup_{n=m}^{\infty } \{\|\theta_{n} \|<\rho \}. 
\end{align*}
Let also $\tilde{C}_{2}\in [1,\infty )$ stand for a (global) Lipschitz constant of $\nabla f(\cdot )$
and for an upper bound of $\|\nabla f(\cdot )\|$ on $Q$. Finally, let  
$\tilde{C}_{3} = 2\tilde{C}_{2} \exp(2\tilde{C}_{2} )$, 
$\tilde{C}_{4} = 12\tilde{C}_{1}\tilde{C}_{2}\tilde{C}_{3}$, 
while $\tau=4^{-1}\tilde{C}_{4}^{-1} \varepsilon c^{2}$.  

In order to prove the theorem's assertion, it is sufficient to show $N_{0}^{c}\subseteq\Lambda$
(i.e., to establish that on $N_{0}^{c}$, $\|\theta_{n}\|\leq \rho_{2}$ for all, 
but finitely many $n$).\footnote
{Assumption \ref{a1.2} is a consequence of Assumption \ref{aa2.2},  
and therefore, 
Parts (ii) -- (iv) directly follow from Part (i) 
and Theorem \ref{theorem1.1}. } 
To prove this, we use contradiction. 
We assume that 
$\|\theta_{n} \|>\rho_{2}$ for infinitely many $n$ and some $\omega\in N_{0}^{c}$. 
Notice that all formulas which follow in the proof correspond to $\omega$.

Owing to (\ref{aa2.2.1}), (\ref{aa2.2.3}), there exists an integer $k_{1}\geq 0$ (depending on $\omega$)
such that 
\begin{align}\label{ta2.1.1}
	g(\theta_{n} ) 
	\max_{n\leq j<a(n,T) } 
	\left\|
	\sum_{i=n}^{j} \alpha_{i} \zeta_{i} 
	\right\| 
	\leq 
	\tau^{2}, 
	\;\;\;\;\; 
	h(\theta_{n} ) \|\eta_{n} \|
	\leq
	\delta
\end{align}
for $n\geq k_{1}$. 
Due to Assumption \ref{a1.1} and (\ref{aa2.2.1}), we also have 
\begin{align}\label{ta2.1.3051}
	\lim_{n\rightarrow\infty } g(\theta_{n} ) \|\alpha_{n}\zeta_{n} \|
	=
	\lim_{n\rightarrow\infty } g(\theta_{n} ) \|\alpha_{n}\eta_{n} \|
	=
	0. 
\end{align}
Since 
\begin{align*}
	&
	g(\theta_{n} ) \|\theta_{n+1} - \theta_{n} \|
	\leq
	\alpha_{n}
	+
	g(\theta_{n} ) \|\alpha_{n}\zeta_{n} \|
	+
	g(\theta_{n} ) \|\alpha_{n}\eta_{n} \|
\end{align*}
for $n\geq 0$, Assumption \ref{a1.1} and (\ref{ta2.1.3051}) imply 
$\lim_{n\rightarrow\infty } g(\theta_{n} ) \|\theta_{n+1} - \theta_{n} \| = 0$. 
Then, (\ref{1.1501}) implies that there exists an integer $k_{2}\geq 0$ (depending on $\omega$)
such that 
\begin{align}\label{ta2.1.3001}
	\sum_{i=n}^{a(n,\tau)-1} \alpha_{i}
	\geq
	(1-\varepsilon ) \tau,
	\;\;\;\;\; 
	g(\theta_{n} ) \|\theta_{n+1} - \theta_{n} \|
	\leq
	\tau
\end{align}
for $n\geq k_{2}$. 

Let $k_{0}=\max\{k_{1}, k_{2} \}$. 
Moreover, let $l_{0}, m_{0}, n_{0}$ be the integers defined as follows. 
If $\omega\in\Lambda$
(i.e., if $\|\theta_{n} \|<\rho$ for infinitely many $n$), let 
\begin{align}\label{ta2.1.3005}
	l_{0}
	=
	\min\{n>k_{0}: \|\theta_{n-1} \|<\rho \}, 
	\;\;\;\: 
	m_{0}
	=
	\min\{n>l_{0}: 
	\|\theta_{n} \|>\rho_{2} \}, 
	\;\;\;\;\: 
	n_{0}
	=
	\max\{n\leq m_{0}: \|\theta_{n-1} \|<\rho \}. 
\end{align}
Otherwise, if $\omega\in\Lambda^{c}$
(i.e., if $\|\theta_{n} \|<\rho$ for finitely many $n$), let 
\begin{align*}
	l_{0}
	=
	\max\{n>0: \|\theta_{n-1} \|<\rho \}, 
	\;\;\;\;\; 
	m_{0}
	=
	\infty,
	\;\;\;\;\; 
	n_{0}
	=
	\max\{k_{0}, l_{0} \}. 
\end{align*}
Then, we have $k_{0}<n_{0}\leq m_{0}$
and $\|\theta_{n} \|\geq\rho$ for $n_{0}\leq n<m_{0}$. 

Let $\phi_{n}(\tau),\phi_{1,n}(\tau),\phi_{2,n}(\tau)$ 
have the same meaning as in Section \ref{section1.bc*}. 
Now, the asymptotic properties of $\phi_{n}(\tau )$ are analyzed. 
As $\|\theta_{n} \|\geq\rho$ for $n_{0}\leq n<m_{0}$, (\ref{ta2.1.1}) implies 
\begin{align}\label{ta2.1.51}
	\left\|
	\sum_{i=n}^{j} \alpha_{i} \xi_{i}
	\right\|
	\leq 
	\left\|
	\sum_{i=n}^{j} \alpha_{i} \zeta_{i}
	\right\|
	+
	\sum_{i=n}^{j} \alpha_{i} \|\eta_{i} \|
	\leq 
	\tau^{2} g^{-1}(\theta_{n} ) 
	+
	\delta
	\sum_{i=n}^{j} 
	\alpha_{i} \|\nabla f(\theta_{i} ) \|
\end{align}
for $n_{0}\leq n\leq j< \min\{m_{0}, a(n,T)\}$
(notice that $\|\eta_{i} \|\leq\delta \|\nabla f(\theta_{i} ) \|$
when $\|\theta_{i} \|\geq\rho$). 
Therefore, 
\begin{align*}
	\|\nabla f(\theta_{j} ) \|
	\leq &
	\|\nabla f(\theta_{n} ) \|
	+
	\|\nabla f(\theta_{j} ) - \nabla f(\theta_{n} ) \|
	\nonumber\\
	\leq &
	\|\nabla f(\theta_{n} ) \| 
	+
	\tilde{C}_{2} \|\theta_{j} - \theta_{n} \| 
	\nonumber\\
	\leq &
	\|\nabla f(\theta_{n} ) \|
	+
	\tilde{C}_{2} \sum_{i=n}^{j-1} \alpha_{i} \|\nabla f(\theta_{i} ) \| 
	+
	\tilde{C}_{2} 
	\left\|
	\sum_{i=n}^{j-1} \alpha_{i} \xi_{i} 
	\right\| 
	\nonumber\\
	\leq &
	\|\nabla f(\theta_{n} ) \| 
	+ 
	\tilde{C}_{2} \tau^{2} g^{-1}(\theta_{n} ) 
	+
	2\tilde{C}_{2} \sum_{i=n}^{j-1} \alpha_{i} \|\nabla f(\theta_{i} ) \| 
\end{align*}
for $n_{0}\leq n< j \leq \min\{m_{0}-1, a(n,\tau) \}$.\footnote
{Notice that $\tau$, $T$ are defined as $\tau=4^{-1}\tilde{C}_{4}^{-1} \varepsilon c^{2}$, 
$T=2\tilde{C}_{1}\varepsilon^{-1}c^{-2}$. 
Notice also $\tau<1<T$ 
since $\tilde{C}_{1}, \tilde{C}_{4} \in [1,\infty)$, $\varepsilon, c \in (0,1)$. }
Combining this with Bellman-Gronwall inequality
(see e.g. \cite[Appendix B]{borkar}), we deduce
\begin{align*}
	\|\nabla f(\theta_{j} ) \|
	\leq &
	\left(\|\nabla f(\theta_{n} ) \| + \tilde{C}_{2}\tau^{2} g^{-1}(\theta_{n} ) \right) 
	\exp\left(
	2\tilde{C}_{2} \sum_{i=n}^{j-1} \alpha_{i}  
	\right) 
	\nonumber\\
	\leq &
	\left(\|\nabla f(\theta_{n} ) \| + \tilde{C}_{2}\tau^{2} g^{-1}(\theta_{n} ) \right) 
	(1 + \tilde{C}_{3}\tau )
	\nonumber \\
	\leq &
	\|\nabla f(\theta_{n} ) \| 
	+
	(\tilde{C}_{3}\tau + \tilde{C}_{2}\tau^{2} + \tilde{C}_{2}\tilde{C}_{3}\tau^{3} )
	g^{-1}(\theta_{n} ) 
	\nonumber\\
	\leq &
	\|\nabla f(\theta_{n} ) \| 
	+
	\tilde{C}_{4}\tau g^{-1}(\theta_{n} ) 
\end{align*}
for $n_{0} \leq n\leq j \leq \min\{m_{0}-1, a(n,\tau) \}$.\footnote
{Notice that $\sum_{i=n}^{j-1} \alpha_{i} \leq \tau<1$
when $n \leq j \leq a(n,\tau)$. 
Notice also  
$g^{-1}(\theta_{n} ) > \|\nabla f(\theta_{n} ) \|$ 
and
$\exp(2\tilde{C}_{2}\tau )\leq 2\tilde{C}_{2}\tau\exp(2\tilde{C}_{2} )= \tilde{C}_{3}\tau$.} 
Then, (\ref{ta2.1.51}) implies 
\begin{align}\label{ta2.1.3007}
	\left\|
	\sum_{i=n}^{j} \alpha_{i} \xi_{i} 
	\right\|
	\leq 
	\tau^{2} g^{-1}(\theta_{n} ) 
	+
	\delta
	\left(
	\|\nabla f(\theta_{n} ) \|
	+
	\tilde{C}_{4}\tau g^{-1}(\theta_{n} ) 
	\right)
	\sum_{i=n}^{j} \alpha_{i} 
	\leq 
	\delta\tau \|\nabla f(\theta_{n} ) \|
	+
	2\tilde{C}_{4}\tau^{2} g^{-1}(\theta_{n} ) 
\end{align}
for $n_{0}\leq n\leq j< \min\{m_{0}, a(n,\tau ) \}$. 
Consequently, 
\begin{align}\label{ta2.1.5}
	\|\theta_{j} - \theta_{n} \|
	\leq &
	\sum_{i=n}^{j-1} \alpha_{i} \|\nabla f(\theta_{i} ) \| 
	+ 
	\left\|
	\sum_{i=n}^{j-1} \alpha_{i} \xi_{i} 
	\right\| 
	\nonumber	\\
	\leq &
	\left(
	\|\nabla f(\theta_{n} ) \|
	+
	\tilde{C}_{4}\tau g^{-1}(\theta_{n} ) 
	\right)
	\left(
	\sum_{i=n}^{j-1} \alpha_{i} 
	+
	\delta\tau
	\right)
	+
	2\tilde{C}_{4}\tau^{2} g^{-1}(\theta_{n} )
	\nonumber	\\
	\leq &
	3 \tau g^{-1}(\theta_{n} )
\end{align}
for $n_{0} \leq n\leq j \leq \min\{m_{0}-1, a(n,\tau) \}$
(notice that $\delta<1$, $\tilde{C}_{4}\tau\leq 1/4$). 
Therefore, 
\begin{align*}
	|\phi_{1,n}(\tau ) |
	\leq &
	\tilde{C}_{2} 
	\|\nabla f(\theta_{n} ) \|
	\sum_{i=n}^{a(n,\tau)-1} \alpha_{i} \|\theta_{i} - \theta_{n} \|
	\leq 
	3\tilde{C}_{2} \tau 
	g^{-1}(\theta_{n} )
	\|\nabla f(\theta_{n} ) \|
	\sum_{i=n}^{a(n,\tau)-1} \alpha_{i} 
	\leq 
	3\tilde{C}_{2} \tau^{2} 
	g^{-2}(\theta_{n} )
\end{align*}
for $n\geq n_{0}$ satisfying $a(n,\tau )<m_{0}$. 
We also have 
\begin{align*}
	|\phi_{2,n}(\tau ) |
	\leq &
	\tilde{C}_{2} \|\theta_{a(n,\tau)} - \theta_{n} \|^{2}
	\leq 
	9\tilde{C}_{2} \tau^{2} 
	g^{-2}(\theta_{n} )
\end{align*}
for $n\geq n_{0}$ satisfying $a(n,\tau )<m_{0}$.  
Thus, 
\begin{align}\label{ta2.1.307}
	|\phi_{n}(\tau ) |
	\leq &
	\tilde{C}_{4} \tau^{2} 
	g^{-2}(\theta_{n} )
\end{align}
when $n\geq n_{0}$, $a(n,\tau )<m_{0}$. 
Additionally, as a result of (\ref{ta2.1.3001}), (\ref{ta2.1.3007}), 
we get
\begin{align*}
	\|\nabla f(\theta_{n} ) \|
	\sum_{i=n}^{a(n,\tau)-1} \alpha_{i} 
	-
	\left\|
	\sum_{i=n}^{a(n,\tau)-1} \alpha_{i}\xi_{i}
	\right\|
	\geq &
	(1-\delta-\varepsilon)\tau 
	\|\nabla f(\theta_{n} ) \| 
	-
	2\tilde{C}_{4}\tau^{2} g^{-1}(\theta_{n} )
	\nonumber\\
	= &
	5\varepsilon\tau \|\nabla f(\theta_{n} ) \| 
	-
	2\tilde{C}_{4}\tau^{2} g^{-1}(\theta_{n} )
	\nonumber\\
	\geq &
	3\varepsilon\tau \|\nabla f(\theta_{n} ) \|
\end{align*}
when $n\geq n_{0}$,  $a(n,\tau )<m_{0}$.\footnote
{Notice that $1-\delta=6\varepsilon$, 
$\varepsilon\geq\varepsilon c\geq 2\tilde{C}_{4}\tau$. 
Notice also that 
$2\varepsilon\tau \|\nabla f(\theta_{n} ) \| \geq 
\varepsilon\tau \|\nabla f(\theta_{n} ) \| + \varepsilon\tau c \geq 
2\tilde{C}_{4}\tau^{2} g^{-1}(\theta_{n} )$ 
for $n_{0}\leq n < m_{0}$. }
Then, (\ref{1.1*}), (\ref{ta2.1.307}) imply 
\begin{align}\label{ta2.1.7}
	f(\theta_{a(n,\tau) } ) - f(\theta_{n} ) 
	\leq&
	-
	3\varepsilon\tau\|\nabla f(\theta_{n} ) \|^{2}
	+
	\tilde{C}_{4}\tau^{2} g^{-2}(\theta_{n} ) 
	\leq 
	-\varepsilon\tau \|\nabla f(\theta_{n} ) \|^{2} 
	\leq 
	-
	\varepsilon\tau c^{2}
\end{align}
for $n\geq n_{0}$ satisfying $a(n,\tau )<m_{0}$.\footnote{
Notice that 
$2\varepsilon\|\nabla f(\theta_{n} ) \|^{2} \geq
\varepsilon \|\nabla f(\theta_{n} ) \|^{2} + \varepsilon c^{2} \geq 
\tilde{C}_{4}\tau g^{-2}(\theta_{n} )$ for $n_{0}\leq n<m_{0}$. }

Let $\{n_{k} \}_{k\geq 0}$ be the sequence recursively defined by 
$n_{k+1}=a(n_{k},\tau )$ for $k\geq 0$. 
Now, we show by contradiction $\omega\in\Lambda$ 
(i.e., $\|\theta_{n} \|<\rho$ for infinitely many $n$). 
We assume the opposite. 
Then, $m_{0}=\infty$ and $\|\theta_{n} \|\geq\rho$ for $n\geq n_{0}$, 
while (\ref{ta2.1.7}) implies 
$f(\theta_{n_{k+1} } ) - f(\theta_{n_{k} } ) \leq -\varepsilon\tau c^{2}$ 
for $k\geq 0$. 
Hence, $\lim_{k\rightarrow\infty } f(\theta_{n_{k} } ) = -\infty$.  
However, this is impossible due to Assumption \ref{aa2.1}. 
Thus, $\omega\in\Lambda$ (i.e., $\|\theta_{n} \|<\rho$ for infinitely many $n$). 
Therefore, $m_{0}, n_{0}$ are defined through (\ref{ta2.1.3005}), 
and thus, $\|\theta_{n_{0}-1} \|<\rho$, $\|\theta_{m_{0} } \|>\rho_{2}$. 
Combining this with (\ref{ta2.1.3001}), we conclude
\begin{align*}
	&
	\|\theta_{n_{0} } - \theta_{n_{0}-1} \| 
	\leq 
	\tau g^{-1}(\theta_{n_{0}-1} ) 
	\leq 
	\tau(\tilde{C}_{2} + 1 ) 
	\leq 
	1/2
\end{align*}
(notice that $\|\nabla f(\theta_{n_{0}-1} ) \|\leq \tilde{C}_{2}$, 
$\tilde{C}_{2}\tau\leq 1/4$). 
Consequently, 
\begin{align}\label{ta2.1.3009}
	&
	\|\theta_{n_{0} } \|
	\leq 
	\|\theta_{n_{0}-1} \|
	+
	\|\theta_{n_{0} } - \theta_{n_{0}-1} \|
	\leq \rho+1/2
	<\rho_{2}. 
\end{align}
Hence, $n_{0}<m_{0}$, 
$f(\theta_{n_{0} } ) \leq \tilde{C}_{1}$. 

Let $i_{0}, j_{0}$ be the integers defined by 
$j_{0} = \max\{j\geq 0: n_{j}< m_{0} \}$, 
$i_{0} = n_{j_{0} }$. 
Then, we have $n_{0}\leq i_{0}=n_{j_{0} }< n_{j_{0}+1} = m_{0}\leq a(i_{0},\tau )$. 
As a result of this and (\ref{ta2.1.5}), we get
\begin{align*}
	&
	\|\theta_{m_{0} } - \theta_{i_{0} } \| 
	\leq 
	3 \tau g^{-1}(\theta_{i_{0} } ) 
	\leq 
	3 \tau (\tilde{C}_{2} + 1 )
	\leq 
	1/2  
\end{align*}
(notice that $\|\nabla f(\theta_{i_{0} } ) \|\leq \tilde{C}_{2}$, 
$\tilde{C}_{2}\tau\leq 1/12$). 
Consequently, 
\begin{align}
	&\label{ta2.1.703}
	\|\theta_{i_{0} } \|
	\geq 
	\|\theta_{m_{0} } \| 
	-
	\|\theta_{m_{0} } - \theta_{i_{0} } \|
	\geq 
	\rho_{2} - 1/2
	>
	\rho_{1}. 
\end{align}

Let $\{\gamma_{n} \}_{n\geq 0}$, $\theta_{0}(\cdot )$ 
have the same meaning as in Section \ref{section1.a*}. 
Now, we show by contradiction that $\gamma_{i_{0} } - \gamma_{n_{0} }\geq T$. 
We assume the opposite. 
Then, (\ref{ta2.1.51}), (\ref{ta2.1.3009}) imply 
\begin{align}\label{ta2.1.705}
	\|\theta_{0}(t) \|
	=
	\|\theta_{j} \|
	\leq&
	\|\theta_{n_{0} } \| 
	+
	\sum_{i=n_{0} }^{j-1} \alpha_{i} \|\nabla f(\theta_{i} ) \| 
	+
	\left\|
	\sum_{i=n_{0} }^{j-1} \alpha_{i} \xi_{i} 
	\right\|
	\nonumber\\
	\leq&
	\|\theta_{n_{0} } \| 
	+
	\tau^{2} g^{-1}(\theta_{n_{0} } )
	+
	2\sum_{i=n_{0} }^{j-1} \alpha_{i} \|\nabla f(\theta_{i} ) \|   
	\nonumber\\
	\leq&
	\rho+1
	+
	2\sum_{i=n_{0} }^{j-1} \alpha_{i} \phi(\|\theta_{i} \| )  
	\nonumber\\
	\leq&
	\rho+1
	+
	2\int_{\gamma_{n_{0} } }^{t} \phi(\|\theta_{0}(s) \| ) ds
\end{align}
for $t\in [\gamma_{j}, \gamma_{j+1} )$, $n_{0}\leq j\leq i_{0}$.\footnote
{As $j\leq i_{0} < m_{0}$, 
we have  
$\gamma_{j}-\gamma_{n_{0} }\leq \gamma_{i_{0} }-\gamma_{n_{0} } \leq T$
and $j\leq\min\{m_{0}-1,a(n_{0},T)\}$. 
We also have $\tau^{2} g^{-1}(\theta_{n_{0} } ) \leq \tau^{2} (\tilde{C}_{2} + 1 ) \leq 1/2$.}
Applying the comparison principle (see \cite[Section 3.4]{khalil}) to
(\ref{ta2.1.701}), (\ref{ta2.1.705}), we conclude 
$\|\theta_{0}(t) \| \leq 
\lambda(t-\gamma_{n_{0} };\rho+1) \leq \rho_{1}$
for all $t\in[\gamma_{{n}_{0} }, \gamma_{i_{0} } ]$. 
Thus, $\|\theta_{i_{0} } \|=\|\theta_{0}(\gamma_{i_{0} } ) \|\leq \rho_{1}$. 
However, this is impossible, due to (\ref{ta2.1.703}). 
Hence, $\gamma_{i_{0} }-\gamma_{n_{0} }\geq T$. 
Consequently, 
\begin{align}\label{ta2.1.707}
	T
	\leq 
	\gamma_{i_{0} } - \gamma_{n_{0} }
	=
	\sum_{j=0}^{j_{0}-1} (\gamma_{n_{j+1} } - \gamma_{n_{j} } )
	\leq 
	j_{0}\tau 
\end{align}
(notice that $n_{j_{0} } = i_{0}$, 
$\gamma_{n_{j+1} }-\gamma_{n_{j} } = \sum_{i=n_{j} }^{n_{j+1}-1} \alpha_{i} \leq \tau$). 

Owing to (\ref{ta2.1.7}), we have 
$f(\theta_{n_{j+1} } ) - f(\theta_{n_{j} } ) \leq -\varepsilon\tau c^{2}$
for $0\leq j \leq j_{0}$. 
Combining this with (\ref{ta2.1.707}), we get 
\begin{align*}
	f(\theta_{i_{0} } )
	=
	f(\theta_{n_{j_{0} } } ) 
	\leq 
	f(\theta_{n_{0} } ) - j_{0}\varepsilon\tau c^{2}
	\leq 
	\tilde{C}_{1} - \varepsilon c^{2} T
	\leq
	-\tilde{C}_{1}. 
\end{align*}
However, this is impossible, since $f(\theta ) > -\tilde{C}_{1}$ for all 
$\theta\in\mathbb{R}^{d_{\theta } }$. 
Hence, $\|\theta_{n} \|>\rho_{2}$ for finitely many $n$. 
\end{sproof}

In the rest of the section, Theorem \ref{theorem1.1} is extended to 
randomly projected stochastic gradient algorithms. 
These algorithms are defined by the following difference equations: 
\begin{align}\label{appendix2.1}
	&
	\vartheta_{n}
	=
	\theta_{n}
	-
	\alpha_{n} (\nabla f(\theta_{n} ) + \xi_{n} ), 
	\nonumber\\
	&
	\theta_{n+1}
	=
	\vartheta_{n} 
	I_{ \{\|\vartheta_{n} \| \leq \beta_{\sigma_{n} } \} }
	+
	\theta_{0} 
	I_{ \{\|\vartheta_{n} \| > \beta_{\sigma_{n} } \} }, 
	\nonumber\\
	&
	\sigma_{n+1}
	=
	\sigma_{n}
	+
	I_{ \{\|\vartheta_{n} \| > \beta_{\sigma_{n} } \} }, 
	\;\;\;\;\; 
	n\geq 0. 
\end{align}
Here, $\nabla f(\cdot )$, $\{\alpha_{n} \}_{n\geq 0}$, $\{\xi_{n} \}_{n\geq 0}$ 
have the same meaning as in Section \ref{section1}, 
while $\{\beta_{n} \}_{n\geq 0}$ is an increasing sequence of positive real numbers satisfying
$\lim_{n\rightarrow\infty } \beta_{n} = \infty$. 
$\theta_{0}\in\mathbb{R}^{d}$ is a (deterministic) vector satisfying $\|\theta_{0} \|\leq \beta_{0}$, 
while $\sigma_{0}=0$. 
For further details on randomly projected stochastic gradient search and 
stochastic approximation, see \cite{chen3}, \cite{tadic1} and references cited therein. 

To study the asymptotic behavior of (\ref{appendix2.1}), we introduce the following two assumption. 

\begin{assumptionappendix}\label{aa2.3} 
$f(\cdot )$ is uniformly lower bounded 
(i.e., $\inf_{\theta\in \mathbb{R}^{d} } f(\theta ) > -\infty$)
and $\nabla f(\cdot )$ is locally Lipschitz continuous. 
Moreover, there exist real numbers $c\in(0,1)$, $\rho\in[1,\infty )$
such that $\|\nabla f(\theta ) \|\geq c$ for all $\theta\in\mathbb{R}^{d_{\theta } }$
satisfying $\|\theta \|\geq\rho$. 
\end{assumptionappendix}

\begin{assumptionappendix}\label{aa2.4} 
$\{\xi_{n} \}_{n\geq 0}$ admits the decomposition $\xi_{n}=\zeta_{n}+\eta_{n}$ for each $n\geq 0$, 
where  
$\{\zeta_{n} \}_{n\geq }$ and $\{\eta_{n} \}_{n\geq 0}$ are 
$\mathbb{R}^{d_{\theta } }$-valued stochastic processes 
satisfying 
\begin{align}\label{aa2.4.1}
	\lim_{n\rightarrow\infty } 
	\max_{n\leq j<a(n,t) } 
	\left\| 
	\sum_{i=n}^{j} \alpha_{i} \zeta_{i} 
	\right\|
	I_{ \{\tau_{Q,n}>j \} }
	=
	0, 
	\;\;\;\;\; 
	\limsup_{n\rightarrow\infty } 
	\|\eta_{n} \| I_{ \{\theta_{n}\in Q \} }
	<
	\infty
\end{align}
almost surely for all $t\in(0,\infty )$ and any compact set $Q\subset\mathbb{R}^{d_{\theta } }$. 
In addition to this, there exists a real number $\delta_{Q}\in(0,1)$ such that 
\begin{align}\label{aa2.4.3}
	\lim_{n\rightarrow\infty } 
	h(\theta_{n} ) \|\eta_{n} \| I_{ \{\theta_{n}\in Q \} }
	<
	\delta_{Q}
\end{align}
almost surely for any compact set $Q\subset\mathbb{R}^{d_{\theta} }$. 
Here, $\tau_{Q,n}$ and $h(\cdot )$ are (respectively) the stopping time and the (scaling) function defined by 
\begin{align*}
	\tau_{Q,n}
	=
	\inf\left(\{j\geq n: \theta_{j}\neq\vartheta_{j-1} \text{ or } \theta_{j}\not\in Q \} \cup \{\infty \} \right), 
	\;\;\;\;\; 
	h(\theta )
	=
	\begin{cases}
	\|\nabla f(\theta ) \|^{-1}, &\text{ if } \|\theta \|\geq\rho, 
	\\
	0, &\text{ otherwise }
	\end{cases}
\end{align*}
for $n\geq 0$, $\theta\in\mathbb{R}^{d_{\theta } }$
($\rho$ is specified in Assumption \ref{aa2.3}). 
\end{assumptionappendix}

Assumption \ref{aa2.3} is a stability condition. 
It is one of the weakest conditions under which the stability of the ODE
$d\theta/dt=-\nabla f(\theta )$ can be demonstrated.  
On the other side, Assumption \ref{aa2.4} is a noise condition. 
It can be considered as a version of the noise conditions adopted in \cite{tadic1}. 

Our results on the asymptotic behavior of algorithm (\ref{appendix2.1}) are provided in the next theorem. 

\begin{theoremappendix}\label{theorema2.2}
Let $\{\theta_{n} \}_{n\geq 0}$ be generated by recursion (\ref{appendix2.1}). 
Suppose that Assumptions \ref{a1.1}, \ref{aa2.3} and \ref{aa2.4} hold. 
Then, all conclusions of Theorem \ref{theorema2.1} are true. 
\end{theoremappendix}

\begin{sproof}
Due to Assumption \ref{aa2.3}, 
there exists a real number 
$\tilde{C}_{1} \in [1,\infty )$ 
such that the following is true: 
(i) $f(\theta )> -\tilde{C}_{1}$ for all $\theta\in\mathbb{R}^{d_{\theta } }$, 
and 
(ii) $f(\theta )\leq \tilde{C}_{1}$
for any $\theta\in\mathbb{R}^{d_{\theta } }$ satisfying $\|\theta \|\leq \rho+1$. 
Without loss of generality, it can also be assumed $\|\theta_{0} \|<\rho$. 
On the other side, owing to Assumption \ref{aa2.4}, 
there exists an event $N_{0}\in {\cal F}$ with the following properties: 
(i) $P(N_{0} ) = 0$,
and 
(ii) (\ref{aa2.4.1}), (\ref{aa2.4.3}) hold on $N_{0}^{c}$ for all $t\in(0,\infty )$
and any compact set $Q\subset\mathbb{R}^{d_{\theta } }$. 

Let $\varepsilon = (1-\delta )/5$, 
$T=2\tilde{C}_{1}\varepsilon^{-1}c^{-2}$, 
while $\phi(\cdot )$, $\lambda(\cdot\:;\cdot )$ 
have the same meaning as in the proof of Theorem \ref{theorema2.1}. 
Then, there exists $\rho_{1}\in [1,\infty )$ such that 
$\rho_{1}\geq\rho+1$ and such that 
$|\lambda(t;z) |\leq\rho_{1}$ for all $t\in [0,T]$, $z\in [0,\rho+1]$. 
Moreover, (\ref{ta2.1.701}) holds for all $t,z\in [0,\infty )$. 

Let $\rho_{2}=\rho_{1}+1$, $Q=\{\theta\in\mathbb{R}^{d_{\theta } }: \|\theta \|\leq \rho_{2} \}$. 
Moreover, let
\begin{align*}
	\sigma
	=
	\lim_{n\rightarrow\infty } \sigma_{n}, 
	\;\;\;\;\; 
	\Lambda
	=
	\limsup_{n\rightarrow\infty } \{\|\theta_{n} \|<\rho \} 
	=
	\bigcap_{m=0}^{\infty } \bigcup_{n=m}^{\infty } \{\|\theta_{n} \|<\rho \},  
\end{align*}
while 
$\tilde{\rho}=\rho_{2} I_{\Lambda } + \beta(\sigma) I_{\Lambda^{c} }$, 
$\tilde{Q}=\{\theta\in\mathbb{R}^{d_{\theta } }: \|\theta \|\leq \tilde{\rho} \}$. 
As $\sigma<\infty$ on $\Lambda^{c}$, 
we have $\theta_{n}, \vartheta_{n}\in \tilde{Q}$ for $n\geq 0$ on the same event. 
We also have $\tilde{\rho}<\infty$ everywhere. 
Consequently, $\delta_{\tilde{Q} }<1$ everywhere, while 
\begin{align}\label{ta2.2.5001}
	\lim_{n\rightarrow\infty } 
	\max_{n\leq j<a(n,t) }
	\left\|
	\sum_{i=n}^{j} \alpha_{i}\zeta_{i}
	\right\|
	I_{ \{\tau_{\tilde{Q},n}>j \} }
	=0, 
	\;\;\;\;\; 
	\limsup_{n\rightarrow\infty} 
	\|\eta_{n} \| I_{ \{\theta_{n} \in \tilde{Q} \} }
	<\infty, 
	\;\;\;\;\; 
	\limsup_{n\rightarrow\infty } 
	h(\theta_{n} ) \|\eta_{n} \| I_{ \{\theta_{n}\in \tilde{Q} \} }
	< 
	\delta_{\tilde{Q} }
\end{align}
for all $t\in(0,\infty )$ on $N_{0}^{c}$. 

Let $\tilde{C}_{2}\in [1,\infty )$ stand for a local Lipschitz constant of $\nabla f(\cdot )$
on $\tilde{Q}$ and for an upper bound of $\|\nabla f(\cdot )\|$ on the same set. 
In addition to this, let  
$\tilde{C}_{3} = 2\tilde{C}_{2} \exp(2\tilde{C}_{2} )$, 
$\tilde{C}_{4} = 20\tilde{C}_{1}\tilde{C}_{3}^{3}$, 
while $\tau=2^{-1}\tilde{C}_{4}^{-1} \varepsilon c^{2}$.  

In order to prove the theorem's assertion, it is sufficient to show $N_{0}^{c}\subseteq\Lambda_{Q}$
(i.e., to demonstrate that on $N_{0}^{c}$, $\|\vartheta_{n} \|\leq\rho_{2}$ for all, 
but finitely many $n$).\footnote
{On $\Lambda_{Q}$, the following holds: 
$\sigma<\infty$ and $\theta_{n}=\vartheta_{n}$, $\tau_{Q,n}=\infty$
for $n>\sigma$. 
Hence, algorithm (\ref{appendix2.1}) asymptotically reduces to (\ref{1.1}) on $\Lambda_{Q}$, 
while (\ref{a1.2.1}) holds almost surely on the same event. 
Therefore, Parts (ii) -- (iv) of the theorem directly follow from Part (i) 
and Theorem \ref{theorem1.1}. } 
We use contradiction to demonstrate this: 
We assume that 
$\|\vartheta_{n} \|>\rho_{2}$ for infinitely many $n$ and some $\omega\in N_{0}^{c}$. 
Notice that all formulas which follow in the proof correspond to $\omega$. 

Let $\delta=\delta_{\tilde{Q} }$. 
As $\{\beta(\sigma_{n} ) \}_{n\geq 0}$ is non-decreasing, 
we have $\beta(\sigma_{n} ) > \rho_{2}$ for all, but finitely many $n$.\footnote
{If $\sigma<\infty$, 
then $\rho_{2}<\|\theta_{n} \| = \|\vartheta_{n-1} \| \leq \beta(\sigma_{n-1} )$ 
for all, but finitely many $n$. 
On the other side, if $\sigma=\infty$, 
then $\lim_{n\rightarrow\infty } \beta(\sigma_{n} ) = \infty$. } 
Hence, there exists an integer $k_{1}$ (depending on $\omega$) such that 
$\beta(\sigma_{n} )> \rho_{2}$ for $n\geq k_{1}$. 
On the other side, 
due to (\ref{ta2.2.5001}), there exists an integer $k_{2}\geq 0$ (depending on $\omega$)
such that 
\begin{align}\label{ta2.2.1}
	\max_{n\leq j<a(n,T) } 
	\left\|
	\sum_{i=n}^{j} \alpha_{i} \zeta_{i} 
	\right\| 
	I_{ \{\tau_{\tilde{Q},n}>j\} }
	\leq 
	\tau^{2}, 
	\;\;\;\;\; 
	h(\theta_{n} ) \|\eta_{n} \| I_{ \{\theta_{n}\in\tilde{Q} \} } 
	\leq
	\delta
\end{align}
for $n\geq k_{2}$. 
Owing to Assumption \ref{a1.1} and (\ref{ta2.2.5001}), we also have 
\begin{align}\label{ta2.2.3051}
	\lim_{n\rightarrow\infty } \|\alpha_{n}\zeta_{n} \| I_{ \{\theta_{n}\in\tilde{Q} \} } 
	=
	\lim_{n\rightarrow\infty } \|\alpha_{n}\eta_{n} \| I_{ \{\theta_{n}\in\tilde{Q} \} } 
	=
	0. 
\end{align}
Since 
\begin{align*}
	&
	\|\vartheta_{n} - \theta_{n} \| I_{ \{\theta_{n}\in\tilde{Q} \} } 
	\leq
	\left(
	\tilde{C}_{2} \alpha_{n}
	+
	\|\alpha_{n}\zeta_{n} \|
	+
	\|\alpha_{n}\eta_{n} \|
	\right)
	I_{ \{\theta_{n}\in\tilde{Q} \} } 
\end{align*}
for $n\geq 0$, Assumption \ref{a1.1} and (\ref{ta2.2.3051}) imply 
$\lim_{n\rightarrow\infty } \|\vartheta_{n} - \theta_{n} \| I_{ \{\theta_{n}\in\tilde{Q} \} } = 0$. 
Then, (\ref{1.1501}) implies that there exists an integer $k_{3}\geq 0$ (depending on $\omega$)
such that 
\begin{align}\label{ta2.2.3001}
	\sum_{i=n}^{a(n,\tau)-1} \alpha_{i}
	\geq
	(1-\varepsilon ) \tau,
	\;\;\;\;\; 
	\|\vartheta_{n} - \theta_{n} \| I_{ \{\theta_{n}\in\tilde{Q} \} } 
	\leq
	\tau
\end{align}
for $n\geq k_{3}$. 

Let $k_{0}=\max\{k_{1}, k_{2}, k_{3} \}$. 
Moreover, let $l_{0}, m_{0}, n_{0}$ be the integers defined as follows. 
If $\omega\in\Lambda$
(i.e., if $\|\theta_{n} \|<\rho$ for infinitely many $n$), let 
\begin{align}\label{ta2.2.3005}
	l_{0}
	=
	\min\{n>k_{0}: \|\theta_{n-1} \|<\rho \}, 
	\;\;\;\: 
	m_{0}
	=
	\min\{n>l_{0}: 
	\|\vartheta_{n-1} \|>\rho_{2} \}, 
	\;\;\;\;\: 
	n_{0}
	=
	\max\{n\leq m_{0}: \|\theta_{n-1} \|<\rho \}. 
\end{align}
Otherwise, if $\omega\in\Lambda^{c}$
(i.e., if $\|\theta_{n} \|<\rho$ for finitely many $n$), let 
\begin{align*}
	l_{0}
	=
	\max\{n>0: \|\theta_{n-1} \|<\rho \}, 
	\;\;\;\;\; 
	m_{0}
	=
	\infty,
	\;\;\;\;\; 
	n_{0}
	=
	\max\{k_{0}, l_{0} \}. 
\end{align*}
Then, we have $k_{0}<n_{0}\leq m_{0}$
and $\theta_{n}=\vartheta_{n-1}$, $\rho\leq\|\theta_{n} \|\leq\tilde{\rho}$ 
for $n_{0}\leq n<m_{0}$.\footnote
{If $\theta_{n}\neq\vartheta_{n-1}$, we have $\|\theta_{n} \|=\|\theta_{0} \|<\rho$. 
On the other side, if $\omega\in\Lambda$, 
then $\|\theta_{n} \|=\|\vartheta_{n-1} \|\leq\rho_{2}=\tilde{\rho}$ 
for $n_{0}\leq n<m_{0}$. 
Moreover, if $\omega\in\Lambda^{c}$, 
then $\|\theta_{n} \|\leq\beta(\sigma_{n-1} ) \leq \beta(\sigma )=\tilde{\rho}$
for $n>0$. }
Therefore, 
\begin{align}\label{ta2.2.3091}
	c\leq \|\nabla f(\theta_{n} ) \|\leq \tilde{C}_{2}, 
	\;\;\;\;\; 
	\theta_{n}\in\tilde{Q}, 
	\;\;\;\;\; 
	\tau_{n,\tilde{Q} }\geq m_{0}
\end{align}
for $n_{0}\leq n<m_{0}$, while 
\begin{align}\label{ta2.2.3071}
	\theta_{j}
	=
	\theta_{n}
	-
	\sum_{i=n}^{j-1} \alpha_{i} \nabla f(\theta_{i} ) 
	-
	\sum_{i=n}^{j-1} \alpha_{i} \xi_{i}
\end{align}
for $n_{0}\leq n<j<m_{0}$. 

Let $\phi_{n}(\tau),\phi_{1,n}(\tau),\phi_{2,n}(\tau)$ 
have the same meaning as in Section \ref{section1.bc*}. 
Now, the asymptotic properties of $\phi_{n}(\tau )$ are analyzed relying on similar arguments
as in the proof of Theorem \ref{theorema2.1}. 
Due to (\ref{ta2.2.1}), (\ref{ta2.2.3091}), we have
\begin{align}\label{ta2.2.51}
	\left\|
	\sum_{i=n}^{j} \alpha_{i} \xi_{i}
	\right\|
	\leq 
	\left\|
	\sum_{i=n}^{j} \alpha_{i} \zeta_{i}
	\right\|
	+
	\sum_{i=n}^{j} \alpha_{i} \|\eta_{i} \|
	\leq 
	\tau^{2} 
	+
	\delta
	\sum_{i=n}^{j} 
	\alpha_{i} \|\nabla f(\theta_{i} ) \|
\end{align}
for $n_{0}\leq n\leq j< \min\{m_{0}, a(n,T) \}$.
Using (\ref{ta2.2.3071}), (\ref{ta2.2.51}), we deduce 
\begin{align*}
	\|\nabla f(\theta_{j} ) \|
	\leq &
	\|\nabla f(\theta_{n} ) \|
	+
	\|\nabla f(\theta_{j} ) - \nabla f(\theta_{n} ) \|
	\nonumber\\
	\leq &
	\|\nabla f(\theta_{n} ) \| 
	+
	\tilde{C}_{2} \|\theta_{j} - \theta_{n} \| 
	\nonumber\\
	\leq &
	\|\nabla f(\theta_{n} ) \|
	+
	\tilde{C}_{2} \sum_{i=n}^{j-1} \alpha_{i} \|\nabla f(\theta_{i} ) \| 
	+
	\tilde{C}_{2} 
	\left\|
	\sum_{i=n}^{j-1} \alpha_{i} \xi_{i} 
	\right\| 
	\nonumber\\
	\leq &
	\|\nabla f(\theta_{n} ) \| 
	+ 
	\tilde{C}_{2} \tau^{2} 
	+
	2\tilde{C}_{2} \sum_{i=n}^{j-1} \alpha_{i} \|\nabla f(\theta_{i} ) \| 
\end{align*}
for $n_{0}\leq n< j \leq \min\{m_{0}-1, a(n,\tau) \}$
(notice that $\tau<T$ and $\theta_{n},\theta_{j}\in\tilde{Q}$ for $n_{0}\leq n< j <m_{0}$). 
Then, Bellman-Gronwall inequality (see e.g., \cite[Appendix B]{borkar}) and (\ref{ta2.2.3091}) imply
\begin{align*}
	\|\nabla f(\theta_{j} ) \|
	\leq &
	\left(\|\nabla f(\theta_{n} ) \| + \tilde{C}_{2}\tau^{2} \right) 
	\exp\left(
	2\tilde{C}_{2} \sum_{i=n}^{j-1} \alpha_{i}  
	\right) 
	\nonumber\\
	\leq &
	\left(\|\nabla f(\theta_{n} ) \| + \tilde{C}_{2} \tau^{2} \right) 
	(1 + \tilde{C}_{3}\tau )
	\nonumber \\
	\leq &
	\|\nabla f(\theta_{n} ) \| 
	+
	\tilde{C}_{2}\tilde{C}_{3}\tau + \tilde{C}_{2}\tau^{2} + \tilde{C}_{2}\tilde{C}_{3}\tau^{3} 
	\nonumber\\
	\leq &
	\|\nabla f(\theta_{n} ) \| 
	+
	\tilde{C}_{4}\tau 
\end{align*}
for $n_{0} \leq n\leq j \leq \min\{m_{0}-1, a(n,\tau) \}$.\footnote
{Notice that $\sum_{i=n}^{j-1} \alpha_{i} \leq \tau<1$ when $n \leq j \leq a(n,\tau)$. 
Notice also that 
$\exp(2\tilde{C}_{2}\tau )\leq 2\tilde{C}_{2}\tau\exp(2\tilde{C}_{2}\tau )\leq \tilde{C}_{3}\tau$.} 
As a result of this and (\ref{ta2.2.51}), we get
\begin{align}\label{ta2.2.3007}
	\left\|
	\sum_{i=n}^{j} \alpha_{i} \xi_{i} 
	\right\|
	\leq 
	\tau^{2} 
	+
	\delta
	\left(
	\|\nabla f(\theta_{n} ) \|
	+
	\tilde{C}_{4}\tau 
	\right)
	\sum_{i=n}^{j} \alpha_{i} 
	\leq 
	\delta\tau \|\nabla f(\theta_{n} ) \|
	+
	2\tilde{C}_{4}\tau^{2}  
\end{align}
for $n_{0}\leq n\leq j< \min\{m_{0}, a(n,\tau ) \}$. 
Owing to  (\ref{ta2.2.3091}), (\ref{ta2.2.3071}), (\ref{ta2.2.3007}), we have 
\begin{align}\label{ta2.2.5}
	\|\theta_{j} - \theta_{n} \|
	\leq &
	\sum_{i=n}^{j-1} \alpha_{i} \|\nabla f(\theta_{i} ) \| 
	+ 
	\left\|
	\sum_{i=n}^{j-1} \alpha_{i} \xi_{i} 
	\right\| 
	\leq 
	\left(
	\|\nabla f(\theta_{n} ) \|
	+
	2\tilde{C}_{4}\tau  
	\right)
	\left(
	\sum_{i=n}^{j-1} \alpha_{i} 
	+
	\tau
	\right)
	\leq 
	4\tilde{C}_{2} \tau 
\end{align}
for $n_{0} \leq n\leq j \leq \min\{m_{0}-1, a(n,\tau) \}$
(notice that $\tilde{C}_{4}\tau\leq 1/2$). 
Consequently, 
\begin{align*}
	|\phi_{1,n}(\tau ) |
	\leq &
	\tilde{C}_{2} 
	\|\nabla f(\theta_{n} ) \|
	\sum_{i=n}^{a(n,\tau)-1} \alpha_{i} \|\theta_{i} - \theta_{n} \|
	\leq 
	4\tilde{C}_{2}^{3} \tau 
	\|\nabla f(\theta_{n} ) \|
	\sum_{i=n}^{a(n,\tau)-1} \alpha_{i} 
	\leq 
	4\tilde{C}_{2}^{3} \tau^{2} 
\end{align*}
for $n\geq n_{0}$ satisfying $a(n,\tau )<m_{0}$
(notice that $\|\nabla f(\theta_{n} ) \|\leq\tilde{C}_{2}$ 
and $\theta_{n}, \theta_{i} \in \tilde{Q}$
for $n_{0}\leq n\leq i < m_{0}$). 
We also have 
\begin{align*}
	|\phi_{2,n}(\tau ) |
	\leq &
	\tilde{C}_{2} \|\theta_{a(n,\tau)} - \theta_{n} \|^{2}
	\leq 
	16\tilde{C}_{2}^{3} \tau^{2} 
\end{align*}
for $n\geq n_{0}$ satisfying $a(n,\tau )<m_{0}$
(notice that 
and $\theta_{n}, \theta_{a(n,\tau ) } \in \tilde{Q}$
when $n\geq n_{0}$, $a(n,\tau)<m_{0}$). 
Hence, 
\begin{align}\label{ta2.2.307}
	|\phi_{n}(\tau ) |
	\leq &
	\tilde{C}_{4} \tau^{2} 
\end{align}
when $n\geq n_{0}$, $a(n,\tau )<m_{0}$. 
On the other side, (\ref{ta2.2.3001}), (\ref{ta2.2.3091}), (\ref{ta2.2.3007}) 
yield 
\begin{align*}
	\|\nabla f(\theta_{n} ) \|
	\sum_{i=n}^{a(n,\tau)-1} \alpha_{i} 
	-
	\left\|
	\sum_{i=n}^{a(n,\tau)-1} \alpha_{i}\xi_{i}
	\right\|
	\geq &
	(1-\delta-\varepsilon)\tau 
	\|\nabla f(\theta_{n} ) \| 
	-
	2\tilde{C}_{4}\tau^{2} 
	\nonumber\\
	= &
	4\varepsilon\tau \|\nabla f(\theta_{n} ) \| 
	-
	2\tilde{C}_{4}\tau^{2} 
	\nonumber\\
	\geq &
	2\varepsilon\tau \|\nabla f(\theta_{n} ) \|
\end{align*}
for $n\geq n_{0}$ satisfying $a(n,\tau )<m_{0}$.\footnote
{Notice that $1-\delta=5\varepsilon$, $\varepsilon c \geq \tilde{C}_{4}\tau$. 
Notice also that 
$\varepsilon\tau \|\nabla f(\theta_{n} ) \|\geq 
\varepsilon\tau c \geq 
\tilde{C}_{4}\tau^{2}$ for $n_{0}\leq n < m_{0}$. }
Then, (\ref{1.1*}), (\ref{ta2.2.3091}), (\ref{ta2.2.307}) imply 
\begin{align}\label{ta2.2.7}
	f(\theta_{a(n,\tau) } ) - f(\theta_{n} ) 
	\leq&
	-
	2\varepsilon\tau\|\nabla f(\theta_{n} ) \|^{2}
	+
	\tilde{C}_{4}\tau^{2}  
	\leq 
	-\varepsilon\tau \|\nabla f(\theta_{n} ) \|^{2} 
	\leq 
	-
	\varepsilon\tau c^{2}
\end{align}
for $n\geq n_{0}$ satisfying $a(n,\tau )<m_{0}$
(notice that 
$\varepsilon\|\nabla f(\theta_{n} ) \|^{2}\geq 
\varepsilon c^{2} \geq \tilde{C}_{4}\tau$ when $n_{0}\leq n<m_{0}$). 

Let $\{n_{k} \}_{k\geq 0}$ be the sequence recursively defined by 
$n_{k+1}=a(n_{k},\tau )$ for $k\geq 0$. 
As in the proof of Theorem \ref{theorema2.1}, 
we now show by contradiction $\omega\in\Lambda$ 
(i.e., $\|\theta_{n} \|<\rho$ for infinitely many $n$). 
We assume the opposite. 
Then, $m_{0}=\infty$ and $\theta_{n}=\vartheta_{n}$, while (\ref{ta2.2.7}) yields  
$
	f(\theta_{n_{k+1} } ) - f(\theta_{n_{k} } )
	\leq 
	-\varepsilon\tau c^{2}
$
for $k\geq 0$. 
Hence, $\lim_{k\rightarrow\infty } f(\theta_{n_{k} } ) = -\infty$.  
However, this is impossible due to Assumption \ref{aa2.3}. 
Thus, $\omega\in\Lambda$ (i.e., $\|\theta_{n} \|<\rho$ for infinitely many $n$). 
Therefore, $m_{0}, n_{0}$ are defined through (\ref{ta2.2.3005}), 
while $\|\theta_{n_{0}-1} \|<\rho$, $\|\vartheta_{m_{0}-1} \|>\rho_{2}$. 
Combining this with (\ref{ta2.2.3001}), we conclude
$	
	\|\vartheta_{n_{0}-1} - \theta_{n_{0}-1} \|
	\leq 
	\tau 
	\leq 
	1/2
$.
Consequently, 
\begin{align}\label{ta2.2.3009}
	&
	\|\vartheta_{n_{0}-1} \|
	\leq 
	\|\theta_{n_{0}-1} \|
	+
	\|\vartheta_{n_{0}-1} - \theta_{n_{0}-1} \|
	\leq \rho+1/2
	<\rho_{2}. 
\end{align}
Hence, $n_{0}<m_{0}$, 
$f(\theta_{n_{0} } ) \leq \tilde{C}_{1}$
(notice that $\|\theta_{n_{0} } \| = \|\vartheta_{n_{0}-1} \|\leq\rho+1$). 

Let $i_{0}, j_{0}$ be the integers defined by 
$
	j_{0}
	=
	\max\{j\geq 0: n_{j}< m_{0} \}
$, 
$ 
	i_{0} 
	=
	n_{j_{0} }
$. 
Then, we have $n_{0}\leq i_{0}=n_{j_{0} }< m_{0}\leq a(i_{0},\tau ) = n_{j_{0}+1}$. 
Combining this with (\ref{ta2.2.3001}), (\ref{ta2.2.5}), we get
\begin{align*}
	&
	\|\vartheta_{m_{0}-1} - \theta_{m_{0}-1} \|
	\leq
	\tau
	\leq
	1/2, 
	\;\;\;\;\; 
	\|\theta_{i_{0} }-  \theta_{m_{0}-1} \| 
	\leq 
	4\tilde{C}_{2} \tau 
	\leq 
	1/2.  
\end{align*}
Therefore, 
\begin{align}
	&\label{ta2.2.703}
	\|\theta_{i_{0} } \|
	\geq 
	\|\vartheta_{m_{0}-1} \| 
	-
	\|\vartheta_{m_{0} } - \theta_{m_{0}-1} \|
	-
	\|\theta_{i_{0} }-  \theta_{m_{0}-1} \| 
	>
	\rho_{2} - 1
	=
	\rho_{1}. 
\end{align}

Let $\{\gamma_{n} \}_{n\geq 0}$, $\theta_{0}(\cdot )$ 
have the same meaning as in Section \ref{section1.a*}. 
As in the proof of Theorem \ref{theorema2.1}, 
we now show by contradiction that $\gamma_{i_{0} } - \gamma_{n_{0} }\geq T$. 
We assume the opposite. 
Then, (\ref{ta2.2.3071}), (\ref{ta2.2.51}), (\ref{ta2.2.3009}) yield  
\begin{align}\label{ta2.2.705}
	\|\theta_{0}(t) \|
	=
	\|\theta_{j} \|
	\leq&
	\|\theta_{n_{0} } \| 
	+
	\sum_{i=n_{0} }^{j-1} \alpha_{i} \|\nabla f(\theta_{i} ) \| 
	+
	\left\|
	\sum_{i=n_{0} }^{j-1} \alpha_{i} \xi_{i} 
	\right\|
	\nonumber\\
	\leq&
	\|\theta_{n_{0} } \| 
	+
	\tau^{2} 
	+
	2\sum_{i=n_{0} }^{j-1} \alpha_{i} \|\nabla f(\theta_{i} ) \|   
	\nonumber\\
	\leq&
	\rho+1
	+
	2\sum_{i=n_{0} }^{j-1} \alpha_{i} \phi(\|\theta_{i} \| )  
	\nonumber\\
	\leq&
	\rho+1
	+
	2\int_{\gamma_{n_{0} } }^{t} \phi(\|\theta_{0}(s) \| ) ds
\end{align}
for $t\in [\gamma_{j}, \gamma_{j+1} )$, $n_{0}\leq j\leq i_{0}$.\footnote
{Since $j\leq i_{0}<m_{0}$, we have 
$\gamma_{j}-\gamma_{n_{0} }\leq \gamma_{i_{0} }-\gamma_{n_{0} } \leq T$
and $j\leq\min\{m_{0}-1,a(n_{0},T) \}$. 
We also have $\tau^{2}\leq 1/2$. }
Owing to the comparison principle (see \cite[Section 3.4]{khalil}) and 
(\ref{ta2.1.701}), (\ref{ta2.2.705}), we have  
$\|\theta_{0}(t) \| \leq 
\lambda(t-\gamma_{n_{0} };\rho+1) \leq \rho_{1}$
for all $t\in[\gamma_{{n}_{0} }, \gamma_{i_{0} } ]$. 
Thus, $\|\theta_{i_{0} } \|=\|\theta_{0}(\gamma_{i_{0} } ) \|\leq \rho_{1}$. 
However, this is impossible, due to (\ref{ta2.2.703}). 
Hence, $\gamma_{i_{0} }-\gamma_{n_{0} }\geq T$. 
Consequently, 
\begin{align}\label{ta2.2.707}
	T
	\leq 
	\gamma_{i_{0} } - \gamma_{n_{0} }
	=
	\sum_{j=0}^{j_{0}-1} (\gamma_{n_{j+1} } - \gamma_{n_{j} } )
	\leq 
	j_{0}\tau 
\end{align}
(notice that $n_{j_{0} } = i_{0}$, 
$\gamma_{n_{j+1} }-\gamma_{n_{j} } = \sum_{i=n_{j} }^{n_{j+1}-1} \alpha_{i} \leq \tau$). 

Due to (\ref{ta2.2.7}), 
we have 
$
	f(\theta_{n_{j+1} } ) - f(\theta_{n_{j} } ) 
	\leq 
	-\varepsilon\tau c^{2}
$
for $0\leq j \leq j_{0}$. 
Then, (\ref{ta2.2.707}) implies 
\begin{align*}
	f(\theta_{i_{0} } )
	=
	f(\theta_{n_{j_{0} } } ) 
	\leq 
	f(\theta_{n_{0} } ) - j_{0}\varepsilon\tau c^{2}
	\leq 
	\tilde{C}_{1} - \varepsilon c^{2} T
	\leq 
	-\tilde{C}_{1}. 
\end{align*}
However, this is impossible, since $f(\theta ) > -\tilde{C}_{1}$ for all 
$\theta\in\mathbb{R}^{d_{\theta } }$. 
Hence, $\|\theta_{n} \|>\rho_{2}$ for finitely many $n$. 
\end{sproof}

\refstepcounter{appendixcounter}\label{appendix3}
\section*{Appendix \arabic{appendixcounter} }

In this section, a global version of Theorem \ref{theorem2.1} is presented.
It is also shown how Theorem \ref{theorem2.1} can be extended 
to the randomly projected stochastic gradient search with Markovian dynamics. 
The results provided in this section can be considered as 
a combination of Theorems \ref{theorema2.1}, \ref{theorema2.2}
(Appendix \ref{appendix2}) with Theorem \ref{theorem2.1} (Section \ref{section2}). 

First, the stability and the global asymptotic behavior of algorithm (\ref{2.1}) are studied.
To analyze these properties, we use the following two assumptions. 

\begin{assumptionappendix}\label{aa3.1} 
There exists a Borel-measurable function $\varphi:\mathbb{R}^{d_{z} }\rightarrow[1,\infty )$
such that 
\begin{align*}
	&
	\max\{\|F(\theta,z) \|, \|\tilde{F}(\theta,z) \|, \|(\Pi\tilde{F} )(\theta,z) \| \}
	\leq
	\varphi(z) (\|\nabla f(\theta ) \| + 1 ), 
	\\
	&
	\|(\Pi\tilde{F} )(\theta',z) - (\Pi\tilde{F} )(\theta'',z) \|
	\leq 
	\varphi(z) \|\theta' - \theta'' \|
\end{align*}
for all $\theta,\theta',\theta''\in\mathbb{R}^{d_{\theta } }$, 
$z\in\mathbb{R}^{d_{z} }$. 
In addition to this, 
\begin{align*}
	\sup_{n\geq 0} 
	E(\varphi^{2}(Z_{n} )|\theta_{0}=\theta,Z_{0}=z)
	<\infty
\end{align*}
for all $\theta\in\mathbb{R}^{d_{\theta } }$, 
$z\in\mathbb{R}^{d_{z} }$. 
\end{assumptionappendix}

\begin{assumptionappendix}\label{aa3.2} 
$\eta_{n}=\eta(\theta_{n} )$ for $n\geq 0$, 
where $\eta:\mathbb{R}^{d_{\theta } }\rightarrow\mathbb{R}^{d_{\theta } }$ is a continuous function. 
Moreover, there exists a real number $\delta\in(0,1)$ such that 
$\|\eta(\theta ) \|\leq\delta\|\nabla f(\theta ) \|$ for all $\theta\in\mathbb{R}^{d_{\theta } }$
satisfying $\|\theta \|\geq\rho$ 
($\rho$ is specified in Assumption \ref{aa2.1}). 
\end{assumptionappendix}

Assumption \ref{aa3.1} is a global version of Assumption \ref{a2.3}.
In a similar form, it is involved in the stability analysis of 
stochastic approximation carried out in \cite[Section II.1.9]{benveniste}. 
On the other side, Assumption \ref{aa3.2} is related to the bias of the gradient estimator. 
It requires the bias $\{\eta_{n} \}_{n\geq 0}$ to be a deterministic function of 
the algorithm iterates $\{\theta_{n} \}_{n\geq 0}$. 
As demonstrated in Sections \ref{section3} -- \ref{section4}, 
this is often satisfied in practice. 
Assumption \ref{aa3.2} can be considered as one of the weakest conditions 
under which the stability of the perturbed ODE $d\theta/dt=-(\nabla f(\theta ) + \eta(\theta ) )$ can be shown. 

Our results on the stability and asymptotic bias of algorithm (\ref{2.1}) are provided in the next theorem. 

\begin{theoremappendix}\label{theorema3.1}
Suppose that Assumptions \ref{a2.1}, \ref{a2.2}, \ref{aa2.1}, \ref{aa3.1} and \ref{aa3.2} hold. 
Then, the following is true: 
\begin{compactenum}[(i)]
\item
If $f(\cdot )$ (specified in Assumption \ref{a2.2}) satisfies Assumption \ref{a1.3.a}, 
Part (i) of Theorem \ref{theorema2.1} holds. 
\item
If $f(\cdot )$ (specified in Assumption \ref{a2.2}) satisfies Assumption \ref{a1.3.b}, 
Part (ii) of Theorem \ref{theorema2.1} holds. 
\item
If $f(\cdot )$ (specified in Assumption \ref{a2.2}) satisfies Assumption \ref{a1.3.c}, 
Part (iii) of Theorem \ref{theorema2.1} holds. 
\end{compactenum}
\end{theoremappendix}

\begin{sproof}
Let $g(\cdot )$, $h(\cdot )$ be the functions defined in Assumption \ref{aa2.2}. 
Then, due to Assumption \ref{aa3.2}, 
$g(\theta )\eta(\theta )$ is uniformly bounded in $\theta\in\mathbb{R}^{d_{\theta } }$, 
while $h(\theta)\eta(\theta )\leq\delta$ for all $\theta\in\mathbb{R}^{d_{\theta } }$
satisfying $\|\theta \|\geq\rho$. 
Let $C\in[1,\infty )$ stand for a (global) Lipschitz constant of $\nabla f(\cdot )$
and for an (global) upper bound of $g(\cdot )\eta(\cdot )$. 
Define $\tau=1/(18C^{2} )$ and let $\{\zeta_{n} \}_{n\geq 0}$, 
$\{\zeta_{1,n} \}_{n\geq 0}$, $\{\zeta_{2,n} \}_{n\geq 0}$, $\{\zeta_{3,n} \}_{n\geq 0}$
have the same meaning as in the proof of Theorem \ref{theorem2.1}, 
while $\tau_{n}$ is the stopping time defined by 
\begin{align*}
	\tau_{n}
	=
	\min\left(
	\left\{
	j\geq n: g(\theta_{n} ) g^{-1}(\theta_{j} ) > 3 
	\right\}
	\cup
	\{\infty \}
	\right)
\end{align*}
for $n\geq 0$. 
Finally, for $\theta\in\mathbb{R}^{d_{\theta } }$, $z\in\mathbb{R}^{d_{z} }$, 
let $E_{\theta,z}(\cdot )$ denote the conditional mean given $\theta_{0}=\theta$, $Z_{0}=z$. 

As a direct consequence of Assumptions \ref{a2.1}, \ref{aa3.1}, we get
\begin{align*}
	E_{\theta,z}\left(\sum_{n=0}^{\infty } \alpha_{n}^{2} \varphi^{2}(Z_{n+1} ) \right)
	<\infty
\end{align*}
for all $\theta\in\mathbb{R}^{d_{\theta } }$, $z\in\mathbb{R}^{d_{z} }$. 
We also have 
\begin{align*}
	&
	g(\theta_{n} ) \|\zeta_{n} \| 
	\leq
	\varphi(Z_{n+1} ) + 1 
	\leq 
	2\varphi(Z_{n+1} ) 
\end{align*}
for $n\geq 0$. 
Consequently, 
\begin{align}\label{ta3.1.1}
	\lim_{n\rightarrow\infty } 
	\alpha_{n}\varphi(Z_{n+1} )
	=
	\lim_{n\rightarrow\infty } 
	\alpha_{n} g(\theta_{n} ) \|\zeta_{n} \|
	=0
\end{align}
almost surely. 

Let $\{m_{k} \}_{k\geq 0}$ be the sequence recursively defined by 
$m_{0}=0$ and $m_{k+1}=a(m_{k},\tau)$ for $k\geq 0$.  
Moreover, let ${\cal F}_{n}=\sigma\{\theta_{0},Z_{0},\dots,\theta_{n},Z_{n} \}$ for $n\geq 0$. 
Due to Assumption \ref{a2.2}, we have
\begin{align*}
	E_{\theta,z}
	\left(
	g(\theta_{n} ) \zeta_{1,j} 
	I_{ \{\tau_{n} > j \} }
	|
	{\cal F}_{j} 
	\right)
	=
	g(\theta_{n} ) 
	\left(
	E_{\theta,z}
	(
	\tilde{F}(\theta_{j}, Z_{j+1} )
	|
	{\cal F}_{j} 
	)
	-
	(\Pi\tilde{F} )(\theta_{j},Z_{j} )
	\right) 
	I_{ \{\tau_{n} > j \} }
	=
	0
\end{align*}
almost surely for each $\theta\in\mathbb{R}^{d_{\theta } }$, $z\in\mathbb{R}^{d_{z} }$, $0\leq n\leq j$
(notice that $\{\tau_{n}>j\}$ is measurable with respect to ${\cal F}_{j}$). 
Moreover, Assumption \ref{aa3.1} implies 
\begin{align*}
	g(\theta_{n} ) \|\zeta_{1,j} \| I_{ \{\tau_{n}>j \} }
	\leq 
	g(\theta_{n} ) g^{-1}(\theta_{j} ) (\varphi(Z_{j} ) + \varphi(Z_{j+1} ) ) I_{ \{\tau_{n}>j \} }
	\leq 
	3 (\varphi(Z_{j} ) + \varphi(Z_{j+1} ) ) 
\end{align*}
for $0\leq n\leq j$. 
Then, as a result of Doob inequality, we get 
\begin{align*}
	E_{\theta,z}\left(
	\max_{n<j<a(n,\tau ) }
	\left\|\sum_{i=n+1}^{j}\alpha_{i}g(\theta_{n} ) \zeta_{1,i} \right\|^{2} I_{ \{\tau_{n}>j\} } 
	\right)
	\leq&
	E_{\theta,z}\left(
	\max_{n<j<a(n,\tau ) }
	\left\|\sum_{i=n+1}^{j}\alpha_{i}g(\theta_{n} ) \zeta_{1,i} I_{ \{\tau_{n}>i\} } \right\|^{2}
	\right)
	\nonumber\\
	\leq&
	4E_{\theta,z}\left(
	\sum_{i=n+1}^{a(n,\tau )-1} \alpha_{i}^{2} 
	g^{2}(\theta_{n} ) \|\zeta_{1,i} \|^{2} I_{ \{\tau_{n}>i\} } 
	\right)
	\nonumber\\
	\leq& 
	72E_{\theta,z}\left(\sum_{i=n+1}^{a(n,\tau) } \alpha_{i}^{2} 
	\left(\varphi^{2}(Z_{i} ) + \varphi^{2}(Z_{i+1} ) \right)
	\right)
\end{align*}
for all $\theta\in\mathbb{R}^{d_{\theta } }$, $z\in\mathbb{R}^{d_{z} }$, $n\geq 0$. 
Combining this with Assumptions \ref{a2.1}, \ref{aa3.1}, we deduce 
\begin{align*}
	E_{\theta,z}\left(
	\sum_{k=0}^{\infty } 
	g^{2}(\theta_{m_{k} } ) 
	\max_{m_{k}<j<m_{k+1} }
	\left\|\sum_{i=m_{k} }^{j}\alpha_{i}\zeta_{1,i} \right\|^{2} I_{ \{\tau_{m_{k} }>j\} } 
	\right)
	\leq &
	72 E_{\theta,z}\left(\sum_{n=0}^{\infty } 
	(\alpha_{i}^{2} + \alpha_{i+1}^{2} ) \varphi^{2}(Z_{i+1} ) \right)
	<
	\infty
\end{align*}
for each $\theta\in\mathbb{R}^{d_{\theta } }$, $z\in\mathbb{R}^{d_{z} }$, $n\geq 0$. 
Therefore, 
\begin{align}\label{ta3.1.5}
	\lim_{k\rightarrow\infty } 
	g(\theta_{m_{k} } ) 
	\max_{m_{k}<j<m_{k+1} }
	\left\|\sum_{i=m_{k} }^{j}\alpha_{i}\zeta_{1,i} \right\| I_{ \{\tau_{m_{k} }>j\} } 
	=
	0
\end{align}
almost surely. 

Since $\alpha_{n}\alpha_{n+1} = O(\alpha_{n}^{2} )$, 
$\alpha_{n}-\alpha_{n+1} = O(\alpha_{n}^{2} )$ for $n\rightarrow\infty$
(see the proof of Theorem \ref{theorem2.1}), 
Assumptions \ref{a2.1}, \ref{aa3.1} yield
\begin{align*}
	E_{\theta,z}
	\left(
	\sum_{n=0}^{\infty } 
	\alpha_{n}\alpha_{n+1} \varphi^{2}(Z_{n+1} )
	\right)
	<\infty, 
	\;\;\;\;\;
	E_{\theta,z}
	\left(
	\sum_{n=0}^{\infty } |\alpha_{n} - \alpha_{n+1} | \varphi^{2}(Z_{n+1} ) 
	\right)
	<\infty
\end{align*}
for all $\theta\in\mathbb{R}^{d_{\theta } }$, $z\in\mathbb{R}^{d_{z} }$. 
Additionally, 
due to Assumptions \ref{aa3.1}, \ref{aa3.2}, we have 
\begin{align*}
	g(\theta_{n} ) \|\zeta_{2,j} \| I_{ \{\tau_{n}>j \} }
	\leq &
	g(\theta_{n} ) \varphi(Z_{j} ) \|\theta_{j} - \theta_{j-1} \| I_{ \{\tau_{n}>j-1 \} } 
	\\
	\leq &
	\alpha_{j-1} g(\theta_{n} ) \varphi(Z_{j} ) 
	(\|F(\theta_{j-1},Z_{j} ) \| + \|\eta_{j-1} \| ) I_{ \{\tau_{n}>j \} } 
	\\
	\leq &
	\alpha_{j-1} g(\theta_{n} ) g^{-1}(\theta_{j-1} ) \varphi(Z_{j} ) 
	(\varphi(Z_{j} ) + C ) I_{ \{\tau_{n}>j \} } 
	\\
	\leq &
	6C\alpha_{j-1} \varphi^{2}(Z_{j} ) 
\end{align*}
for $0\leq n<j$ (notice that $\varphi(z)\geq 1$ for any $z\in\mathbb{R}^{d_{z} }$).
We also have 
\begin{align*}
	&
	g(\theta_{n} ) \|\zeta_{3,j} \| I_{ \{\tau_{n}>j \} } 
	\leq 
	g(\theta_{n} ) g^{-1}(\theta_{j} ) \varphi(Z_{j+1} ) I_{ \{\tau_{n}>j \} } 
	\leq 
	3 \varphi(Z_{j+1} ) 
	\leq 
	3 \varphi^{2}(Z_{j+1} )
\end{align*}
for $0\leq n\leq j$. 
Hence, 
\begin{align*}
	&
	g(\theta_{n} ) 
	\left\|
	\sum_{i=n+1}^{j} \alpha_{i} \zeta_{2,i} 
	\right\|
	I_{ \{\tau_{n}>j \} }
	\leq 
	\sum_{i=n+1}^{j} \alpha_{i} g(\theta_{n} ) \|\zeta_{2,i} \| I_{ \{\tau_{n}>i \} }
	\leq 
	6C \sum_{i=n}^{j} \alpha_{i} \alpha_{i+1} \varphi^{2}(Z_{i+1} ), 
	\\
	&
	g(\theta_{n} ) 
	\left\|
	\sum_{i=n+1}^{j} (\alpha_{i} - \alpha_{i+1} ) \zeta_{3,i} 
	\right\|
	I_{ \{\tau_{n}>j \} }
	\leq 
	\sum_{i=n+1}^{j} |\alpha_{i} - \alpha_{i+1} | g(\theta_{n} ) \|\zeta_{3,i} \| I_{ \{\tau_{n}>i \} }
	\leq 
	3 \sum_{i=n+1}^{j} |\alpha_{i} - \alpha_{i+1} | \varphi^{2}(Z_{i+1} )
\end{align*}
for $0\leq n<j$. 
Consequently, 
\begin{align}\label{ta3.1.21}
	&
	\lim_{n\rightarrow\infty }
	g(\theta_{n} ) 
	\max_{j>n} 
	\left\|
	\sum_{i=n+1}^{j} \alpha_{i} \zeta_{2,i} 
	\right\|
	I_{ \{\tau_{n}>j \} }
	=
	\lim_{n\rightarrow\infty } 
	g(\theta_{n} ) 
	\max_{j>n} 
	\left\|
	\sum_{i=n+1}^{j} (\alpha_{i} - \alpha_{i+1} ) \zeta_{3,i} 
	\right\|
	I_{ \{\tau_{n}>j \} }
	=
	0
\end{align}
almost surely
(notice that $\alpha_{j+1}/\alpha_{j}=O(1)$ for $j\rightarrow\infty$).  
Moreover, (\ref{ta3.1.1}) yields 
\begin{align}\label{ta3.1.23}
	\lim_{n\rightarrow\infty } 
	g(\theta_{n} ) 
	\max_{j\geq n} \alpha_{j+1} \|\zeta_{3,j} \| I_{ \{ \tau_{n}>j \} }
	=
	0
\end{align}
almost surely. 
Combining (\ref{ta3.1.1}) -- (\ref{ta3.1.23}) with (\ref{2.3*}), 
we deduce 
\begin{align}\label{ta3.1.25}
	\lim_{k\rightarrow\infty } 
	g(\theta_{n_{k} } ) 
	\max_{m_{k}\leq j < m_{k+1} }
	\left\|
	\sum_{i=m_{k} }^{j} \alpha_{i} \zeta_{i}
	\right\|
	I_{ \{\tau_{m_{k} }>j \} }
	=
	0
\end{align}
almost surely. 

Owing to Assumptions \ref{aa2.1}, \ref{aa3.2}, we have 
\begin{align*}
	g^{-1}(\theta_{j+1} ) I_{ \{\tau_{n}>j \} }
	\leq &
	g^{-1}(\theta_{n} ) 
	+
	\|\nabla f(\theta_{j+1} ) - \nabla f(\theta_{n} ) \| I_{ \{\tau_{n}>j \} }
	\nonumber \\
	\leq &
	g^{-1}(\theta_{n} ) 
	+
	C\|\theta_{j+1} - \theta_{n} \| I_{ \{\tau_{n}>j \} }
	\nonumber\\
	\leq &
	g^{-1}(\theta_{n} ) 
	+
	C\sum_{i=n}^{j} \alpha_{i} \|\nabla f(\theta_{i} ) \| I_{ \{\tau_{n}>j \} }
	+
	C\left\|\sum_{i=n}^{j}\alpha_{i}\zeta_{i} \right\| I_{ \{\tau_{n}>j \} }
	+
	C\sum_{i=n}^{j} \alpha_{i} \|\eta_{i} \| I_{ \{\tau_{n}>j \} }
	\nonumber\\
	\leq &
	g^{-1}(\theta_{n} ) 
	+
	C\left\|\sum_{i=n}^{j}\alpha_{i}\zeta_{i} \right\| I_{ \{\tau_{n}>j \} }
	+
	2C^{2} \sum_{i=n}^{j} \alpha_{i} g^{-1}(\theta_{i} ) I_{ \{\tau_{n}>j \} }
\end{align*}
for $0\leq n\leq j$
(notice that $\|\eta(\theta ) \|\leq Cg^{-1}(\theta )$ for each $\theta\in\mathbb{R}^{d_{\theta } }$). 
Combining this with Bellman-Gronwall inequality (see e.g., \cite[Appendix B]{borkar}), we conclude 
\begin{align*}
	g^{-1}(\theta_{j+1} ) I_{ \{\tau_{n}>j \} }
	\leq &
	\left(
	g^{-1}(\theta_{n} ) 
	+
	C \max_{n\leq j<a(n,\tau ) } 
	\left\|\sum_{i=n}^{j} \alpha_{i}\zeta_{i} \right\| I_{ \{\tau_{n}>j \} }
	\right)
	\exp\left(2C^{2} \sum_{i=n}^{j-1} \alpha_{i} \right)
	\\
	\leq &
	2 g^{-1}(\theta_{n} ) 
	\left(
	1
	+
	C g(\theta_{n} ) 
	\max_{n\leq j<a(n,\tau ) } 
	\left\|\sum_{i=n}^{j} \alpha_{i}\zeta_{i} \right\| I_{ \{\tau_{n}>j \} }
	\right)
\end{align*}
for $0\leq n\leq j\leq a(n,\tau )$.\footnote
{Notice that $\sum_{i=n}^{j-1}\alpha_{i}\leq\tau$ for $n\leq j\leq a(n,\tau)$. 
Notice also that $\exp(2C^{2}\tau )\leq\exp(1/2)\leq 2$. }
Then, (\ref{ta3.1.25}) yields
\begin{align}\label{ta3.1.27}
	\limsup_{k\rightarrow\infty } 
	g(\theta_{m_{k} } )
	\max_{m_{k}\leq j<m_{k+1} } g^{-1}(\theta_{j+1} ) I_{ \{\tau_{m_{k} }>j \} }
	\leq 
	2
\end{align}
almost surely. 

Let $N_{0}$ be the event where (\ref{ta3.1.25}) or (\ref{ta3.1.27}) does not hold. 
Then, in order to prove the theorem's assertion, it is sufficient to show that 
(\ref{aa2.2.1}), (\ref{aa2.2.3}) are satisfied on $N_{0}^{c}$ for any $t\in(0,\infty )$. 
Let $\omega$ be any sample in $N_{0}^{c}$, while $t\in(0,\infty )$ is any real number.  
Notice that all formula which follow in the proof correspond to $\omega$. 

Due to Assumption \ref{aa3.2}, we have 
\begin{align*}
	\limsup_{n\rightarrow\infty } g(\theta_{n} ) \|\eta_{n} \|
	\leq
	C
	< \infty, 
	\;\;\;\;\; 
	\limsup_{n\rightarrow\infty } h(\theta_{n} ) \|\eta_{n} \|
	\leq
	\delta
	< 1. 
\end{align*}
Moreover, Assumption \ref{a2.1} and 
(\ref{1.1501}), (\ref{ta3.1.27}) imply that there exists 
an integer $k_{0}\geq 0$ (depending on $\omega$)
such that 
\begin{align}\label{ta3.1.71}
	\sum_{i=m_{k} }^{m_{k+1}-1} \alpha_{i}
	\geq
	\tau/2, 
	\;\;\;\;\;
	g(\theta_{m_{k} } ) 
	\left\|\sum_{i=m_{k} }^{j} \alpha_{i} \zeta_{i} \right\| I_{ \{\tau_{m_{k} }>j \} }
	\leq
	\tau, 
	\;\;\;\;\; 
	g(\theta_{m_{k} } ) g^{-1}(\theta_{j+1} ) I_{ \{\tau_{m_{k} }>j \} }
	\leq 3
\end{align}
for $k\geq k_{0}$, $m_{k}\leq j<m_{k+1}$. 
As $\tau_{n}>n$ for $n\geq 0$, we conclude 
$\tau_{m_{k} }>m_{k+1}$ for $k\geq k_{0}$.\footnote
{If $\tau_{m_{k} }\leq m_{k+1}$, 
then $\tau_{m_{k} }=j$ and 
$g(\theta_{m_{k} } ) g^{-1}(\theta_{j} ) I_{ \{\tau_{m_{k} }>j-1 \} } =
g(\theta_{m_{k} } ) g^{-1}(\theta_{j} ) > 3$ for some $j$
satisfying $m_{k}<j\leq m_{k+1}$. } 
Consequently, 
$I_{ \{\tau_{m_{k} }>j \} } = 1$ for $k\geq k_{0}$, $m_{k}\leq j\leq m_{k+1}$. 
Combining this with (\ref{ta3.1.71}), we get
$g(\theta_{m_{k} } ) \leq 3 g(\theta_{j+1} )$ and 
\begin{align}\label{ta3.1.43}
	g^{-1}(\theta_{j+1} ) 
	\geq &
	g^{-1}(\theta_{m_{k} } ) 
	-
	\|\nabla f(\theta_{j+1} ) - \nabla f(\theta_{n} ) \|
	\nonumber\\
	\geq &
	g^{-1}(\theta_{m_{k} } ) 
	-
	C\|\theta_{j+1} - \theta_{n} \|
	\nonumber\\
	\geq &
	g^{-1}(\theta_{m_{k} } ) 
	-
	C\sum_{i=m_{k} }^{j} \alpha_{i} \|\nabla f(\theta_{i} ) \|
	-
	C\left\|\sum_{i=m_{k} }^{j} \alpha_{i}\zeta_{i} \right\|
	-
	C\sum_{i=m_{k} }^{j} \alpha_{i} \|\eta_{i} \|
	\nonumber\\
	\geq &
	g^{-1}(\theta_{m_{k} } ) 
	-
	2C^{2} \sum_{i=m_{k} }^{j} \alpha_{i} g^{-1}(\theta_{i} ) 
	-
	C\left\|\sum_{i=m_{k} }^{j} \alpha_{i}\zeta_{i} \right\|
	\nonumber\\
	\geq &
	g^{-1}(\theta_{m_{k} } ) 
	(1-6C^{2}\tau-C\tau )
	\nonumber\\
	\geq &
	3^{-1} g^{-1}(\theta_{m_{k} } ) 
\end{align}
for $k\geq k_{0}$, $m_{k}\leq j<m_{k+1}$.\footnote
{Notice that $g^{-1}(\theta_{i} )\leq 3g^{-1}(\theta_{m_{k} } )$, 
$\sum_{m_{k} }^{m_{k+1}\!-\!1}\!\! \alpha_{i}\leq\tau$ when $k\geq k_{0}$, $m_{k}\leq i<m_{k+1}$.
Notice also that $6C^{2}\tau=1/3$, $C\tau\leq 1/3$. } 
Hence, 
$3^{-1} g(\theta_{m_{k} } ) \leq g(\theta_{j} ) \leq 3 g(\theta_{m_{k} } )$
for $k\geq k_{0}$, $m_{k}\leq j\leq m_{k+1}$. 

Let $n_{0}=m_{k_{0} }$, while $k(n)=\max\{k\geq 0: m_{k}\leq n \}$, 
$m(n)=m_{k(n) }$ for $n\geq 0$. 
Then, (\ref{ta3.1.43}) implies 
$g(\theta_{n} )\leq 3g(\theta_{m(n) } )$, 
$g(\theta_{m_{k} } )\leq 3g(\theta_{m_{k+1} } )$
for $n\geq n_{0}$, $k\geq k_{0}$
(notice that $k(n)\geq k_{0}$, $m_{k(n) }\leq n<m_{k(n)+1}$ when $n\geq n_{0}$). 
Hence, $g(\theta_{n} )\leq C_{n,k} \:g(\theta_{m_{k} } )$ for $n\geq n_{0}$, $k\geq m(n)$, 
where $C_{n,k}=3^{k-k(n)+1}$.\footnote{
Notice that 
$g(\theta_{n} ) g^{-1}(\theta_{m(n) } ) \leq 3$, 
$g(\theta_{m(n) } ) g^{-1}(\theta_{m_{k} } ) \leq 3^{k-k(n)}$
when $n\geq n_{0}$, $k\geq m(n)$. 
Notice also  
$g(\theta_{n} ) = 
\left( g(\theta_{n} ) g^{-1}(\theta_{m(n) } ) \right)
\left( g(\theta_{m(n) } ) g^{-1}(\theta_{m_{k} } ) \right) g(\theta_{m_{k} } )$. 
} 
Since 
\begin{align*}
	2^{-1} (k(j)-k(n) ) \tau
	\leq 
	\sum_{k=k(n)+1}^{k(j) } \sum_{i=m_{k} }^{m_{k+1}-1} \alpha_{i}
	\leq 
	\sum_{i=n}^{j} \alpha_{i}
	\leq 
	t 
\end{align*}
for $n_{0}\leq n\leq j\leq a(n,\tau )$, 
we conclude $k(j)-k(n)\leq 2t/\tau$ for the same $n,j$. 
Consequently, 
\begin{align*}
	g(\theta_{n} ) 
	\left\|
	\sum_{i=n}^{j} \alpha_{i} \zeta_{i}
	\right\|
	=&
	g(\theta_{n} ) 
	\left\|
	\sum_{k=k(n)}^{k(j) }
	\sum_{i=m_{k} }^{m_{k+1}-1} \alpha_{i}\zeta_{i}
	-
	\sum_{i=m(n)}^{n-1} \alpha_{i}\zeta_{i} 
	+
	\sum_{i=m(j)}^{j} \alpha_{i}\zeta_{i} 
	\right\|
	\\
	\leq&
	\sum_{k=k(n)}^{k(j)-1}
	C_{n,k} \: g(\theta_{m_{k} } ) 
	\left\|
	\sum_{i=m_{k} }^{m_{k+1}-1} \alpha_{i}\zeta_{i}
	\right\|
	+
	C_{n,k(n) } \: g(\theta_{m(n) } ) 
	\left\|
	\sum_{i=m(n)}^{n-1} \alpha_{i}\zeta_{i} 
	\right\|
	\\
	&+
	C_{n,k(j) } \: g(\theta_{m(j) } ) 
	\left\|
	\sum_{i=m(j)}^{j} \alpha_{i}\zeta_{i} 
	\right\|
	\\
	\leq &
	C(t) 
	\max_{\stackrel{\scriptstyle m_{k}\leq l<m_{k+1} }{\scriptstyle k(n)\leq k } } 
	g(\theta_{m_{k} } ) 
	\left\|\sum_{i=m_{k} }^{l} \alpha_{i}\zeta_{i} \right\|
\end{align*}
for $n_{0}\leq n\leq j\leq a(n,t)$,\footnote
{Here, the following convention is used: If the lower limit of a sum is (strictly) greater than 
the upper limit, then the sum is zero. } 
where $C(t)=(2t/\tau+3) 3^{2t/\tau+3}$.
Since $\tau_{m_{k} }>m_{k+1}$ for $k\geq k_{0}$
(i.e., $I_{ \{\tau_{m_{k} }>j\} } = 1$ for $k\geq k_{0}$, $m_{k}\leq j\leq m_{k+1}$), 
(\ref{ta3.1.25}) implies 
\begin{align*}
	\lim_{n\rightarrow\infty } 
	g(\theta_{n} ) 
	\max_{n\leq j<a(n,t) } 
	\left\|\sum_{i=n}^{j} \alpha_{i}\zeta_{i} \right\|
	=0
\end{align*}
(notice that $\lim_{n\rightarrow\infty } k(n) =\infty$). 
Hence, (\ref{aa2.2.1}), (\ref{aa2.2.3})  hold. 
\end{sproof}

In the rest of the section, Theorem \ref{theorem2.1}
is extended to randomly projected stochastic gradient algorithms with Markovian dynamics.
These algorithms are defined by the following difference equations: 
\begin{align}\label{appendix3.1}
	&
	\vartheta_{n}
	=
	\theta_{n}
	-
	\alpha_{n} (F(\theta_{n}, Z_{n+1} ) + \eta_{n} ), 
	\nonumber\\
	&
	\theta_{n+1}
	=
	\vartheta_{n} 
	I_{ \{\|\vartheta_{n} \| \leq \beta_{\sigma_{n} } \} }
	+
	\theta_{0} 
	I_{ \{\|\vartheta_{n} \| > \beta_{\sigma_{n} } \} }, 
	\nonumber\\
	&
	\sigma_{n+1}
	=
	\sigma_{n}
	+
	I_{ \{\|\vartheta_{n} \| > \beta_{\sigma_{n} } \} }, 
	\;\;\;\;\; 
	n\geq 0. 
\end{align}
Here, $H(\cdot,\cdot)$, $\{\alpha_{n} \}_{n\geq 0}$, $\{Z_{n} \}_{n\geq 0}$, $\{\eta_{n} \}_{n\geq 0}$
have the same meaning as in Section \ref{section2}, 
while $\theta_{0}$, $\{\beta_{n} \}_{n\geq 0}$ have the same meaning 
as in the case of recursion (\ref{appendix2.1}). 

To analyze the asymptotic behavior of (\ref{appendix3.1}), we use the following two assumptions. 

\begin{assumptionappendix}\label{aa3.3} 
For any compact set $Q\subset \mathbb{R}^{d_{\theta } }$, 
there exists a Borel-measurable function 
$\varphi_{Q}: \mathbb{R}^{d_{z} } \rightarrow [1,\infty )$ such that 
\begin{align*}
	&
	\max\{
	\|F(\theta,z ) \|, \|\tilde{F}(\theta,z ) \|, \|(\Pi \tilde{F} )(\theta,z ) \|
	\}
	\leq 
	\varphi_{Q}(z ), 
	\\
	&
	\|(\Pi \tilde{F} )(\theta',z ) - (\Pi \tilde{F} )(\theta'',z ) \|
	\leq 
	\varphi_{Q}(z) \|\theta' - \theta'' \| 
\end{align*}
for all 
$\theta, \theta', \theta'' \in Q$, $z \in \mathbb{R}^{d_{z} }$.  
In addition to this, 
\begin{align*} 
	\sup_{n\geq 0}
	E\left(
	\varphi_{Q}^{2}(Z_{n} ) 
	|\theta_{0}=\theta, Z_{0}=z 
	\right)
	< 
	\infty
\end{align*}
for all $\theta \in \mathbb{R}^{d_{\theta } }$, $z \in \mathbb{R}^{d_{z} }$. 
\end{assumptionappendix}

\begin{assumptionappendix}\label{aa3.4} 
$\eta_{n}=\eta(\theta_{n} )$ for $n\geq 0$, 
where $\eta:\mathbb{R}^{d_{\theta } }\rightarrow\mathbb{R}^{d_{\theta } }$ is a continuous function. 
Moreover, 
$\|\eta(\theta ) \|<\|\nabla f(\theta ) \|$ for all $\theta\in\mathbb{R}^{d_{\theta } }$
satisfying $\|\theta \|\geq\rho$ 
($\rho$ is specified in Assumption \ref{aa2.3}). 
\end{assumptionappendix}

In a similar form, Assumptions \ref{aa3.3} and \ref{aa3.4} are involved in the analysis
of randomly projected stochastic approximation carried out in \cite{tadic1}. 

Our result on the asymptotic behavior of algorithm (\ref{appendix3.1}) are provided in the next theorem. 

\begin{theoremappendix}\label{theorema3.2}
Let $\{\theta_{n} \}_{n\geq 0}$ be generated by recursion (\ref{appendix3.1}). 
Suppose that Assumptions \ref{a2.1}, \ref{a2.2}, \ref{aa2.3}, \ref{aa3.3} and \ref{aa3.4} hold. 
Then, all conclusions of Theorem \ref{theorema2.2} are true. 
\end{theoremappendix}

\begin{sproof}
Let $Q\subset\mathbb{R}^{d_{\theta } }$ be any compact set, 
while $t\in(0,\infty )$ is any real number. 
Moreover, let $C_{Q}\in[1,\infty )$ be an upper bound of 
$\|\nabla f(\cdot )\|$, $\|\eta(\cdot ) \|$ on $Q$. 
In order to prove the theorem's assertion, it is sufficient to show that 
(\ref{aa2.4.1}), (\ref{aa2.4.1}) hold almost surely. 

Due to Assumption \ref{aa3.4}, we have 
\begin{align*}
	\limsup_{n\rightarrow\infty } \|\eta_{n} \| I_{ \{\theta_{n}\in Q \} }
	\leq
	C_{Q}
	< \infty, 
	\;\;\;\;\; 
	\limsup_{n\rightarrow\infty } h(\theta_{n} ) \|\eta_{n} \| I_{ \{\theta_{n} \in Q \} }
	\leq
	\delta_{Q}
	< 1
\end{align*}
almost surely ($h(\cdot )$ is specified in Assumption \ref{aa2.4}). 
On the other side, 
Assumptions \ref{a2.1}, \ref{aa3.3} imply 
\begin{align*}
	E_{\theta,z}\left(\sum_{n=0}^{\infty } \alpha_{n}^{2} \varphi_{Q}^{2}(Z_{n+1} ) \right)
	<\infty
\end{align*}
for all $\theta\in\mathbb{R}^{d_{\theta } }$, $z\in\mathbb{R}^{d_{z} }$. 
Assumption \ref{aa3.3} also yields 
\begin{align*}
	&
	\|\zeta_{n} \| I_{ \{\theta_{n}\in Q \} }
	\leq 
	(\|F(\theta_{n}, Z_{n+1} ) \| + \|\nabla f(\theta_{n} ) \| ) I_{ \{\theta_{n}\in Q \} }
	\leq 
	\varphi_{Q}(Z_{n+1} ) + C_{Q} 
	\leq 
	2C_{Q} \varphi_{Q}(Z_{n+1} ) 
\end{align*}
for $n\geq 0$. 
Consequently, 
\begin{align}\label{ta3.2.1}
	\lim_{n\rightarrow\infty } 
	\alpha_{n}\varphi_{Q}(Z_{n+1} )
	=
	\lim_{n\rightarrow\infty } 
	\alpha_{n} \|\zeta_{n} \| I_{ \{\theta_{n}\in Q \} }
	=0
\end{align}
almost surely. 

Let ${\cal F}_{n}=\sigma\{\theta_{0},Z_{0},\dots,\theta_{n},Z_{n} \}$ for $n\geq 0$. 
Owing to Assumption \ref{a2.2}, we have
\begin{align*}
	E_{\theta,z}
	\left(
	\zeta_{1,n} 
	I_{ \{\theta_{n}\in Q \} }
	|
	{\cal F}_{n} 
	\right)
	=
	\left(
	E_{\theta,z}
	(
	\tilde{F}(\theta_{n}, Z_{n+1} )
	|
	{\cal F}_{n} 
	)
	-
	(\Pi\tilde{F} )(\theta_{n},Z_{n} )
	\right) 
	I_{ \{\theta_{n}\in Q \} }
	=
	0
\end{align*}
almost surely for each $\theta\in\mathbb{R}^{d_{\theta } }$, $z\in\mathbb{R}^{d_{z} }$, $n\geq 0$. 
On the other side, Assumption \ref{aa3.3} implies 
\begin{align*}
	\|\zeta_{1,n} \| I_{ \{\theta_{n} \in Q \} }
	\leq 
	\varphi_{Q}(Z_{n} ) + \varphi_{Q}(Z_{n+1} )   
\end{align*}
for $n\geq 0$. 
Combining this with Assumptions \ref{a2.1}, \ref{aa3.3}, we get 
\begin{align*}
	E_{\theta,z}\left(
	\sum_{n=0}^{\infty } \alpha_{n}^{2} \|\zeta_{1,n} \|^{2} I_{ \{\theta_{n}\in Q \} }
	\right)
	\leq
	2E_{\theta,z}\left(
	\sum_{n=0}^{\infty } (\alpha_{n}^{2} + \alpha_{n+1}^{2} ) \varphi_{Q}^{2}(Z_{n+1} )
	\right)
	<\infty
\end{align*}
for all $\theta\in\mathbb{R}^{d_{\theta } }$, $z\in\mathbb{R}^{d_{z} }$. 
Then, using Doob theorem, we conclude that 
$\sum_{n=0}^{\infty } \alpha_{n} \zeta_{1,n} I_{ \{\theta_{n} \in Q \} }$
converges almost surely. 
Since
\begin{align*}
	\left\|\sum_{i=n}^{j}\alpha_{i}\zeta_{1,i} I_{ \{\theta_{i}\in Q \} } \right\| 
	\leq
	\left\|\sum_{i=n}^{j}\alpha_{i}\zeta_{1,i} \right\| I_{ \{\tau_{Q,n}>j \} } 
\end{align*}
for $0\leq n\leq j$ (notice that $\theta_{i}\in Q$ for $n\leq i<\tau_{Q,n}$), 
we deduce  
\begin{align}\label{ta3.2.5}
	\lim_{n\rightarrow\infty }  
	\max_{j\geq n}
	\left\|\sum_{i=n}^{j}\alpha_{i}\zeta_{1,i} \right\| I_{ \{\tau_{Q,n}>j \} } 
	=
	0
\end{align}
almost surely. 

As $\alpha_{n}\alpha_{n+1} = O(\alpha_{n}^{2} )$, 
$\alpha_{n}-\alpha_{n+1} = O(\alpha_{n}^{2} )$ for $n\rightarrow\infty$
(see the proof of Theorem \ref{theorem2.1}), 
Assumptions \ref{a2.1}, \ref{aa3.3} yield
\begin{align*}
	E_{\theta,z}
	\left(
	\sum_{n=0}^{\infty } 
	\alpha_{n}\alpha_{n+1} \varphi_{Q}^{2}(Z_{n+1} )
	\right)
	<\infty, 
	\;\;\;\;\;
	E_{\theta,z}
	\left(
	\sum_{n=0}^{\infty } |\alpha_{n} - \alpha_{n+1} | \varphi_{Q}^{2}(Z_{n+1} ) 
	\right)
	<\infty
\end{align*}
for all $\theta\in\mathbb{R}^{d_{\theta } }$, $z\in\mathbb{R}^{d_{z} }$. 
On the other side, 
owing to Assumptions \ref{aa3.3}, \ref{aa3.4}, we have 
\begin{align*}
	\|\zeta_{2,j} \| I_{ \{\tau_{Q,n}>j \} }
	\leq &
	\varphi_{Q}(Z_{j} ) \|\theta_{j} - \theta_{j-1} \| I_{ \{\tau_{Q,n}>j\} } 
	\\
	\leq &
	\alpha_{j-1} \varphi(Z_{j} ) 
	(\|F(\theta_{j-1},Z_{j} ) \| + \|\eta_{j-1} \| ) I_{ \{\theta_{j-1}\in Q \} } 
	\\
	\leq &
	\alpha_{j-1} \varphi_{Q}(Z_{j} ) 
	(\varphi_{Q}(Z_{j} ) + C_{Q} ) 
	\\
	\leq &
	2C_{Q}\alpha_{j-1} \varphi_{Q}^{2}(Z_{j} ) 
\end{align*}
for $0\leq n<j$ (notice that $\varphi_{Q}(z)\geq 1$ for any $z\in\mathbb{R}^{d_{z} }$). 
We also have 
\begin{align*}
	&
	\|\zeta_{3,n} \| I_{ \{\theta_{n}\in Q \} } 
	\leq 
	\varphi_{Q}(Z_{n+1} )  
	\leq 
	\varphi_{Q}^{2}(Z_{n+1} )
\end{align*}
for $n\geq 0$. 
Thus, 
\begin{align*}
	&
	\left\|
	\sum_{i=n+1}^{j} \alpha_{i} \zeta_{2,i} 
	\right\|
	I_{ \{\tau_{Q,n}>j \} }
	\leq 
	\sum_{i=n+1}^{j} \alpha_{i} \|\zeta_{2,i} \| I_{ \{\tau_{Q,n}>i \} }
	\leq 
	2C_{Q} \sum_{i=n}^{j} \alpha_{i} \alpha_{i+1} \varphi_{Q}^{2}(Z_{i+1} ), 
	\\
	&
	\left\|
	\sum_{i=n+1}^{j} (\alpha_{i} - \alpha_{i+1} ) \zeta_{3,i} 
	\right\|
	I_{ \{\tau_{Q,n}>j \} }
	\leq 
	\sum_{i=n+1}^{j} |\alpha_{i} - \alpha_{i+1} | \: \|\zeta_{3,i} \| I_{ \{\theta_{i}\in Q \} }
	\leq 
	\sum_{i=n+1}^{j} |\alpha_{i} - \alpha_{i+1} | \varphi_{Q}^{2}(Z_{i+1} )
\end{align*}
for $0\leq n<j$. 
Consequently, 
\begin{align}\label{ta3.2.21}
	&
	\lim_{n\rightarrow\infty }
	\max_{j>n} 
	\left\|
	\sum_{i=n+1}^{j} \alpha_{i} \zeta_{2,i} 
	\right\|
	I_{ \{\tau_{Q,n}>j \} }
	=
	\lim_{n\rightarrow\infty } 
	\max_{j>n} 
	\left\|
	\sum_{i=n+1}^{j} (\alpha_{i} - \alpha_{i+1} ) \zeta_{3,i} 
	\right\|
	I_{ \{\tau_{Q,n}>j \} }
	=
	0
\end{align}
almost surely. 
On the other side, (\ref{ta3.2.1}) yields 
\begin{align}\label{ta3.2.23}
	\lim_{n\rightarrow\infty } 
	\alpha_{n+1} \|\zeta_{3,n} \| I_{ \{ \theta_{n}\in Q \} }
	=
	0
\end{align}
almost surely. 

Since $\theta_{i}=\vartheta_{i-1}$ for $n\leq i<\tau_{Q,n}$, Assumption \ref{a2.2} and (\ref{2.3*}) yield 
\begin{align*}
	\left\|
	\sum_{i=n+1}^{j}\! \alpha_{i}\zeta_{i}
	\right\|
	I_{ \{\tau_{Q,n}>j \} }
	=&
	\left\|
	\sum_{i=n+1}^{j}\! 
	\alpha_{i} \zeta_{1,i} 
	+
	\sum_{i=n+1}^{j}\! 
	\alpha_{i} \zeta_{2,i} 
	-
	\sum_{i=n+1}^{j}\!
	(\alpha_{i} - \alpha_{i+1} ) \zeta_{3,i} 
	-
	\alpha_{j+1} \zeta_{3,j} 
	+
	\alpha_{n+1} \zeta_{3,n} 
	\right\|
	I_{ \{\tau_{Q,n}>j \} }
	\\
	\leq &
	\left\|
	\sum_{i=n+1}^{j}\! 
	\alpha_{i} \zeta_{1,i} 
	\right\|
	I_{ \{\tau_{Q,n}>j \} }
	+
	\left\|
	\sum_{i=n+1}^{j}\! 
	\alpha_{i} \zeta_{2,i} 
	\right\|
	I_{ \{\tau_{Q,n}>j \} }
	+
	\left\|
	\sum_{i=n+1}^{j}\! 
	(\alpha_{i} - \alpha_{i+1} ) \zeta_{3,i} 
	\right\|
	I_{ \{\tau_{Q,n}>j \} }
	\\
	&
	+
	\alpha_{j+1} \|\zeta_{3,j} \| I_{ \{\theta_{j}\in Q \} }
	+
	\alpha_{n+1} \|\zeta_{3,n} \| I_{ \{\theta_{n}\in Q \} }
\end{align*}
for $0\leq n<j$. 
Combining this with (\ref{ta3.2.1}) -- (\ref{ta3.2.23}), 
we deduce 
\begin{align*}
	\lim_{n\rightarrow\infty } 
	\max_{n\leq j<a(n,t) }
	\left\|
	\sum_{i=n}^{j} \alpha_{i} \zeta_{i}
	\right\|
	I_{ \{\tau_{Q,}>j \} }
	=
	0
\end{align*}
almost surely. 
Thus, (\ref{aa2.4.1}), (\ref{aa2.4.3}) hold almost surely. 
\end{sproof}

\end{document}